\documentclass[12pt, reqno]{amsart}
\usepackage{mathrsfs}
\usepackage{amsfonts}
\usepackage[nobysame,sorted,compressed-cites,sorted-cites,initials]{amsrefs}
\usepackage[centertags]{amsmath}
\usepackage{amssymb,array}
\usepackage{amsthm,youngtab}
\usepackage{graphicx}
\usepackage{caption}
\usepackage{ytableau}
\usepackage{mathbbol}
\usepackage{diagbox}
\DeclareFontFamily{U}{MnSymbolC}{}
\DeclareFontShape{U}{MnSymbolC}{m}{n}{
  <-6> MnSymbolC5
  <6-7> MnSymbolC6
  <7-8> MnSymbolC7
  <8-9> MnSymbolC8
  <9-10> MnSymbolC9
  <10-12> MnSymbolC10
  <12-> MnSymbolC12}{}
\DeclareSymbolFont{MnSyC}{U}{MnSymbolC}{m}{n}
\DeclareMathSymbol{\smallsquare}{\mathbin}{MnSyC}{'150}
\newcommand{\cybe}[1]{\mathbb{\Lbrack}#1,#1\mathbb{\Rbrack}}
\newcommand{\cyb}[2]{\mathbb{\Lbrack}#1,#2\mathbb{\Rbrack}}
\usepackage{enumitem}
\SetLabelAlign{center}{\hfill #1\hfill}
\usepackage[textwidth=16cm, hmarginratio=1:1]{geometry}
\usepackage{tikz-cd}
\usepackage{rotating}
\usepackage[all,cmtip]{xy}
\usepackage[titletoc, title]{appendix}
\usepackage{amscd}
\usepackage[colorlinks]{hyperref}
\usepackage{float}
\hypersetup{
bookmarksnumbered,
pdfstartview={FitH},
breaklinks=true,
linkcolor=blue,
urlcolor=blue,
citecolor=blue,
bookmarksdepth=2
}
\usepackage[utf8]{inputenc}

\newcommand{\wh}[1]{\widehat{#1}}
\newcommand{\sfR}{\mathsf{R}}
\SetLabelAlign{center}{\hfil#1\hfil}

\allowdisplaybreaks[4]
\usetikzlibrary{matrix,arrows,decorations.pathmorphing}
\newcommand{\stirling}[2]{\genfrac{\{}{\}}{0pt}0{#1}{#2}}
\newcommand{\tstirling}[2]{\genfrac{\{}{\}}{0pt}1{#1}{#2}}

\newcommand{\lact}{\triangleright}
\newtheorem{theorem}{Theorem}[section]
\newtheorem{corollary}[theorem]{Corollary}
\newtheorem{mthm}[theorem]{Main Theorem}
\newtheorem{lemma}[theorem]{Lemma}
\newtheorem{proposition}[theorem]{Proposition}
\theoremstyle{definition}

\newtheorem{remark}[theorem]{Remark}
\newtheorem{example}[theorem]{Example}
\newtheorem{question}[theorem]{Question}

\newtheorem{conjecture}[theorem]{Conjecture}

\newcommand{\ascprod}{\mathop{\overrightarrow{\prod}}}
\newcommand{\dscprod}{\mathop{\overleftarrow{\prod}}}
\newcommand{\ascprodst}{\mathop{\overrightarrow{\prod}^\ast}}
\newcommand{\dscprodst}{\mathop{\overleftarrow{\prod}^\ast}}
\newcommand{\ascprodtens}{\mathop{\overrightarrow{\prod}^\tensor}}

\numberwithin{equation}{section}

\newcommand{\tensor}{\otimes}

\newcommand{\kk}{\mathbb{k}}\newcommand{\lie}[1]{{\mathfrak{#1}}}
\newcommand{\ZZ}{{\mathbb Z}}
\newenvironment{enumalph}{\begin{enumerate}[label={\rm(\alph*)},leftmargin=*]}{\end{enumerate}}
\newenvironment{enumrom}{\begin{enumerate}[label={\rm(\roman*)},leftmargin=*]}{\end{enumerate}}

\DeclareMathOperator*{\Ob}{Ob}
\DeclareMathOperator*{\diag}{diag}
\DeclareMathOperator{\id}{id}

\DeclareMathOperator{\ad}{ad}

\DeclareMathOperator{\Mat}{Mat}

\DeclareMathOperator{\End}{End}

\DeclareMathOperator{\Hom}{Hom}

\DeclareMathOperator{\sign}{sign}
\DeclareMathOperator{\Inv}{Inv}
\DeclareMathOperator{\Inj}{Inj}
\DeclareMathOperator{\Map}{Map}
\DeclareMathOperator{\Part}{Part}

\newcommand{\EE}{{\mathcal{E}}}
\newcommand{\PP}{{\mathcal {P}}}
\newcommand{\sr}[3]{\mathsf r_{#1,#2}^{(#3)}}
\newcommand{\sgna}[1]{\boldsymbol{\epsilon}(#1)}
\begin{document}

\title{Monomial bialgebras}
\author{Arkady Berenstein, Jacob Greenstein, and Jian-Rong Li}
\address{Arkady Berenstein, Department of Mathematics, University of Oregon, Eugene, OR 97403, USA}
\email{arkadiy@math.uoregon.edu}
\address{Jacob Greenstein, Department of Mathematics, University of California, Riverside, CA 92521, USA}
\email{jacobg@ucr.edu}
\address{Jian-Rong Li, Department of Mathematics, Universit\"at Wien,
Vienna, Austria}
\email{lijr07@gmail.com}
\date{}
\thanks{This work was partially supported by the Simons Foundation Collaboration Grant no.~636972 (A.~Berenstein), the Simons
foundation collaboration grant no.~245735 (J.~Greenstein), and Austrian Science Fund (FWF): P 34602, Grant DOI: 10.55776/P34602, and PAT 9039323, Grant-DOI 10.55776/PAT9039323 (J.-R. Li).
}

\begin{abstract}
Starting from 
a single solution of QYBE (or CYBE) we produce an infinite family of 
solutions of QYBE (or CYBE) parametrized by transitive arrays and, in particular, by signed permutations. We are especially interested in
cases when such solutions yield quasi-triangular
structures on direct powers of Lie bialgebras and tensor 
powers of Hopf algebras. We obtain infinite families of such structures as well and study the corresponding
Poisson-Lie structures and co-quasi-triangular algebras.
\end{abstract}

\maketitle 

\tableofcontents

\section{Introduction}
Classical and quantum Yang Baxter equations (CYBE and QYBE, respectively) play a fundamental role in representation theory, low-dimensional topology and mathematical physics,
especially in the theory of integrable
systems and statistical mechanics. This makes the construction of new solutions of CYBE and QYBE quite important. 

In the present work, starting from 
a single solution of CYBE (QYBE), we produce an infinite family of solutions of CYBE (QYBE) parametrized by {\em transitive arrays} and, in particular, by signed permutations. More precisely, we start with a quasi-triangular Lie
bialgebra $\lie g$ with a classical r-matrix~$r\in \lie g\tensor \lie g$.
As customary, let~$U(\lie l)$ be the universal enveloping algebra of a Lie algebra~$\lie l$, and it is convenient to identify~$\lie l$ as the subspace (of all primitive elements)
inside the Hopf algebra~$U(\lie l)$.
Then~$U(\lie g)^{\tensor n}$, $n\ge 1$, is naturally isomorphic to~$U(\lie g^{\oplus n})$. Given~$x=\sum_t a_t\tensor b_t\in \lie g\tensor \lie g$ and~$1\le i<j\le m$, denote
\begin{equation}\label{eq:copy in tensor}
x_{i,j}=\sum_t
1^{\tensor (i-1)}\tensor a_t\tensor 1^{\tensor(j-i-1)}\tensor b_t\tensor 1^{\tensor (m-j)}\in U(\lie g)^{\tensor m}
\end{equation}
and set~$x_{j,i}=(\tau(x))_{i,j}$ where~$\tau$ is the permutation of factors. 
For 
any $n\times n$-matrix $\boldsymbol\epsilon $ with entries in $\{1,-1\}$
define 
$\mathbf r^{(\boldsymbol\epsilon)}\in 
\lie g^{\oplus n}\tensor 
\lie g^{\oplus n}\subset U(\lie g)^{\tensor 2n}$
by 
\begin{equation}\label{eq:r c gen}
\mathbf r^{(\boldsymbol\epsilon)}=\sum_{1\le i,j\le n} r_{j,i+n}^{(\epsilon_{i,j})}
\end{equation}
where $r^{(1)}=r$, $r^{(-1)}=-r_{2,1}$.
A remarkable result~\cite{LuMou}*{Theorem~6.2} can
be restated as follows to justify the notation.
\begin{theorem}[\cite{LuMou}*{Theorem~6.2 and Remark~6.4}]\label{thm:thm 1}
Let $\lie g$
be a quasi-triangular Lie bialgebra with a classical r-matrix~$r$ and let~$n\ge 2$. Then,
for any~$w\in S_n$ and~$\boldsymbol d\in\{1,-1\}^n$,
$\lie g^{\oplus n}$ is a quasi-triangular Lie bialgebra with the classical r-matrix~$\mathbf r^{(\boldsymbol\epsilon(w,\boldsymbol d))}$ where 
\begin{equation}\label{eq:eps w alpha defn}
\boldsymbol\epsilon(w,\boldsymbol d)_{i,j}=\delta_{i,j}d_i+\sign(w(j)-w(i)),
\qquad 1\le i,j\le n.
\end{equation}
\end{theorem}
By~\cite{RSTS}, under the additional assumption that~$r+r_{2,1}\in S^2(\lie g)$ is non-degenerate, 
$\lie g\oplus\lie g$ with
the r-matrix corresponding to~$\sgna{\id,(1,1)}$
is isomorphic to the Drinfeld double of~$\lie g$.
The Lie bracket dual to the Lie cobracket
corresponding to $\mathbf r^{(\sgna{\id,(1,\dots,1)})}$
for arbitrary~$n$ was constructed in~\cite{EK-III}*{Proposition~1.9}.

It turns out that matrices
$\boldsymbol\epsilon(w,\boldsymbol d)$, $w\in S_n$
and~$\boldsymbol d\in\{1,-1\}^n$ have a special combinatorial property. In the spirit of~\cites{ABGJ,ABGJepr}, we say that an $n\times n$-matrix $\boldsymbol a=(a_{i,j})_{1\le i,j\le n}$
with entries in some set~$C$ is {\em transitive} if
$a_{i,k}\in\{a_{i,j},a_{j,k}\}$ for all $1\le i,j,k\le n$. By Lemma~\ref{lem:In trans to trans}, every transitive $n\times n$ matrix~$\boldsymbol{a}$
with entries~$a_{i,j}\in\{1,-1\}$, $1\le i,j\le n$ which is {\em almost skew-symmetric},
that is, $a_{j,i}=-a_{i,j}$ for all~$1\le i<j\le n$,
is equal to~$\boldsymbol{\epsilon}(w,\boldsymbol d)$
for some unique~$w\in S_n$ and~$\boldsymbol d\in\{1,-1\}^n$.
In particular, there are $2^nn!$ of such matrices. 
More generally, transitive $n\times n$ matrices
with entries in~$\{1,-1\}$ naturally identify with 
{\em bitransitive} relations (\cite{FZ}) on the set~$\{1,\dots,n\}$, and their numbers
$\{B_n\}_{n\ge 1}$ are given by the sequence~\href{https://oeis.org/A004123}{A004123} in~\cite{OEIS}, the first few terms being $2$, $10$, $74$, $730$, $9002$, $133210,\dots$ (see~\cites{FZ,HK,Wag} and~\S\ref{subs:pol trans}). By~\cite{FZ}*{Proposition~4}
\begin{equation}\label{eq:Bn asymptotic}
\frac{B_n}{2^n n!}=O(\lambda^n),\qquad n\to\infty,
\end{equation}
where~$\lambda=(2\log(\frac32))^{-1}>\frac65$.
The following is an ultimate justification of our notion of transitivity.
\begin{conjecture}\label{conj:CYBE for trans 1,-1}
Let~$\boldsymbol{\epsilon}$ be a
transitive $n\times n$ matrix with entries in~$\{1,-1\}$
and let~$r\in \lie g\tensor \lie g$
be a solution of CYBE. 
Then $\mathbf r^{(\boldsymbol{\epsilon})}$ solves CYBE.
\end{conjecture}
We verified this conjecture for~$n\le 4$ and, by the above, it holds for all transitive matrices with entries in~$\{1,-1\}$ which are almost skew-symmetric.
We expect that, for a generic~$r$, this exhausts all solutions of CYBE of the form~\eqref{eq:r c gen}. It should be noted that the symmetric group~$S_n$ acts
naturally on~$\lie g^{\oplus n}$, and each r-matrix
provided by Theorem~\ref{thm:thm 1} is equivalent
to~$\mathbf r^{(\boldsymbol{\epsilon}(\id,\boldsymbol d))}$ for some~$\boldsymbol d\in\{1,-1\}^n$, while
bialgebra structures on~$\lie g^{\oplus n}$ given by these r-matrices are all isomorphic. 
This is no longer the case for a general transitive~$\boldsymbol\epsilon$. While the class of r-matrices provided by Conjecture~\ref{conj:CYBE for trans 1,-1} is closed under the action of~$S_n$ on~$\lie g^{\oplus n}$, by~\eqref{eq:Bn asymptotic} the number of $S_n$-orbits grows exponentially faster than~$2^n$.

Theorem~\ref{thm:thm 1} can be recovered as a special case of a much more general Theorem~\ref{thm:main thm trans}, there the role of transitivity, albeit of a slightly different kind, becomes even more prominent. 
Namely, generalizing the basic family~$\{r^{(1)},r^{(-1)}\}=
\{r,-r_{2,1}\}$, let
$\{r^{(c)}\}_{c\in C}\subset \lie g\tensor\lie g$ be a 
family of r-matrices
corresponding to the {\em same} Lie cobracket on~$\lie g$. By Proposition~\ref{prop:color CYBE}, such a family solves 
the {\em transitive CYBE} in~$U(\lie g)^{\tensor 3}$, which appears to be new, namely
\begin{equation}\label{eq:trans CYBE}
[r_{1,2}^{(c)},r_{1,3}^{(c')}]+[r_{1,2}^{(c)},r_{2,3}^{(c'')}]+[r_{1,3}^{(c')},
r_{2,3}^{(c'')}]=0
\end{equation}
whenever~$c'\in\{c,c''\}\subset C$. The ordinary CYBE corresponds to the
case when~$c=c'=c''$. 

Given an upper triangular $n$-array $\mathbf c=(c_{i,j})_{1\le i<j\le n}$ with entries in a set~$C$,
$\boldsymbol d=(d_1,\dots,d_n)\in C^n$ and $\mathbf r=\{r^{(c)}\}_{c\in C}\subset 
\lie g\tensor \lie g$, define
\begin{equation}\label{eq:r c d}
\mathbf r(\mathbf c,\boldsymbol d):=\sum_{1\le i\le n}(r^{(d_{i})})_{i,i+n}
+\sum_{1\le i<j\le n} (r^{(c_{i,j})})_{j,i+n}-(r^{(c_{i,j})})_{j+n,i}.
\end{equation}
We say that an upper triangular $n$-array $\mathbf c=(c_{i,j})_{1\le i<j\le n}$ is {\em transitive} (cf.~\cites{ABGJ,ABGJepr}) 
if $c_{i,k}\in\{c_{i,j},c_{j,k}\}$ for all~$1\le i<j<k\le n$. Our main result in the classical case is
\begin{mthm}\label{thm:main thm trans}
Let~$\lie g$ be a quasi-triangular Lie bialgebra with a family~$\mathbf r=\{r^{(c)}\}_{c\in C}$ of r-matrices corresponding to the same Lie cobracket.
Then for 
any $\boldsymbol d\in C^n$ and
any transitive $n$-array $\mathbf c$ with entries in~$C$,
$\lie g^{\oplus n}$ is a quasi-triangular Lie bialgebra with the $r$-matrix~$\mathbf r(\mathbf c,
\boldsymbol d)$, that is,
$\tilde\delta_{\mathbf c}:\lie g^{\oplus n}\to \lie g^{\oplus n}\tensor \lie g^{\oplus n}$ defined by 
$$\tilde\delta_{\mathbf c}(x)=[\mathbf r(\mathbf c,\boldsymbol{d}),
1\tensor x+x\tensor 1],\qquad x\in\lie g^{\oplus n}
$$
is a Lie cobracket and~$\mathbf r(\mathbf c,\boldsymbol d)$ solves the CYBE.
\end{mthm}
It should be noted that $\tilde\delta_{\mathbf c}$,
unlike~$\mathbf r(\mathbf c,\boldsymbol d)$, does not 
depend on~$\boldsymbol{d}$, hence the notation.
We prove this Theorem in~\S\ref{subs:pf m thm 1} by constructing, for each transitive
$n$-array~$\mathbf c$, a {\em classical Drinfeld twist}~$j_{\mathbf c}(\mathbf r)$ which deforms the natural Lie bialgebra structure 
on~$\lie g^{\oplus n}$ to~$\tilde\delta_{\mathbf c}$. Classical twists were defined by
Drinfeld in the foundational work~\cite{Drinf} and studied in~\cites{AM,EH,Halb,KPSST,LuMou}
to name but a few. We proceed by induction on~$n$,
and the main ingredient in the inductive step is 
a generalization of classical Drinfeld twist which
we refer to as a {\em relative classical twist} for a pair
of Lie bialgebras (see~\S\ref{subs:rel clas twist}).
Another important tool is 
a ``transitive version'' $\lie{qtr}_n(C)$, which we introduce in~\S\ref{subs:Cqtr}, of the Lie algebra~$\lie{qtr}_n$ defined in~\cite{BEER}. 

As shown in~\cite{ABGJ} (cf.~Lemma~\ref{lem:trans sign} for the precise statement), transitive $n$-arrays 
with entries in~$\{1,-1\}$
are 
in a natural bijection with~$S_n$
via 
\begin{equation}\label{eq:sgna defn}w\mapsto\sgna{w}:=(\sign(w(j)-w(i)))_{1\le i<j\le n},\quad w\in S_n.
\end{equation}
Theorem~\ref{thm:thm 1} is a special case of
Theorem~\ref{thm:main thm trans} for~$\mathbf r=\{r^{(1)},r^{(-1)}\}=\{r,-r_{2,1}\}$ and~$\mathbf c=\boldsymbol{\epsilon}(w)$ (see Remark~\ref{rem:pf thm 1.1}). 

If~$C$ is finite, the number of transitive $n$-arrays with entries in~$C$
equals~$p_n(|C|)$ where
\begin{equation}\label{eq:p_n pochhammer}
p_n(x)=\sum_{1\le k\le n-1} 
p_{n,k}\, x(x-1)\cdots (x-k+1)
\end{equation}
with~$p_{n,k}\in\ZZ_{>0}$, $1\le k\le n-1$ (\cite{ABGJepr}*{Proposition~2.13}, see also Corollary~\ref{cor:enum transitive}). In particular, 
$p_{n,1}=1$, $p_{n,2}=\frac12 n!-1$, whence~$p_n(2)=n!$, and, by~\cite{ABGJ}*{Theorem~2.17}, $p_{n,n-1}$ is the $(n-1)$th Catalan
number. The coefficient~$p_{n,n-2}$ is given in Conjecture~\ref{conj:K n n-2}. 
We discuss this polynomiality phenomenon in much greater generality in~\S\ref{subs:polynom}.

To take the full advantage of Theorem~\ref{thm:main thm trans}, we pose the following 
\begin{question}\label{prob:family}
Which quasi-triangular Lie bialgebras admit non-trivial families  of classical r-matrices $\{r^{(c)}\}_{c\in C}$ with~$|C|>2$?
\end{question}
We answer this question when the underlying Lie algebra is of the form~$\lie t=\lie t(V,\lie g):=V\rtimes \lie g$, where~$\lie g$ is a Lie algebra and $V$ is
(a quotient of) the adjoint~$\lie g$-module~$\lie g_{\ad}$. Lie algebras of this type appeared in various contexts (see e.g.~\cites{CG,AN,Chaf,ChKR,GM,Pan,Tak}). Namely, suppose that~$\lie g$ is quasi-triangular with an r-matrix~$r$ and fix a surjective homomorphism of~$\lie g$-modules $f:\lie g_{\ad}\to V$.
We prove (Proposition~\ref{prop:takiff}) that
$\lie t$ is quasi-triangular with the r-matrix
 $\wh r:=(f\tensor\id_{\lie g}+\id_{\lie g}\tensor f)(r)$, which seems to be new. Finally, let~$\Omega$ be a $\lie g$-invariant in $V\tensor V$ (for example, if~$\lie g$ is simple then~$\Omega$ is the image of the Casimir element in~$\lie g\tensor\lie g$) and
 suppose that $r$ is skew-symmetric. Then~$\{\wh r^{(c)}\}_{c\in \kk}$ where~$\wh r^{(c)}=\wh r+c\,\Omega$
 provides an answer to Question~\ref{prob:family}
(Proposition~\ref{prop:takiff}).

One consequence of our construction is the following observation. The diagonal embedding of Lie algebras $\lie g\hookrightarrow \lie g^{\oplus n}$ is not, generally speaking, a homomorphism of Lie bialgebras for the direct sum bialgebra structure. However, this Nature's mistake 
is rectified if~$\lie g^{\oplus n}$ is regarded as 
a Lie bialgebra with the cobracket $\tilde\delta_{\mathbf c}$ (Theorem~\ref{thm:diag embed bialg}).

The following is a natural generalization of Conjecture~\ref{conj:CYBE for trans 1,-1}, also  verified for~$n\le 4$.
\begin{conjecture}\label{conj:CYBE for trans}
Let~$\boldsymbol{a}$ be a
transitive $n\times n$ matrix with entries in~$C$
and let~$\mathbf r=\{r^{(c)}\}_{c\in C}\subset \lie g\tensor \lie g$
be any solution of the transitive CYBE~\eqref{eq:trans CYBE}. Then
     $\mathbf r^{(\boldsymbol{a})}$ solves CYBE.
\end{conjecture}
Once~$\mathbf r^{(\boldsymbol{a})}$
satisfies CYBE, the necessary and sufficient 
condition for it to be an r-matrix for~$\lie g^{\oplus n}$ is provided by Lemma~\ref{lem:quasi invariant}
and is reminiscent the almost skew-symmetry discussed above.

It is curious that the number of transitive~$n\times n$ matrices with entries
in a finite~$C$ is again a polynomial~$q_n$ in~$|C|$ of degree~$n$ which is dramatically easier to compute than~$p_n$ from~\eqref{eq:p_n pochhammer}. Namely,
$q_n(x)=2x^n-x+(\frac12 B_n-2^n+1)x(x-1)$ 
(see Proposition~\ref{prop:enum trans} and
Corollary~\ref{cor:poly q_n}; in particular, $q_n(2)=B_n$ as expected).
In general, under a very mild assumption on~$C$ and the family~$\mathbf r$, 
$\mathbf r(\mathbf c,\boldsymbol d)$ can be
written as~$\mathbf r^{(\widehat{\mathbf c})}$
where the diagonal of~$\widehat{\mathbf c}$ is~$\boldsymbol{d}$ and the off-diagonal part is
obtained by a certain ``skew-symmetrization'' of~$\mathbf c$. 
Yet~$\widehat{\mathbf c}$ so obtained is not
a transitive $n\times n$-matrix and satisfies a rather different
combinatorial condition.

We conclude the discussion of the ``classical story'' with applications to Poisson geometry. Let~$G$
be an algebraic Poisson-Lie group whose Lie algebra is 
the Lie bialgebra~$\lie g$. Then its coordinate algebra~$\kk[G]$ is Poisson
with the Poisson bracket induced by
the cobracket on~$\lie g$ (see~\S\ref{subs:Poisson alg} for details). In particular, any Lie bialgebra structure on~$\lie g^{\oplus n}$ yields a Poisson algebra structure on~$\kk[G]^{\tensor n}\cong \kk[G^{\times n}]$.
\begin{theorem}\label{thm:Poisson mult}
Let~$\mathbf r=\{r^{(c)}\}_{c\in C}$ be 
a family of r-matrices for a Lie bialgebra~$\lie g$, let~$\mathbf c$ be a transitive $n$-array with entries in~$C$, $n\in\ZZ_{>1}$, and let~$G$ be an algebraic Poisson-Lie group with the Lie algebra~$\lie g$. Then
the diagonal embedding $G\hookrightarrow G^{\times n}$ is Poisson, that is the multiplication~$\kk[G]^{\tensor n}\to\kk[G]$ 
is a homomorphism of Poisson algebras, where 
the Poisson structure on~$\kk[G]^{\tensor n}$ is induced by the cobracket~$\tilde\delta_{\mathbf c}$ on~$\lie g^{\oplus n}$ corresponding to~$\mathbf r$ and~$\mathbf c$ in the notation of Theorem~\ref{thm:main thm trans}.
\end{theorem}
We provide a proof of Theorem~\ref{thm:Poisson mult},
along with an explicit formula for the corresponding Poisson bracket
$\{\cdot,\cdot\}_{\mathbf c}$ on~$\kk[G]^{\tensor n}$,
in~\S\ref{subs:Poisson}.
As a first example (which we did not find in the literature), taking $G$ to be the group~$GL_m$ or even the monoid~$\Mat_m$ and starting from the standard Poisson structure~\eqref{eq:Am Pois bracket} on~$\kk[G]$, yields the Poisson structure~\eqref{eq:Am,n Pois bracket} on
$\kk[G]^{\tensor n}$.

For~$\mathbf r=
\{r^{(1)},r^{(-1)}\}=
\{r,-r_{21}\}$ where~$r$ is an r-matrix for~$\lie g$, Poisson structures on~$\kk[G]^{\tensor n}$ 
corresponding to
different transitive $n$-arrays are naturally isomorphic
(see Remark~\ref{rem:Pois perms}). However, this is no longer the case if~$|C|>2$ 
(moreover, in the quantum case even when $C=\{1,-1\}$ the corresponding quantum algebras need not be isomorphic, see Proposition~\ref{prop:perm isom}
and~\S\ref{subs:twisted tens quantum matrices}). An example is provided by
Poisson algebra structures on tensor powers of
the coordinate algebra~$\kk[V\rtimes G]$ of the 
algebraic group~$V\rtimes G$, where~$V$ is (a quotient of) the adjoint $\lie g$-module. 
As an algebra, $\kk[V\rtimes G]$ identifies with~$\kk[V]\tensor \kk[G]$ where~$\kk[G]$ is 
a Poisson-commutative Poisson ideal, while~$\kk[V]\cong S(V^*)$ is a Poisson subalgebra whose Poisson bracket extends the (non-abelian) Lie bracket
on~$V^*$ induced by the Lie cobracket
on~$\lie t=V\rtimes \lie g$ (see~\S\ref{subs:Poiss Takiff}). It is quite involved even for~$G=GL_m$ and~$r\in\lie g\tensor\lie g$ belonging to the skew-symmetric family~\eqref{eq:skew symm r matr GLm} which extends the
well-known skew-symmetric solution of CYBE for~$\lie{gl}_2$
(cf.~\eqref{eq:Lie algebra V*}). Since~$\lie t$,
as discussed above, admits an infinite family of non-equivalent r-matrices,
our construction yields a family~\eqref{eq:Poiss Takiff tens} of 
Poisson brackets on~$\kk[V\rtimes G]^{\tensor n}
\cong\kk[(V\rtimes G)^{\times n}]$ parametrized
by transitive $n$-arrays with entries in~$\kk$ which appear to be non-isomorphic
for~$|C|>2$.
Quite remarkably, for~$n>1$ the natural image of~$\kk[V]^{\tensor n}$ in~$\kk[V\rtimes G]^{\tensor n}$
is no longer a Poisson subalgebra, while~$\kk[G]^{\tensor n}$
remains a Poisson-commutative Poisson ideal. 

We will now discuss the quantum case, which turns out
to be even more spectacular. 
Let~$H$ be an associative algebra.
Given $w\in S_n$
and an invertible $R$ in (a suitable completion $H\widehat\tensor H$, see~\S\ref{subs:compl}, of) $H\tensor H$, define 
$J_w=J_w(R)\in H^{\widehat\tensor 2n}$ by 
$$
J_w=(R_{2,n+1}^{(\sgna w_{1,2})})
(R_{3,n+2}^{(\sgna w_{2,3})}R_{3,n+1}^{(\sgna w_{1,3})})\cdots 
(R_{n,2n-1}^{(\sgna w_{n-1,n})}\cdots R_{n,n+1}^{(\sgna w_{1,n})})
$$
where~$\sgna w$ is defined by~\eqref{eq:sgna defn} and
$$
R_{i,j}^{(\epsilon)}=\begin{cases}
R_{i,j},&\epsilon=1,\\
R_{j,i}^{-1},&\epsilon=-1.
\end{cases}
$$
The notation~$R_{i,j}$ has the same meaning as in~\eqref{eq:copy in tensor} with~$U(\lie g)$ now replaced by~$H$.

Henceforth, let~$H$ be a bialgebra with the comultiplication $\Delta:H\to H\tensor H$. Then~$H^{\tensor n}$ is naturally a 
bialgebra with the comultiplication 
$\Delta_{H^{\tensor n}}:H^{\tensor n}\to H^{\tensor n}\tensor H^{\tensor n}$ determined by 
$$
\Delta_{H^{\tensor n}}(1^{\tensor(i-1)}\tensor h
\tensor 1^{\tensor (n-i)})=\Delta(h)_{i,n+i},\qquad h\in H,\quad 
1\le i\le n.
$$
In the sequel, we consider quasi-triangular bialgebras, rather than Hopf algebras, as we do not need the antipode for our constructions.
The following is our first main result in the quantum case.
\begin{mthm}\label{thm:main thm 2}
Let $H$ be a quasi-triangular bialgebra 
with a (universal) R-matrix~$R$ and
let~$w\in S_n$. Then
\begin{enumalph}
\item\label{thm:main thm 2.a}
$J_w=J_w(R)$ is a Drinfeld twist for~$\Delta_{H^{\tensor n}}$, hence
$\Delta_w:H^{\tensor n}\to H^{\widehat\tensor 2n}$ defined by
$\Delta_w(h)=J_w^{-1}\Delta_{H^{\tensor n}}(h)J_w$, $h\in H^{\tensor n}$
equips 
 $H^{\tensor n}$ with a structure 
 of a (topological) bialgebra~$H^{\tensor n,w}$;
 \item\label{thm:main thm 2.b}
 For every~$w\in S_n$, $\boldsymbol{d}=(d_1,\dots,d_n)\in\{1,-1\}^n$, 
 $$\mathbf R(\sgna{w},\boldsymbol d):=(J_w(R)^{op})^{-1} R_{1,n+1}^{(d_1)}\cdots R_{n,2n}^{(d_n)}\, J_w(R)
$$
is an R-matrix for $H^{\tensor n,w}$, where for
$X=X_1\tensor\cdots\tensor X_{2n}\in H^{\tensor 2n}$
in Sweedler-like notation we abbreviate $X^{op}:=X_{n+1}\tensor \cdots\tensor X_{2n}
\tensor X_1\tensor\cdots\tensor X_n
$.
\end{enumalph}
 \end{mthm}
This result was inspired by~\cite{RSTS}*{Theorem~2.9}
for~$n=2$ and its generalization~\cite{Mou}*{\S6.2}
for~$w=\id\in S_n$.
Similarly to the classical setup, we prove Theorem~\ref{thm:main thm 2} in~\S\ref{subs:fam Drinf twist} as a special case of a more general result (Theorem~\ref{thm:trans twist}).
This result is highly non-trivial,  since, unlike
the classical situation (Theorem~\ref{thm:thm 1}),
bialgebras corresponding to different $w\in S_n$
do not have to be isomorphic (see for example~\S\ref{subs:small quantum}). In particular, unlike
in the classical case,
there is no obvious connection between new classes of solutions of QYBE
provided by Theorem~\ref{thm:main thm 2}\ref{thm:main thm 2.b} for various~$w\in S_n$.
Furthermore, similarly to Conjecture~\ref{conj:CYBE for trans 1,-1}, we expect a new large class of solution of QYBE as a non-commutative manifestation of transitivity.
\begin{conjecture}\label{conj:QCYBE trans 1,-1}
Let~$R$ be any solution of QYBE. Then
for any transitive $n\times n$-matrix $\boldsymbol{\epsilon}=(\epsilon_{i,j})_{1\le i,j\le n}$ with~$\epsilon_{i,j}\in\{1,-1\}$, $1\le i,j\le n$,
\begin{equation}\label{eq:R eps defn}
\mathbf R^{(\boldsymbol{\epsilon})}:=
(R_{1,2n}^{(\epsilon_{n,1})}\cdots 
R_{n,2n}^{(\epsilon_{n,n})}) 
(R_{1,2n-1}^{(\epsilon_{n-1,1})}\cdots 
R_{n,2n-1}^{(\epsilon_{n-1,n})})
\cdots (R_{1,n+1}^{(\epsilon_{1,1})}\cdots R_{n,n+1}^{(\epsilon_{1,n})})
\end{equation}
solves QYBE.
\end{conjecture}
One can show (Proposition~\ref{prop:Spec trans conj})
that if $\boldsymbol\epsilon$ is transitive almost skew-symmetric
then $\mathbf R^{(\boldsymbol{\epsilon})}=\mathbf R(\sgna w,\boldsymbol{d})$ for some~$w\in S_n$ and
$\boldsymbol{d}\in\{1,-1\}^n$, which in particular
verifies Conjecture~\ref{conj:QCYBE trans 1,-1} in that case. We also verified it for~$n\le 4$.

Another difference between the classical and the quantum situations is that while in the former
a twist only affects the cobracket but not the comultiplication on~$U(\lie g)$, in the latter new comultiplications heavily depend on a choice of a {\em 
Drinfeld twist}. These were first introduced by Drinfeld in~\cite{Drinf} and have been extensively
studied, in particular in connection with lattice models in statistical physics and as a tool
for turning cocommutative bialgebras into non-commutative ones (see~\cites{CPbook,BB,Ter,BoW,MS,Tor,Majid,Ne} to name but a few).

As in the classical world, suppose that our bialgebra~$H$
admits a family of R-matrices $\mathbf R=\{R^{(c)}\}_{c\in C}$
for the {\em same} comultiplication.
By Proposition~\ref{prop:cQYBE}, such a family solves the {\em transitive QYBE}
\begin{equation}\label{eq:trans QYBE}
R_{1,2}^{(c)}R_{1,3}^{(c')}R_{2,3}^{(c'')}
=R_{2,3}^{(c'')}R_{1,3}^{(c')}R_{1,2}^{(c)}
\end{equation}
whenever~$c'\in\{c,c''\}\subset C$, which
is the multiplicative analogue
of our transitive CYBE \eqref{eq:trans CYBE}.
Our second quantum main result generalizes Theorems~\ref{thm:main thm trans}
and~\ref{thm:main thm 2}.
\begin{mthm}\label{thm:trans twist}
Let~$H$ be a bialgebra with a family
$\mathbf R=\{R^{(c)}\}_{c\in C}$ of (universal) R-matrices.
Then for any transitive $n$-array $\mathbf c=(c_{i,j})_{1\le i<j\le n}$
$$
J_{\mathbf c}=J_{\mathbf c}(\mathbf R):=R_{2,n+1}^{(c_{1,2})}
(R_{3,n+2}^{(c_{2,3})}R_{3,n+1}^{(c_{1,3})})
\cdots (R_{n,2n-1}^{(c_{n-1,n})}\cdots 
R_{n,n+1}^{(c_{1,n})})
$$
is a Drinfeld twist for~$H^{\tensor n}$ with its
standard comultiplication~$\Delta_{H^{\tensor n}}$. 
In particular, $$
\mathbf R(\mathbf c,\boldsymbol{d}):=
(J_{\mathbf c}^{op})^{-1} R_{1,n+1}^{(d_1)}
\cdots R_{n,2n}^{(d_n)}\,J_{\mathbf c}
$$
is an~R-matrix for~$H^{\tensor n}$ with 
the comultiplication twisted by~$J_{\mathbf c}$
and thus solves QYBE.
\end{mthm}
We prove Theorem~\ref{thm:trans twist} in~\S\ref{subs:fam Drinf twist}
using the notion of a {\em relative Drinfeld twist}
for a pair of bialgebras (see~\S\ref{subs:rel twist})
and a transitive generalization~$\mathsf{QTr}_n(C)$
of the group~$\mathsf{QTr}_n$ defined in~\cite{BEER}
(see~\S\ref{subs:QTR(C)} for the details).

Furthermore, starting with any family~$\mathbf R=\{R^{(c)}\}_{c\in C}\subset H\wh\tensor H$ and an $n\times n$ matrix~$\boldsymbol{a}=(a_{i,j})_{1\le i,j\le n}$
define 
$$
\mathbf R^{(\boldsymbol{a})}=
(R_{1,2n}^{(a_{n,1})}\cdots 
R_{n,2n}^{(a_{n,n})}) 
(R_{1,2n-1}^{(a_{n-1,1})}\cdots 
R_{n,2n-1}^{(a_{n-1,n})})
\cdots (R_{1,n+1}^{(a_{1,1})}\cdots R_{n,n+1}^{(a_{1,n})})
$$
which is the quantum counterpart of~\eqref{eq:r c gen}.
Mirroring Conjecture~\ref{conj:CYBE for trans}, we formulate 
\begin{conjecture}\label{conj:QYBE for trans}
Given a solution~$\mathbf R=\{R^{(c)}\}_{c\in C}\in H\wh\tensor H$ of the transitive QYBE \eqref{eq:trans QYBE}, 
$\mathbf R^{(\boldsymbol{a})}$ solves QYBE for any transitive $n\times n$-matrix $\boldsymbol{a}$ with 
entries in~$C$.
\end{conjecture}
This conjecture was verified for~$n\le 4$. It should
be noted that for~$|C|>2$ solutions $\mathbf R(\mathbf c,\boldsymbol{d})$ of QYBE provided by Theorem~\ref{thm:trans twist} are not of the form~$\mathbf R^{(\boldsymbol{a})}$ for a transitive
matrix~$\boldsymbol a$, and it would be interesting to
find a class of solutions of QYBE encompassing these two.

Similarly to Question~\ref{prob:family}, it is only natural to raise the following
\begin{question}\label{prob:family quantum}
Which bialgebras~$H$ admit families 
of~R-matrices $\{R^{(c)}\}_{c\in C}$?
\end{question}
Any quantization (see, for example, \cites{EH,EK-I,EK-II,EK-III,Drinf-YBE}) of an answer to Question~\ref{prob:family} answers Question~\ref{prob:family quantum}. However, there are also finite-dimensional examples (see~\S\ref{subs:small quantum})
which are not directly obtainable via quantization. 
Yet even in finite-dimensional cases twisting the 
comultiplication can have very non-trivial 
consequences, as e.g. in~\cite{BB}.

We finish the discussion of the quantum story with the dual picture which is even more attractive since it does not require completions. Namely, 
for a family of co-quasi-triangular structures
$\{\mathcal R^{(c)}\}_{c\in C}\subset \Hom_\kk(A^{\tensor 2},\kk)$ on a bialgebra~$A$ and 
a transitive $n$-array~$\mathbf c$, we construct 
(Theorem~\ref{thm:dual Jc twist})
a {\em dual Drinfeld twist} $\mathcal J_{\mathbf c}\in 
\Hom_\kk(A^{\tensor 2n},\kk)$ (see~\S\ref{subs:quant dual}
and~\S\ref{subs:dual Jc}
for the details).
This yields a family of associative multiplications on~$A^{\tensor n}$ with the same coalgebra structure. 
In addition to avoiding completions, 
the twisted multiplication can often be expressed in
a more compact form (Proposition~\ref{prop:Jc twisted mult}).
In particular, for~$n=2$ this recovers the multiplication
introduced by Takeuchi (\cite{Take}*{Section~8})
and studied by Majid for bialgebras (see e.g.~\cite{Majid}). It turns out that the (iterated) multiplication
$A^{\tensor n}\to A$ becomes a homomorphism of
bialgebras for the twisted algebra structure on~$A^{\tensor n}$ (Theorem~\ref{thm:mult is homomorphism}). Once again, in the inductive argument
we utilize a {\em relative dual Drinfeld twist} (see~\S\ref{subs:quant dual}). As an example,
we obtain a family of twisted algebra structures
parametrized
by permutations on tensor powers of ``quantum matrices''  (\S\ref{subs:twisted tens quantum matrices}), which, unlike their $q=1$ limits, are 
no longer isomorphic under permutations of factors.
We expect that they are pairwise non-isomorphic.
Of course, they become isomorphic after passing to the completion, but to construct such an isomorphism is a serious challenge. Perhaps, an even stronger 
challenge is to find all {\em rational} quantizations of 
a given classical object.
Moreover, the above ``homomorphism assertion'' generalizes to the situation when the natural map $A\tensor B\to C$ for any sub-bialgebras $A$, $B$ of a bialgebra~$C$
becomes an algebra homomorphism when the algebra
structure on~$A\tensor B$ is twisted in an appropriate way
(Proposition~\ref{prop:twist hom alg}). 
Based on the above and other observations, we expect these algebras to be of interest and, in particular, to admit quantum cluster-like structures,
which we plan to explore. 

\subsection*{Acknowledgments}
The main part of this work was carried out while
the first two authors were visiting University of Vienna, Austria, whose hospitality is gratefully acknowledged.
We are indebted to Pavel Etingof and Milen Yakimov for  important references and other helpful suggestions.
We also thank Dennis Nguyen and Shane Rankin for stimulating discussions. 

\section{Preliminaries}

\subsection{General notation}
All vector spaces, algebras and coalgebras are over some fixed base field~$\kk$ of characteristic zero. The symbol~$\tensor$, when unadorned, stands for the tensor product over~$\kk$.

Given $a,b\in \mathbb Z$, denote~$[a,b]:=\{x\in\ZZ\,:\,a\le x\le b\}$. We abbreviate
$[n]:=[1,n]$ for~$n\in\ZZ_{>0}$.
Define~$\Upsilon:\mathbb R\to \{0,1\}$ by~$\Upsilon(x)=1$
if~$x>0$ and~$\Upsilon(x)=0$ if~$x\le 0$, i.e.
\begin{equation}\label{eq:Ups formula}
\Upsilon(x)=\tfrac12(1+\sign(x)-\delta_{x,0}),\qquad x\in\mathbb R. 
\end{equation}

Let~$\mathsf M$ be a multiplicative monoid. 
Given any finite subset~$I\subset \ZZ$ and
a family $X_i$, $i\in I$ of elements of~$\mathsf M$ we set
$$
\ascprod_{i\in I} X_i=X_{i_1}\cdots X_{i_r},\qquad
\dscprod_{i\in I} X_i=X_{i_r}\cdots X_{i_1}.
$$
where $I=\{i_1,\dots,i_r\}$ with~$i_1<\cdots<i_r$. 

Given any collection~$V_1,\dots,V_n$, $n\ge2$ of vector spaces
and a permutation~$\sigma\in S_n$, define $\widehat\sigma:
\in\Hom_\kk(V_1\tensor\cdots\tensor V_n, V_{\sigma(1)}\tensor\cdots\tensor V_{\sigma(n)})$
by~$\widehat\sigma(v_1\tensor \cdots\tensor v_n)=
v_{\sigma(1)}\tensor\cdots\tensor v_{\sigma(n)}$, $v_i\in V_i$, $1\le i\le n$.
For~$\sigma=(i,j)$, $1\le i<j\le n$, we denote this map by~$\tau_{i,j}$ and often abbreviate~$\tau_{1,2}$ as~$\tau$. 
It is easy to see that~$\widehat{\sigma\rho}=\widehat{\rho}\circ\widehat{\sigma}$
in $\Hom_\kk(V_1\tensor\cdots\tensor V_n,V_{\sigma\rho(1)}\tensor\cdots\tensor V_{\sigma\rho(n)})$, $\sigma,\rho\in S_n$.

Let~$B$ be a unital algebra. Given~$N\ge n\in\ZZ_{>0}$ and~$\mathbf i=(i_1,\dots,i_n)$ with
$1\le i_r\le N$, $1\le r\le n$, 
and $i_r\not=i_s$, $1\le r<s\le n$, we denote by~$\phi_{\mathbf i}:B^{\tensor n}\to B^{\tensor N}$ the unique homomorphism of algebras satisfying 
$$
\phi_{\mathbf i}(1_B^{\tensor (r-1)}\tensor b\tensor 
1_B^{\tensor (n-r)})=1_B^{\tensor (i_r-1)}\tensor b
\tensor 1_B^{\tensor (N-i_r)},\qquad b\in B,\,1\le r\le n.
$$
We will often abbreviate $X_{\mathbf i}:=\phi_{\mathbf i}(X)$, $X\in B^{\tensor n}$.
For example, for~$X=x_1\tensor x_2\tensor x_3\in B^{\tensor 3}$ in Sweedler-like notation, 
$$
X_{2,6,4}=1_B\tensor x_1\tensor 1_B\tensor x_3 \tensor 1_B\tensor x_2\in B^{\tensor 6}.
$$
When we write a sequence of indices as a set,
we assume that it is ordered increasingly; for example, $[1,2n]\setminus\{n\}$ stands for the sequence $(1,\dots,n-1,n+1,\dots,2n)$. 
Note that if $N\ge \max(k,l)$, $X\in B^{\tensor k}$, $Y\in B^{\tensor l}$ then, for any disjoint sequences $\mathbf i=(i_1,\dots,i_k)$
and~$\mathbf j=(j_1,\dots,j_l)$
with $1\le i_r,j_s\le N$, 
$X_{\mathbf i}$ and~$Y_{\mathbf j}$
commute in~$B^{\tensor N}$.

We will also use similar notation for tensor products of different (bi)algebras.

For any associative algebra~$A$, we denote its center by~$Z(A)$. If~$B$ is an associative algebra and
$\varphi:A\to B$ is a homomorphism of algebras,
we say that~$\partial\in\Hom_\kk(A,B)$ is a $\varphi$-{\em derivation} if $\partial(aa')=\partial(a)\varphi(a')+\varphi(a)\partial(a')$
for all~$a,a'\in A$.

\subsection{Completions}\label{subs:compl}
Let~$A$ be an algebra and let~$K_A$ be an ideal in~$A$. We say
that~$(A,K_A)$ is a {\em locally finite pair} 
if
\begin{equation}\label{eq:cnd loc finite}
\dim(A/K_A^{r})<\infty,\qquad r\ge 1.\end{equation}
The following well-known fact provides 
a large class of locally finite pairs.
\begin{lemma}\label{lem:counit locally finite}
Let~$A$ be finitely generated and let~$K_A\subset A$
be an ideal of codimension~$1$. Then~$(A,K_A)$ is 
a locally finite pair. 
Moreover, the dimension of
$A/K_A^r$, $r\ge 0$ is bounded above by~$r N^{r-1}$ where~$N$
is the number of generators of~$A$.
\end{lemma}
\begin{proof}
Let~$V\subset K_A$ be a finite-dimensional subspace generating~$A$. Let~$K=\bigoplus_{r\ge 1}V^{\tensor r}$ be the augmentation ideal of the tensor algebra~$T(V)$ of~$V$
and let~$I$ be the kernel of the
canonical projection~$T(V)\to A$. Then~$K_A$ identifies
with~$(K+I)/I$, hence~$K_A^r$ identifies with~$(K^r+I)/I$.
Therefore, $A/K_A^{r}$ is isomorphic 
to~$T(V)/(K^{r}+I)$, which is a homomorphic 
image of~$T(V)/K^{r}\cong \kk\oplus V\oplus\cdots 
\oplus V^{\tensor (r-1)}$ as a vector space. 
\end{proof}
Let~$\wh A=\underset{\longleftarrow}\lim  A/K_A^r$
be the completion of~$A$ with respect to~$K_A$. 
We say that~$f\in A^*$ is {\em locally finite} with respect to~$K_A$ if~$f(K_A^r)=0$ for some~$r>0$ and 
let~$A^{\smallsquare}\subset A^*$ be the subspace of all locally finite linear functionals with respect to~$K_A$.
Clearly,
Sweedler finite dual~$A^o$ of~$A$ is
contained in~$A^{\smallsquare}$, which justifies the notation.  
The following is immediate.
\begin{lemma}\label{lem:restr dual}
The evaluation pairing~$A^{\smallsquare}\tensor A\to \kk$ naturally lifts to a well-defined pairing
$A^{\smallsquare}\tensor \wh A\to \kk$. 
\end{lemma}
\begin{lemma}\label{lem:tens completion}
Suppose that~$(A,K_A)$ and~$(B,K_B)$ are locally finite pairs and let~$K_{A\tensor B}=K_A\tensor B+A\tensor K_B$,
which is an ideal in~$A\tensor B$.
\begin{enumalph}
    \item\label{lem:tens completion.a} $(A\tensor B,K_{A\tensor B})$ is a locally finite pair;
    \item\label{lem:tens completion.b} $A^{\smallsquare}\tensor B^{\smallsquare}
    \subset (A\tensor B)^{\smallsquare}$;
    \item\label{lem:tens completion.c} Let $A\wh\tensor B$ be the completion
    of~$A\tensor B$ with respect to~$K_{A\tensor B}$.
    Then the evaluation pairing $(A^{\smallsquare}\tensor B^{\smallsquare})\tensor 
    (A\tensor B)\to \kk$ naturally lifts to a 
    well-defined pairing $(A^{\smallsquare}\tensor B^{\smallsquare})\tensor (A\wh\tensor B)\to\kk$.
\end{enumalph}
\end{lemma}
\begin{proof}
Clearly, $K_{A\tensor B}^r\supset K_A^r\tensor B+
A\tensor K_B^r$ hence $(A\tensor B)/K_{A\tensor B}^r$
is a homomorphic image of~$(A\tensor B)/(K_A^r\tensor B+A\tensor K_B^r)\cong A/K_A^r\tensor B/K_B^r$. This proves~\ref{lem:tens completion.a}. Part~\ref{lem:tens completion.b} is obvious while~\ref{lem:tens completion.c}
follows from~\ref{lem:tens completion.a}, \ref{lem:tens completion.b} and Lemma~\ref{lem:restr dual}
\end{proof}
\begin{corollary}\label{cor:complet tensor}
Let~$H$ be a finitely generated algebra with counit and
let~$H^{\wh\tensor n}$ be the natural completion of~$H^{\tensor n}$ in the spirit of Lemma~\ref{lem:tens completion}. Then for any~$R\in H^{\wh\tensor n}$
the assignments~$f_1\tensor\cdots\tensor f_1\mapsto 
(f_1\tensor\cdots\tensor f_n)(R)$,
$f_1,\dots,f_n\in H^{\smallsquare}$ define a 
linear map~$\mathcal R:H^{\smallsquare}{}^{\tensor n}\to \kk$.
\end{corollary}

\subsection{Coalgebras}
Let $C$ be a coalgebra with comultiplication $\Delta:C\to C\tensor C$ and counit $\varepsilon:C\to \kk$. 
Recall that~$\sigma\in\End_\kk C$ is 
a coalgebra endomorphism if 
$(\sigma\tensor \sigma)\circ\Delta=
\Delta\circ \sigma$ and $\varepsilon\circ \tau=\varepsilon$. 
Clearly, this notion is dual to that of
an endomorphism of unital algebras
and that coalgebra endomorphisms form
a monoid with respect to composition.
\begin{lemma}\label{lem:op coalg end}
Let~$\sigma\in\End_\kk C$ be a coalgebra endomorphism. Then
$\sigma$ is also an endomorphism of 
$C^{cop}=(C,\Delta^{op})$ where~$\Delta^{op}=\tau_{1,2}\circ\Delta$ is the opposite
comultiplication.
\end{lemma}
\begin{proof}
Let~$c\in C$. Since~$\sigma$ is a coalgebra endomorphism, we have $\Delta(\sigma(c))=(\sigma(c))_{(1)}\tensor(\sigma(c))_{(2)}=
\sigma(c_{(1)})\tensor\sigma(c_{(2)})$.
Then $\Delta^{op}(\sigma(c))=
(\sigma(c))_{(2)}\tensor(\sigma(c))_{(1)}
=\sigma(c_{(2)})\tensor\sigma(c_{(1)})
=(\sigma\tensor\sigma)\Delta^{op}(c)$.
\end{proof}

Let~$C$, $D$ be coalgebras with respective comultiplications 
$\Delta_C:C\to C\tensor C$, $\Delta_D:D\to D\tensor D$ and respective
counits $\varepsilon_C:C\to\kk$, $\varepsilon_D:D\to \kk$. 
Then $C\tensor D$ is naturally a coalgebra with the
counit~$\varepsilon_C\tensor\varepsilon_D$ and with the
comultiplication $\Delta_{C\tensor D}:C\tensor D\to (C\tensor D)\tensor (C\tensor D)$ defined by
$$
\Delta_{C\tensor D}(c\tensor d)=c_{(1)}\tensor d_{(1)}\tensor c_{(2)}\tensor d_{(2)}
$$
where~$\Delta_C(c)=c_{(1)}\tensor c_{(2)}$ in Sweedler notation.
In particular, for any coalgebra~$C$ and~$n>0$, $C^{\tensor n}$
is naturally a coalgebra with 
$$
\Delta_{C^{\tensor n}}(c^1\tensor\cdots\tensor c^n)=c^1_{(1)}\tensor \cdots \tensor c^n_{(1)}\tensor c^1_{(2)}\tensor \cdots\tensor c^n_{(2)}
$$
for all $c^1,\dots,c^n\in C$.
\begin{lemma}
The symmetric group~$S_n$ acts 
on~$C^{\tensor n}$ by coalgebra automorphisms via
$$
\sigma\triangleright (c^1\tensor\cdots c^n)=c^{\sigma^{-1}(1)}\tensor\cdots \tensor c^{\sigma^{-1}(n)},
\quad c^1,\dots,c^n\in C,\sigma\in S_n.
$$
\end{lemma}
\begin{proof}
Given $\sigma\in S_n$, denote $\tilde\sigma$ the $\kk$-linear automorphism of~$C^{\tensor n}$ defined by
$$
\tilde\sigma(c^1\tensor\cdots\tensor c^n)
=c^{\sigma^{-1}(1)}\tensor\cdots\tensor c^{\sigma^{-1}(n)},
c^1,\dots,c^n\in C.
$$
It is a standard fact that $\widetilde{\sigma\circ\tau}=\tilde\sigma\circ\tilde\tau$ for all~$\sigma,\tau\in S_n$. 
We have, for any $c^1,\dots,c^n\in C$
\begin{align*}
\Delta&_{C^{\tensor n}}(\tilde\sigma(c^1\tensor\cdots c^n))=
\Delta_{C^{\tensor n}}(c^{\sigma^{-1}(1)} \tensor\cdots \tensor c^{\sigma^{-1}(n)})\\
&=(c^{\sigma^{-1}(1)})_{(1)} \tensor\cdots \tensor (c^{\sigma^{-1}(n)})_{(1)}\tensor 
(c^{\sigma^{-1}(1)})_{(2)} \tensor\cdots \tensor (c^{\sigma^{-1}(n)})_{(2)}\\
&=(\tilde\sigma\tensor 
\tilde\sigma)(c^1_{(1)}\tensor\cdots\tensor c^n_{(1)})\tensor (c^1_{(2)}\tensor\cdots\tensor c^n_{(2)})
=(\tilde\sigma\tensor 
\tilde\sigma)\Delta_{C^{\tensor n}}
(c^1\tensor\cdots\tensor c^n)
.\qedhere
\end{align*}
\end{proof}
Let $B$ be a bialgebra. We say that $\sigma\in\End_\kk B$ is a bialgebra endomorphism if it is an algebra and a coalgebra endomorphism.

Let~$C$ be a coalgebra.  
Given~$f\in\Hom_\kk(C^{\tensor r},\kk)$, $r\ge 1$,
and a sequence~$(i_1,\dots,i_r)\in [N]^r$, $N\ge r$
with~$i_s\not=i_t$, $1\le s<t\le r$
define~$f_{i_1,\dots,i_r}\in\Hom_\kk(C^{\tensor N},\kk)$
by
$$
f_{i_1,\dots,i_r}(c_1\tensor \cdots\tensor c_N)=
f(c_{i_1}\tensor\cdots\tensor c_{i_r})\prod_{j\in [N]\setminus \{i_1,\dots,i_r\}} \varepsilon(c_j),
\qquad c_1,\dots,c_N\in C.
$$
The convolution product $\ast$ is defined on~$\Hom_\kk(C,\kk)$ via
$$
(f\ast g)(c)=f(c_{(1)})g(c_{(2)}),\qquad 
f,g\in\Hom(C,\kk),\,c\in C,
$$
and is easily seen to be associative, while~$\varepsilon$ is the unity. It is immediate from the definition that for any homomorphism of coalgebras~$\varphi:C\to C'$
\begin{equation}\label{eq:hom coalg}
(f\circ \varphi)\ast (g\circ \varphi)=(f\ast g)\circ\varphi,\qquad f,g\in\Hom_\kk(C',\kk),
\end{equation}
whence if~$f\in \Hom_\kk(C',\kk)$ is $\ast$-invertible
then so is~$f\circ\varphi$
and 
\begin{equation}\label{eq:hom coalg inv}
(f\circ \varphi)^{\ast-1}=f^{\ast-1}\circ\varphi,
\end{equation}
where~$f^{\ast-1}$ is the~$\ast$-inverse of~$f$.
Note the following
elementary yet very useful 
\begin{lemma}\label{lem:old mult from twisted}
Let~$C$ be a coalgebra and let~$\mathcal S\in\Hom(C,\kk)$
be $\ast$-invertible.
Then the following are equivalent for
a $\kk$-vector space~$V$ and $f,g\in\Hom_\kk(C,V)$
\begin{enumrom}
\item\label{lem:old mult from twisted.i} $(\mathcal S\tensor f)\circ\Delta=(g\tensor\mathcal S)\circ\Delta$; \item\label{lem:old mult from twisted.ii}
$f=(\mathcal S^{\ast-1}\tensor g\tensor \mathcal S)\circ(\Delta\tensor\id_C)\circ\Delta$;
\item\label{lem:old mult from twisted.iii}
$(f\tensor \mathcal S^{\ast -1})\circ\Delta=
(\mathcal S^{\ast-1}\tensor g)\circ\Delta$.
\end{enumrom}
\end{lemma}
\begin{proof}
To prove implications \ref{lem:old mult from twisted.i}$\implies$\ref{lem:old mult from twisted.ii}
and~\ref{lem:old mult from twisted.iii}$\implies$
\ref{lem:old mult from twisted.ii} 
it suffices to observe that~$(\mathcal S^{\ast-1}\tensor\mathcal S\tensor h)\circ(\Delta\tensor\id_C)
\circ\Delta=h=(h\tensor \mathcal S^{\ast-1}\tensor 
\mathcal S)\circ(\Delta\tensor\id_C)\circ\Delta$
for all~$h\in\Hom_\kk(C,V)$.
Suppose that~\ref{lem:old mult from twisted.ii} holds,
that is~$f(c)=\mathcal S^{\ast-1}(c_{(1)})\mathcal S(c_{(3)})g(c_{(2)})$
for all~$c\in C$. Then
$$
\mathcal S(c_{(1)})f(c_{(2)})
=\mathcal S(c_{(1)})\mathcal S^{\ast-1}(c_{(2)})
\mathcal S(c_{(4)})g(c_{(3)})
=\mathcal S(c_{(3)})g(c_{(2)}),
$$
which is \ref{lem:old mult from twisted.i}, and
similarly
\begin{equation*}
\mathcal S^{\ast-1}(c_{(2)})f(c_{(1)})
=\mathcal S^{\ast-1}(c_{(1)})\mathcal S(c_{(3)})
\mathcal S^{\ast-1}(c_{(4)})g(c_{(2)})
=\mathcal S^{\ast-1}(c_{(1)})g(c_{(2)}),
\end{equation*}
which is~\ref{lem:old mult from twisted.iii}.
\end{proof}

If~$A$ and~$B$ are unital bialgebras, we can write
$\Delta_{A\tensor B}(a\tensor b)$ as $(\Delta_A(a))_{(1,3)}\cdot  (\Delta_B(b))_{(2,4)}$, and we will use the shorthand 
$\Delta_{A\tensor B}=(\Delta_A)_{1,3}\circ(\Delta_B)_{2,4}$ in this 
situation. Likewise, the standard comultiplication on~$B^{\tensor n}$ can be presented as $\Delta_{B^{\tensor n}}=\Delta_{1,n+1}\circ\cdots\circ \Delta_{n,2n}$.

\subsection{Lie bialgebras}
Let~$\lie g$ be a Lie algebra and let~$U(\lie g)$ be its universal enveloping algebra. Let~$\Delta:U(\lie g)\to U(\lie g)\tensor U(\lie g)$ be the standard 
comultiplication which is uniquely defined by $\Delta(x)=x\tensor 1+1\tensor x$, $x\in\lie g$. We identify $\lie g\tensor \lie g$ with its natural image in~$U(\lie g)\tensor U(\lie g)$.

Recall that $\delta\in\Hom_\kk(\lie g,\lie g\tensor \lie g)$ is a Lie {\em cobracket} 
if it satisfies 
\begin{enumerate}[label={($LB_\arabic*$)}]
\item\label{cnd.3} $\delta([x,y]_{\lie g})=[\delta(x),\Delta(y)]-[\delta(y),\Delta(x)]
$ in $U(\lie g)\tensor U(\lie g)$
for all $x,y\in\lie g$;
\item\label{cnd.1} $\tau_{1,2}\circ\delta=-\delta$;
\item\label{cnd.2} $(\id_{\lie g^{\tensor 3}}+\tau_{1,2}\tau_{2,3}+\tau_{2,3}\tau_{1,2})\circ (\delta\tensor\id_{\lie g})\circ \delta=0$
as a linear map~$\lie g\to \lie g^{\tensor 3}$.
\end{enumerate}
The first condition is equivalent to the requirement that~$\delta$ extends to a $\Delta$-derivation
$U(\lie g)\to U(\lie g)\tensor U(\lie g)$,
that is
\begin{equation}\label{eq:delta prod}
\delta(x y)=\delta(x)\Delta(y)+\Delta(x)\delta(y),\quad x,y\in U(\lie g).
\end{equation}
Note that if~$\delta$ satisfies~\ref{cnd.1}
then~\ref{cnd.2} can be rewritten as 
\begin{equation}\label{eq:equiv coJacobi}
(\delta\tensor\id_{\lie g})\circ \delta =
(\id_{\lie g}\tensor\delta)\circ\delta+
\tau_{2,3}(\delta\tensor\id_{\lie g})\circ\delta.
\end{equation}
The following is well-known (see e.g.~\cite{CPbook}).
\begin{lemma}\label{lem:coLeib}
Let~$(\lie g,\delta)$ be a Lie bialgebra. Then 
$(\Delta\tensor\id)\circ\delta=(\id\tensor\delta)\circ\Delta+
\tau_{2,3}(\delta\tensor\id)\circ\Delta$ on~$U(\lie g)$.
\end{lemma}
\begin{proof}
The argument is by induction on the canonical filtration~$\{U_n(\lie g)\}_{n\ge 0}$ on~$U(\lie g)$. The assertion for~$n=0$ is trivial while for~$n=1$ it is easily checked. Suppose that the Lemma is proved for all
$u\in U_{n-1}(\lie g)$. Then for any~$x\in\lie g$, $u\in U_{n-1}(\lie g)$
\begin{align*}
(\Delta\tensor\id)\delta(xu)&=(\Delta\tensor\id)\delta(x)\cdot (\Delta\tensor\id)\Delta(u)+(\Delta\tensor\id)\Delta(x)\cdot 
(\Delta\tensor\id)\delta(u)\\
&=(\id\tensor\delta)\Delta(x)\cdot (\id\tensor\Delta)\Delta(u)+
\tau_{2,3}(\delta\tensor\id)\Delta(x)\cdot(\Delta\tensor\id)\Delta(u)\\
&\phantom{=}+
(\id\tensor\Delta)\Delta(x)\cdot 
(\id\tensor \delta)\Delta(u)+
(\Delta\tensor\id)\Delta(x)\cdot 
\tau_{2,3}(\delta\tensor \id)\Delta(u)\\
&=(\id\tensor\delta)(\Delta(xu))+\tau_{2,3}(\delta\tensor\id)(\Delta(xu)),
\end{align*}
where we used~\eqref{eq:delta prod} and the cocommutativity and the co-associativity of~$\Delta$. Since~$U_n(\lie g)=U_{n-1}(\lie g)+\lie gU_{n-1}(\lie g)$, this proves the inductive step.
\end{proof}

Define $\cyb{\cdot}{\cdot}\in\Hom_\kk(\lie g^{\tensor 2}\tensor\lie g^{\tensor 2},U(\lie g)^{\tensor 3})$ via
$$
\cyb{s}{s'}=
[(\id_{\lie g}\tensor\Delta)(s),(\Delta\tensor\id_{\lie g})(s')]=[s_{1,2}+s_{1,3},s'_{1,3}+s'_{2,3}],\qquad s,s'\in \lie g\tensor\lie g.
$$
The following is well-known
(see e.g.~\cite{Majid}*{\S8.1} and~\cites{Drinf,EK-I,EK-II,EK-III}).
We provide a proof here for the reader's convenience, since some of the intermediate identities will be 
needed later. 
\begin{proposition}\label{prop:compat cond}
Let~$r\in \lie g\tensor\lie g$ and define~$\delta_{r}\in\Hom_\kk(\lie g, \lie g\tensor \lie g)$ by $\delta_{r}(x)=
[r,\Delta(x)]$, $x\in\lie g$.
\begin{enumalph}
\item\label{prop:compat cond.a} 
$\delta_{r}$ satisfies~\ref{cnd.3}
and so extends to a $\Delta$-derivation
$\delta_{r}:U(\lie g)\to U(\lie g)\tensor U(\lie g)$. In particular,
$\delta_{r}(u)=[r,\Delta(u)]$ for all~$u\in U(\lie g)$.
\item\label{prop:compat cond.b}
$\delta_{r}$ satisfies
\ref{cnd.1} if and only if $r+\tau_{1,2}(r)$ centralizes 
$\Delta(U(\lie g))\subset U(\lie g)\tensor U(\lie g)$ or, equivalently, if and only
if $r+\tau_{1,2}(r)$ is
$\lie g$-invariant with respect to the natural
diagonal action of~$\lie g$ on~$\lie g\tensor\lie g$; 
\item\label{prop:compat cond.c} 
Suppose that~$\delta_{r}$ satisfies~\ref{cnd.1}. Then
for all~$x\in\lie g$
$$
((\id_{\lie g^{\tensor 3}}+\tau_{1,2}\tau_{2,3}+\tau_{2,3}\tau_{1,2})\circ(\delta_{r}\tensor\id)\circ\delta_{r})(x)=
[\cybe{r},(\Delta\tensor\id)\circ\Delta(x)].
$$
In particular, under this assumption $\delta_{r}$ satisfies~\ref{cnd.2} if and only if
$\cybe{r}$ centralizes 
$(\Delta\tensor\id)\circ\Delta(U(\lie g))\subset U(\lie g)^{\tensor 3}$
or, equivalently, is $\lie g$-invariant
with respect to the natural diagonal 
$\lie g$-action on~$U(\lie g)^{\tensor 3}$.
\item\label{prop:compat cond.d}
$\cybe{r}=0$ if and only if $(\delta_{r}\tensor\id_{\lie g})(r)=[r_{2,3},r_{1,3}]$ if and only if
$(\id_{\lie g}\tensor \delta_{r})(r)=[r_{1,2},
r_{1,3}]$.
\end{enumalph}
\end{proposition}
\begin{proof}
We need the following 
\begin{lemma}\label{lem:prop Delta-der}
    Suppose that~$\delta:U(\lie g)\to U(\lie g)\tensor U(\lie g)$ is a $\Delta$-derivation.
    Then for all~$u\in U(\lie g)$, $r\in\lie g\tensor\lie g$ we have in $U(\lie g)\tensor U(\lie g)$
    \begin{equation}\label{eq:cobrack commutator}(\delta\tensor\id)([r,\Delta(u)])=[(\delta\tensor\id)(r),(\Delta\tensor\id)\Delta(u)]+
    [r_{1,3}+r_{2,3},(\delta\tensor\id)\Delta(u)].
    \end{equation} 
\end{lemma}
\begin{proof}
Write~$r=r_1\tensor r_2$ in Sweedler-like notation. Then
\begin{align*}
(\delta\tensor\id)(r\cdot \Delta(u))&=(\delta\tensor\id)
(r_1 u_{(1)}\tensor r_2 u_{(2)})\\
&=\delta(r_1)\cdot \Delta(u_{(1)})\tensor r_2 u_{(2)}+\Delta(r_1)\cdot \delta(u_{(1)})\tensor r_2u_{(2)}\\
&=(\delta\tensor\id)(r)\cdot (\Delta\tensor\id)\Delta(u)+
(\Delta\tensor\id)(r)\cdot 
(\delta\tensor\id)\Delta(u)\\
&=(\delta\tensor\id)(r)\cdot (\Delta\tensor\id)\Delta(u)+
(r_{1,3}+r_{2,3})\cdot (\delta\tensor\id)\Delta(u).
\end{align*}
The identity for~$(\delta\tensor\id)(\Delta(u)\cdot r)$ is obtained similarly, and the assertion follows.
\end{proof}

To prove part~\ref{prop:compat cond.a}, note that for all~$x,y\in\lie g$,
\begin{align*}
\delta_{r}([x,y]_{\lie g})&=
[r,\Delta([x,y]_{\lie g})]=
[r,[\Delta(x),\Delta(y)]]
\\
&=
[[r,\Delta(x)],\Delta(y)]+[
\Delta(x),[r,\Delta(y)]]
=[\delta_{r}(x),\Delta(y)]-[\delta_{r}(y),\Delta(x)].
\end{align*}
Thus, $\delta_{r}$ satisfies~\ref{cnd.3} and hence extends to a $\Delta$-derivation $U(\lie g)\to U(\lie g)\tensor U(\lie g)$. The second assertion of~\ref{prop:compat cond.a} follows since~$\lie g$
generates~$U(\lie g)$ as an associative algebra.

Since~$\Delta$ is cocommutative,
$\tau_{1,2}\circ\delta_{r}(x)+\delta_{r}(x)=[\tau_{1,2}(r)+r,\Delta(x)]=0$ for all~$x\in\lie g$
if and only if~$\tau_{1,2}(r)+r$ commutes with~$\Delta(x)$ for all~$x\in\lie g$. Since~$\lie g$ generates~$U(\lie g)$ as an associative algebra, part~\ref{prop:compat cond.b} follows.

By Lemma~\ref{lem:prop Delta-der} we have
for all~$x\in\lie g$
\begin{align*}
(\delta_{r}&\tensor\id)\delta_{r}(x)
=[(\delta_{r}\tensor\id)(r),(\Delta\tensor\id)\Delta(x)]+[r_{1,3}+r_{2,3},(\delta_{r}\tensor \id)\Delta(x)]\\
&=[(\delta_{r}\tensor\id)(r),(\Delta\tensor\id)\Delta(x)]+[r_{1,3}+r_{2,3},\delta_{r}(x)_{1,2}].
\end{align*}
Furthermore, 
\begin{align}\label{eq:delta r of r}
(\delta_{r}\tensor\id)(r)=[r_{1,2},r_{1,3}+r_{2,3}]=-[r_{2,1},r_{1,3}+r_{2,3}]
\end{align}
and
$$
\delta_{r}(x)_{1,2}=
[r_{1,2},x\tensor 1\tensor 1+1\tensor x\tensor 1]=
[r_{1,2},(\Delta\tensor \id)\Delta(x)]=-[r_{2,1},(\Delta\tensor\id)\Delta(x)].
$$
Thus, 
\begin{align*}
(\delta_{r}\tensor\id)\delta_{r}(x)&=
[[r_{1,2},r_{1,3}+r_{2,3}],(\Delta\tensor\id)(x)]+
[r_{1,3}+r_{2,3},[r_{1,2},(\Delta\tensor\id)\Delta(x)]]\\
&=[r_{1,2},[r_{1,3}+r_{2,3},(\Delta\tensor\id)\Delta(x)]]\\
&=-[[r_{2,1},r_{1,3}+r_{2,3}],(\Delta\tensor\id)(x)]-
[r_{1,3}+r_{2,3},[r_{2,1},(\Delta\tensor\id)\Delta(x)]]\\
&=-[r_{2,1},[r_{1,3}+r_{2,3},(\Delta\tensor\id)\Delta(x)]]
\end{align*}
Note that, since $r+\tau_{1,2}(r)$ commutes with~$\Delta(\lie g)$, 
$r_{i,j}+r_{j,i}$ commutes with $(\Delta\tensor\id)\Delta(\lie g)$ for all $1\le i<j\le 3$. 
Applying $\id_{\lie g^{\tensor 3}}+\tau_{2,3}\tau_{1,2}+\tau_{1,2}\tau_{2,3}$ 
and taking into account that~$\Delta$ is cocommutative and coassociative we obtain
\begin{align*}
(\id_{\lie g^{\tensor 3}}&+\tau_{2,3}\tau_{1,2}+\tau_{1,2}\tau_{2,3})((\delta_{r}\tensor\id)\delta_{r}(x))\\
&=[r_{1,2},[r_{1,3}+r_{2,3},z]]
-[r_{1,3},[r_{3,2}+r_{1,2},z]]+[r_{2,3},[r_{2,1}+r_{3,1},z]]
\\
&=[r_{1,2},[r_{1,3}+r_{2,3},z]]-
[r_{1,3},[r_{1,2}-r_{2,3},z]]-[r_{2,3},[r_{1,2}+r_{1,3},z]]\\
&=([r_{12},[r_{2,3},z]]+[r_{2,3},[z,r_{12}]])
-([r_{1,3},[r_{1,2},z]]+
[r_{1,2},[z,r_{1,3}]])\\
&\phantom{=}-([r_{2,3},[r_{1,3},z]]+[r_{1,3},[z,r_{2,3}]])\\
&=[ [r_{1,2},r_{2,3}]
-[r_{1,3},r_{1,2}]
-[r_{2,3},r_{1,3}],z]
=[\cybe{r},z],
\end{align*}
where we abbreviated~$z=(\Delta\tensor\id)\Delta(x)$. Part~\ref{prop:compat cond.c} is now immediate.

Finally, note that by~\eqref{eq:delta r of r}
$$
\cybe{r}=[r_{1,2}+r_{1,3},r_{1,3}+r_{2,3}]
=(\delta_{r}\tensor\id_{\lie g})(r)+[r_{1,3},r_{2,3}].
$$
Similarly, since~$(\id_{\lie g}\tensor \delta_{r})(r)=
[r_{2,3},r_{1,2}+r_{1,3}]$,
$$
\cybe{r}=-(\id_{\lie g}\tensor\delta_{r})(r)+[r_{1,2},r_{1,3}].
$$
These identities prove part~\ref{prop:compat cond.d}.
\end{proof}

\subsection{Quasi-triangular Lie bialgebras}\label{subs:qtr Lie bialg}
Following~\cite{Drinf}, Lie bialgebra $(\lie g,\delta)$ is 
called {\em quasi-triangular} if there
is $r\in\lie g\tensor \lie g$, called a
classical r-matrix, such that
$\delta=\delta_{r}$
and~$\cybe{r}=0$. 

The equation~$\cybe{r}=0$ is called the {\em Classical Yang-Baxter equation} (CYBE). 
More generally, we have the following
\begin{proposition}\label{prop:color CYBE}
Let $(\lie g,\delta)$ be a Lie bialgebra and let 
$\{ r^{(c)}\}_{c\in C}\subset \lie g\tensor \lie g$ be a family
of classical r-matrices for~$(\lie g,\delta)$.
Then in $U(\lie g)^{\tensor 3}$
\begin{equation}\label{eq:color CYBE}
[r_{i,j}^{(c)},r_{i,k}^{(c')}]+[r_{i,j}^{(c)},r_{j,k}^{(c'')}]+[r_{i,k}^{(c')},r_{j,k}^{(c'')}]=0
\end{equation}
for all~$c'\in \{c,c''\}\subset C$, $\{i,j,k\}=\{1,2,3\}$.
\end{proposition}
\begin{proof}
Since~$S_3$ acts by algebra automorphisms on~$U(\lie g)^{\tensor 3}$, it suffices
to prove~\eqref{eq:color CYBE} for $(i,j,k)=(1,2,3)$.
Let~$c,c''\in C$.
Suppose first that~$c'=c''$. Since~$\delta=\delta_{r^{(c)}}=
\delta_{r^{(c'')}}$, we have by Proposition~\ref{prop:compat cond}\ref{prop:compat cond.d}
$$
0=(\delta\tensor\id_{\lie g})(r^{(c'')})+[r^{(c'')}_{1,3},
r^{(c'')}_{2,3}]
=[r^{(c)}_{1,2},r^{(c'')}_{1,3}+r^{(c'')}_{2,3}]+
[r^{(c'')}_{1,3},r^{(c'')}_{2,3}],
$$
which is~\eqref{eq:color CYBE} with~$c'=c''$. Similarly, if~$c'=c$ then
by Proposition~\ref{prop:compat cond}\ref{prop:compat cond.d}
$$
0=[r^{(c)}_{1,2},
r^{(c)}_{1,3}]-(\id_{\lie g}\tensor\delta)(r^{(c)}) 
=[r^{(c)}_{1,2},r^{(c)}_{1,3}]-
[r^{(c'')}_{2,3},r^{(c)}_{1,2}+r^{(c)}_{1,3}],
$$
which yields~\eqref{eq:color CYBE} with~$c'=c$.
\end{proof}
The basic example of such a family is provided by the following
\begin{lemma}\label{lem:bas classical family}
Let~$(\lie g,\delta)$ be a quasi-triangular Lie bialgebra
with a classical r-matrix $r$.
Then $-\tau_{1,2}(r)$ is also a classical~r-matrix 
for $(\lie g,\delta)$. In particular, if $r^{(1)}=r$
and~$r^{(-1)}=-\tau_{1,2}(r)$ then 
$$
[r^{(\epsilon)}_{i,j},r^{(\epsilon')}_{i,k}]+
[r^{(\epsilon)}_{i,j},r^{(\epsilon'')}_{j,k}]+
[r^{(\epsilon')}_{i,k},r^{(\epsilon'')}_{j,k}]=0
$$
provided that~$\epsilon'\in\{\epsilon,\epsilon''\}\subset\{1,-1\}$ and~$\{i,j,k\}=\{1,2,3\}$.
\end{lemma}
\begin{proof}
Suppose that~$\delta=\delta_{r}$ for some~$r\in\lie g\tensor\lie g$
satisfying~$\cybe{r}=0$.
Then, since~$\Delta$ is cocommutative and~$\tau_{1,2}\circ\delta=-\delta$, we have 
$\delta=\delta_{-\tau_{1,2}(r)}$. Furthermore,
\begin{align*}
\cybe{-\tau_{1,2}(r)}&=[(\id_{\lie g}\tensor\Delta)(\tau_{1,2}(r)),(\Delta\tensor\id)(
\tau_{1,2}(r)]\\
&=\tau_{2,3}\tau_{1,2}([(\Delta\tensor\id_{\lie g})(r),
(\id_{\lie g}\tensor\Delta)(r)])=-\tau_{2,3}\tau_{1,2}(\cybe{r})=0.
\end{align*}
Thus, $-\tau_{1,2}(r)$ is also a classical r-matrix for~$(\lie g,\delta)$.
The remaining assertion is then immediate from Proposition~\ref{prop:color CYBE}.
\end{proof}

\subsection{Classical Drinfeld twists}
Given~$r\in\lie g\tensor \lie g$, denote~$r^-:=r-\tau_{1,2}(r)$.

A {\em weak classical Drinfeld twist} $ j $ for~$(\lie g,\delta)$ is an element 
of~$\lie g\tensor\lie g$ such that $\tilde\delta_{ j }:=\delta+
\delta_{ j ^-}$
is a Lie cobracket.  
\begin{remark}
This definition is slightly different from that of the classical twist given in~\cite{Drinf} to emphasize the parallel with the quantum situation.
\end{remark}
The following is well-known (see~\cite{Drinf} and~\cite{Majid}*{\S8.1}).
\begin{proposition}\label{prop:cnd classical twist}
Let~$(\lie g,\delta)$ be a Lie bialgebra and $ j \in \lie g\tensor\lie g$.
\begin{enumalph}
\item \label{prop:cnd classical twist.a}
$ j $ is a weak classical Drinfeld twist if and only if 
\begin{equation}
\cybe{ j ^-}_{\delta}:=(\delta\tensor\id)( j ^-)-(\id\tensor\delta)( j ^-)
-\tau_{2,3}(\delta\tensor\id)( j ^-)
+\cybe{ j ^-}\label{eq:class Drinfeld twist}
\end{equation}
centralizes~$(\Delta\tensor\id)\Delta(U(\lie g))\subset U(\lie g)^{\tensor 3}$ or, equivalently, is~$\lie g$-invariant
with respect to the natural diagonal $\lie g$-action on~$U(\lie g)^{\tensor 3}$.

\item \label{prop:cnd classical twist.b}
Suppose that~$(\lie g,\delta)$ is quasi-triangular with a classical r-matrix
$r$. Then 
\begin{equation}\cybe{ j ^-}_\delta=\cybe{r+ j ^-}.
\label{eq:class Drinfeld twist'}
\end{equation}
In particular, $(\lie g,\delta+\delta_{ j ^-})$
is quasi-triangular with the classical r-matrix~$r+ j ^-$
if and only if~$\cybe{ j ^-}_\delta=0$.
\end{enumalph}
\end{proposition}
\begin{proof}
Note that both~\ref{cnd.3} and~\ref{cnd.1}
are linear in~$\delta$.
Since~$\delta$ satisfies~\ref{cnd.3} and~\ref{cnd.1} by assumption while~$\delta_{ j ^-}$ satisfies~\ref{cnd.3} by Proposition~\ref{prop:compat cond}\ref{prop:compat cond.a} and~\ref{cnd.1}
by Proposition~\ref{prop:compat cond}\ref{prop:compat cond.b} as~$\tau_{1,2}( j ^-)+ j ^-=0$, 
it follows 
that~$\tilde\delta_{ j }=\delta+\delta_{ j ^-}$ satisfies~\ref{cnd.3} and~\ref{cnd.1}.

We have 
\begin{align*}
(\tilde\delta_{ j }\tensor\id)\circ\tilde\delta_{ j }
=(\delta\tensor\id)\circ\delta 
+(\delta\tensor\id)\circ\delta_{ j ^-}
+(\delta_{ j ^-}\tensor\id)\circ\delta
+(\delta_{ j ^-}\tensor\id)\circ \delta_{ j ^-}.
\end{align*}
Note that~$\delta$ satisfies~\ref{cnd.2} by assumption and 
$$
(\id_{ \lie g^{\tensor 3}}+
\tau_{2,3}\tau_{1,2}+\tau_{1,2}\tau_{2,3})\circ 
(\delta_{ j ^-}\tensor \id)\circ \delta_{ j ^-}(x)=
[\cybe{ j ^-},(\Delta\tensor\id)\Delta(x)]
$$
by Proposition~\ref{prop:compat cond}\ref{prop:compat cond.c}. Since
$$
(\delta_{ j ^-}\tensor\id)\delta(x)=
[ j ^-_{1,2},(\Delta\tensor\id)\delta(x)]
=[ j ^-_{1,2},\delta(x)_{1,3}+\delta(x)_{2,3}],
$$
we have 
\begin{align}
(\id_{\tensor \lie g^{\tensor 3}}&+
\tau_{2,3}\tau_{1,2}+\tau_{1,2}\tau_{2,3})\circ(\delta_{ j ^-}\tensor\id)\circ \delta(x)\nonumber\\
&=[ j ^-_{1,2},\delta(x)_{1,3}+\delta(x)_{2,3}]+
[ j ^-_{3,1},\delta(x)_{3,2}+\delta(x)_{1,2}]
+[ j ^-_{2,3},\delta(x)_{2,1}+\delta(x)_{3,1}]\nonumber\\
&=-[ j ^-_{1,3}+ j ^-_{2,3},\delta(x)_{1,2}]+
[ j ^-_{1,2}- j ^-_{2,3},\delta(x)_{1,3}]
+[ j ^-_{1,2}+ j ^-_{1,3},\delta(x)_{2,3}].\label{eq:term III}
\end{align}
Since
$$
(\delta\tensor\id)\delta_{ j ^-}(x)=
[(\delta\tensor\id)( j ^-),(\Delta\tensor\id)\Delta(x)]+
[ j ^-_{1,3}+ j ^-_{2,3},\delta(x)_{1,2}]
$$
by Lemma~\ref{lem:prop Delta-der}, $\Delta$ is coassociative and cocommutative and $\tau_{1,2}( j ^-)=- j ^-$,
it follows that 
\begin{align*}
(&\id_{\lie g^{\tensor 3}}+
\tau_{2,3}\tau_{1,2}+\tau_{1,2}\tau_{2,3})\circ
(\delta\tensor\id)\circ \delta_{ j ^-}(x)\\
&=[(\delta\tensor\id)( j ^-)-\tau_{2,3}(\delta\tensor\id)( j ^-)-(\id\tensor\delta)( j ^-),(\Delta\tensor\id)\Delta(x)]\\
&\phantom{=}+[ j ^-_{1,3}+ j ^-_{2,3},\delta(x)_{1,2}]
+[ j ^-_{3,2}+ j ^-_{1,2},\delta(x)_{3,1}]
+[ j ^-_{2,1}+ j ^-_{3,1},\delta(x)_{2,3}]\\
&=[\cybe{ j ^-}_\delta-\cybe{ j ^-},(\Delta\tensor\id)\Delta(x)]-((\id_{\tensor \lie g^{\tensor 3}}+
\tau_{2,3}\tau_{1,2}+\tau_{1,2}\tau_{2,3})\circ(\delta_{ j ^-}\tensor\id)\circ\delta)(x),
\end{align*}
where we used~\eqref{eq:term III}. Thus,
$$
(\id_{\lie g^{\tensor 3}}+
\tau_{2,3}\tau_{1,2}+\tau_{1,2}\tau_{2,3})\circ(\tilde\delta_{ j }\tensor\id)\circ \tilde \delta_{ j }(x)=[ \cybe{ j ^-}_\delta,(\Delta\tensor\id)\Delta(x)],
$$
and part~\ref{prop:cnd classical twist.a} is now
immediate.

Suppose that~$\delta=\delta_{r}$ with~$\cybe{r}=0$. Then
\begin{align*}
&\cybe{r+ j ^-}=
\cybe{r}+\cyb{r}{ j ^-}+\cyb{ j ^-}{r}
+\cybe{ j ^-}\\
&\qquad=\cybe{ j ^-}+[r_{1,2}+r_{1,3}, j ^-_{1,3}+ j ^-_{2,3}]-[r_{1,3}+r_{2,3}, j ^-_{1,2}+ j ^-_{1,3}]\\
&\qquad=\cybe{ j ^-}+
[r_{1,2}, j ^-_{1,3}+ j ^-_{2,3}]
+[r_{1,3}, j ^-_{2,3}- j ^-_{1,2}]-[r_{2,3}, j ^-_{1,2}+ j ^-_{1,3}]
\\
&\qquad=\cybe{ j ^-}+(\delta_{r}\tensor\id)( j ^-)-(\id\tensor\delta_{r})( j ^-)-[r_{1,3}, j ^-_{1,2}+ j ^-_{3,2}]\\
&\qquad=\cybe{ j ^-}+(\delta_{r}\tensor\id)( j ^-)-(\id\tensor\delta_{r})( j ^-)-\tau_{2,3}([r_{1,2}, j ^-_{1,3}+ j ^-_{2,3}])=\cybe{ j ^-}_\delta.
\end{align*}
This proves the first assertion in part~\ref{prop:cnd classical twist.b}.
To prove the second, note that since~$\delta=\delta_{r}$,
$\tilde\delta_{ j }=\delta_{r+ j ^-}$. The assertion
is then immediate from~\eqref{eq:class Drinfeld twist'}.
\end{proof}
From now on, we call $ j \in\lie g\tensor\lie g$ satisfying
$\cybe{ j ^-}_\delta=0$ a {\em classical Drinfeld twist} for $(\lie g,\delta)$.

\subsection{Relative classical twist}\label{subs:rel clas twist}
Recall that $U(\lie g\oplus\lie h)$ is isomorphic 
to $U(\lie g)\tensor U(\lie h)$ as a bialgebra,
with $(x,y)\mapsto x\tensor 1+1\tensor y$, $x\in\lie g$, $y\in \lie h$. If~$(\lie g,\delta_{\lie g})$ and~$(\lie h,\delta_{\lie h})$ are Lie bialgebras, then 
$\lie g\oplus \lie h$ is a Lie bialgebra 
with $\delta_{\lie g\oplus\lie h}((x,y))=
\delta_{\lie g}(x)_{1,3}+
\delta_{\lie h}(y)_{2,4}$ in~$(U(\lie g)\tensor U(\lie h))^{\tensor 2}$ for all~$x\in\lie g$, $y\in\lie h$.
Note that if $(\lie g,\delta_{\lie g})$ and~$(\lie h,\delta_{\lie h})$
are quasi-triangular with respective classical r-matrices 
$r_{\lie g}$, $r_{\lie h}$ then~$\lie g\oplus \lie h$
is quasi-triangular, its classical r-matrix 
being $(r_{\lie g})_{1,3}+(r_{\lie h})_{2,4}$.

Let $(\lie g,\delta_{\lie g})$ and~$(\lie h,\delta_{\lie h})$ be Lie bialgebras. 
We say that $\mathbf f\in \lie h\tensor\lie g$
is a {\em relative classical Drinfeld twist} for $(\lie g\oplus\lie h,\delta_{\lie g\oplus\lie h})$
if $\mathbf f_{2,3}=1_{U(\lie g)}\tensor 
\mathbf f\tensor 1_{U(\lie h)}\in  
U(\lie g)\tensor \lie h\tensor \lie g\tensor U(\lie h)$ 
is a classical Drinfeld twist for~$(\lie g\oplus\lie h,\delta_{\lie g\oplus\lie h})$.
\begin{proposition}\label{prop:rel class Drinfeld}
Let $(\lie g,\delta_{\lie g})$ and~$(\lie h,\delta_{\lie h})$ be Lie bialgebras and let
$\mathbf f\in\lie h\tensor\lie g$. Then $\mathbf f$ is a relative 
Drinfeld twist if and only if 
\begin{multline*}
(\id_{(U(\lie g)\tensor U(\lie h))^{\tensor3}}-\tau_{2,4}\tau_{1,3}+\tau_{2,6}\tau_{3,5}\tau_{1,3})((\id_\lie h\tensor\delta_{\lie g})(\mathbf f)_{2,3,5})\\-
(\id_{(U(\lie g)\tensor U(\lie h))^{\tensor3}}-\tau_{4,6}\tau_{3,5}+\tau_{2,4}\tau_{4,6}\tau_{1,5})((\delta_{\lie h}\tensor \id_{\lie g})(\mathbf f)_{2,4,5})
\\
=[\mathbf f_{2,3},\mathbf f_{2,5}]+[\mathbf f_{2,5},\mathbf f_{4,5}]+
[\mathbf f_{4,5},\mathbf f_{4,1}]
+[\mathbf f_{4,1},\mathbf f_{6,1}]+[\mathbf f_{6,1},\mathbf f_{6,3}]
+[\mathbf f_{6,3},\mathbf f_{2,3}].
\end{multline*}
\end{proposition}
\begin{proof}
Since~$\mathbf f_{23}^-=\mathbf f_{2,3}-\mathbf f_{4,1}$, we have 
\begin{align*}
(\delta_{\lie g\oplus\lie h}&\tensor\id_{U(\lie g)\tensor U(\lie h)})(\mathbf f_{2,3}-\mathbf f_{4,1})
=(\delta_{\lie h}\tensor \id_{\lie g})(\mathbf f)_{2,4,5}-
(\id_{\lie h}\tensor\delta_{\lie g})(\mathbf f)_{6,1,3},
\end{align*}
whence
\begin{align*}
&(\id_{(U(\lie g)\tensor U(\lie h))^{\tensor3}}+\tau_{\mathbf 2,\mathbf 3}\tau_{\mathbf 1,\mathbf 2}+\tau_{\mathbf 1,\mathbf 2}\tau_{\mathbf 2,\mathbf 3})\circ
(\delta_{\lie g\oplus\lie h}\tensor\id_{U(\lie g)\tensor U(\lie h)})(\mathbf f_{2,3}-\mathbf f_{4,1})
\\
&\qquad=(\delta_{\lie h}\tensor \id_{\lie g})(\mathbf f)_{2,4,5}
+(\delta_{\lie h}\tensor \id_{\lie g})(\mathbf f)_{6,2,3}
+(\delta_{\lie h}\tensor \id_{\lie g})(\mathbf f)_{4,6,1}\\
&\qquad\phantom{=}-
(\id_{\lie h}\tensor\delta_{\lie g})(\mathbf f)_{6,1,3}
-(\id_\lie h\tensor\delta_{\lie g})(\mathbf f)_{4,5,1}-(\id_\lie h\tensor\delta_{\lie g})(\mathbf f)_{2,3,5}\\
&\qquad=
(\id_{(U(\lie g)\tensor U(\lie h))^{\tensor3}}-\tau_{4,6}\tau_{3,5}+\tau_{2,4}\tau_{4,6}\tau_{1,5})((\delta_{\lie h}\tensor \id_{\lie g})(\mathbf f)_{2,4,5})
\\
&\qquad\phantom{=}-(\id_{(U(\lie g)\tensor U(\lie h))^{\tensor3}}-\tau_{2,4}\tau_{1,3}+\tau_{2,6}\tau_{3,5}\tau_{1,3})((\id_\lie h\tensor\delta_{\lie g})(\mathbf f)_{2,3,5})
\\
\intertext{where $\tau_{\mathbf 1,\mathbf 2}=\tau_{2,3}\tau_{1,2}\tau_{3,4}\tau_{2,3}$,
$\tau_{\mathbf 2,\mathbf 3}=\tau_{4,5}\tau_{3,4}\tau_{5,6}\tau_{4,5}$,
while}
&\cybe{\mathbf f_{2,3}-\mathbf f_{4,1}}=
[\mathbf f_{2,3}-\mathbf f_{4,1}+\mathbf f_{2,5}-\mathbf f_{6,1},
\mathbf f_{2,5}-\mathbf f_{6,1}+\mathbf f_{4,5}-\mathbf f_{6,3}]\\
&\qquad=[\mathbf f_{2,3},\mathbf f_{2,5}-\mathbf f_{6,3}]
+[\mathbf f_{4,1},\mathbf f_{6,1}-\mathbf f_{4,5}]+
[\mathbf f_{2,5},\mathbf f_{4,5}]+[\mathbf f_{6,1},\mathbf f_{6,3}].
\end{align*}
The assertion is now immediate.
\end{proof}
\begin{proposition}\label{prop:basic twist class}
Let~$(\lie g,\delta)$ be a quasi-triangular bialgebra with a classical r-matrix~$r$. Then $r$ is a relative classical Drinfeld twist.
In particular, in the notation of Lemma~\ref{lem:bas classical family},
$(\lie g\oplus\lie g,\delta_{\lie g\oplus\lie g})$ is quasi-triangular 
with the classical r-matrix $r_{2,3}^{(\epsilon)}+r_{1,4}^{(-\epsilon)}+r_{1,3}^{(\epsilon')}+r_{2,4}^{(\epsilon'')}$
for any~$\epsilon,\epsilon',\epsilon''\in \{1,-1\}$.
\end{proposition}
\begin{proof}
Using Proposition~\ref{prop:compat cond}\ref{prop:compat cond.d} we obtain
\begin{align*}
&(\id_{(\lie g\tensor \lie g)^{\tensor3}}-\tau_{2,4}\tau_{1,3}+\tau_{2,6}\tau_{3,5}\tau_{1,3})((\id_\lie g\tensor\delta)(r)_{2,3,5})\\&-
(\id_{(\lie g\tensor \lie g)^{\tensor3}}-\tau_{4,6}\tau_{3,5}+\tau_{2,4}\tau_{4,6}\tau_{1,5})((\delta\tensor \id_{\lie g})(r)_{2,4,5})\\
&\quad = (\id_{(\lie g\tensor \lie g)^{\tensor3}}-\tau_{2,4}\tau_{1,3}+\tau_{2,6}\tau_{3,5}\tau_{1,3})([r_{2,3},r_{2,5}])\\
&\quad\phantom{=}-(\id_{(\lie g\tensor \lie g)^{\tensor3}}-\tau_{4,6}\tau_{3,5}+\tau_{2,4}\tau_{4,6}\tau_{1,5})([r_{4,5},r_{2,5}])\\
&\quad = [r_{2,3},r_{2,5}]+[r_{4,5},r_{4,1}]
+[r_{6,1},r_{6,3}]
+[r_{2,5},r_{4,5}]+
[r_{6,3},r_{2,3}]+[r_{4,1},r_{6,1}].
\end{align*}
The first assertion with~$\epsilon=1$ is now immediate by Proposition~\ref{prop:rel class Drinfeld}. To prove it for~$\epsilon=-1$ it suffices to note that~$-\tau_{1,2}(r)$
is also a classical r-matrix for the same cobracket~$\delta$ by Lemma~\ref{lem:bas classical family}.
To prove the second assertion, it remains to apply Proposition~\ref{prop:cnd classical twist}\ref{prop:cnd classical twist.b} together with Lemma~\ref{lem:bas classical family}.
\end{proof}

\begin{remark}\label{rem:trivial legs}
Note that the identity in Proposition~\ref{prop:rel class Drinfeld}
can be rewritten as 
\begin{align*}
&(\id_{\lie h}\tensor\delta_{\lie g})(\mathbf f)_{2,3,5}-
(\id_{\lie h}\tensor\delta_{\lie g})(\mathbf f)_{4,1,5}-
(\delta_{\lie h}\tensor\id_{\lie g})(\mathbf f)_{2,4,5}
+[\mathbf f_{4,1},\mathbf f_{4,5}]+[\mathbf f_{4,5},\mathbf f_{2,5}]+[\mathbf f_{2,5},\mathbf f_{2,3}]\\
&=(\id_{\lie h}\tensor\delta_{\lie g})(\mathbf f)_{6,1,3}-(\delta_{\lie h}\tensor \id_{\lie g})(\mathbf f)_{6,2,3}
-(\delta_{\lie h}\tensor \id_{\lie g})(\mathbf f)_{4,6,1}\\
&\mskip400mu+
[\mathbf f_{4,1},\mathbf f_{6,1}]+[\mathbf f_{6,1},\mathbf f_{6,3}]+[\mathbf f_{6,3},
\mathbf f_{2,3}].
\end{align*}
The left hand side is contained in~$(U(\lie g)\tensor U(\lie h))^{\tensor 2}\tensor U(\lie g)\tensor 1_{U(\lie h)}$, while the right hand side is contained in 
$(U(\lie g)\tensor U(\lie h))^{\tensor 2}\tensor 1_{U(\lie g)}\tensor U(\lie h)$.
It follows that this reduces to an identity in~$(U(\lie g)\tensor U(\lie h))^{\tensor 2}$.
\end{remark}
\begin{proposition}\label{prop:tensortwist}
Let~$(\lie g,\delta_{\lie g})$, $(\lie h,\delta_{\lie h})$ be 
Lie bialgebras and let~$ j _{\lie g}$, $ j _{\lie h}$
be respective classical Drinfeld twists. Abbreviate $\delta'_{\lie g}=
\widetilde{\delta_{\lie g}}_{ j _{\lie g}}$, 
$\delta'_{\lie h}=\widetilde{\delta_{\lie h}}_{ j _{\lie h}}$
and let~$\delta'_{\lie g\oplus\lie h}$ be the Lie cobracket on~$\lie g\oplus\lie h$
obtained from Lie cobrackets $\delta'_{\lie g}$ and~$\delta'_{\lie h}$.
Let~$ j _{\lie g\oplus\lie h}\in (\lie g\tensor \lie h)^{\tensor 2}$
and set~$ j '_{\lie g\oplus\lie h}:= j _{\lie g\oplus\lie h}-( j _{\lie g})_{1,3}-( j _{\lie h})_{2,4}$. Then
$\cybe{ j ^-_{\lie g\oplus\lie h}}_{\delta_{\lie h\oplus\lie h}}=
\cybe{( j '_{\lie g\oplus\lie h})^-}_{\delta'_{\lie h\oplus\lie h}}$.
In particular, $ j _{\lie g\oplus \lie h}$ is a classical Drinfeld 
twist for~$(\lie g\oplus\lie h,\delta_{\lie g\oplus\lie h})$ if and only if
$ j _{\lie g\oplus\lie h}'$ is
a classical Drinfeld twist for $(\lie g\oplus\lie h,\delta'_{\lie g\oplus\lie h})$.
\end{proposition}
\begin{proof}

Note that for all~$x\in\lie g$, $y\in\lie h$
\begin{align*}
\delta'_{\lie g\oplus\lie h}((x,y))&=\delta'_{\lie g}(x)_{1,3}
+\delta'_{\lie h}(y)_{2,4}=\delta_{\lie g}(x)_{1,3}+\delta_{\lie h}(y)_{2,4}
+
[ j ^-_{\lie g},\Delta(x)]_{1,3}+[ j ^-_{\lie h},\Delta(y)]_{2,4}\\
&=\delta_{\lie g\oplus\lie h}((x,y))+
[( j ^-_{\lie g})_{1,3}+( j ^-_{\lie h})_{2,4},\Delta_{U(\lie g)\tensor U(\lie h)}((x,y))].
\end{align*}
Thus,
\begin{align*}
(\delta'_{\lie g\oplus\lie h}\tensor&\id_{U(\lie g)\tensor U(\lie h)})( j ^-_{\lie g\oplus\lie h}-( j ^-_{\lie g})_{1,3}-( j ^-_{\lie h})_{2,4})
\\
&=(\delta_{\lie g\oplus\lie h}\tensor\id_{U(\lie g)\tensor U(\lie h)})( j ^-_{\lie g\oplus\lie h})-(\delta_{\lie g}\tensor\id_{\lie g})( j ^-_{\lie g})_{1,3,5}
-(\delta_{\lie h}\tensor\id_{\lie h})( j ^-_{\lie h})_{2,4,6}\\
&\phantom{=}
+[( j ^-_{\lie g})_{1,3}+( j ^-_{\lie h})_{2,4},(\Delta_{U(\lie g)\tensor U(\lie h)}\tensor\id_{U(\lie g)\tensor U(\lie h)})( j ^-_{\lie g\oplus\lie h})]\\
&\phantom{=}-[( j ^-_{\lie g})_{1,3},( j ^-_{\lie g})_{1,5}+( j ^-_{\lie g})_{3,5}]-[( j ^-_{\lie h})_{2,4},( j ^-_{\lie h})_{2,6}+( j ^-_{\lie g})_{4,6}]
\end{align*}
and
\begin{align*}
(\id_{U(\lie g)\tensor U(\lie h)}&\tensor\delta'_{\lie g\oplus\lie h})( j ^-_{\lie g\oplus\lie h}-( j ^-_{\lie g})_{1,3}-( j ^-_{\lie h})_{2,4})
\\
&=(\id_{U(\lie g)\tensor U(\lie h)}\tensor\delta_{\lie g\oplus\lie h})( j ^-_{\lie g\oplus\lie h})-(\id_{\lie g}\tensor\delta_{\lie g})( j ^-_{\lie g})_{1,3,5}
-(\id_{\lie h}\tensor\delta_{\lie h})( j ^-_{\lie h})_{2,4,6}\\
&\phantom{=}
+[( j ^-_{\lie g})_{3,5}+( j ^-_{\lie h})_{4,6},(\id_{U(\lie g)\tensor U(\lie h)}\tensor\Delta_{U(\lie g)\tensor U(\lie h)})( j ^-_{\lie g\oplus\lie h})]\\
&\phantom{=}-[( j ^-_{\lie g})_{3,5},( j ^-_{\lie g})_{1,3}+( j ^-_{\lie g})_{1,5}]-[( j ^-_{\lie h})_{4,6},( j ^-_{\lie h})_{2,4}+( j ^-_{\lie g})_{2,6}].
\end{align*}
Therefore
\begin{align*}
((\delta'&_{\lie g\oplus\lie h}\tensor\id_{U(\lie g)\tensor U(\lie h)})-
(\id_{U(\lie g)\tensor U(\lie h)}\tensor\delta'_{\lie g\oplus\lie h})-
\tau_{\mathbf 2,\mathbf 3}(\delta'_{\lie g\oplus\lie h}\tensor\id_{U(\lie g)\tensor U(\lie h)}))(( j '_{\lie g\oplus\lie h})^-)\\
&=((\delta_{\lie g\oplus\lie h}\tensor\id_{U(\lie g)\tensor U(\lie h)})-
(\id_{U(\lie g)\tensor U(\lie h)}\tensor\delta_{\lie g\oplus\lie h})-
\tau_{\mathbf 2,\mathbf 3}(\delta_{\lie g\oplus\lie h}\tensor\id_{U(\lie g)\tensor U(\lie h)}))( j _{\lie g\oplus\lie h}^-)
\\
&\phantom{=}-(((\delta_{\lie g}\tensor\id_{\lie g})-
(\id_{\lie g}\tensor\delta_{\lie g})-
\tau_{2,3}(\delta_{\lie g}\tensor\id_{\lie g}))( j _{\lie g}^-))_{1,3,5}\\
&\phantom{=}-(((\delta_{\lie h}\tensor\id_{\lie h})-
(\id_{\lie h}\tensor\delta_{\lie h})-
\tau_{2,3}(\delta_{\lie h}\tensor\id_{\lie h}))( j _{\lie h}^-))_{2,4,6}\\
&\phantom{=}
+[( j ^-_{\lie g})_{1,3}+( j ^-_{\lie h})_{2,4},(\Delta_{U(\lie g)\tensor U(\lie h)}\tensor\id_{U(\lie g)\tensor U(\lie h)})( j ^-_{\lie g\oplus\lie h})]\\
&\phantom{=}-[( j ^-_{\lie g})_{1,3},( j ^-_{\lie g})_{1,5}+( j ^-_{\lie g})_{3,5}]-[( j ^-_{\lie h})_{2,4},( j ^-_{\lie h})_{2,6}+( j ^-_{\lie g})_{4,6}]]\\
&\phantom{=}
-[( j ^-_{\lie g})_{1,5}+( j ^-_{\lie h})_{2,6},\tau_{\mathbf 2,\mathbf 3}(\Delta_{U(\lie g)\tensor U(\lie h)}\tensor\id_{U(\lie g)\tensor U(\lie h)})( j ^-_{\lie g\oplus\lie h})]\\
&\phantom{=}+[( j ^-_{\lie g})_{1,5},( j ^-_{\lie g})_{1,3}-( j ^-_{\lie g})_{3,5}]+[( j ^-_{\lie h})_{2,6},( j ^-_{\lie h})_{2,4}-( j ^-_{\lie g})_{4,6}]\\
&\phantom{=}-[( j ^-_{\lie g})_{3,5}+( j ^-_{\lie h})_{4,6},(\id_{U(\lie g)\tensor U(\lie h)}\tensor\Delta_{U(\lie g)\tensor U(\lie h)})( j ^-_{\lie g\oplus\lie h})]\\
&\phantom{=}+[( j ^-_{\lie g})_{3,5},( j ^-_{\lie g})_{1,3}+( j ^-_{\lie g})_{1,5}]+[( j ^-_{\lie h})_{4,6},( j ^-_{\lie h})_{2,4}+( j ^-_{\lie g})_{2,6}]\\
&=\cybe{ j ^-_{\lie g\oplus\lie h}}_{\delta_{\lie g\oplus\lie h}}-
\cybe{ j ^-_{\lie g\oplus\lie h}}-(\cybe{ j ^-_{\lie g}}_{\delta_{\lie g}})_{1,3,5}-(\cybe{ j ^-_{\lie h}}_{\delta_{\lie h}})_{2,4,6}\\
&\phantom{=}-\cybe{ j ^-_{\lie g}}_{1,3,5}-\cybe{ j ^-_{\lie h}}_{1,3,5}\\
&\phantom{=}
+[( j ^-_{\lie g})_{1,3}+( j ^-_{\lie h})_{2,4},(\Delta_{U(\lie g)\tensor U(\lie h)}\tensor\id_{U(\lie g)\tensor U(\lie h)})( j ^-_{\lie g\oplus\lie h})]\\
&\phantom{=}
-[( j ^-_{\lie g})_{1,5}+( j ^-_{\lie h})_{2,6},\tau_{\mathbf 2,\mathbf 3}(\Delta_{U(\lie g)\tensor U(\lie h)}\tensor\id_{U(\lie g)\tensor U(\lie h)})( j ^-_{\lie g\oplus\lie h})]\\
&\phantom{=}-[( j ^-_{\lie g})_{3,5}+( j ^-_{\lie h})_{4,6},(\id_{U(\lie g)\tensor U(\lie h)}\tensor\Delta_{U(\lie g)\tensor U(\lie h)})( j ^-_{\lie g\oplus\lie h})]\\
&=\cybe{ j ^-_{\lie g\oplus\lie h}}_{\delta_{\lie g\oplus\lie h}}-
\cybe{ j ^-_{\lie g\oplus\lie h}}
-\cybe{ j ^-_{\lie g}}_{1,3,5}-\cybe{ j ^-_{\lie h}}_{1,3,5}\\
&\phantom{=}
+\cyb{( j ^-_{\lie g})_{1,3}+( j ^-_{\lie h})_{2,4}}{ j ^-_{\lie g\oplus\lie h}}-
[( j ^-_{\lie g})_{1,5}+( j ^-_{\lie h})_{2,6},(\Delta_{U(\lie g)\tensor U(\lie h)}\tensor\id_{U(\lie g)\tensor U(\lie h)})( j ^-_{\lie g\oplus\lie h})]\\
&\phantom{=}
-[( j ^-_{\lie g})_{1,5}+( j ^-_{\lie h})_{2,6},\tau_{\mathbf 2,\mathbf 3}(\Delta_{U(\lie g)\tensor U(\lie h)}\tensor\id_{U(\lie g)\tensor U(\lie h)})( j ^-_{\lie g\oplus\lie h})]\\
&\phantom{=}+\cyb{ j ^-_{\lie g\oplus\lie h}}{( j ^-_{\lie g})_{1,3}+( j ^-_{\lie h})_{2,4}}-[(\id_{U(\lie g)\tensor U(\lie h)}\tensor\Delta_{U(\lie g)\tensor U(\lie h)})( j ^-_{\lie g\oplus\lie h}),( j ^-_{\lie g})_{1,5}+( j ^-_{\lie h})_{2,6}]\\
&=\cybe{ j ^-_{\lie g\oplus\lie h}}_{\delta_{\lie g\oplus\lie h}}-
\cybe{ j ^-_{\lie g\oplus\lie h}-( j ^-_{\lie g})_{1,3}-( j ^-_{\lie h})_{2,4}}\\
&\phantom{=}-[( j ^-_{\lie g})_{1,5}+( j ^-_{\lie h})_{2,6},((\Delta_{U(\lie g)\tensor U(\lie h)}\tensor\id_{U(\lie g)\tensor U(\lie h)})
+\tau_{\mathbf 2,\mathbf 3}(\Delta_{U(\lie g)\tensor U(\lie h)}\tensor\id_{U(\lie g)\tensor U(\lie h)})\\
&\phantom{=-[( j ^-_{\lie g})_{1,5}+( j ^-_{\lie h})_{2,6},}-(\id_{U(\lie g)\tensor U(\lie h)}\tensor\Delta_{U(\lie g)\tensor U(\lie h)}))
( j ^-_{\lie g\oplus\lie h})].
\end{align*}
We need the following
\begin{lemma}
Let~$\lie l$ be a Lie algebra and let~$s\in\lie l\tensor\lie l$.
Then $((\Delta\tensor\id_{\lie l})+\tau_{2,3}(\Delta\tensor\id_{\lie l})
-(\id_{\lie l}\tensor\Delta))(s^-)=0$ in~$U(\lie l)^{\tensor 3}$.
\end{lemma}
\begin{proof}
Indeed, 
\begin{align*}
((\Delta&\tensor\id_{\lie l})+\tau_{2,3}(\Delta\tensor\id_{\lie l})
-(\id_{\lie l}\tensor\Delta))(s^-)
=s^-_{1,3}+s^-_{2,3}+s^-_{1,2}+s^-_{3,2}-
s^-_{1,2}-s^-_{1,3}\\
&=s^-_{2,3}+s^-_{3,2}=0.\qedhere 
\end{align*}
\end{proof}
It follows from the Lemma applied to~$\lie l=\lie g\oplus\lie h$ and~$s=
 j _{\lie g\oplus\lie h}$ that 
\begin{multline*}
((\delta'_{\lie g\oplus\lie h}\tensor\id_{U(\lie g)\tensor U(\lie h)})-
(\id_{U(\lie g)\tensor U(\lie h)}\tensor\delta'_{\lie g\oplus\lie h})
-\tau_{\mathbf 2,\mathbf 3}(\delta'_{\lie g\oplus\lie h}\tensor\id_{U(\lie g)\tensor U(\lie h)}))(( j '_{\lie g\oplus\lie h})^-)\\
=\cybe{ j ^-_{\lie g\oplus\lie h}}_{\delta_{\lie g\oplus \lie h}}-
\cybe{( j '_{\lie g\oplus\lie h})^-}
\end{multline*}
or
$
\cybe{ j ^-_{\lie g\oplus\lie h}}_{\delta_{\lie g\oplus \lie h}}=
\cybe{( j '_{\lie g\oplus\lie h})^-}_{\delta'_{\lie g\oplus \lie h}}$, 
which is the first assertion of the Proposition. The second is immediate
from the first and the definition of a classical Drinfeld twist.
\end{proof}
\begin{corollary}\label{cor:clas twist from rel twist}
Let~$(\lie g,\delta_{\lie g})$ and~$(\lie h,\delta_{\lie h})$ be Lie bialgebras.
Suppose 
that~$\mathbf f\in\lie h\tensor \lie g$ is a relative classical Drinfeld twist for~$(\lie g\oplus\lie h,\delta_{\lie g\oplus \lie h})$ and that $\delta_{\lie g}$ (respectively, $\delta_{\lie h}$)
is obtained by twisting another cobracket $\delta'_{\lie g}$ (respectively, $\delta'_{\lie h})$ by some classical Drinfeld twist~$ j _{\lie g}$ (respectively, $ j _{\lie h}$). Then $\mathbf f_{2,3}+( j _{\lie g})_{1,3}+( j _{\lie h})_{2,4}$ 
 a Drinfeld twist for $(\lie g\oplus \lie h,\delta'_{\lie g\oplus\lie h})$.
\end{corollary}

\subsection{The dual picture: Poisson algebras}\label{subs:Poisson alg}
Consider~$U(\lie g)^*$ with the algebra structure defined by the convolution product.
The natural left and right $U(\lie g)$-actions on~$ U(\lie g)^*$ are given by
$$
(x\triangleright f)(u)=f(u x),\quad 
(f\triangleleft x)(u)=f(x u)
$$
for all~$x\in \lie g$, $f\in   U(\lie g)^* $ and~$u\in U(\lie g)$.
For an algebraic group~$G$ with~$\operatorname{Lie}(G)=\lie g$,
$U(\lie g)$ and~$\kk[G]$ admit a natural Hopf pairing via 
$\langle u,f\rangle=(u\triangleright f)(1_G)$ and so~$\kk[G]$ identifies
with a subalgebra of~$U(\lie g)^*$. The natural left 
and right actions of~$\lie g$ on~$\kk[G]$ are given, respectively, by
\begin{equation}\label{eq:g act k[G]}
(x\triangleright f)(g)=\frac{d}{dt}\Big|_{t=0}f(g\exp(tx)),
\qquad (f\triangleleft x)(g)=\frac{d}{dt}\Big|_{t=0}
f(\exp(tx)g)
\end{equation}
for all~$f\in\kk[G]$, $g\in G$ and~$x\in \lie g$.

If~$(\lie g,\delta)$ is a Lie bialgebra then~$\delta$ induces a Poisson structure on~$ U(\lie g)^*$
via
$$
\{f,f'\}:=(f\tensor f')\circ \delta,\qquad f,f'\in  U(\lie g)^*.
$$
Indeed, \ref{cnd.1} and~\ref{cnd.2} imply that~$\{\cdot,\cdot\}$
is skew-symmetric and satisfies the Jacobi identity. Furthermore, for all
$f,f',f''\in  U(\lie g)^*$
we have by Lemma~\ref{lem:coLeib}
\begin{align*}
\{f\cdot f',f''\}&=(f\cdot f'\tensor f'')\circ\delta=(f\tensor f'\tensor f'')\circ (\Delta\tensor\id_{U(\lie g)})\circ \delta
\\
&=(f\tensor f'\tensor f'')\circ(\id\tensor\delta)\circ\Delta+(f\tensor f''\tensor f')\circ (\delta\tensor\id)\circ\Delta
\\
&=f\cdot \{f',f''\}+\{f,f''\}\cdot f'.
\end{align*}

\begin{proposition}\label{prop:pois twist}
Let~$(\lie g,\delta)$ be a Lie bialgebra and let~$\{\cdot,\cdot\}$ be
the Poisson bracket on~$ U(\lie g)^*$ induced by~$\delta$.
Let $j \in\lie g\tensor\lie g$
be a weak classical Drinfeld twist. Then~$\{\cdot,\cdot\}_{ j }: U(\lie g)^*\tensor  U(\lie g)^*\to 
 U(\lie g)^*$ given by
\begin{equation}\label{eq:pois twist}
\{f,f'\}_{ j }=\{f,f'\}+\mu(  j ^-\bowtie (f\tensor f')),\qquad f,f'\in  U(\lie g)^*,
\end{equation}
where~$\mu:  U(\lie g)^* \tensor   U(\lie g)^* \to   U(\lie g)^* $ is the multiplication map and $$(x\tensor y)\bowtie (f\tensor f'):=(f\triangleleft x)\tensor (f'\triangleleft y)-(x\triangleright f)\tensor (y\triangleright f'),
\quad 
f,f'\in   U(\lie g)^* ,\, x,y\in\lie g,
$$
is a Poisson bracket on~$  U(\lie g)^* $. In particular, if $(\lie g,\delta)$
is quasi-triangular with a classical r-matrix~$r$ then
$$
\{f,f'\}_{ j }=\mu(((r+ j ^-)\bowtie (f\tensor f')),
\qquad f,f'\in   U(\lie g)^* 
$$
and restricts to a Poisson bracket on~$\kk[G]$.
\end{proposition}
\begin{proof}
We need the following 
\begin{lemma}
Let~$r\in\lie g\tensor\lie g$. Then
$$
(f\tensor f')\circ\delta_{r}=\mu(r\bowtie (f\tensor f')),\quad f,f'\in   U(\lie g)^* .
$$
\end{lemma}
\begin{proof}
For any~$u\in U$,
\begin{align*}
(f\tensor f')&(\delta_{r}(u))=
(f\tensor f')(r\cdot \Delta(u)-\Delta(u)\cdot r)\\
&=((f\tensor f')\triangleleft r)(\Delta(u))-
(r\triangleright (f\tensor f'))(\Delta(u))=\mu(r\bowtie (f\tensor f'))(u). \qedhere
\end{align*}
\end{proof}
Since~$ j $ is a weak classical Drinfeld twist, $\tilde\delta_{ j }=\delta+\delta_{ j ^-}$ is a cobracket and hence induces a Poisson bracket on~$  U(\lie g)^* $. We claim that this bracket coincides with~$\{\cdot,\cdot\}_{ j }$. 
Indeed, $(f\tensor f')\circ \tilde\delta_{ j }
=(f\tensor f')\circ \delta+(f\tensor f')\circ \delta_{ j ^-}=
\{f,f'\}+(f\tensor f')\circ\delta_{ j ^-}$ for all~$f,f'\in   U(\lie g)^* $, and it remains to
apply the Lemma. The second assertion follows from the first and Proposition~\ref{prop:cnd classical twist}\ref{prop:cnd classical twist.b}.
\end{proof}

\subsection{Quasi-triangular bialgebras}
Let~$B$ be a bialgebra. Recall (see e.g.~\cites{Drinf,CPbook,RSTS,Majid}) that~$B$ is called quasi-triangular if there is an invertible element $R$ 
in a (suitable completion~$B\wh\tensor B$ of) $B\tensor B$ satisfying 
$$
\Delta^{op}(b)=R\Delta(b)R^{-1},\qquad b\in B
$$
where $\Delta^{op}=\tau_{1,2}\circ\Delta$,
and
$$
(\Delta\tensor \id_B)(R)=R_{13}R_{23},\quad 
(\id_B\tensor\Delta)(R)=R_{13}R_{12}.
$$
Note that~$(\varepsilon_B\tensor\id_B)(R),(\id_B\tensor\varepsilon_B)(R)\in Z(B)$.
It is well-known (see e.g.~\cite{RSTS}*{Proposition~2.2}) that if~$B$ is quasi triangular then
$R$ satisfies the quantum Yang-Baxter equation (QYBE)
\begin{equation}
R_{12}R_{13}R_{23}=R_{23}R_{13}R_{12}.\label{eq:QYBE}
\end{equation}
More generally, we have the following
\begin{proposition}\label{prop:cQYBE}
Let~$C$ be a set and let 
$\{R^{(c)}\}_{c\in C}$ be a family of R-matrices for the 
same bialgebra~$(B,\Delta)$. Then for all~$i\not=j\not=k\in[n]$
\begin{equation}\label{eq:cQYBE}
R_{jk}^{(c)}R_{ik}^{(c')}R_{ij}^{(c'')}=
R_{ij}^{(c'')}R_{ik}^{(c')}R_{jk}^{(c)},\qquad 
c'\in\{c,c''\}\subset C.
\end{equation} 
\end{proposition}
\begin{proof}
Let~$c,c''\in C$. Since the symmetric group~$S_n$
acts on~$B^{\tensor n}$ by algebra automorphisms, it suffices to prove~\eqref{eq:cQYBE}
for~$(i,j,k)=(1,2,3)$.
We have
\begin{align*}
R_{23}^{(c)}R_{13}^{(c)}&=\tau_{1,2}(R_{13}^{(c)}R_{23}^{(c)})=\tau_{1,2}((\Delta\tensor\id_B)(R^{(c)}))\\
&=R_{12}^{(c'')}(\Delta\tensor\id_B)(R^{(c)})(R_{12}^{(c'')})^{-1}=R_{12}^{(c'')}R_{13}^{(c)}R_{23}^{(c)}(R_{12}^{(c'')})^{-1},
\end{align*}
which implies~\eqref{eq:cQYBE} for~$c'=c$. The identity for~$c'=c''$ is proved similarly by using~$\id_B\tensor\Delta^{op}$.
\end{proof}
The most basic example of such a family is provided by the following
\begin{lemma}\label{lem:Rop inv}
Let~$B$ be a quasi-triangular bialgebra with an R-matrix~$R$.
Then~$\tau_{1,2}(R)^{-1}$ is also an R-matrix for the same comultiplication. In particular, if $R^{(1)}=R$ and~$R^{(-1)}=\tau_{1,2}(R)^{-1}$ then
$$
R_{23}^{(\epsilon)}R_{13}^{(\epsilon')}R_{12}^{(\epsilon'')}
=R_{12}^{(\epsilon'')}R_{13}^{(\epsilon')}R_{23}^{(\epsilon)}
$$
provided that~$\epsilon'\in\{\epsilon,\epsilon''\}\subset\{1,-1\}$.
\end{lemma}
\begin{proof}
Let~$b\in B$. Then 
$$
\tau_{1,2}(R)^{-1}\Delta(b)=\tau_{1,2}(R^{-1}\Delta^{op}(b))=\tau_{1,2}(\Delta(R^{-1}))=\Delta^{op}(b)
\tau_{1,2}(R)^{-1}.
$$
Furthermore, 
\begin{align*}
(\Delta\tensor\id_B)(\tau_{1,2}(R)^{-1})&=((\Delta\tensor\id_B)(\tau_{1,2}(R))^{-1}=(\tau_{2,3}\tau_{1,2}(\id_B\tensor \Delta)(R))^{-1}\\
&=\tau_{2,3}\tau_{1,2}(R_{12}^{-1}R_{13}^{-1})=(\tau_{1,2}(R)^{-1})_{12} (\tau_{1,2}(R)^{-1})_{23}.
\end{align*}
The remaining identity is proved similarly.
\end{proof}
\begin{lemma}
Let~$B$ be a quasi-triangular bialgebra with an R-matrix~$R$ and let $\sigma:B\to B$ be a bialgebra automorphism. Then~$(\sigma\tensor\sigma)(R)$
is also an R-matrix for~$B$.
\end{lemma}
\begin{proof}
Let~$R^\sigma=(\sigma\tensor\sigma)(R)$.
Then for any~$b\in B$
\begin{align*}
R^\sigma &\Delta(b)(R^{\sigma})^{-1}=
(\sigma\tensor\sigma)(R\Delta(\sigma^{-1}(b))R^{-1})
=(\sigma\tensor\sigma)\Delta^{op}(\sigma^{-1}(b))
=\Delta^{op}(b),
\end{align*}
where we used Lemma~\ref{lem:op coalg end}. 

Furthermore,
\begin{align*}
(\Delta\tensor\id_B)(R^\sigma)=(\sigma\tensor\sigma\tensor\sigma)((\Delta\tensor\id_B)(R))
=(\sigma\tensor\sigma\tensor\sigma)(R_{13}R_{23})=R_{13}^\sigma R_{23}^{\sigma},
\end{align*}
and similarly for the remaining identity.
\end{proof}

\subsection{Drinfeld twists}
Let~$B$ be a bialgebra with the comultiplication~$\Delta$. We say
that an invertible~$J\in B\wh \tensor B$ is a (right) {\em weak Drinfeld twist} if $\Delta_J:B\to B\wh\tensor B$, where  $\Delta_J(b)=J^{-1}\Delta(b) J$, $b\in B$,
is a (topological)
comultiplication.

We will need the following standard facts (see for example~\cite{Majid}*{Theorem~2.3.4}). We provide proofs here for the reader's convenience
and to introduce the notation that will be used later.
\begin{proposition}\label{prop:Drinfeld twist}
Let~$B$ be a bialgebra with the comultiplication~$\Delta:B\to B\tensor B$ and the counit~$\varepsilon:B\to\kk$.
Let $J$ be an invertible element of (a suitable completion 
$B\wh\tensor B$ of) $B\tensor B$. Then
\begin{enumalph}
    \item \label{prop:Drinfeld twist.a} 
$J$ is a weak Drinfeld twist if and only if 
$(\Delta\tensor \id_B)(J)\cdot (J\tensor  J^{-1})\cdot (\id_B\tensor\Delta)(J^{-1})$
centralizes $(\Delta\tensor\id_B)\Delta(B)$
in~$B^{\tensor 3}$
and $(\varepsilon\tensor\id_B)(J),(\id_B\tensor\varepsilon)(J)\in Z(B)$.
In particular, if $(\varepsilon\tensor\id_B)(J),(\id_B\tensor\varepsilon)(J)\in Z(B)$ and 
\begin{equation}
(\Delta\tensor \id_B)(J)\cdot (J\tensor 1)=(\id_B\tensor\Delta)(J)\cdot (1\tensor J)
\label{eq:Drinfeld twist}    
\end{equation}
then~$J$ is a weak Drinfeld twist. 

\item\label{prop:Drinfeld twist.b} 
Suppose that~$B$ is quasi-triangular with an R-matrix $R$.
If~$J$ satisfies~\eqref{eq:Drinfeld twist} 
then $B$ is a (topological) quasi-triangular bialgebra with respect to~$\Delta_J$ with the R-matrix $R_J=J_{21}^{-1}RJ$.
\item\label{prop:Drinfeld twist.c}
Let $\sigma$ be a bialgebra automorphism of~$B$. If~$J$
satisfies~\eqref{eq:Drinfeld twist} then so does~$(\sigma\tensor\sigma)(J)$.
\end{enumalph}
\end{proposition}
\begin{proof}
Clearly, $\Delta_J$ is a homomorphism of algebras. 
Abbreviate $J_{12,3}=(\Delta\tensor \id_B)(J)$ and~$J_{1,23}=(\id_B\tensor\Delta)(J)$. In this notation~\eqref{eq:Drinfeld twist} becomes
\begin{equation}\label{eq:Drinfeld twist'}
J_{12,3}J_{12}=J_{1,23}J_{23}.
\end{equation}
We have, for all~$x\in B$
\begin{align*}
(\Delta_J\tensor &\id_B)\Delta_J(x)=
(\Delta_J\tensor \id_H)(J^{-1}\Delta(x)J)\\
&=(J^{-1}\tensor 1)\cdot (\Delta\tensor \id_B)(J^{-1}\Delta_H(h)J)\cdot (J\tensor 1)\\
&=(J^{-1}\tensor 1)\cdot ((\Delta\tensor\id_B)(J^{-1}))\cdot (\Delta\tensor\id_B)\Delta(x)\cdot 
(\Delta\tensor\id_B)(J)(J\tensor 1)\\
&=J_{12}^{-1}J_{12,3}^{-1}(\Delta\tensor\id_B)\Delta(x)
J_{12,3}J_{12}
\\
\intertext{and similarly}
(\id_B\tensor &\Delta_J)\Delta_J(x)=
(\id_B\tensor \Delta_J)(J^{-1}\Delta(x)J)\\
&=(1\tensor J^{-1})\cdot (\id_B\tensor \Delta)(J^{-1}\Delta(x)J)\cdot (1\tensor J)\\
&=(1\tensor J^{-1})\cdot ((\id_B\tensor\Delta)(J^{-1}))\cdot (\id_B\tensor\Delta)\Delta(x)\cdot 
(\id_B\tensor\Delta)(J)(1\tensor J)\\
&=J_{23}^{-1}J_{1,23}^{-1}(\id_B\tensor\Delta)\Delta(x)
J_{1,23}J_{23}.
\end{align*}
Since~$\Delta$ is coassociative, it follows that~$\Delta_J$
is coassociative if and only $J_{12,3}J_{12}J_{23}^{-1}J_{1,23}^{-1}$
commutes with~$(\Delta\tensor\id_B)\Delta(x)$
for all~$x\in B$.
We also have, for all~$x\in B$, 
$(\varepsilon\tensor\id_B)(J^{-1}\Delta(x)J)=
(\varepsilon\tensor \id_B)(J^{-1})x (\varepsilon\tensor\id_B)(J)$, which is equivalent to~$(\varepsilon\tensor\id_B)(J)\in Z(B)$. The remaining
identity is proven similarly.

To prove~\ref{prop:Drinfeld twist.b}, note first that, for all $x\in B$, 
\begin{align*}
R_J\Delta_J(x)&=J_{21}{}^{-1} R\Delta(x)J=J_{21}^{-1}\Delta^{op}(x)R J=\Delta^{op}_J(x) R_J.
\end{align*}
Furthermore, we need to prove that
\begin{equation}\label{eq:RJ qtriang}
(\Delta_J\tensor \id_B)(R_J)=(R_J)_{13}(R_J)_{23},\quad 
(\id_B\tensor\Delta_J)(R_J)=(R_J)_{13}(R_J)_{12}.
\end{equation}
We have 
\begin{align*}
(\Delta_J\tensor \id_B)(R_J)&=J_{12}^{-1}(\Delta\tensor \id_B)(J_{21}^{-1}R J)J_{12}=J_{12}^{-1}(\Delta\tensor \id_B)(J_{21}^{-1})R_{13}R_{23}J_{12,3}J_{12}\\
&=J_{12}^{-1}(\Delta\tensor \id_B)(J_{21}^{-1})R_{13}R_{23}J_{1,23}J_{23}
\end{align*}
while~$
(R_J)_{13}(R_J)_{23}=J_{31}^{-1}R_{13}J_{13}J_{32}^{-1}R_{23}J_{23}
$.
Thus, the first identity in~\eqref{eq:RJ qtriang} is equivalent to
$$
J_{12}^{-1}(\Delta\tensor \id_B)(J_{21}^{-1})R_{13}R_{23}J_{1,23}=J_{31}^{-1}R_{13}J_{13}J_{32}^{-1}R_{23}
$$
or
\begin{equation}\label{eq:interm 0}
J_{12}^{-1}((\Delta\tensor \id_B)(J_{21}))^{-1}R_{13}J_{1,32}=J_{31}^{-1}R_{13}J_{13}J_{32}^{-1}
\end{equation}
where $J_{1,32}:=(\id_B\tensor\Delta^{op})(J)$; here we used that $R\Delta(x)R^{-1}=\Delta^{op}(x)$ for all $x\in B$.
We also have, by~\eqref{eq:Drinfeld twist'}
$
J_{1,32}J_{32}=\tau_{2,3}(J_{1,23}J_{23})
=\tau_{2,3}(J_{12,3}J_{12})=\tau_{2,3}(J_{12,3})J_{13}$,
whence~\eqref{eq:interm 0} is equivalent to
$$
J_{12}^{-1}((\Delta\tensor \id_B)(J_{21}))^{-1}R_{13}\tau_{2,3}(J_{12,3})=J_{31}^{-1}R_{13}
$$
Since~$R_{13}=\tau_{2,3}(R_{12})$,
the last identity is equivalent to
$$
J_{12}^{-1}((\Delta\tensor \id_B)(J_{21}))^{-1}\tau_{2,3}(R_{12}J_{12,3}R_{12}^{-1})=J_{31}^{-1}
$$
or 
$$
J_{12}^{-1}((\Delta\tensor \id_B)(J_{21}))^{-1}\tau_{2,3}(J_{21,3})=J_{31}^{-1}
$$
where~$J_{21,3}=(\Delta^{op}\tensor\id_B)(J)=\tau_{1,2}(J_{12,3})$.
This in turn is equivalent to
\begin{equation}\label{eq:intermed-I}
\tau_{2,3}\tau_{1,2}(J_{12,3})J_{31}=(\Delta\tensor\id_B)(J_{21})J_{12}.
\end{equation}
Furthermore,
$
(\Delta\tensor \id_B)(J_{21})=\tau_{2,3}\tau_{1,2}((\id_B\tensor\Delta)(J))=\tau_{2,3}\tau_{1,2}(J_{1,23})
$, 
and so~\eqref{eq:intermed-I} is equivalent to
$$
\tau_{2,3}\tau_{1,2}(J_{12,3})J_{31}=\tau_{2,3}\tau_{1,2}(J_{1,23})J_{12}.
$$
Applying~$\tau_{1,2}\tau_{2,3}$ to both sides we conclude that~\eqref{eq:intermed-I} and hence
the first identity in~\eqref{eq:RJ qtriang} are equivalent to
$$
J_{12,3}\tau_{1,2}\tau_{2,3}(J_{31})=J_{1,23}\tau_{1,2}\tau_{2,3}(J_{12})
$$
which is manifestly equivalent
to~\eqref{eq:Drinfeld twist'}. The argument for the second identity in~\eqref{eq:RJ qtriang} is similar and is omitted.

To prove part~\ref{prop:Drinfeld twist.c}, denote~$J^\sigma=(\sigma\tensor\sigma)(J)$. Then
$$
(\Delta\tensor\id_B)(J^\sigma)
\cdot(J^\sigma\tensor 1)
=(\sigma\tensor\sigma\tensor\sigma)(
(\Delta\tensor \id_B)(J)\cdot (J\tensor 1)).
$$
Similarly, 
$$
(\id_B\tensor\Delta)(J^\sigma)\cdot(1\tensor J^\sigma)
=(\sigma\tensor\sigma\tensor\sigma)((\id_H\tensor\Delta)(J)\cdot(1\tensor J)).
$$
Thus, $J^{\sigma}$ satisfies~\eqref{eq:Drinfeld twist}. 
Furthermore, $(\varepsilon\tensor\id_B)(J^\sigma)=(\varepsilon\circ\sigma\tensor \sigma)(J)=\sigma((\varepsilon\tensor\id_B)(J))$ which is central in~$B$ since~$\sigma $ is an automorphism of~$B$ and~$(\varepsilon\tensor\id_B)(J)$ is central. 
\end{proof}

\begin{remark}\label{rem:left twist}
In the literature, the term ``Drinfeld twist'' is also used for the {\em left} Drinfeld twist, that is, an invertible element in (a suitable completion of) $B\tensor B$ such that $\Delta'_J:B\to B\tensor B$ defined by
$\Delta'_J(b)=J \Delta(b) J^{-1}$, $b\in B$, is a comultiplication. 
Clearly, $J$ is a (right) Drinfeld twist if and only if~$J^{-1}$
is a left Drinfeld twist, and if~$B$ is quasi-triangular
with an R-matrix $R$ then $R$ is a left Drinfeld twist.
\end{remark}
From now on, we refer to a weak Drinfeld twist satisfying~\eqref{eq:Drinfeld twist} as a {\em Drinfeld twist}. 

\subsection{Relative Drinfeld twists}\label{subs:rel twist}
Let~$A$ and~$B$ be bialgebras with respective comultiplications $\Delta_A$ and~$\Delta_B$ and 
let~$F$ in (a suitable completion~$B\wh\tensor A$ of) $B\tensor A$ be invertible. We say that~$F$ is a {\em relative Drinfeld twist} if $1_A\tensor F\tensor 1_B$ is a Drinfeld twist for~$A\tensor B$ with respect to its standard comultiplication~$\Delta_{A\tensor B}$.
The following is a quantum analogue of Proposition~\ref{prop:rel class Drinfeld}.
\begin{proposition}\label{prop:rel Drinfeld twist}
Let $(A,\Delta_A)$, $(B,\Delta_B)$ be 
bialgebras and let~$F\in B\tensor A$ be invertible. Then $F$
is a relative Drinfeld twist if and only if
$(\varepsilon_B\tensor\id_A)(F)\in Z(A)$, $(\id_B\tensor\varepsilon_A)(F)\in Z(B)$
and
\begin{equation}\label{eq:rel Drinfeld twist}
F_{13,4}F_{12}=F_{1,24}F_{34}
\end{equation}
in~$B\tensor A\tensor B\tensor A$ where~$F_{13,4}=
[(\Delta_B\tensor \id_A)(F)]_{1,3,4}$ and~$
F_{1,24}=
[(\id_B\tensor \Delta_A)(F)]_{1,2,4}$.
\end{proposition}
\begin{proof}
The identity
$$
(\Delta_{A\tensor B}\tensor \id_{A\tensor B})(F_{23})(F_{23}\tensor 1_{A\tensor B})=
(\id_{A\tensor B}\tensor\Delta_{A\tensor B})(F_{23})(1_{A\tensor B}\tensor F_{23})
$$
in~$(A\tensor B)^{\tensor 3}$ is equivalent to 
$$
[(\Delta_B\tensor\id_A)(F)]_{2,4,5}F_{23}=
[(\id_B\tensor\Delta_A)(F)]_{2,3,5}F_{45}
$$
where both sides are in $1_A\tensor B\tensor A\tensor B\tensor A\tensor 1_B\subset (A\tensor B)^{\tensor 3}$. This 
is clearly equivalent to~\eqref{eq:rel Drinfeld twist}. The remaining assertions are straightforward.
\end{proof}
Note that $F_{12}\in B\tensor A\tensor 1_B\tensor 1_A$ and $F_{34}\in 
1_B\tensor 1_A\tensor B\tensor A$ commute in $(B\tensor A)^{\tensor 2}$. Thus, \eqref{eq:rel Drinfeld twist} is equivalent to 
\begin{equation}\label{eq:rel twist divided form}
F_{13,4}F_{34}^{-1}=F_{1,24}F_{12}^{-1}
\end{equation}
It follows that if~$F$ is a relative Drinfeld twist then $F_{13,4}F_{34}^{-1}$ and
$F_{1,24}F_{12}^{-1}$ are in $B\tensor 1_A\tensor 1_B\tensor A$ and thus~\eqref{eq:rel twist divided form} is essentially an identity in~$B\tensor A$. It should be noted that, by Remark~\ref{rem:trivial legs},  the corresponding identity in the classical case encompasses the entire~$(U(\lie g)\tensor U(\lie h))^{\tensor 2}$ and cannot be collapsed to just two of its tensor factors.

The following recovers~\cite{RSTS}*{Theorem~2.9 and~(2.26)}.
\begin{proposition}\label{prop:basic twist}
Let~$B$ be quasi-triangular with an R-matrix~$R$. Then 
$J_{\epsilon}=1\tensor R^{(\epsilon)}\tensor 1$, $\epsilon\in\{1,-1\}$ in
the notation from Lemma~\ref{lem:Rop inv}, is a Drinfeld twist for~$B^{\tensor 2}\tensor B^{\tensor 2}$. In particular, for any~$\epsilon,\epsilon',\epsilon''\in\{1,-1\}$,
$B^{\tensor 2}$
with the comultiplication twisted by~$J_\epsilon$ is quasi-triangular
with the R-matrix $R_{14}^{(-\epsilon)}R_{13}^{(\epsilon')} R_{24}^{(\epsilon'')} R_{23}^{(\epsilon)}$.
\end{proposition}
\begin{proof}
Let~$\epsilon=1$.
We apply Proposition~\ref{prop:rel Drinfeld twist} with~$A=B$ and $F=R$. Since
$F_{13,4}=(R_{13}R_{23})_{(1,3,4)}=R_{14}R_{34}$, while $F_{1,24}=
(R_{13}R_{12})_{(1,2,4)}=R_{14}R_{12}$,  
it follows that
$F_{13,4}F_{34}^{-1}=R_{14}=F_{1,24}F_{12}^{-1}$
and so~\eqref{eq:rel twist divided form} holds. 
Since~$(\varepsilon\tensor\id_B)(R)=(\id_B\tensor \varepsilon)(R)$ are central, the second condition in Proposition~\ref{prop:rel Drinfeld twist} follows. It remains to apply Proposition~\ref{prop:Drinfeld twist}\ref{prop:Drinfeld twist.b} and Lemma~\ref{lem:Rop inv},
taking into account that in this case~$J_{\mathbf{21}}=J_{(3,4,1,2)}$ is obtained from~$J$ by 
applying $\tau_{1,3}\tau_{2,4}$. The case~$\epsilon=-1$ is similar with~$R$ replaced
by~$R_{21}^{-1}$, which is also an~R-matrix for~$B$ by Lemma~\ref{lem:Rop inv}.
\end{proof}
\begin{remark}
Taking~$\epsilon=\epsilon''=1$, $\epsilon'=-1$ yields~\cite{RSTS}*{(2.6)}.
\end{remark}

The next two results are quantum analogues of, respectively, Proposition~\ref{prop:tensortwist} and
Corollary~\ref{cor:clas twist from rel twist}.
\begin{lemma}\label{lem:rel twist from twist}
Let~$A$ (respectively, $B$) be unital bialgebras with
respective comultiplications $\Delta_A$, $\Delta_B$. Let~$J_A$, $J_B$ be Drinfeld twists on~$(A,\Delta_A)$, $(B,\Delta_B)$, respectively. Let~$\Delta'_A$, $\Delta'_B$ be comultiplications 
twisted by, respectively, $J_A$ and~$J_B$. 
Then
$J_{A\tensor B}$ is a Drinfeld twist for~$(A\tensor B,\Delta_{A\tensor B})$ if and only if~$J'_{A\tensor B}:=(J_A)_{1,3}^{-1}
(J_B)_{2,4}^{-1}J_{A\tensor B}$ is a Drinfeld twist for~$(A\tensor B,\Delta'_{A\tensor B})$.
\end{lemma}
\begin{proof}
We have
\begin{align*}
(\Delta'&_{A\tensor B}\tensor\id_{A\tensor B})(J'_{A\tensor B})(J'_{A\tensor B}\tensor 1_A\tensor 1_B)\\&=
(J_A^{-1})_{1,3}(J_B^{-1})_{2,4}(\Delta_{A\tensor B}\tensor 
\id_{A\tensor B})((J_A^{-1})_{1,3}(J_B^{-1})_{2,4}J_{A\tensor B})
(J_{A\tensor B}\tensor 1_A\tensor 1_B)\\
&=[(J_A^{-1}\tensor 1_A)\cdot(\Delta_A\tensor\id_A)(J_A^{-1})]_{(1,3,5)}[(J_B^{-1}\tensor 1_B)\cdot (\Delta_B\tensor\id_B)(J_B^{-1})]_{(2,4,6)}\cdot\\
&\mskip300mu
(\Delta_{A\tensor B}\tensor 
\id_{A\tensor B})(J_{A\tensor B})
(J_{A\tensor B}\tensor 1_A\tensor 1_B)\\
&=[(1_A \tensor J_A^{-1})\cdot(\id_A\tensor\Delta_A)(J_A^{-1})]_{(1,3,5)}[(1_B\tensor J_B^{-1})\cdot (\id_B\tensor\Delta_B)(J_B^{-1})]_{(2,4,6)}\cdot\\
&\mskip300mu
(\id_{A\tensor B}\tensor 
\Delta_{A\tensor B})(J_{A\tensor B})
(1_A\tensor 1_B\tensor J_{A\tensor B})\\
&=[(1_A \tensor J_A^{-1})\cdot(\id_A\tensor\Delta_A)(J_A^{-1})]_{(1,3,5)}[(1_B\tensor J_B^{-1})\cdot (\id_B\tensor\Delta_B)(J_B^{-1})]_{(2,4,6)}\cdot\\
&\mskip300mu
(\id_{A\tensor B}\tensor 
\Delta_{A\tensor B})(J_{A\tensor B})
(1_A\tensor 1_B\tensor J_{A\tensor B})\\
&=(J_A^{-1})_{3,5}(J_B^{-1})_{4,6}(\id_{A\tensor B}\tensor\Delta_{A\tensor B})(J'_{A\tensor B})(1_A\tensor 1_B\tensor J_{A\tensor B})\\
&=(J_A^{-1})_{3,5}(J_B^{-1})_{4,6}(\id_{A\tensor B}\tensor\Delta_{A\tensor B})(J'_{A\tensor B})(J_A)_{3,5}(J_B)_{4,6}
(1_A\tensor 1_B\tensor J'_{A\tensor B})\\
&=(\id_{A\tensor B}\tensor\Delta'_{A\tensor B})(J'_{A\tensor B})(1_A\tensor 1_B\tensor J'_{A\tensor B}),
\end{align*}
where we used~\eqref{eq:Drinfeld twist} for~$J_A$, $J_B$
and~$J_{A\tensor B}$. 
Furthermore, 
\begin{align*}
(\varepsilon_{A\tensor B}&\tensor\id_{A\tensor B})(J'_{A\tensor B})=((\varepsilon_A\tensor\id_A)(J_A)\tensor
(\varepsilon_A\tensor\id_B)(J_B))\cdot \\
&\mskip300mu(\varepsilon_A\tensor\varepsilon_B\tensor\id_A\tensor\id_B)(J_{A\tensor B}).
\end{align*}
The first factor is contained in $Z(A)\tensor Z(B)\subset Z(A\tensor B)$ while the second is contained in~$Z(A\tensor B)$.
The assertion for $\id_{A\tensor B}\tensor\varepsilon_{A\tensor B}$ is proved similarly. The converse follows by interchanging the role of~$\Delta$ and~$\Delta'$ and using Remark~\ref{rem:left twist}.
\end{proof}
\begin{corollary}\label{cor:twist from rel twist}
Let~$A$ and~$B$ be bialgebras with respective comultiplications~$\Delta_A:A\to A\tensor A$ and~$\Delta_B:B\to B\tensor B$.  
Suppose 
that~$F\in B\tensor A$ is a relative Drinfeld twist for~$(A\tensor B,\Delta_{A\tensor B})$ and that~$\Delta_A$ (respectively, $\Delta_B$)
is obtained by twisting another comultiplication $\Delta'_A:A\to A\tensor A$ (respectively, $\Delta'_B:B\to B\tensor B$)
by some Drinfeld twist~$J_A\in A\tensor A$ (respectively, $J_B\in B\tensor B$).
Then $(J_A)_{13}(J_B)_{24}F_{23}$ is a Drinfeld twist for~$((A\tensor B)^{\tensor 2},\Delta'_{A\tensor B})$ where $\Delta'_{A\tensor B}=(\Delta'_A)_{1,3}\circ(\Delta'_B)_{2,4}$.
\end{corollary}
We conclude this section with a natural generalization
of the defining property of an R-matrix.
\begin{proposition}\label{prop:rel twist comult prop}
Let~$A$, $B$ be bialgebras with respective comultiplications $\Delta_A$, $\Delta_B$
and let~$F\in B\wh\tensor A$ be a relative Drinfeld twist. Let $U$
be a coalgebra with the comultiplication $\Delta_U$
and let $\psi_A:U\to A$, $\psi_B:U\to B$ be homomorphisms of coalgebras. Then $(\psi_A\tensor\psi_B)\circ \Delta_U:U\to A\tensor B$ is a homomorphism of coalgebras,
where the natural comultiplication on~$A\tensor B$ is twisted by~$F_{2,3}$, if and only if
$$
F\cdot((\psi_B\tensor\psi_A)\circ\Delta_U)(u))=(\tau\circ(\psi_A\tensor\psi_B)\circ\Delta_U(u))\cdot F,\qquad u\in U.
$$
\end{proposition}
\begin{proof}
Abbreviate~$\boldsymbol{\Delta}=(\psi_A\tensor\psi_B)\circ\Delta_U$. Then $\boldsymbol{\Delta}:U\to A\tensor B$ is
a homomorphism of coalgebras if and only if 
\begin{equation}\label{eq:Bdelt hom coalg}
(\Delta_{A\tensor B}(\boldsymbol{\Delta}(u)))\cdot F_{2,3}=F_{2,3}\cdot (\boldsymbol{\Delta}\tensor\boldsymbol{\Delta})(\Delta_U(u)),\qquad u\in U.
\end{equation}
Since
\begin{align*}
\Delta_{A\tensor B}(\boldsymbol{\Delta}(u))&=
\Delta_{A\tensor B}(\psi_A(u_{(1)})\tensor \psi_B(u_{(2)}))=
\Delta_A(\psi_A(u_{(1)}))_{1,3}\cdot 
\Delta_B(\psi_B(u_{(2)}))_{2,4}\\
&=(\psi_A\tensor\psi_A)(u_{(1)}\tensor u_{(2)})_{1,3}
\cdot(\psi_B\tensor\psi_B)(u_{(3)}\tensor u_{(4)})_{2,4}\\
&=\psi_A(u_{(1)})\tensor\psi_B(u_{(3)})\tensor \psi_A(u_{(2)})\tensor\psi_B(u_{(4)}),
\end{align*}
while
\begin{align*}
  (\boldsymbol{\Delta}\tensor\boldsymbol{\Delta})(\Delta_U(u))&=\boldsymbol{\Delta}(u_{(1)})\tensor 
  \boldsymbol{\Delta}(u_{(2)})
  =
  \psi_A(u_{(1)})\tensor \psi_B(u_{(2)})\tensor \psi_A(u_{(3)})\tensor \psi_B(u_{(4)}),
\end{align*}
it follows that~\eqref{eq:Bdelt hom coalg} is equivalent to
$$
F\cdot (\psi_B(u_{(1)})\tensor \psi_A(u_{(2)}))=
(\psi_B(u_{(2)})\tensor \psi_A(u_{(1)}))\cdot F,\qquad u\in U
$$
which is the assertion.
\end{proof}

\subsection{Duality in the quantum case}\label{subs:quant dual}
In this section we go over the dual analogues 
of various constructions discussed above. Since most 
arguments are very similar to those in previous
quantum sections, we omit most of them for the sake of brevity.

We say that
$\mathcal R\in\Hom_\kk(B\tensor B,\kk)$ is a 
{\em co-quasi-triangular structure} on a bialgebra~$B$
if~$\mathcal R$ is $\ast$-invertible
and
\begin{align}
&\mathcal R(a_{(1)},a'_{(1)})a_{(2)}a'_{(2)}=\mathcal R(a_{(2)},a'_{(2)})a'_{(1)}a_{(1)},\qquad a,a'\in B
\label{eq:coquas triang 1}\\
&\mathcal R\circ (m_B\tensor\id_B)=\mathcal R_{1,3}\ast \mathcal R_{2,3},\quad  
\mathcal R\circ (\id_B\tensor m_B)=\mathcal R_{1,3}\ast \mathcal R_{1,2},\label{eq:coquas triang 2}
\end{align}
where~$m_B:B\tensor B\to B$ is the multiplication map. Henceforth we will often use the ``bilinear form'' notation for elements of~$\Hom_\kk(B\tensor B,\kk)$.
A particularly important case is when~$\mathcal R$
is {\em counital}, that is
\begin{equation}\label{eq:counital}
\mathcal R(a,1)=\varepsilon(a)=\mathcal R(1,a),\qquad a\in B,
\end{equation}
which holds, for example, for quantized coordinate algebras on
reductive groups. 
The following is the co-quasi-triangular 
analogue of Proposition~\ref{prop:cQYBE}. 
\begin{proposition}\label{prop:dual cQYBE}
Let~$\{\mathcal R^{(c)}\}_{c\in C}\subset \Hom_\kk(B\tensor B,\kk)$ be a family 
of co-quasi-triangular structures on~$B$.
Then for all~$\{i,j,k\}=\{1,2,3\}$
\begin{equation}\label{eq:coquas cQYBE}
\mathcal R_{i,j}^{(c)}\ast \mathcal R_{i,k}^{(c')}
\ast \mathcal R_{j,k}^{(c'')}=\mathcal R_{j,k}^{(c'')}
\ast \mathcal R_{i,k}^{(c')}\ast \mathcal R_{i,j}^{(c)},\qquad 
c'\in \{c,c''\}\subset C.
\end{equation}
\end{proposition}
\begin{proof}
As before, it suffices to prove~\eqref{eq:coquas cQYBE} for $(i,j,k)=(1,2,3)$.
Let~$a,a',a''\in B$.
Applying~$\mathcal R^{(c'')}(-,a'')$
to the left hand side of~\eqref{eq:coquas triang 1} 
with~$\mathcal R=\mathcal R^{(c)}$ yields 
\begin{align*}
\mathcal R^{(c)}(&a_{(1)},a'_{(1)})\mathcal R^{(c'')}(a_{(2)}a'_{(2)},a'')=
\mathcal R^{(c)}(a_{(1)},a'_{(1)})\mathcal R^{(c'')}(a_{(2)},a''_{(1)})\mathcal R^{(c'')}(a'_{(2)},a''_{(2)})
\\
&=\mathcal R_{1,2}^{(c)}(a_{(1)}\tensor a'_{(1)}\tensor a''_{(1)})
\mathcal R_{1,3}^{(c'')}(a_{(2)}\tensor a'_{(2)}\tensor a''_{(2)})\mathcal R_{2,3}^{(c'')}(a_{(3)}\tensor a'_{(3)}\tensor a''_{(3)})
\\
&=(\mathcal R_{1,2}^{(c)}\ast \mathcal R_{1,3}^{(c'')}\ast 
\mathcal R_{2,3}^{(c'')})(a\tensor a'\tensor a''),
\end{align*}
where we used the first equality in~\eqref{eq:coquas triang 2}. 
On the other hand, applying~$\mathcal R^{(c'')}(-,a'')$
to the right hand side of~\eqref{eq:coquas triang 1}
with the same convention
we obtain 
\begin{align*}
\mathcal R^{(c)}(&a_{(2)},a'_{(2)})\mathcal R^{(c'')}(a'_{(1)}a_{(1)},a'')=\mathcal R^{(c)}(a_{(2)},a'_{(2)})\mathcal R^{(c'')}(a'_{(1)},a''_{(1)})\mathcal R^{(c'')}(a_{(1)},a''_{(2)})
\\
&=\mathcal R_{2,3}^{(c'')}(a_{(1)}\tensor a'_{(1)}\tensor a''_{(1)})\mathcal R_{1,3}^{(c'')}(a_{(2)}\tensor a'_{(2)}\tensor a''_{(2)})
\mathcal R_{1,2}^{(c)}(a_{(3)}\tensor a'_{(3)}\tensor a''_{(3)})\\
&=(\mathcal R_{2,3}^{(c'')}\ast\mathcal R_{1,3}^{(c'')}
\ast\mathcal R_{1,2}^{(c)})(a\tensor a'\tensor a''),
\end{align*}
which yields~\eqref{eq:coquas cQYBE} with~$c'=c''$. The
case~$c=c'$ is proven similarly
by applying $\mathcal R^{(c)}(a'',-)$ to both sides of \eqref{eq:coquas triang 1} with~$\mathcal R=\mathcal R^{(c'')}$ and using
the second equality in~\eqref{eq:coquas triang 2}.
\end{proof}
Like in the case of R-matrices, the basic example of such a family is provided by
\begin{lemma}\label{lem:basic coquas family}
Let~$B$ be a bialgebra and let~$\mathcal R\in\Hom_\kk(B\tensor B,\kk)$ be a co-quasi-triangular
structure. Then~$\mathcal R^{(-1)}:=\mathcal R^{\ast-1}\circ\tau$ is
also a co-quasi-triangular structure.
\end{lemma}
\begin{proof}
Using~\eqref{eq:coquas triang 1} with $a$ and~$a'$
interchanged and 
Lemma~\ref{lem:old mult from twisted} with~$C=B\tensor B$,
$\mathcal S=\mathcal R$, $V=B$,
$f=m_B$ and~$g=m_B\circ\tau$ we obtain for all $a,a'\in B$,
\begin{align*}
\mathcal R^{\ast-1}(a'_{(2)},a_{(2)})a'_{(1)}a_{(1)}
=\mathcal R^{\ast-1}(a'_{(1)},a_{(1)})a_{(2)}a'_{(2)},
\end{align*}
which is~\eqref{eq:coquas triang 1} for~$\mathcal R^{(-1)}$. Furthermore, $\mathcal R^{(-1)}_{1,3}\ast \mathcal R^{(-1)}_{2,3}=(\mathcal R^{\ast-1}_{1,2}\ast\mathcal R^{\ast-1}_{1,3})\circ\tau_{1,2}\circ\tau_{2,3}
=(\mathcal R_{1,3}\ast\mathcal R_{1,2})^{\ast-1}\circ
\tau_{1,2}\circ\tau_{2,3}$, and so to prove
that~$\mathcal R^{(-1)}$ satisfies the first 
identity in~\eqref{eq:coquas triang 2}, 
it suffices
to prove that~$\mathcal R^{(-1)}\circ (m_B\tensor\id_B)
\circ\tau_{2,3}\circ\tau_{1,2}=(\mathcal R_{1,2}\ast\mathcal R_{1,3})^{\ast-1}$, or,
equivalently, that~$\mathcal R^{\ast-1}\circ (\id_B\tensor m_B)=(\mathcal R_{1,3}\ast\mathcal R_{1,2})^{\ast-1}$. The latter
follows from~\eqref{eq:coquas triang 2} and~\eqref{eq:hom coalg inv} since~$\id_B\tensor m_B:B^{\tensor 3}\to B^{\tensor 2}$
is a homomorphism of coalgebras. The second identity in~\eqref{eq:coquas triang 2} for~$\mathcal R^{(-1)}$ is proven similarly.
\end{proof}

Note that any (restricted) dual~$B^\star\subset B^\smallsquare$ of~$B$ is also a bialgebra, the product being~$\ast$
while~$\Delta(f)(b\tensor b')=f(bb')$,
$f\in B^\star$, $b,b'\in B$ and~$\varepsilon(f)=f(1)$. 
\begin{lemma}\label{lem:natural pairing}
Let~$B$ be a finitely generated bialgebra, let
$R,R'\in B^{\wh\tensor n}$ and let~$B^\star\subset B^\smallsquare$ be a bialgebra. Define $f,f':B^\star{}^{\tensor n}\to\kk$ by, respectively,
$f(a):=a(R)$, $f'(a):=a(R')$, $a\in B^\star{}^{\tensor n}$. Then
$(f\ast f')(a)=a(RR')$ for all~$a\in B^\star{}^{\tensor n}$.
In particular, if~$R$ is invertible then~$f$ is~$\ast$-invertible and~$f^{\ast-1}(a)=a(R^{-1})$, $a\in B^{\star\tensor n}$.
\end{lemma}
\begin{proof}
Let~$a\in B^\star{}^{\tensor n}$ and 
write~$\Delta_{B^\star{}^{\tensor n}}(a)=a_{(1)}\tensor a_{(2)}$ in Sweedler notation.
Then~$(f\ast f')(a)=f(a_{(1)})f'(a_{(2)})=
a_{(1)}(R)a_{(2)}(R')=a(RR')$ for all~$a\in B^\star$. The
last assertion is immediate.
\end{proof}
The following is immediate from the definition,
Lemma~\ref{lem:natural pairing}
and Corollary~\ref{cor:complet tensor}.
\begin{lemma}
Let~$B$ be a (topological) quasi-triangular bialgebra with an R-matrix~$R$. Then for any (restricted) dual bialgebra~$B^\star\subset B^\smallsquare$,
$\mathcal R\in\Hom_\kk(B^\star\tensor B^\star,\kk)$
defined by~$\mathcal R(f,f')=(f\tensor f')(R)$,
$f,f'\in B^\star$, is a co-quasi-triangular structure.
\end{lemma}

We say that a $\ast$-invertible
~$\mathcal J\in\Hom_\kk(B\tensor B,\kk)$ is a {\em dual Drinfeld twist} if
\begin{align}
&\mathcal J(a_{(1)},1)a_{(2)}=\mathcal J(a_{(2)},1)a_{(1)},\qquad 
\mathcal J(1,a_{(1)})a_{(2)}=\mathcal J(1,a_{(2)})a_{(1)},\label{eq:dual Drinf eq 1}\\
&(\mathcal J\circ (m_B\tensor \id_B))\ast \mathcal J_{1,2}
=(\mathcal J\circ (\id_B\tensor m_B))\ast \mathcal J_{2,3}.\label{eq:dual Drinf eq 2}
\end{align}
The following is immediate from the definition.
\begin{lemma}
Let~$B$ be a (topological) bialgebra and let $J\in B\wh\tensor B$ be a Drinfeld twist. Let~$B^\star\subset B^\smallsquare$ be a (restricted) dual of~$B$.
Then
$\mathcal J\in\Hom_\kk(B^\star\tensor B^\star,\kk)$ defined by $\mathcal J(f,f')=(f\tensor f')(J)$,
$f,f'\in B^\star$,
is a dual Drinfeld twist. 
\end{lemma}
Define $\bullet_{\mathcal J}\in\Hom_\kk(B\tensor B,B)$ by
$$
a\bullet_{\mathcal J} b=\mathcal J^{\ast-1}(a_{(1)},b_{(1)})\mathcal J(a_{(3)},b_{(3)})
a_{(2)}b_{(2)},\qquad a,b\in B
$$
The following is well-known (see, e.g.~\cite{Majid}).
\begin{lemma}\label{lem:twisted product}
Let~$B$ be a bialgebra and let~$\mathcal J\in\Hom_\kk(B^{\tensor 2},\kk)$ be a dual Drinfeld twist. Then $B$ is a bialgebra with the multiplication~$\bullet_{\mathcal J}$
and with the comultiplication~$\Delta:B\to B\tensor B$.
\end{lemma}
\begin{proof}
We have, for all~$a\in B$,
$$
a\bullet_{\mathcal J}1=\mathcal J^{\ast-1}(a_{(1)},1)\mathcal J(a_{(3)},1)a_{(2)}
=\mathcal J^{\ast-1}(a_{(1)},1)\mathcal J(a_{(2)},1)a_{(3)}=\varepsilon(a_{(1)})a_{(2)}=a.
$$
Similarly, $1\bullet_{\mathcal J}a=a$. 
Furthermore, we have, for all $a,b\in B$
$$
(\Delta\tensor\id_B)\Delta(a\bullet_{\mathcal J} b)=
\mathcal J^{\ast-1}(a_{(1)},b_{(1)})\mathcal J(a_{(5)},b_{(5)})a_{(2)}b_{(2)}\tensor a_{(3)}b_{(3)}\tensor a_{(4)}b_{(4)}.
$$
Then, 
for all $a,b,c\in \mathcal J$,
\begin{align*}
(a\bullet_{\mathcal J} b)\bullet_{\mathcal J} c&=
\mathcal J^{\ast-1}((a\bullet_{\mathcal J}b)_{(1)},c_{(1)})
\mathcal J((a\bullet_{\mathcal J}b)_{(3)},c_{(3)})(a\bullet_{\mathcal J}b)_{(2)}c_{(2)}\\
&=\mathcal J^{\ast-1}(a_{(1)},b_{(1)})\mathcal J(a_{(5)},b_{(5)})
\mathcal J^{\ast-1}(a_{(2)}b_{(2)},c_{(1)})\mathcal J(a_{(4)}b_{(4)},c_{(3)})a_{(3)}b_{(3)}c_{(2)}\\
&=\mathcal J^{\ast-1}(b_{(1)},c_{(1)})\mathcal J^{\ast-1}(a_{(1)},b_{(2)}c_{(2)})
\mathcal J(b_{(5)},c_{(5)})\mathcal J(a_{(3)},b_{(4)}c_{(4)})a_{(2)}b_{(3)}c_{(3)}
\\
&=\mathcal J^{\ast-1}(a_{(1)},(b\bullet_{\mathcal J}c)_{(1)})
\mathcal J(a_{(3)},(b\bullet_{\mathcal J}c)_{(3)}) a_{(2)}(b\bullet_{\mathcal J}c)_{(2)}=a\bullet_{\mathcal J}(b\bullet_{\mathcal J}c).
\end{align*}
Finally, for all~$a,b\in B$
\begin{align*}
\Delta(a)\bullet_{\mathcal J}\Delta(b)&=
a_{(1)}\bullet_{\mathcal J}b_{(1)}\tensor 
a_{(2)}\bullet_{\mathcal J}b_{(2)}
\\
&=\mathcal J^{\ast-1}(a_{(1)},b_{(1)})
\mathcal J^{\ast-1}(a_{(4)},b_{(4)})\mathcal J(a_{(3)},b_{(3)})\mathcal J(a_{(6)},b_{(6)})a_{(2)}b_{(2)}\tensor a_{(5)}b_{(5)}\\
&=\mathcal J^{\ast-1}(a_{(1)},b_{(1)})
\mathcal J(a_{(4)},b_{(4)})a_{(2)}b_{(2)}\tensor a_{(3)}b_{(3)}=\Delta(a\bullet_{\mathcal J}b).\qedhere
\end{align*}
\end{proof}
Note the following analogue of Lemma~\ref{lem:rel twist from twist}.
\begin{lemma}\label{lem:dual rel twist from twist}
Let~$A$ (respectively, $B$) be unital bialgebras and let~$\mathcal J_A$, $\mathcal J_B$ be Drinfeld twists on~$A$ (respectively, $B$). 
Then $\mathcal J_{A\tensor B}:(A\tensor B)^{\tensor 2}\to \kk$ is a Drinfeld twist for~$A\tensor B$ if and only if~$\mathcal J'_{A\tensor B}:=(\mathcal J_A)_{1,3}^{\ast-1}
\ast (\mathcal J_B)_{2,4}^{\ast-1}\ast \mathcal J_{A\tensor B}$ is a Drinfeld twist for~$A\tensor B$
with the multiplication defined by~$(a\tensor b)\cdot (a'\tensor b')=(a\bullet_{\mathcal J_A} a')\tensor 
(b\bullet_{\mathcal J_B}b')$, $a,a'\in A$, $b,b'\in B$.
\end{lemma}
The argument mimics that in the proof of Lemma~\ref{lem:rel twist from twist} and is left to the reader as an exercise. 

Let~$A$, $B$ be bialgebras. We say that~$\mathcal F\in\Hom_\kk(B\tensor A,\kk)$ is a {\em relative dual Drinfeld twist} if~$\mathcal F_{2,3}\in\Hom_\kk((A\tensor B)^{\tensor 2},\kk)$
is a dual Drinfeld twist. 
The following is easily checked. 
\begin{lemma}\label{eq:}
Let~$A$ and~$B$ be bialgebras. 
A $\ast$-invertible~$\mathcal F\in\Hom_\kk(B\tensor A,\kk)$
is a relative dual Drinfeld twist is and only if
$\mathcal F(b_{(1)},1_A)b_{(2)}=\mathcal F(b_{(2)},1_A)b_{(1)}$ for all~$b\in B$,
$\mathcal F(1_B,a_{(1)})a_{(2)}=
\mathcal F(1_B,a_{(2)})a_{(1)}$ for all~$a\in A$ and
\begin{equation}\label{eq:rel twist eq'}
\mathcal F_{13,4}\ast \mathcal F_{1,2}=\mathcal F_{1,24}\ast \mathcal F_{3,4},
\end{equation}
where~$\mathcal F_{13,4}(b\tensor a\tensor b'\tensor a')=
\mathcal F(bb',a')\varepsilon(a)$ and~$\mathcal F_{1,24}(b\tensor a
\tensor b'\tensor a')=
\mathcal F(b,aa')\varepsilon(b')$, $a,a'\in A$, $b,b'\in B$.
\end{lemma}
\begin{lemma}\label{lem:basic rel dual}
Let~$A=B$ be a bialgebra and let~$\mathcal R\in\Hom_\kk(B\tensor B,\kk)$ be a co-quasi\-tri\-angular structure. Then~$\mathcal R$
is a relative dual Drinfeld twist.
\end{lemma}
\begin{proof}
We only need to check that~\eqref{eq:rel twist eq'} holds.
Indeed, for any~$a,a',b,b'\in B$ we have 
\begin{align*}
(\mathcal R_{13,4}\ast &\mathcal R_{1,2})(b\tensor a\tensor b'\tensor a')=\mathcal R_{13,4}(b_{(1)}\tensor a_{(1)}\tensor b'\tensor a')\mathcal R(b_{(2)}\tensor a_{(2)})\\
&=\mathcal R(b_{(1)}b',a')\mathcal R(b_{(2)},a)
=\mathcal R(b_{(1)},a'_{(1)})\mathcal R(b',a'_{(2)})
\mathcal R(b_{(2)},a)\\
&=\mathcal R(b,aa'_{(1)})\mathcal R(b',a'_{(2)})
=\mathcal R_{1,24}(b\tensor a\tensor b'_{(1)}\tensor a'_{(1)})\mathcal R(b'_{(2)},a'_{(2)})\\
&=(\mathcal R_{1,24}\ast \mathcal R_{3,4})(b\tensor a\tensor b'\tensor a').\qedhere
\end{align*}
\end{proof}
The following analogue
of Corollary~\ref{cor:twist from rel twist} is immediate from Lemma~\ref{lem:dual rel twist from twist}.
\begin{corollary}\label{cor:dual twist from rel twist}
Let~$A$ and~$B$ be bialgebras and suppose
that~$\mathcal F\in\Hom_\kk(B\tensor A,\kk)$ is a relative dual Drinfeld twist where multiplications 
on~$A$ and~$B$ are twisted by~$\mathcal J_A\in\Hom_\kk(A\tensor A,\kk)$ and~$\mathcal J_B\in 
\Hom_\kk(B\tensor B,\kk)$, respectively. 
Then $\mathcal J_{\mathcal F}:=(\mathcal J_A)_{1,3}\ast (\mathcal J_B)_{2,4}\ast \mathcal F_{2,3}$ is a dual Drinfeld twist for~$A\tensor B$ 
with respect to the non-twisted multiplication and
$$
(a\tensor b)\bullet_{\mathcal J_{\mathcal F}}(a'\tensor b')=(\bullet_{\mathcal J_A}\tensor\bullet_{\mathcal J_B})
(\id_A\tensor \Psi_{\mathcal F}\tensor \id_B)(a\tensor b\tensor a'\tensor b'),
$$
where~$\Psi_{\mathcal F}:B\tensor A\to A\tensor B$ is defined by
$$
\Psi_{\mathcal F}(b\tensor a)=\mathcal F^{\ast-1}(b_{(1)},a_{(1)})\mathcal F(b_{(3)},a_{(3)})a_{(2)}\tensor b_{(2)},\qquad a\in A,\,b\in B.
$$
\end{corollary}
We conclude this section with a natural generalization of~\eqref{eq:coquas triang 1} and the counterpart of Proposition~\ref{prop:rel twist comult prop}.
\begin{proposition}\label{prop:twist hom alg}
Let~$A$, $B$ be bialgebras and let~$\mathcal F\in\Hom_\kk(B\tensor A,\kk)$ be a relative dual Drinfeld twist. Let~$U$ be a unital algebra with
the multiplication~$m_U:U\tensor U\to U$ and 
let~$\varphi_A:A\to U$, $\varphi_B:B\to U$ be homomorphisms of unital algebras.
Then~$m_U\circ(\varphi_A\tensor\varphi_B)$ is 
homomorphism of algebras
$(A\tensor B,\bullet_{\mathcal F_{2,3}})\to U$ if and only if
\begin{equation}
\mathcal F(b_{(1)},a_{(1)})\varphi_B(b_{(2)})\varphi_A(a_{(2)})=
\mathcal F(b_{(2)},a_{(2)})\varphi_A(a_{(1)})\varphi_B(b_{(2)}),
\qquad a\in A,\,b\in B.\label{eq:cnd hom F}
\end{equation}
\end{proposition}
\begin{proof}
Note that $\mathbf m:=m_U\circ(\varphi_A\tensor\varphi_B)$ is a homomorphism of algebras
if and only if
$$
\mathbf m((a'\tensor b)\bullet(a\tensor b'))=
\varphi_A(a')\varphi_B(b)\varphi_A(a)\varphi_B(b'),
\quad a,a'\in A,\,b,b'\in B.
$$
Using Corollary~\ref{cor:dual twist from rel twist}
with trivial twists~$\mathcal J_A$ and~$\mathcal J_B$ we obtain
\begin{align*}
\mathbf m((a'\tensor b)\bullet(a\tensor b'))&=
\mathcal F^{\ast-1}(b_{(1)},a_{(1)})\mathcal F(b_{(3)},a_{(3)})
\mathbf m(a'a_{(2)}\tensor b_{(2)}b')
\\
&=\mathcal F^{\ast-1}(b_{(1)},a_{(1)})\mathcal F(b_{(3)},a_{(3)})
\varphi_A(a'a_{(2)})\varphi_B(b_{(2)}b').
\end{align*}
Thus, $\mathbf m$ is a homomorphism of algebras if and only if
$$
\varphi_A(a')\mathbf m(\Psi_{\mathcal F}(b\tensor a))
\varphi(b')=\varphi_A(a')\varphi_B(b)\varphi_A(a)\varphi_B(b')
$$
for all $a,a'\in A$, $b,b'\in B$ or, equivalently,
if and only if
$$
\varphi_B(b)\varphi_A(a)=\mathcal F^{\ast-1}(b_{(1)},a_{(1)})\mathcal F(b_{(3)},a_{(3)})
\varphi_A(a_{(2)})\varphi_B(b_{(2)}),\qquad a\in A,\,b\in B,
$$
which in turn is equivalent to~\eqref{eq:cnd hom F} by
Lemma~\ref{lem:old mult from twisted} applied with~$C=B\tensor A$, $\mathcal S=\mathcal F$,
$V=U$, 
$f=m_U\circ (\varphi_B\tensor\varphi_A)$ and~$g=m_U\circ\tau\circ (\varphi_B\tensor\varphi_A)$.
\end{proof}

\section{Combinatorial background}

\subsection{Transitivity and permutations}\label{subs:trans}
Let~$n\in\ZZ_{>0}$.
We say that~$S\subset [n]\times [n]$ is {\em transitive} if $(i,j),(j,k)\in S$
implies that~$(i,k)\in S$ (see e.g.~\cite{ABGJ}*{\S4.2}). We 
can identify $[n]\times[n]$ with a complete bioriented graph~$\overset{\leftrightarrow\circlearrowleft}{K_n}$ with vertex set~$[n]$
and with loops. In this language, a transitive subset corresponds 
to a subgraph of~$\overset{\leftrightarrow\circlearrowleft}{K_n}$ which contains the oriented edge $i\to k$
provided that it contains oriented edges $i\to j$ and~$j\to k$, $i,j,k\in[n]$.

Let~$S\subset [n]\times[n]$ be transitive and let~$C$ be any set. We say that 
$\mathbf c:S\to C$
is {\em transitive} (cf.~\cite{ABGJ}*{Definition~1.1} and~\cite{ABGJepr}*{Definition~1.2}) if 
$\mathbf c(i,k)\in \{ \mathbf c(i,j),\mathbf c(j,k)\}$ for all $(i,j),(j,k)\in S$. Clearly, if $S'\subset S$ is transitive
and~$\mathbf c:S\to C$ is transitive, then so is~$\mathbf c|_{S'}$.

Given a permutation~$w\in S_n$, denote~$\Inv(w)=\{ (i,j)\,:\,
1\le i<j\le n,\, w(i)>w(j)\}$. In particular, we abbreviate 
$I_n=\Inv(w_\circ)=\{ (i,j)\,:\, 1\le i<j\le n\}$ where 
$w_\circ$ is the longest element of~$S_n$.
It is well-known
that~$S\subset I_n$ is equal to~$\Inv(w)$ for some~$w\in S_n$
if and only if both~$S$ and~$I_n\setminus S$ are transitive. Moreover, the map~$w\mapsto \Inv(w)$, $w\in S_n$, is injective. 

The following was established in~\cite{ABGJ}*{Observation~4.5}.
\begin{lemma}\label{lem:trans sign}
For any~$w\in S_n$, $\sgna{w}:I_n\to \{1,-1\}$ 
defined by
$\sgna{w}(i,j)=\sign(w(j)-w(i))$, $1\le i<j\le n$
(cf.~\eqref{eq:sgna defn}) is transitive. Moreover, the assignments 
$w\mapsto \sgna{w}$, $w\in S_n$ define a bijection from~$S_n$
onto the set of all transitive maps $I_n\to \{1,-1\}$.
\end{lemma}
\begin{lemma}\label{lem:In trans to trans}
Let~$w\in S_n$, $\boldsymbol{\alpha}:[n]\to \{1,-1\}$ and define $\boldsymbol\epsilon(w,\boldsymbol{\alpha}):[n]\times[n]\to \{1,-1\}$ by
$$
\boldsymbol\epsilon(w,\boldsymbol{\alpha})(i,j)=\delta_{i,j}\alpha(i)+
\sign(w(j)-w(i)),\qquad i,j\in [n]
$$
(cf.~\eqref{eq:eps w alpha defn}).
Then~$\boldsymbol\epsilon(w,\boldsymbol{\alpha})$ is transitive. Moreover,
all transitive maps~$\mathbf c:[n]\times[n]\to\{1,-1\}$ satisfying
$\mathbf c(j,i)=-\mathbf c(i,j)$, $(i,j)\in I_n$, are obtained this way.
In particular, there are $n! 2^n$ such maps.
\end{lemma}
\begin{proof}
Abbreviate $\epsilon_{ij}=\boldsymbol\epsilon(w,\boldsymbol{\alpha})(i,j)$,
$i,j\in [n]$. We need to prove that~$\epsilon_{ik}\in\{\epsilon_{ij},
\epsilon_{jk}\}$ for all~$i,j,k\in [n]$. Note that~$\epsilon_{ij}=-\epsilon_{ji}$
if~$i\not=j$. 

If~$i=j=k$ then there is nothing to prove. Suppose that~$|\{i,j,k\}|=2$. If~$i\not=k$ then~$j\in\{i,k\}$ and the assertion is obvious.
If~$i=k$ then~$i\not=j$ and so~$\{\epsilon_{ij},\epsilon_{jk}\}=\{\epsilon_{ij},\epsilon_{ji}\}
=\{1,-1\}\ni\epsilon_{ik}$.

Assume now that~$|\{i,j,k\}|=3$. Suppose first that $i<k$. If~$i<j<k$ then~$\epsilon_{ik}\in\{\epsilon_{ij},\epsilon_{jk}\}$ by Lemma~\ref{lem:trans sign}.
If~$k<j$ then~$\epsilon_{ij}\in \{\epsilon_{ik},\epsilon_{kj}\}$ by Lemma~\ref{lem:trans sign}. Thus either
$\epsilon_{ik}=\epsilon_{ij}$ or~$\epsilon_{ik}=-\epsilon_{ij}=-\epsilon_{kj}
=\epsilon_{jk}$. 
If~$j<i$ then~$\epsilon_{jk}\in\{\epsilon_{ji},\epsilon_{ik}\}$ by Lemma~\ref{lem:trans sign} and so either~$\epsilon_{ik}=\epsilon_{jk}$
or~$\epsilon_{ik}=-\epsilon_{jk}=-\epsilon_{ji}=\epsilon_{ij}$.
Finally, if~$k<i$ then $\epsilon_{ki}\in \{ \epsilon_{kj},\epsilon_{ji}\}$ by the above and so~$\epsilon_{ik}=-\epsilon_{ki}\in \{-\epsilon_{ji},-\epsilon_{kj}\}=
\{\epsilon_{ij},\epsilon_{jk}\}$.

To prove the converse, note that~$\mathbf c:[n]\times[n]\to \{1,-1\}$
satisfying~$\mathbf c(j,i)=-\mathbf c(i,j)$, $(i,j)\in I_n$,  
is uniquely determined by $(\mathbf c(i,i))_{i\in[n]}$ and
by~$\mathbf c|_{I_n}\to \{1,-1\}$ which is transitive. It remains to apply Lemma~\ref{lem:trans sign}.
\end{proof}

\subsection{Admissible functors and polynomiality}\label{subs:polynom}
Let~$M$ be a set and 
denote by $\Part(M)$ the set of all partitions of~$M$ into nonempty disjoint subsets. Given~$m\in M$ and~$\mathcal P\in 
\Part(M)$, let $\mathcal P[m]$ be the necessarily unique set in~$\mathcal P$
containing~$m$. 
Clearly, if~$M$ is finite then $\Part(M)=\bigsqcup\limits_{k\ge 1} \Part_k(M)$ where  $\Part_k(M)$ is the set of all 
$\PP\in \Part(M)$ with exactly $k$ parts.

Given sets $M$ and $C$,  
denote~$\Map(M,C)$ the set of all maps from~$M$ to~$C$ and let $\Inj(M,C)\subset \Map(M,C)$
be the set of all injective maps from~$M$
to~$C$.
Clearly,
each~$f\in \Map(M,C)$ defines a partition $\PP_f=\{f^{-1}(c)\}_{c\in f(M)}\in\Part(M)$.

Denote $\Map_k(M,C):=\{\varepsilon\in \Map(M,C) : |\varepsilon(M)|=k\}$ for any  sets $M$, $C$ and $k\ge 1$. Clearly,  $\Map_k(M,C)=\emptyset$ if and only if $k>|C|$ and if either $M$ or $C$ is finite, then $\Map(M,C)=\bigsqcup\limits_{k\ge 1} \Map_k(M,C)$. Note that $\PP_\varepsilon\in \Part_k(M)$ if and only $\varepsilon\in \Map_k(M,C)$.
Given~$\EE\subset\Map(M,C)$ and~$k\in\ZZ_{>0}$ we denote $$
\EE_k:=\EE\cap \Map_k(M,C),\qquad \Part_k^\EE(M):=\{\PP_\varepsilon\in \Part_k(M) : \varepsilon \in \EE\}.
$$
It is immediate from definitions that $\Part_k^\EE(M)=\Part_k^{\EE_k}(M)$ and $\Part_k^\EE(M)\subset \Part_k^{\EE'}(M)$ whenever $\EE_k\subset \EE'_k\subset \Map_k(M,C)$.

\begin{lemma} 
\label{le:injective equality}
$\Part_k^\EE(M)=\Part_k^{f\circ \EE}(M)$ for any $\EE\subset \Map(M,C)$ and any $f\in \Inj(C,C')$, where we abbreviated $f\circ \EE=\{f\circ \varepsilon : \varepsilon \in \EE\}$. 
\end{lemma} 
\begin{proof}
It is sufficient to observe that if~$f\in\Inj(C,C')$ then $\PP_\varepsilon=\PP_{f\circ \varepsilon}$ for any $\varepsilon\in \Map(M,C)$.
\end{proof}

\begin{proposition}\label{prop:general enumeration} For any finite sets $M$ and $C$, $\EE\subset \Map(M,C)$, one has
\begin{align}
\label{eq:general enumeration}
\begin{split}
& |\EE|=\sum_{k\ge 1} |\Part_k^\EE(M)|\cdot (|C|)_k,
\end{split}
\end{align}
where $(x)_k := x(x-1) \cdots (x-k+1)$. 
\end{proposition}

\begin{proof} 

For any sets $M,C$ and any $\mathcal P\in \Part(M)$ and any $\EE\subset \Map(M,C)$ denote
\begin{equation}
\EE_\PP=\{\varepsilon \in \EE  :  \PP_\varepsilon = \PP\} \ .
\end{equation}
Clearly, $\EE_\PP= \emptyset$ for $\PP\in \Part_k(M)$ if and only if $\Part_k^\EE(M)=\emptyset$.

\begin{lemma} $|\EE_\PP|=(|C|)_k$ for any finite $M$ and $C$, $\EE\subset \Map(M,C)$ and $\mathcal P\in \Part_k^\EE(M)$, $k\ge 1$. 
\end{lemma}

\begin{proof}
Given $\PP\in \Part_k(M)$, fix a surjective map $\pi:M\twoheadrightarrow [k]$ such that $\PP_\pi=\PP$. 
Clearly, $\varepsilon\in \EE_\PP$  if and only if~$\varepsilon=\underline \varepsilon\circ \pi$ for some $\underline \varepsilon\in  \Inj([k],C)$. Thus, the assignments $\varepsilon\mapsto \underline \varepsilon$ define a 
bijection $\EE_\PP\xrightarrow{\sim} \Inj([k],C)$.
Finally, note that $|\Inj([k],C)|=(|C|)_k$ for any finite set $C$.
\end{proof}
Since~$\mathcal E=\bigsqcup_{\mathcal P\in\Part(M)} \EE_\PP$, we have 
\begin{align*}|\EE| &=\sum\limits_{\PP\in \Part(M)}|\EE_\PP|=\sum\limits_{k\ge 1} \sum\limits_{\PP\in \Part_k^\EE(P)} (|C|)_k=\sum\limits_{k\ge 1} |\Part_k^\EE(M)|\cdot (|C|)_k,
\end{align*}
which is~\eqref{eq:general enumeration}.
\end{proof}

We now describe a categorical framework for constructing $\EE(C)\subset \Map(M,C)$ for a given $M$ and various $C$ such that $\Part_k^{\EE(C)}(M)$ does not depend on $C$. If this is the case then $|\EE(C)|$ is a polynomial in $|C|$ by Proposition~\ref{prop:general enumeration}. More precisely, $|\EE(C)|=\sum_{k\ge 1} 
a_{M,k}(|C|)_k$ where~$a_{M,k}\in\ZZ_{\ge 0}$
and $a_{M,k}=0$ for~$k\gg0$.

Given a category  ${\mathcal C}$  and $M\in \Ob({\mathcal C})$, we say that a  family $\{F(C)\subset \Hom_{\mathcal C}(M,C)\,:\,C\in \Ob({\mathcal C})\}$ 
is an $M$-{\it admissible} functor  $F:{\mathcal C}\to \mathbf{Set}$ if 
\begin{equation}
\label{eq:functoriality of F}
f\circ F(C)\subset F(C')
\end{equation}
for any $f\in \Hom_{\mathcal C}(C,C')$ and any $C,C'\in \Ob({\mathcal C})$.

Let $\mathbf{Set}_0$ be a full subcategory of $\mathbf{Set}$ containing all the $[c]$, $c\in \ZZ_{\ge 0}$ and closed with respect to subsets.
Given a finite set $M$ and an $M$-admissible functor $F:{\bf Set_0}\to \mathbf{Set}$  we denote
${\tstirling{M}{k}}_F:=|\Part_k^{F([k])}(M)|$, $\quad k\ge 1$. Clearly, $\tstirling{M}{k}_F=0$ if~$k>|M|$
and~$\tstirling{M}{1}_F=1$. Define
$$
p_F=  \displaystyle\sum\limits _{k\ge 1}  {\stirling{M}{k}}_F\cdot (x)_k\in \ZZ_{\ge 0}[x].
$$

\begin{theorem}
\label{th:polynomial} 
Let~$M$ be a finite set and $F$ be an $M$-admissible functor $\mathbf{Set}_0\to \mathbf{Set}$. Then
$|F(C)|=p_F(|C|)$ 
for any finite $C\in \Ob(\mathbf{Set}_0)$.
\end{theorem}

\begin{proof}We need the following result.
\begin{proposition} 
\label{prop:F-invariance}
Let $F:\mathbf{Set}_0\to \mathbf{Set}$ be any $M$-admissible functor. Then
$$\Part^{F(C)}_k(M)=\Part_k^{F([k])}(M)$$ 
for all $C\in \Ob(\mathbf{Set}_0)$ such that $|C|\ge k$.
\end{proposition}
\begin{proof} 
We need the following
\begin{lemma} Let~$F$ be an $M$-admissible functor. Then
$F(C)_k=\bigcup\limits_{f\in \Inj([k],C)} f\circ F([k])$ for all $C\in \Ob(\mathbf{Set}_0)$, $k\ge 1$.
\end{lemma}
\begin{proof} Denote  $F(C)'_k:=\bigcup\limits_{f\in \Inj([k],C)} f\circ F([k])$. Clearly, $F(C)'_k\subset F(C)_k$ by  \eqref{eq:functoriality of F} which guarantees that $f\circ F([k])\subset F(C)_k$ for any $f\in \Inj([k],C)$. 

To prove the opposite inclusion, given $\varepsilon\in F(C)_k$, choose $f\in \Inj([k],C)$ such that $f([k])=\varepsilon(M)$ and let $g$ be a (necessarily surjective) map 
$C\twoheadrightarrow [k]$ such that $g\circ f=\id_{[k]}$. By construction, both $f$ and $g$ are morphisms in $\mathbf{Set}_0$, such that $f\circ g\circ \varepsilon=\varepsilon$.

Denote $\underline \varepsilon:=g\circ \varepsilon$. Clearly, $\underline \varepsilon\in F([k])$ by \eqref{eq:functoriality of F} and 
$f\circ \underline \varepsilon=\varepsilon$. That is, $\varepsilon\in f\circ F([k])$.
Therefore, $F(C)_k\subset F(C)'_k$, which completes the proof of the Lemma.
\end{proof}

We have
$$\Part_k^{F(C)}(M)=\Part_k^{F(C)_k}(M)=\bigcup\limits_{f\in \Inj([k],C)} 
\{\PP_\varepsilon\,:\,\varepsilon\in f\circ F([k])\}=\Part_k^{F([k])}(M).
$$
This completes the proof of Proposition~\ref{prop:F-invariance}.
\end{proof}
By Proposition \ref{prop:F-invariance}, ${\tstirling{M}{k}}_F=|\Part_k^{F(C)}(M)|$ for any object $C\in \mathbf{Set}_0$. 
Using this  and   \eqref{eq:general enumeration} with $\EE=F(C)$ we obtain
\begin{align*}
|F(C)|=\sum_{k\ge 1} |\Part_k^{F(C)}|\, (|C|)_k=\sum_{k\ge 1} {\stirling{M}{k}}_F\, (|C|)_k=p_F(|C|).
\end{align*}
Theorem \ref{th:polynomial} is proved.
\end{proof}
\begin{corollary}\label{cor:stirling F2}
$\tstirling{M}{2}_F=\frac{1}{2}|F([2])|-1$ under the assumptions of Theorem \ref{th:polynomial}. 
\end{corollary}
\begin{example}\label{example:ordinary Stirling number} Let $M$ be a finite set. The assignments $F(C):=\Map(M,C)$ for all sets $C$ define an $M$-admissible functor $F:\mathbf{Set}\to \mathbf{Set}$.
Then $|F(C)|=|C|^{|M|}$ hence $p_F=x^{|M|}$. In particular, $\tstirling{[n]}{k}_F=\tstirling{n}{k}$ is the number of partitions of $[n]$ into $k$ disjoint nonempty subsets, that is, the Stirling number of the second kind. 
Thus, we can view the $\tstirling{M}{k}_F$ as  generalized Stirling numbers which justifies the notation.
\end{example}

\subsection{Polynomiality and transitivity}\label{subs:pol trans}
Let~$\Gamma=(V,E)$ be a directed graph. A function~$\mathbf c:E\to C$ is called a {\em transitive coloring} of~$\Gamma$ (cf.~\cite{ABGJ}) if for any vertices~$i,j,k\in V$ such that $(i,j),(j,k),(i,k)\in E$,
$\mathbf c(i,k)\in\{\mathbf c(i,j),\mathbf c(j,k)\}$.
For the directed graph~$\vec K_n=([n],I_n)$ (an acyclic tournament),
a transitive coloring is precisely a transitive map
$I_n\to C$. 
For an unoriented graph, the corresponding notion is a {\em Gallai coloring} (cf.~\cite{ABGJ}). Given a graph~$\Gamma$, 
let~$\EE_\Gamma(C)$ be the set of transitive (respectively, Gallai) colorings of~$\Gamma$ with values in~$C$.

\begin{proposition} \label{prop:EGamma is an M admissible functor}
Given a (directed) graph~$\Gamma=(V,E)$, the assignments $C\mapsto \EE_{\Gamma}(C)$ for all sets $C$ define an $E$-admissible functor $\EE_{\Gamma}: \mathbf{Set}\to \mathbf{Set}$.
\end{proposition}

\begin{proof} 
Let~$\Gamma=(V,E)$ be a graph, $C$ be a set and $\mathbf c\in \EE_{\Gamma}(C)$. Let~$f\in\Map(C,C')$.

Suppose first that~$\Gamma$ is directed. Then for any $i,j,k\in V$
such that the oriented edges $(i,j),(j,k),(i,k)\in E$, 
$\mathbf c(i,k)\in \{\mathbf c(i,j),\mathbf c(j,k)\}$ whence 
\begin{align*}
f(\mathbf c(i,k)) \in \{ f(\mathbf c(i,j)), f(\mathbf c(j,k)) \},
\end{align*}
That is, $f\circ \mathbf c \in \EE_{\Gamma}(C')$. Thus, \eqref{eq:functoriality of F} holds and $\EE_{\Gamma}$ is an $E$-admissible functor $\mathbf{Set}\to \mathbf{Set}$.

Similarly, if~$V$ is unoriented then $|\{\mathbf c(i,j),
\mathbf c(j,k),\mathbf c(i,k)\}|\le 2$
for any vertices $i,j,k\in V$
and edges $(i,j),(j,k),(i,k)\in V$, whence 
\begin{align*}
|\{ f(\mathbf c(i,j)), f(\mathbf c(j,k)), f(\mathbf c(i,k)) \}| \le 2.
\end{align*}
Again, $f\circ \mathbf c \in \EE_{\Gamma}(C')$ and so $\EE_{\Gamma}$ is an $E$-admissible functor $\mathbf{Set}\to \mathbf{Set}$.
\end{proof}
\begin{corollary}[cf.~\cite{ABGJepr}*{Proposition~2.13}]\label{cor:enum transitive}
For any finite set~$C$ and any graph~$\Gamma=(V,E)$,
$$
|\mathcal E_\Gamma(C)|=\sum_{k\ge 1}\stirling{\Gamma}{k}\, (|C|)_k
$$
where~$\tstirling{\Gamma}{k}=|\Part_k^{\mathcal E_{\Gamma}([k])}(V)|\in\ZZ_{\ge 0}$.
\end{corollary}
Let~$\Gamma=\vec K_n$. Then
$\Part^{\EE_{\Gamma}([r])}_r(I_n)$
is the set of all~$\mathcal P\in\Part_r(I_n)$
such that for any $(i,k)\in I_n$,
$\mathcal P[(i,k)]\cap \{ (i,j),(j,k)\}\not=\emptyset$ or, equivalently $\mathcal P[(i,k)]$ coincides
with one of the~$\mathcal P[(i,j)]$ and~$\mathcal P[(j,k)]$, 
for all~$1\le i<j<k\le n$. By~\cite{ABGJ}*{Observation~2.15}, $\tstirling{\vec K_n}{k}=0$
if~$k\ge n$. By~\cite{ABGJ}*{Theorem~1.4}), $\tstirling{\vec K_{n}}{n-1}$ is the~$(n-1)$th 
Catalan number~$\frac{1}{n}\binom{2(n-1)}{n-1}$, while~$\tstirling{\vec K_n}{2}=\frac12n!-1$ by~Lemma~\ref{lem:trans sign} or by~Corollary~\ref{cor:stirling F2}. 
\begin{conjecture}\label{conj:K n n-2}
$\tstirling{\vec K_n}{n-2}=(n-2)\binom{2n-3}n+\binom{2(n-2)}n$
for all~$n\ge 3$.
\end{conjecture}
The numbers~$\tstirling{\vec K_n}{k}$ for small values of~$n$ are shown in Table~\ref{tab:I}.
\begin{table}[b]
\begin{tabular}{|c|cccccccc|}
\hline
\diagbox{$n$}{$k$}
&1&2&3&4&5&6&7&8\\
\hline
2&1&&&&&&&\\
3&1&2&&&&&&\\
4&1&11&5&&&&&\\
5&1&59&69&14&&&&\\
6&1&359&756&364&42&&&\\
7&1&2519&7954&6700&1770&132&&\\
8&1&20159&84444&109032&49215&8217&429&\\
9&1&181439&919572&1683550&1150105&321937&37037&1430\\
\hline
\end{tabular}
\caption{$\tstirling{\vec K_n}{k}$ for~$1\le k<n\le 9$}\label{tab:I}
\end{table}
In particular, for~$|C|=3$ (respectively, $|C|=4$), the number of transitive maps~$\mathbf c:I_n\to C$ is given by the sequences
$3$, $15$, $99$, $771$, $6693$, $62841$, $627621,\dots$
(respectively, $4$, $28$, $256$, $2704$, $31192$, $381928$, $4885336,\dots$).

\begin{proposition}\label{prop:enum trans}
For~$\Gamma=\overset{\leftrightarrow\circlearrowleft}{K_n}$, we have $$
\stirling{\Gamma}{k}=\begin{cases}2\stirling{n}{k},& k\ge 3,\\\tfrac12 B_n-1,&k=2,
\end{cases}
$$
where
$B_n :=\sum\limits_{1\le k\le n} \tstirling{n}{k}2^k k!$
is the number of transitive relations~$\mathscr R$ on~$[n]$ such that~$([n]\times [n])\setminus\mathscr R$ is also transitive.
\end{proposition}

\begin{proof}
Given~$i,j\in[n]$ and~$\mathcal P\in\Part([n]\times[n])$,
abbreviate~$\mathcal P(i,j):=\mathcal P[(i,j)]$.
Let~$\mathscr T_n$ be the set of all $\mathcal P\in \Part([n]\times[n])$
such that~$\mathcal P(i,k)\in\{\mathcal P(i,j),\mathcal P(j,k)\}$ for all~$i,j,k\in[n]$. 
It is immediate from the definition that~$\Part_r^{\mathcal E_\Gamma([r])}([n]\times[n])
=\mathscr T_n\cap \Part_r([n]\times[n]):=\mathscr T_{n,r}$.

Given~$\mathcal P\in\Part([n]\times[n])$, define a relation~$\sim_{\mathcal P}$
on~$[n]$ by~$i\sim_{\mathcal P}j$ if and only if $\mathcal P(i,j)=\mathcal P(j,i)$,
$i,j\in[n]$. This relation is clearly reflexive and symmetric.
\begin{lemma}\label{lem:prop sim_P}
Let~$\mathcal P\in\mathscr T_n$. Then
\begin{enumalph}
\item\label{lem:prop sim_P.a} $\sim_{\mathcal P}$ is an equivalence relation;
\item\label{lem:prop sim_P.c} If~$i\sim_{\mathcal P}i'$,
$j\sim_{\mathcal P}j'$ then~$\mathcal P(i,j)=\mathcal P(i',j')$.
\end{enumalph}
\end{lemma}
\begin{proof}
Note that if~$\mathcal P\in\mathscr T_n$ and~$\mathcal P(i,j)=\mathcal P(j,i)$
then~$\mathcal P(i,i)=\mathcal P(i,j)=\mathcal P(j,i)=\mathcal P(j,j)$.
Suppose that~$i\sim_{\mathcal P}j$, $j\sim_{\mathcal P}k$, $i,j,k\in[n]$. Then~$\mathcal P(i,i)=\mathcal P(j,j)=\mathcal P(k,k)$,
$\mathcal P(i,k)\in\{\mathcal P(i,j),\mathcal P(j,k)\}=\{\mathcal P(j,j)\}$
and~$\mathcal P(k,i)\in\{\mathcal P(k,j),\mathcal P(j,i)\}=\{\mathcal P(j,j)\}$.
Therefore $\mathcal P(i,k)=\mathcal P(k,i)$, that is~$i\sim_{\mathcal P}k$.
To prove~\ref{lem:prop sim_P.c}, 
suppose that~$\mathcal P(i,j)\not=\mathcal P(i',j')$. Since
$\mathcal P(i,j)\in\{\mathcal P(i,i'),\mathcal P(i',j)\}$ while
$\mathcal P(i',j')\in\{\mathcal P(i',i),\mathcal P(i,j')\}=
\{\mathcal P(i,i'),\mathcal P(i,j')\}$, it follows that~$\mathcal P(i,j)=
\mathcal P(i',j)\not=\mathcal P(i,j')=\mathcal P(i',j')$. Yet,
$\mathcal P(i',j)\in\{\mathcal P(i',j'),\mathcal P(j',j)\}$, whence~$\mathcal P(i',j)=\mathcal P(j,j')$, while~$\mathcal P(i,j')\in\{\mathcal P(i,j),\mathcal P(j,j')\}$ whence~$\mathcal P(i,j')=\mathcal P(j,j')$. Thus, $\mathcal P(i,j')=\mathcal P(i',j)$ which is a contradiction.
\end{proof}
We need the following
\begin{lemma}\label{lem:surj map part}
Let~$\pi:[n]\to[r]$
be surjective. 
\begin{enumalph}
    \item\label{lem:surj map part.a} Suppose that $\mathcal P\in\Part([n]\times[n])$
    satisfies 
    $\mathcal P(i,j)=\mathcal P(i',j')$ for all $i,i',j,j'\in[n]$ such that $\pi(i)=\pi(i')$ and~$\pi(j)=\pi(j')$.
Then~$(\pi\times\pi)(\mathcal P):=\{ (\pi\times\pi)(P)\,:\, P\in\mathcal P\}\in \Part([r]\times[r])$;
\item\label{lem:surj map part.b} Suppose that~$\mathcal Q\in \Part([r]\times[r])$.
Then~$(\pi\times\pi)^{-1}(\mathcal Q):=\{ (\pi\times\pi)^{-1}(Q)\,:\, Q\in\mathcal Q\}$
is a partition of~$[n]\times[n]$ and~$(\pi\times\pi)((\pi\times\pi)^{-1}(\mathcal Q))=\mathcal Q$.
\end{enumalph}
\end{lemma}
\begin{proof}
Since~$\bigcup_{P\in\mathcal P} P=[n]\times[n]$ and~$\pi$
is surjective, $\bigcup_{P\in\mathcal P}
(\pi\times\pi)(P)=[r]\times[r]$. Suppose that~$(i,j)\in 
(\pi\times\pi)(P)\cap(\pi\times\pi)(P')$, $i,j\in[r]$. Then~$i=(\pi(i'),
\pi(j'))=(\pi(i''),\pi(j''))$ for some~$i',i'',j',j''\in[n]$
and~$P=\mathcal P(i',j')=\mathcal P(i'',j'')=P'$
whence~$(\pi\times\pi)(P)=(\pi\times\pi)(P')$.
Thus, $(\pi\times\pi)(\mathcal P)\in\Part([r]\times[r])$.

Part~\ref{lem:surj map part.b} is immediate.
\end{proof}

Given~$\mathcal P\in\mathscr T_n$, let~$\mathcal C(\mathcal P)$ be
the partition of~$[n]$ into equivalence classes for~$\sim_{\mathcal P}$. Denote~$\underline{\mathscr T}_n$ the set of 
all~$\mathcal P\in\mathscr T_n$ satisfying~$\mathcal P(i,j)\not=\mathcal P(j,i)$ for all~$i\not=j\in[n]$.
\begin{lemma}\label{lem:surj trans part}
Let~$\mathcal P\in\mathscr T_n$ and suppose that~$\mathcal C(\mathcal P)\in\Part_r([n])$. Then
$(\pi\times\pi)(\mathcal P)\in \underline{\mathscr T}_r$ 
for any surjective~$\pi:[n]\to[r]$ such that~$\mathcal C(\mathcal P)=\{ \pi^{-1}(i)\}_{i\in[r]}$.
\end{lemma}
\begin{proof}
By
Lemmata~\ref{lem:prop sim_P}\ref{lem:prop sim_P.c} and~\ref{lem:surj map part}, 
$(\pi\times\pi)(\mathcal P)\in\Part([r]\times[r])$. Let~$i=\pi(i'),j=\pi(j'),k=\pi(k')\in[r]$, $i',j',k'\in[n]$.
Then~$\mathcal P(i',k')\in \{ \mathcal P(i',j'),\mathcal P(j',k')\}$,
and so $(\pi\times\pi)(\mathcal P)(i,k)\in\{(\pi\times\pi)(\mathcal P)(i,j),
(\pi\times\pi)(\mathcal P)(j,k)\}$, that is~$(\pi\times\pi)(\mathcal P)\in\mathscr T_r$.
Finally, suppose that~$(\pi\times\pi)(\mathcal P)(i,j)=(\pi\times\pi)(\mathcal P)(j,i)$ 
for some~$i\not=j\in[r]$. Then
$i=\pi(i')$, $j=\pi(j')$, $i',j'\in[n]$ with~$i'\not\sim_{\mathcal P}j'$
yet~$(\pi\times\pi)(\mathcal P(i',j'))=(\pi\times\pi)(\mathcal P(j',i'))$, that
is, $(i,j)=(\pi(i''),\pi(j''))$ for some~$i'',j''\in[n]$ such that~$\mathcal P(i'',j'')=
\mathcal P(j',i')\not=\mathcal P(i',j')$. Since~$\mathcal P(i'',j'')=\mathcal P(i',j')$
by Lemma~\ref{lem:prop sim_P}\ref{lem:prop sim_P.c}, this is a contradiction.
\end{proof}

Let~$\mathcal J\in\Part_r([n])$.
Define~$\mathcal J_k=\mathcal J[m_k]$ where~$m_k=\min\big([n]
\setminus \bigcup_{1\le t<k} \mathcal J_t\big)$; thus, $m_1=1$
and~$\mathcal J_1=\mathcal J[1]$ and so on.
Let~$\pi_{\mathcal J}$ be the unique surjective map~$\pi:[n]\to[r]$
satisfying~$\pi^{-1}(t)=\mathcal J_t$, $t\in[r]$. 
By Lemma~\ref{lem:surj trans part} 
we obtain a well-defined map~$\Pi_n:\mathscr T_n\to \bigsqcup_{1\le r\le n}
\underline{\mathscr T}_r$, $\mathcal P\mapsto (\pi_{\mathcal C(\mathcal P)}\times \pi_{\mathcal C(\mathcal P)})(\mathcal P)$,
$\mathcal P\in\mathscr T_n$.
\begin{lemma}\label{lem:collapse part}
The map~$\Pi_n$ is surjective and~$|\Pi_n^{-1}(\mathcal P)|=\tstirling{n}{r}$ 
for any~$\mathcal P\in\underline{\mathscr T}_r$. 
\end{lemma}
\begin{proof}
Let~$\mathcal P\in\underline{\mathscr T}_r$, 
$\mathcal J\in\Part_r([n])$ and define~$\mathcal Q:=\mathcal Q(\mathcal P,\mathcal J)=(\pi_{\mathcal J}\times\pi_{\mathcal J})^{-1}(\mathcal P)\in \Part([n]\times[n])$ by Lemma~\ref{lem:surj map part}\ref{lem:surj map part.b} and $(\pi_{\mathcal J}\times\pi_{\mathcal J})(\mathcal Q)=\mathcal P$.
We claim that~$\mathcal Q\in \mathscr T_n$. Indeed,
let~$i,j,k\in[n]$. Then
$$
(\pi_{\mathcal J}\times\pi_{\mathcal J})(\mathcal Q(i,k))=\mathcal P(\pi_{\mathcal J}(i),\pi_{\mathcal J}(k))\in 
\{ \mathcal P(\pi_{\mathcal J}(i),\pi_{\mathcal J}(j)),\mathcal P(\pi_{\mathcal J}(j),\pi_{\mathcal J}(k))\},
$$
whence~$\mathcal Q(i,k)\in\{ \mathcal Q(i,j),
\mathcal Q(j,k)\}$. 

Next, we prove that~$\mathcal C(\mathcal Q)=\mathcal J$,
that is, $\mathcal Q(i,j)=\mathcal Q(j,i)$ if and only if~$\pi_{\mathcal J}(i)=\pi_{\mathcal J}(j)$. Suppose first that~$\mathcal Q(i,j)=\mathcal Q(j,i)$. Then $\mathcal P(\pi_{\mathcal J}(i),\pi_{\mathcal J}(j))=
\mathcal P(\pi_{\mathcal J}(j),\pi_{\mathcal J}(i))
$ which forces~$\pi_{\mathcal J}(i)=\pi_{\mathcal J}(j)$
since~$\mathcal P\in\bigsqcup_{r\in[n]}\underline{\mathscr T}_r$. 
The converse is obvious
since~$\mathcal Q(i,j)=(\pi_{\mathcal J}\times\pi_{\mathcal J})^{-1}(
\mathcal P(\pi_{\mathcal J}(i),\pi_{\mathcal J}(j))$, $i,j\in[n]$.

Thus, we proved that~$\Pi_n(\mathcal Q(\mathcal P,\mathcal J))=
\mathcal P$, that is, $\Pi_n$ is surjective and, moreover,
that~$\Pi^{-1}_n(\mathcal P)=\{ \mathcal Q(\mathcal P,\mathcal J)\,:\,
\mathcal J\in \Part_r([n])\}$. It remains to observe that,
since $\mathcal J=\mathcal C(\mathcal Q(\mathcal P,\mathcal J))$,
the map $\Part_r([n])\to \Pi_n^{-1}(\mathcal P)$, $\mathcal J\mapsto 
\mathcal Q(\mathcal P,\mathcal J)$, $\mathcal J\in\Part_r([n])$
is also injective.
\end{proof}
\begin{lemma}\label{lem:trans count last}
Let~$\mathcal P\in \underline{\mathscr T}_n$.
\begin{enumalph}
\item\label{lem:trans count last.a} If~$\mathcal P(i,i)\not=\mathcal P(j,j)$ for all~$i\not=j\in[n]$
then $\mathcal P=\{ \mathcal P_1,\dots,\mathcal P_n\}$ where 
either~$\mathcal P_i=\{ (i,j)\,:\, j\in[n]\}$ for all~$i\in [n]$
or~$\mathcal P_i=\{ (j,i)\,:\, j\in [n]\}$ for all~$i\in[n]$.

\item\label{lem:trans count last.b} If~$\mathcal P(i,i)=\mathcal P(j,j)$ for some~$i\not=j$ then~$\mathcal P\in \Part_2([n])$. 
\end{enumalph}
\end{lemma}
\begin{proof}
To prove~\ref{lem:trans count last.a}, 
note that for~$\mathcal P\in\underline{\mathscr T}_n$
and satisfying~$\mathcal P(i,i)\not=\mathcal P(j,j)$ for all $j\not=i\in[n]$, since
$\{\mathcal P(i,i),\mathcal P(j,j)\}\subset \{\mathcal P(i,j),\mathcal P(j,i)\}$
it follows that, without loss of generality, 
$\mathcal P(i,i)=\mathcal P(i,j)$ and~$\mathcal P(j,j)=\mathcal P(j,i)$.
If~$n=2$ we are done. If~$n>2$, let~$k\in[n]\setminus\{i,j\}$
and suppose first that $\mathcal P(i,k)=\mathcal P(k,k)$, whence $\mathcal P(k,i)=\mathcal P(i,i)$. Then $\mathcal P(i,i)=\mathcal P(k,i)\in
\{\mathcal P(k,j),\mathcal P(j,i)\}=\{\mathcal P(k,j),\mathcal P(j,j)\}
\subset \{\mathcal P(j,j),\mathcal P(k,k)\}$ which is a contradiction. Thus,
$\mathcal P(i,k)=\mathcal P(i,i)$ and then $\mathcal P(k,i)=\mathcal P(k,k)$. As $\mathcal P(k,k)=\mathcal P(k,i)\in\{\mathcal P(k,j),\mathcal P(j,i)\}
=\{\mathcal P(k,j),\mathcal P(j,j)\}$ and $\mathcal P(k,j)\in\{\mathcal P(j,j),\mathcal P(k,k)\}$ this forces $\mathcal P(k,j)=\mathcal P(k,k)$ and $\mathcal P(j,k)=\mathcal P(j,j)$.
Therefore, $\mathcal P(i,i)=\mathcal P(i,j)=\mathcal P(i,k)$, $\mathcal P(j,i)=\mathcal P(j,j)=\mathcal P(j,k)$ and $\mathcal P(k,i)=\mathcal P(k,j)=\mathcal P(k,k)$.
This implies that $\mathcal P(r,s)=\mathcal P(r,r)$ for all $r,s\in[n]$. 

To prove part~\ref{lem:trans count last.b}, suppose that
$\mathcal P(i,i)=\mathcal P(j,j)$ for some $i\not=j$.
Since $\mathcal P(i,i)\in\{\mathcal P(i,j),\mathcal P(j,i)\}$
and $\mathcal P(i,j)\not=\mathcal P(j,i)$, we may assume without loss of generality that $\mathcal P(i,j)=\mathcal P(i,i)=\mathcal P(j,j)\not=\mathcal P(j,i)$.

Let $k\not=i$. 
Suppose that $\mathcal P(i,k)\notin\{\mathcal P(i,j),\mathcal P(j,i)\}$. Then $\mathcal P(i,k)\in\{\mathcal P(i,j),\mathcal P(j,k)\}$
forces $\mathcal P(j,k)=\mathcal P(i,k)$ while $\mathcal P(i,j)=\mathcal P(i,i)\in\{\mathcal P(i,k),\mathcal P(k,i)\}$ yields $\mathcal P(k,i)=\mathcal P(i,j)$.
But then $\mathcal P(j,i)\in\{\mathcal P(j,k),\mathcal P(k,i)\}=\{\mathcal P(i,j),
\mathcal P(i,k)\}$, which is a contradiction.
If $\mathcal P(k,i)\notin\{\mathcal P(i,j),\mathcal P(j,i)\}$, then
$\mathcal P(j,i)\in\{\mathcal P(j,k),\mathcal P(k,i)\}$ forces $\mathcal P(j,k)=\mathcal P(j,i)$.
Since $\mathcal P(i,j)=\mathcal P(j,j)\in\{\mathcal P(j,k),\mathcal P(k,j)\}$ we have $\mathcal P(k,j)=\mathcal P(i,j)$. Then $\mathcal P(k,i)\in\{\mathcal P(k,j),\mathcal P(j,i)\}
=\{\mathcal P(i,j),\mathcal P(j,i)\}$ which is a contradiction. 

Therefore, for any~$k\in[n]\setminus\{i,j\}$, $\mathcal P(i,k),\mathcal P(k,i)\in \{\mathcal P(i,j),\mathcal P(j,i)\}$.
But then $\mathcal P(k,l)\in\{\mathcal P(k,i),\mathcal P(i,l)\}\subset \{\mathcal P(i,j),\mathcal P(j,i)\}$ for all $k,l\in [n]$. Thus, $\mathcal P\in\Part_2([n]\times[n])$. 
\end{proof}
It follows from Corollary~\ref{cor:stirling F2} and Lemmata~\ref{lem:collapse part}, \ref{lem:trans count last}\ref{lem:trans count last.a} and~\ref{lem:In trans to trans} that~$|\mathscr T_{n,r}|=2\tstirling{n}{r}$,
while~$\mathcal E_{\Gamma}([2])=\sum_{1\le r\le n}\tstirling{n}{r}2^r r!=B_n$. The assertion is now immediate.
\end{proof}
\begin{corollary}\label{cor:poly q_n}
Let~$n\ge 2$. Then for any finite~$C$, $|\mathcal E_{\overset{\leftrightarrow\circlearrowleft}{K_n}}(C)|=q_n(|C|)$ where~$q_n(x)=
2x^n-x+(\frac12B_n-2^n+1)x(x-1)$. In particular, $|\mathcal E_{\overset{\leftrightarrow\circlearrowleft}{K_n}}(\{1,-1\})|=q_n(2)=B_n$.
\end{corollary}
\begin{proof}
Since $\tstirling{n}{2}=2^{n-1}-1$, $n\ge 1$,
by Corollary~\ref{cor:enum transitive} 
and Proposition~\ref{prop:enum trans}
\begin{align*}
q_n(x)&=\sum_{k\ge 1} \stirling{\overset{\leftrightarrow\circlearrowleft}{K_n}}{k}
(x)_k=x+\Big(\tfrac12B_n-1\Big)x(x-1)+2\sum_{3\le k\le n}\stirling{n}{k}(x)_k\\
&=x+\Big(\tfrac12B_n-1\Big)x(x-1)+2x^n-2x-2\stirling{n}{2}x(x-1)\\
&=2x^n-x+\Big(\tfrac12B_n-2^n+1\Big)x(x-1).\qedhere
\end{align*} 
\end{proof}

\subsection{Extensions and restrictions}
Given~$\mathbf c:I_n\to C$, denote $\mathbf c^-=\mathbf c|_{I_{n-1}}$ and define $\mathbf c^+:I_{n-1}\to C$ by $\mathbf c^+(i,j)=\mathbf c(i+1,j+1)$, $(i,j)\in I_{n-1}$.
Furthermore, given~$\boldsymbol\alpha:[n]\to C$, denote $\boldsymbol{\alpha}^-
=\boldsymbol{\alpha}|_{[n-1]}$ and define $\boldsymbol\alpha^+:[n-1]\to C$ by
$\boldsymbol{\alpha}^+(i)=\boldsymbol{\alpha}(i+1)$, $i\in [n-1]$. When convenient, 
we identify $\boldsymbol{\alpha}:[n]\to C$ with $(\boldsymbol{\alpha}(1),\dots,
\boldsymbol{\alpha}(n))\in C^n$. 
Finally, for any~$\mathbf c:I_n\to C$, $\boldsymbol{\alpha}:[n]\to C$ we
define~$\mathbf c^{\boldsymbol{\alpha}}:I_{n+1}\to C$ by $\mathbf c^{\boldsymbol{\alpha}}|_{I_n}=
\mathbf c$ and $\mathbf c^{\boldsymbol\alpha}(i,n+1)=\boldsymbol{\alpha}(i)$, $i\in[n]$.
\begin{lemma}\label{lem:rest ext trans}
Let~$\mathbf c:I_n\to C$, $\boldsymbol{\alpha}:[n]\to C$.
\begin{enumalph} 
\item \label{lem:rest ext trans.c}
$\mathbf c=(\mathbf c^-)^{\boldsymbol\alpha_{\mathbf c}}$ where~$\boldsymbol\alpha_{\mathbf c}(i)=\mathbf c(i,n)$,
$i\in [n-1]$. 
\item \label{lem:rest ext trans.a}
$\mathbf c^{\boldsymbol{\alpha}}$ is transitive if and only if
$\mathbf c$ is transitive and~$\boldsymbol\alpha(i)\in\{\mathbf c(i,j),
\boldsymbol\alpha(j)\}$ for all~$1\le i<j\le n$. 
\item \label{lem:rest ext trans.b}
If~$\mathbf c^{\boldsymbol\alpha}$ is transitive, then so are~$(\mathbf c^\pm)^{\boldsymbol\alpha^\pm}:I_n\to C$.
\end{enumalph}
\end{lemma}
\begin{proof}
Part~\ref{lem:rest ext trans.c} is immediate from definitions.
To prove part~\ref{lem:rest ext trans.a}, note that
if~$\mathbf c^{\boldsymbol\alpha}$ is transitive then so is~$\mathbf c=\mathbf c^{\boldsymbol\alpha}|_{I_n}$. Furthermore, for all~$1\le i<j\le n$
\begin{align*}
\boldsymbol\alpha(i)=\mathbf c^{\boldsymbol{\alpha}}(i,n+1)\in \{ \mathbf c^{\boldsymbol{\alpha}}(i,j),
\mathbf c^{\boldsymbol{\alpha}}(j,n+1)\}=\{\mathbf c(i,j),\boldsymbol{\alpha}(j)\}.
\end{align*}
Conversely, since~$\mathbf c$ is transitive, $\mathbf c^{\boldsymbol\alpha}(i,k)=
\mathbf c(i,k)\in \{\mathbf c(i,j),\mathbf c(j,k)\}=
\{\mathbf c^{\boldsymbol\alpha}(i,j),\mathbf c^{\boldsymbol\alpha}(j,k)\}$ for 
all~$1\le i<j<k\le n$, while 
$$
\mathbf c^{\boldsymbol\alpha}(i,n+1)=\boldsymbol{\alpha}(i)\in\{\mathbf c(i,j),
\boldsymbol{\alpha}(j)\}
=\{\mathbf c^{\boldsymbol{\alpha}}(i,j),\mathbf c^{\boldsymbol{\alpha}}(j,n+1)\},
\qquad 1\le i<j\le n.
$$

To prove part~\ref{lem:rest ext trans.b}, note that if~$\mathbf c^{\boldsymbol{\alpha}}$ is transitive, then $\mathbf c$ is transitive by part~\ref{lem:rest ext trans.a}, hence~$\mathbf c^-$ is also transitive, while for all $1\le i<j<k\le n-1$ we have 
$$
\mathbf c^+(i,k)=\mathbf c(i+1,k+1)\in \{\mathbf c(i+1,j+1),\mathbf c(j+1,k+1)\}=
\{\mathbf c^+(i,j),\mathbf c^+(j,k)\},
$$
whence~$\mathbf c^+$ is transitive. 
Furthermore,  for all~$1\le i<j\le n-1$,
$\boldsymbol{\alpha}(i)\in\{\mathbf c(i,j),\boldsymbol{\alpha}(j)\}=\{\mathbf c^-(i,j),\boldsymbol{\alpha}(j)\}$ while $\boldsymbol\alpha^+(i)=\boldsymbol{\alpha}(i+1)\in\{\mathbf c(i+1,j+1),\boldsymbol{\alpha}(j+1)\}=\{\mathbf c^+(i,j),\boldsymbol\alpha^+(j)\}$. 
Then $(\mathbf c^-)^{\boldsymbol{\alpha}^-}$ and~$(\mathbf c^+)^{\boldsymbol{\alpha}^+}$ are transitive by part~\ref{lem:rest ext trans.a}.
\end{proof}

\section{Main results: the classical case}\label{sect:main classical}

\subsection{\texorpdfstring{$C$}{C}-quasi-triangular Lie algebra}\label{subs:Cqtr}
Given a set~$C$ and $n\ge 2$, let~$\lie{qtr}_n(C)$ be the Lie algebra with generators $\mathsf r_{i,j}^{(c)}$,
$i\not=j\in [n]$, $c\in C$ subject to relations 
\begin{alignat}{3}
&[\sr{i}{j}{c},\sr{k}{l}{c'}]=0,\label{eq:CYB rels comm}
&&i\not=j\not=k\not=l\in[n],&\quad &c,c'\in C,\\
\label{eq:CYB rels}
&[\sr ij{c},\sr ik{c'}]+[\sr ijc,\sr jk{c''}]+[\sr ik{c'},\sr jk{c''}]=0,&\quad &
i\not=j\not=k\in [n],&&c'\in\{c,c''\}\subset C.
\end{alignat}
We call~$\lie{qtr}_n(C)$ a $C$-{\em quasi-triangular Lie algebra} since 
it generalizes the quasi-triangular Lie algebra~$\lie{qtr}_n$ defined in~\cite{BEER} which
is~$\lie{qtr}_n(C)$ with~$|C|=1$. Clearly, $S_n$ acts 
on~$\lie{qtr}_n(C)$ by Lie algebra automorphisms via
$\sigma(\sr ijc)=\sr{\sigma(i)}{\sigma(j)}{c}$, $i\not=j\in[n]$, $c\in C$, $\sigma\in S_n$.

Let $N\ge n$, $C'\supset C$ and fix a sequence $\mathbf i=(i_1,\dots,i_n)\in [N]^n$, with~$i_k\not=i_l$ for all~$1\le k<l\le n$. By~\eqref{eq:CYB rels comm} and~\eqref{eq:CYB rels}, the assignments $\mathsf r_{k,l}^{(c)}\mapsto \mathsf r_{i_k,i_l}^{(c)}$, $k,l\in[n]$, $c\in C$,
define a homomorphism of Lie algebras~$\psi_{\mathbf i}:\lie{qtr}_n(C)\to 
\lie{qtr}_N(C')$.
The following is an immediate consequence of Proposition~\ref{prop:color CYBE}.
\begin{lemma}
Let~$(\lie g,\delta)$ be a quasi-triangular Lie algebra with a family of classical 
r-matrices $\boldsymbol r=\{r^{(c)}\}_{c\in C}\subset \lie g\tensor\lie g$.
The assignments $\sr{i}{j}{c}\mapsto r_{i,j}^{(c)}$, $i\not=j\in[n]$, $c\in C$,
define a homomorphism of Lie algebras~$\Psi_{\boldsymbol r}^{(n)}:
\lie{qtr}_n(C)\to U(\lie g)^{\tensor n}$.
\end{lemma}
Given~$\boldsymbol\gamma:I_n\to C$, $n\ge 2$ define $\lie j_{\boldsymbol\gamma}\in\lie{qtr}_{2n}(C)$
by
\begin{equation}\label{eq:jc defn}
\lie j_{\boldsymbol\gamma}:=\sum_{1\le i<j\le n}
\sr{j}{n+i}{\boldsymbol\gamma(i,j)}.
\end{equation}
The following is immediate.
\begin{lemma}\label{lem:rec class j}
Let~$n\ge 3$. Then
for all~$\boldsymbol\gamma:I_n\to C$
\begin{align}\label{eq:rec jc}
\lie j_{\boldsymbol\gamma}&=\psi_{[1,2n-1]\setminus\{n\}}(\lie j_{\boldsymbol\gamma^-})+
\sum_{1\le i\le n-1}
\sr{n}{i+n}{\boldsymbol\gamma(i,n)}
\\
&=\psi_{[2,2n]\setminus\{n+1\}}(\lie j_{\boldsymbol\gamma^+})+\sum_{2\le i\le n} \sr{i}{n+1}{\boldsymbol\gamma(1,i)}.\label{eq:rec jc-2}
\end{align}
\end{lemma}
We now prove an identity which is the key ingredient in our proof of Theorem~\ref{thm:main thm trans}.
\begin{proposition}\label{prop:identity jc}
For any~$m\ge 1$, $\boldsymbol\gamma:I_m\to C$ and $\boldsymbol\alpha=(\alpha_1,\dots,\alpha_m)\in C^m$ such that~$\boldsymbol\gamma^{\boldsymbol{\alpha}}:I_{m+1}\to C$ is transitive
\begin{align}
&\sum_{1\le i\le m}\Big([\psi_{[m+2,2m+1]\cup[2m+3,3m+2]}(\lie j_{\boldsymbol\gamma}^-),
\sr{m+1}{m+1+i}{\alpha_{i}}+\sr{m+1}{2m+2+i}{\alpha_{i}}]\nonumber\\
&\phantom{\sum_{1\le i\le m}}-[\psi_{[1,m]\cup[2m+3,3m+2]}(\lie j_{\boldsymbol\gamma}^-),
\sr{2m+2}{i}{\alpha_{i}}+\sr{2m+2}{2m+2+i}{\alpha_{i}}]\nonumber\\
&\phantom{\sum_{1\le i\le m}}+[\psi_{[1,m]\cup[m+2,2m+1]}(\lie j_{\boldsymbol\gamma}^-),
\sr{3m+3}{i}{\alpha_{i}}+\sr{3m+3}{m+1+i}{\alpha_{i}}]\Big)\nonumber\\
&=\sum_{1\le i\not=j\le m} [\sr{m+1}{i+m+1}{\alpha_{i}},\sr{m+1}{j+2m+2}{\alpha_{j}}]+[\sr{2m+2}{i+2m+2}{\alpha_{i}},\sr{2m+2}{j}{\alpha_{j}}]+[\sr{3m+3}{i}{\alpha_{i}}
\sr{3m+3}{j+m+1}{\alpha_{j}}]\label{eq:key id jc}
\end{align}
in~$\lie{qtr}_{3m+3}(C)$, 
where~$\lie j^-_{\boldsymbol\gamma}=\lie j_{\boldsymbol\gamma}-\sigma_m(\lie j_{\boldsymbol\gamma})$ and
$\sigma_m\in S_{2m}$ is the Grassmann permutation, that is $\sigma_m(i)=i+m$ if~$i\in[n]$, $\sigma_m(i)=i-m$ if~$i\in[m+1,2m]$.
\end{proposition}
\begin{proof}
We use induction on~$m$. The induction base is trivial since both sides of the identity are equal to zero for~$m=1$.

To prove the inductive step, abbreviate~$\gamma_{i,j}=\boldsymbol\gamma(i,j)$, $1\le i<j\le m$ and  denote~$\mathscr L_m(\boldsymbol\gamma,\boldsymbol\alpha)$ (respectively, $\mathscr R_m(\boldsymbol \alpha)$) the left (respectively, the right) hand side of~\eqref{eq:key id jc}, the latter being independent of~$\boldsymbol\gamma$.
By~\eqref{eq:rec Jc} we have 
$$
\lie j_{\boldsymbol\gamma}^-=\lie j'_{\boldsymbol\gamma}+
\sum_{1\le i\le m-1} \sr m{i+m}{\gamma_{i,m}}-\sum_{1\le i\le m-1} \sr{2m}{i}{\gamma_{i,m}}
$$
where we abbreviate $
\lie j'_{\boldsymbol\gamma}=\psi_{[1,m-1]\cup[m+1,2m-1]}((\lie j_{\boldsymbol\gamma^-})^-)$.
Since
\begin{align*}
\psi_{[m+2,2m+1]\cup[2m+3,3m+2]}(\sr m{i+m}{\gamma_{i,m}}-\sr{2m}{i}{\gamma_{i,m}})&=\sr{2m+1}{2m+i+2}{\gamma_{i,m}}-\sr{3m+2}{i+m+1}{\gamma_{i,m}},\\
\psi_{[1,m]\cup[2m+3,3m+2]}(\sr m{i+m}{\gamma_{i,m}}-\sr{2m}{i}{\gamma_{i,m}})&=\sr m{2m+i+2}{\gamma_{i,m}}
-\sr{3m+2}{i}{\gamma_{i,m}},\\
\psi_{[1,m]\cup[m+2,2m+1]}(\sr m{i+m}{\gamma_{i,m}}-\sr{2m}{i}{\gamma_{i,m}})&=\sr m{m+i+1}{\gamma_{i,m}}
-\sr{2m+1}{i}{\gamma_{i,m}}
\end{align*}
we can write 
\begin{align*}
\mathscr L_m(\boldsymbol\gamma,\boldsymbol \alpha)&=\sum_{1\le i\le m}\Big([\psi_{[m+2,2m+1]\cup[2m+3,3m+2]}(\lie j'_{\boldsymbol\gamma}),
\sr{m+1}{m+1+i}{\alpha_{i}}+\sr{m+1}{2m+2+i}{\alpha_{i}}]\nonumber\\
&\phantom{\sum_{1\le i\le m}}-[\psi_{[1,m]\cup[2m+3,3m+2]}(\lie j'_{\boldsymbol\gamma}),
\sr{2m+2}{i}{\alpha_{i}}+\sr{2m+2}{2m+2+i}{\alpha_{i}}]\nonumber\\
&\phantom{\sum_{1\le i\le m}}+[\psi_{[1,m]\cup[m+2,2m+1]}(\lie j'_{\boldsymbol\gamma}),
\sr{3m+3}{i}{\alpha_{i}}+\sr{3m+3}{m+1+i}{\alpha_{i}}]\Big)
+\mathscr L'_m(\boldsymbol\gamma,\boldsymbol\alpha),
\end{align*}
where 
\begin{align*}
\mathscr L'_m(\boldsymbol\gamma,\boldsymbol\alpha):=\sum_{\substack{1\le t\le m-1\\1\le s\le m}}
\Big(&[\sr{2m+1}{2m+t+2}{\gamma_{t,m}}-\sr{3m+2}{t+m+1}{\gamma_{t,m}},\sr{m+1}{m+1+s}{\alpha_{s}}+\sr{m+1}{2m+2+s}{\alpha_{s}}]\\
&-[\sr m{2m+t+2}{\gamma_{t,m}}-\sr{3m+2}{t}{\gamma_{t,m}},\sr{2m+2}{s}{\alpha_{s}}+\sr{2m+2}{2m+2+s}{\alpha_{s}}]\\
&+[\sr m{m+t+1}{\gamma_{t,m}}-\sr{2m+1}{t}{\gamma_{t,m}},\sr{3m+3}{s}{\alpha_{s}}+\sr{3m+3}{m+1+s}{\alpha_{s}}]\Big)\\
&\mskip-100mu=\sum_{1\le t\le m-1} \Big(-[\sr{m+1}{2m+t+2}{\alpha_{t}},\sr{2m+1}{2m+t+2}{\gamma_{t,m}}]-
[\sr{m+1}{2m+1}{\alpha_m},\sr{2m+1}{2m+t+2}{\gamma_{t,m}}]\\
&\mskip-100mu\phantom{=\sum_{1\le t\le m-1}\Big(}+
[\sr{m+1}{m+t+1}{\alpha_{t}},\sr{3m+2}{m+t+1}{\gamma_{t,m}}]
+[\sr{m+1}{3m+2}{\alpha_{m}},\sr{3m+2}{m+t+1}{\gamma_{t,m}}]\\
&\mskip-100mu\phantom{=\sum_{1\le t\le m-1}\Big(}
+[\sr{2m+2}{2m+t+2}{\alpha_{t}},\sr m{2m+t+2}{\gamma_{t,m}}]+[\sr{2m+2}{m}{\alpha_{m}},\sr m{2m+t+2}{\gamma_{t,m}}]\\
&\mskip-100mu\phantom{=\sum_{1\le t\le m-1}\Big(}
-[\sr{2m+2}{t}{\alpha_{t}},\sr{3m+2}{t}{\gamma_{t,m}}]-[\sr{2m+2}{3m+2}{\alpha_{m}},\sr{3m+2}{t}{\gamma_{t,m}}]\\
&\mskip-100mu\phantom{=\sum_{1\le t\le m-1}\Big(}
-[\sr{3m+3}{m+t+1}{\alpha_{t}},\sr m{m+t+1}{\gamma_{t,m}}]-[\sr{3m+3}{m}{\alpha_{m}},\sr m{m+t+1}{\gamma_{t,m}}]\\
&\mskip-100mu\phantom{=\sum_{1\le t\le m-1}\Big(}
+[\sr{3m+3}{t}{\alpha_{t}},\sr{2m+1}{t}{\gamma_{t,m}}]+[\sr{3m+3}{2m+1}{\alpha_{m}},\sr{2m+1}{t}{\gamma_{t,m}}]\Big).
\end{align*}
Since~$\boldsymbol\gamma^{\boldsymbol\alpha}$ is transitive,
$\alpha_t\in\{\gamma_{t,m},\alpha_m\}$, $1\le t\le m-1$ by Lemma~\ref{lem:rest ext trans}\ref{lem:rest ext trans.a}. Using~\eqref{eq:CYB rels}
with $(c,c',c'')=(\alpha_m,\alpha_t,\gamma_{t,m})$ and, respectively,
\begin{multline*}(i,j,k)\in\{ (m+1,2m+1,2m+t+2),(m+1,3m+2,m+t+1),
(2m+2,m,2m+t+2),\\(2m+2,3m+2,t),(3m+3,m,m+t+1),(3m+3,2m+1,t)\}
\end{multline*}
we obtain
\begin{align*}
\mathscr L'_m&(\boldsymbol\gamma,\boldsymbol\alpha)=\sum_{1\le t\le m-1} [\sr{m+1}{2m+1}{\alpha_m},
\sr{m+1}{2m+t+2}{\alpha_t}]-[\sr{m+1}{3m+2}{\alpha_m},\sr{m+1}{m+t+1}{\alpha_t}]
\\&
-[\sr{2m+2}{m}{\alpha_m},\sr{2m+2}{2m+t+2}{\alpha_t}]+[\sr{2m+2}{3m+2}{\alpha_m},
\sr{2m+2}{t}{\alpha_t}]\\
&+[\sr{3m+3}{m}{\alpha_m},\sr{3m+3}{m+t+1}{\alpha_t}]-[\sr{3m+3}{2m+1}{\alpha_m},
\sr{3m+3}{t}{\alpha_t}],
\end{align*}
whence
\begin{align}
\mathscr R&_m(\boldsymbol\alpha)-\mathscr L'_m(\boldsymbol\gamma,\boldsymbol\alpha)\nonumber\\
&=\sum_{1\le i\not=j\le m-1} [\sr{m+1}{i+m+1}{\alpha_{i}},\sr{m+1}{j+2m+2}{\alpha_{j}}]+[\sr{2m+2}{i+2m+2}{\alpha_{i}},\sr{2m+2}{j}{\alpha_{j}}]+[\sr{3m+3}{i}{\alpha_{i}},
\sr{3m+3}{j+m+1}{\alpha_{j}}]\nonumber\\
&=\psi_{[1,3m+3]\setminus\{m,2m+1,3m+2\}}\Big(\sum_{1\le i\not=j\le m-1} [\sr{m}{i+m}{\alpha_{i}},\sr{m}{j+2m}{\alpha_{j}}]+[\sr{2m}{i+2m}{\alpha_{i}},\sr{2m}{j}{\alpha_{j}}]+[\sr{3n}{i}{\alpha_{i}},
\sr{3n}{j+m}{\alpha_{j}}]\Big)\nonumber\\
&=\psi_{[1,3m+3]\setminus\{m,2m+1,3m+2\}}(\mathscr R_{m-1}(\boldsymbol\alpha^-)).\label{eq:interm jc id}
\end{align}
Furthermore,
\begin{align*}
\psi_{[m+2,2m+1]\cup[2m+3,3m+2]}(\lie j'_{\boldsymbol\gamma})&=
\psi_{[m+2,2m]\cup[2m+3,3m+1]}((\lie j_{\boldsymbol\gamma^-})^-)\\
&=
\psi_{[1,3m+3]\setminus\{m,2m+1,3m+2\}}\circ \psi_{[m+1,2m-1]\cup [2m+1,3m-1]}((\lie j_{\boldsymbol\gamma^-})^-),\\
\psi_{[1,m]\cup[2m+3,3m+2]}(\lie j'_{\boldsymbol\gamma})&=\psi_{[1,m-1]\cup[2m+3,3m+1]}((\lie j_{\boldsymbol\gamma^-})^-)\\
&=\psi_{[1,3m+3]\setminus\{m,2m+1,3m+2\}}\circ \psi_{[1,m-1]\cup[2m+1,3m-1]}((\lie j_{\boldsymbol\gamma^-})^-),\\
\psi_{[1,m]\cup[m+2,2m+1]}(\lie j'_{\boldsymbol\gamma})&=\psi_{[1,m-1]\cup[m+2,2m]}((\lie j_{\boldsymbol\gamma^-})^-)\\
&=\psi_{[1,3m+3]\setminus\{m,2m+1,3m+2\}}\circ \psi_{[1,m-1]\cup[m+1,2m-1]}((\lie j_{\boldsymbol\gamma^-})^-).
\end{align*}
Therefore,
\begin{align*}
\mathscr L&_m(\boldsymbol\gamma,\boldsymbol{\alpha})-
\mathscr L'_m(\boldsymbol\gamma,\boldsymbol{\alpha})\\
&=
\psi_{[1,3m+3]\setminus\{m,2m+1,3m+2\}}\Big(\sum_{1\le i\le m}\Big(
[\psi_{[m+1,2m-1]\cup [2m+1,3m-1]}((\lie j_{\boldsymbol\gamma^-})^-),\sr m{m+i}{\alpha_i}+\sr{m}{2m+i}{\alpha_i}]\\
&
\phantom{=
\psi_{[1,3m+3]\setminus\{m,2m+1,3m+2\}}\Big(\sum_{1\le i\le m}}
-[\psi_{[1,m-1]\cup[2m+1,3m-1]}((\lie j_{\boldsymbol\gamma^-})^-),
\sr{2m}{i}{\alpha_{i}}+\sr{2m}{2m+i}{\alpha_{i}}]\\
&\phantom{=
\psi_{[1,3m+3]\setminus\{m,2m+1,3m+2\}}\Big(\sum_{1\le i\le m}}
+[\psi_{[1,m-1]\cup[m+1,2m-1]}((\lie j_{\boldsymbol\gamma^-})^-),\sr{3n}{i}{\alpha_i}+
\sr{3n}{m+i}{\alpha_i}]\Big)\Big).
\end{align*}
The terms corresponding to~$i=m$ in these sums are equal zero by~\eqref{eq:CYB rels comm}, whence 
$$
\mathscr L_m(\boldsymbol\gamma,\boldsymbol\alpha)-\mathscr L'_m(\boldsymbol\gamma,\boldsymbol\alpha)
=\psi_{[1,3m+3]\setminus\{m,2m+1,3m+2\}}(\mathscr L_{m-1}(\boldsymbol\gamma^-,\boldsymbol\alpha^-)).
$$
It follows from~\eqref{eq:interm jc id} and the above that $\mathscr L_m(\boldsymbol\gamma,\boldsymbol \alpha)=\mathscr R_m(\boldsymbol\alpha)$
if and only if 
$$
\psi_{[1,3m+3]\setminus\{m,2m+1,3m+2\}}(\mathscr L_{m-1}(\boldsymbol\gamma^-,\boldsymbol\alpha^-)-
\mathscr R_{m-1}(\boldsymbol{\alpha^-}))=0.
$$
Yet~$(\boldsymbol\gamma^-)^{\boldsymbol{\alpha}^-}$ is transitive by Lemma~\ref{lem:rest ext trans}\ref{lem:rest ext trans.b},
and so~$\mathscr L_{m-1}(\boldsymbol\gamma^-,\boldsymbol\alpha^-)-
\mathscr R_{m-1}(\boldsymbol{\alpha^-})=0$ by the induction hypothesis. This 
completes the proof of the inductive step and hence of the Proposition.
\end{proof}

\subsection{A family of classical twists} We can now establish the main result of this section.
\begin{theorem}\label{thm:trans classical twists}
Let~$(\lie g,\delta)$ be a Lie bialgebra with a family~$
\mathbf r=\{r^{(c)}\}_{c\in C}$ of classical r-matrices for~$\delta$. Then for any~$n\ge 2$ and any
transitive~$\mathbf c:I_n\to C$, $j_{\mathbf c}=j_{\mathbf c}(\mathbf r):=\Psi_{\mathbf r}^{(2n)}(\lie j_{\mathbf c})\in \lie g^{\oplus n}\tensor \lie g^{\oplus n}\subset 
U(\lie g)^{\tensor n}\tensor U(\lie g)^{\tensor n}$ is a classical Drinfeld 
twist for~$(\lie g^{\oplus n},\delta_{\lie g^{\oplus n}})$.
\end{theorem}
\begin{proof}
The argument is by induction on~$n$. The case~$n=2$ was established in Proposition~\ref{prop:basic twist class}. To prove the inductive step, note 
that~$j_{\mathbf c^-}$ is a classical Drinfeld twist for $(\lie g^{\oplus(n-1)},
\delta_{\lie g^{\oplus(n-1)}})$ by the induction hypothesis.
We need the following
\begin{proposition}\label{prop:key prop classical}
Let~$\lie a=\lie g^{\oplus(n-1)}$ with~$\delta_{\lie a}$
being the natural Lie bialgebra structure $\delta_{\lie g^{\oplus(n-1)}}$ twisted by 
$j_{\mathbf c^-}(\mathbf r)$ and let $\lie b=\lie g$
with~$\delta_{\lie b}=\delta$. Then $\mathbf f=\sum\limits_{1\le i\le n-1}
r_{1,i+1}^{(\mathbf c(i,n))}\in \lie b\tensor \lie a$ is a relative Drinfeld twist for~$(\lie a\oplus \lie b,
\delta_{\lie a\oplus\lie b})$. 
\end{proposition}
\begin{proof}
Abbreviate
$c_{i,k}=\mathbf c(i,k)$, $1\le i<k\le n$,
$\mathbf i=[n(i-1)+1,ni-1]$, $i\in \{1,3,5\}$,
$\mathbf k=\{n(k-1)\}$, $k\in \{2,4,6\}$
and identify $U(\lie a)$ with~$U(\lie g)^{\tensor(n-1)}$ as an associative algebra.
By Proposition~\ref{prop:compat cond}\ref{prop:compat cond.d},
$(\delta_{\lie b}\tensor\id_{\lie a})(r_{1,i+1}^{(c)})=[r_{2,i+2}^{(c)},r_{1,i+2}^{(c)}]$,
$c\in C$, $1\le i\le n-1$, whence 
\begin{align*}
(\id_{(U(\lie a)\tensor U(\lie b))^{\tensor3}}&-\tau_{\mathbf 4,\mathbf 6}\tau_{\mathbf 3,\mathbf 5}+\tau_{\mathbf 2,\mathbf 4}\tau_{\mathbf 4,\mathbf 6}\tau_{\mathbf 1,\mathbf 5})((\delta_{\lie b}\tensor \id_{\lie a})(\mathbf f)_{\mathbf 2,\mathbf 4,\mathbf 5})\\
&=\sum_{1\le i\le n-1} [r_{2n,2n+i}^{(c_{i,n})},r_{n,2n+i}^{(c_{i,n})}]-
[r_{3n,n+i}^{(c_{i,n})},r_{n,n+i}^{(c_{i,n})}]+[r_{3n,i}^{(c_{i,n})},r_{2n,i}^{(c_{i,n)}}].
\end{align*}
Furthermore, we have  in $(U(\lie a)\tensor U(\lie b))^{\tensor 3}$
\begin{align*}
[\mathbf f&_{\mathbf 2,\mathbf 3},\mathbf f_{\mathbf 2,\mathbf 5}]+[\mathbf f_{\mathbf 2,\mathbf 5},\mathbf f_{\mathbf 4,\mathbf 5}]+
[\mathbf f_{\mathbf 4,\mathbf 5},\mathbf f_{\mathbf 4,\mathbf 1}]
+[\mathbf f_{\mathbf 4,\mathbf 1},\mathbf f_{\mathbf 6,\mathbf 1}]+[\mathbf f_{\mathbf 6,\mathbf 1},\mathbf f_{\mathbf 6,\mathbf 3}]
+[\mathbf f_{\mathbf 6,\mathbf 3},\mathbf f_{\mathbf 2,\mathbf 3}]\\
&=\sum_{1\le i,k\le n-1} [r_{n,i+n}^{(c_{i,n})},r_{n,k+2n}^{(c_{k,n})}]
+[r_{n,k+2n}^{(c_{k,n})},r_{2n,i+2n}^{(c_{i,n})}]
+[r_{2n,i+2n}^{(c_{i,n})},r_{2n,k}^{(c_{k,n})}]
+[r_{2n,k}^{(c_{k,n})},r_{3n,i}^{(c_{i,n})}]\\
&\phantom{=\sum_{1\le i,k\le n-1}}+[r_{3n,i}^{(c_{i,n})},
r_{3n,k+n}^{(c_{k,n})}]+[r_{3n,k+n}^{(c_{k,n})},r_{n,i+n}^{(c_{i,n})}]\\
&=\sum_{1\le i,k\le n-1} [r_{n,i+n}^{(c_{i,n})},r_{n,k+2n}^{(c_{k,n})}]+[r_{2n,i+2n}^{(c_{i,n})},r_{2n,k}^{(c_{k,n})}]+[r_{3n,i}^{(c_{i,n})},
r_{3n,k+n}^{(c_{k,n})}]\\
&\quad +\sum_{1\le i\le n-1}[r_{n,i+2n}^{(c_{i,n})},r_{2n,i+2n}^{(c_{i,n})}]
+[r_{2n,i}^{(c_{i,n})},r_{3n,i}^{(c_{i,n})}]+[r_{3n,i+n}^{(c_{i,n})},r_{n,i+n}^{(c_{i,n})}]
\\
&=\sum_{1\le i,k\le n-1} [r_{n,i+n}^{(c_{i,n})},r_{n,k+2n}^{(c_{k,n})}]+[r_{2n,i+2n}^{(c_{i,n})},r_{2n,k}^{(c_{k,n})}]+[r_{3n,i}^{(c_{i,n})},
r_{3n,k+n}^{(c_{k,n})}]
\\
&\phantom{=}-(\id_{(U(\lie a)\tensor U(\lie b))^{\tensor3}}-\tau_{\mathbf 4,\mathbf 6}\tau_{\mathbf 3,\mathbf 5}+\tau_{\mathbf 2,\mathbf 4}\tau_{\mathbf 4,\mathbf 6}\tau_{\mathbf 1,\mathbf 5})((\delta_{\lie b}\tensor \id_{\lie a})(\mathbf f)_{\mathbf 2,\mathbf 4,\mathbf 5}).
\end{align*}
Thus, by Proposition~\ref{prop:rel class Drinfeld}, it remains to prove that 
\begin{align}
(\id_{(U(\lie a)\tensor U(\lie b))^{\tensor3}}&-\tau_{\mathbf 2,\mathbf 4}\tau_{\mathbf 1,\mathbf 3}+\tau_{\mathbf 2,\mathbf 6}\tau_{\mathbf 3,\mathbf 5}\tau_{\mathbf 1,\mathbf 3})((\id_{\lie b}\tensor \delta_{\lie a})(\mathbf f)_{\mathbf 2,\mathbf 3,\mathbf 5})\nonumber\\
&=\sum_{1\le i,k\le n-1} [r_{n,i+n}^{(c_{i,n})},r_{n,k+2n}^{(c_{k,n})}]+[r_{2n,i+2n}^{(c_{i,n})},r_{2n,k}^{(c_{k,n})}]+[r_{3n,i}^{(c_{i,n})},
r_{3n,k+n}^{(c_{k,n})}].\label{eq:jc rel equiv}
\end{align}
By Proposition~\ref{prop:compat cond}\ref{prop:compat cond.d} and the definition of~$\delta_{\lie a}$ we have in $U(\lie b)\tensor U(\lie a)\tensor U(\lie a)$
\begin{align*}
(\id_{\lie b}\tensor\delta_{\lie a})(r_{1,i+1}^{(c)})=
[r_{1,i+1}^{(c)},r_{1,n+i}^{(c)}]+[1\tensor (j_{\mathbf c^-})^-,r_{1,i+1}^{(c)}+r_{1,n+i}^{(c)}],
\end{align*}
whence 
$$
(\id_{\lie b}\tensor\delta_{\lie a})(\mathbf f)_{\mathbf 2,\mathbf 3,\mathbf 5}
=\sum_{1\le i\le n-1} [r_{n,n+i}^{(c_{i,n})},r_{n,2n+i}^{(c_{i,n})}]
+[( j_{\mathbf c^-})^-_{[n+1,2n-1]\cup[2n+1,3n-1]},
r_{n,n+i}^{(c_{i,n})}+r_{n,2n+i}^{(c_{i,n})}].
$$
Therefore,
\begin{align*}
(&\id_{(U(\lie a)\tensor U(\lie b))^{\tensor3}}-\tau_{\mathbf 2,\mathbf 4}\tau_{\mathbf 1,\mathbf 3}+\tau_{\mathbf 2,\mathbf 6}\tau_{\mathbf 3,\mathbf 5}\tau_{\mathbf 1,\mathbf 3})((\id_{\lie b}\tensor \delta_{\lie a})(\mathbf f)_{\mathbf 2,\mathbf 3,\mathbf 5})\\
&=\sum_{1\le i\le n-1} 
[r_{n,n+i}^{(c_{i,n})},r_{n,2n+i}^{(c_{i,n})}]-
[r_{2n,i}^{(c_{i,n})},r_{2n,2n+i}^{(c_{i,n})}]
+[r_{3n,i}^{(c_{i,n})},r_{3n,n+i}^{(c_{i,n})}]\\
&\phantom{=}+\sum_{1\le i\le n-1}\Big([( j_{\mathbf c^-})^-_{[n+1,2n-1]\cup[2n+1,3n-1]},
r_{n,n+i}^{(c_{i,n})}+r_{n,2n+i}^{(c_{i,n})}]\\
&\phantom{=}\quad-[( j_{\mathbf c^-})^-_{[1,n-1]\cup[2n+1,3n-1]},
r_{2n,i}^{(c_{i,n})}+r_{2n,2n+i}^{(c_{i,n})}]+[( j_{\mathbf c^-})^-_{[1,n-1]\cup[n+1,2n-1]},
r_{3n,i}^{(c_{i,n})}+r_{3n,n+i}^{(c_{i,n})}]\Big),
\end{align*}
and so~\eqref{eq:jc rel equiv} is equivalent to
\begin{align*}
&\sum_{1\le i\le n-1}\Big([( j_{\mathbf c^-})^-_{[n+1,2n-1]\cup[2n+1,3n-1]},
r_{n,n+i}^{(c_{i,n})}+r_{n,2n+i}^{(c_{i,n})}]\\
&\phantom{\sum_{1\le i\le n-1}}-[( j_{\mathbf c^-})^-_{[1,n-1]\cup[2n+1,3n-1]},
r_{2n,i}^{(c_{i,n})}+r_{2n,2n+i}^{(c_{i,n})}]+[( j_{\mathbf c^-})^-_{[1,n-1]\cup[n+1,2n-1]},
r_{3n,i}^{(c_{i,n})}+r_{3n,n+i}^{(c_{i,n})}]\Big)\\
&=\sum_{1\le i\not=k\le n-1} [r_{n,i+n}^{(c_{i,n})},r_{n,k+2n}^{(c_{k,n})}]+[r_{2n,i+2n}^{(c_{i,n})},r_{2n,k}^{(c_{k,n})}]+[r_{3n,i}^{(c_{i,n})},
r_{3n,k+n}^{(c_{k,n})}].
\end{align*}
Since~$\mathbf c=(\mathbf c^-)^{\boldsymbol{\alpha}_{\mathbf c}}$ by Lemma~\ref{lem:rest ext trans}\ref{lem:rest ext trans.c} and hence is transitive,
it remains to use Proposition~\ref{prop:identity jc} with~$m=n-1$, $\boldsymbol{\gamma}=\mathbf c^-$ and~$\boldsymbol{\alpha}=\boldsymbol{\alpha}_{\mathbf c}$ 
and apply the homomorphism $\Psi^{(3n)}_{\mathbf r}$.
\end{proof}
By Lemma~\ref{lem:rec class j}, $j_{\mathbf c}=(j_{\mathbf c^-})_{[1,2n-1]\setminus\{n\}}+\mathbf f_{[n,2n-1]}=(j_{\mathbf c^-})_{\mathbf 1,\mathbf 3}+
\mathbf f_{\mathbf 2,\mathbf 3}$, where $\mathbf 1=[n-1]$, $\mathbf 2=\{n\}$, $\mathbf 3=[n+1,2n-1]$.
Then by Proposition~\ref{prop:key prop classical} and Corollary~\ref{cor:clas twist from rel twist}, $j_{\mathbf c}$ is a classical Drinfeld twist for~$\lie g^{\oplus n}$ with its standard cobracket~$\delta_{\lie g^{\oplus n}}$.
\end{proof}
\begin{corollary}
For any transitive~$\mathbf c:I_n\to C$, $\sum\limits_{2\le i\le n} r_{i-1,n}^{(\mathbf c(1,i))}$ is a relative Drinfeld twist 
for~$(\lie a\oplus \lie b,\delta_{\lie a\oplus \lie b})$ where~$(\lie a,\delta_{\lie a})=(\lie g,\delta)$ and~$\lie b=\lie g^{\oplus (n-1)}$ with~$\delta_{\lie b}$ obtained by twisting~$\delta_{\lie g^{\oplus(n-1)}}$ by~$j_{\mathbf c^+}$.
\end{corollary}
\begin{proof}
This is an immediate consequence of Lemma~\ref{lem:rec class j}, Proposition~\ref{prop:tensortwist}
and Theorem~\ref{thm:trans classical twists}.
\end{proof}

\subsection{Proof of Theorem~\ref{thm:main thm trans}}\label{subs:pf m thm 1}
We now have all necessary ingredients to the main 
result announced in the Introduction in the classical case.
\begin{proof}[Proof of Theorem~\ref{thm:main thm trans}]
Let~$\mathbf c:I_n\to C$ be transitive 
and let~$\boldsymbol{d}=(d_1,\dots,d_n)\in C^n$.
Then~$\mathbf r^{(\diag \boldsymbol d)}=\sum_{i\in [n]} r_{i,i+n}^{(d_i)}$
is a classical r-matrix for~$\lie g^{\oplus n}$ 
with its standard cobracket~$\delta_{\lie g^{\oplus n}}$.
Since~$j_{\mathbf c}=j_{\mathbf c}(\mathbf r)$ is a classical Drinfeld twist
by Theorem~\ref{thm:trans classical twists},
it follows from Proposition~\ref{prop:cnd classical twist}
\ref{prop:cnd classical twist.b} that~$\mathbf r^{(\diag\boldsymbol d)}+j_{\mathbf c}^-$ is a classical r-matrix for~$\lie g^{\oplus n}$. It remains to observe that
\begin{equation*}
\mathbf r^{(\diag\boldsymbol d)}+j_{\mathbf c}^-=
\sum_{1\le i\le n} r_{i,i+n}^{(d_i)}+
\sum_{1\le i<j\le n} r_{j,i+n}^{(\mathbf c(i,j))}
-\sum_{1\le i<j\le n} r_{j+n,i}^{(\mathbf c(i,j))}
=\mathbf r(\mathbf c,\boldsymbol d).\qedhere
\end{equation*}
\end{proof}
\begin{remark}\label{rem:pf thm 1.1}
Let~$r$ be an r-matrix and let~$\mathbf r=\{r^{(1)},r^{(-1)}\}=\{r,-\tau(r)\}$. 
Since~$\sgna w:I_n\to \{1,-1\}$ is transitive 
by Lemma~\ref{lem:trans sign}, 
$\mathbf r(\sgna w,\boldsymbol d)$,
$\boldsymbol{d}\in\{1,-1\}^n$ is an r-matrix for~$\lie g^{\oplus n}$ by Lemma~\ref{lem:bas classical family} and Theorem~\ref{thm:main thm trans}.
We have 
\begin{align*}
\mathbf r(\sgna w,\boldsymbol{d})&=
\sum_{1\le i\le n} r_{i,i+n}^{(d_i)}+
\sum_{1\le i<j\le n} (r_{j,i+n}^{(\sign(w(j)-w(i)))}-
r_{j+n,i}^{(\sign(w(j)-w(i)))})
\\
&=\sum_{1\le i\le n} r_{i,i+n}^{(d_i)}+
\sum_{1\le i<j\le n} r_{j,i+n}^{(\sign(w(j)-w(i)))}+
r_{i,j+n}^{(\sign(w(i)-w(j)))}\\
&=\sum_{i,j\in[n]} r_{i,j+n}^{(d_i\delta_{i,j}+\sign(w(j)-w(i)))}=
\mathbf r^{(\sgna{w,\boldsymbol{d}})},
\end{align*}
which proves Theorem~\ref{thm:thm 1}.
\end{remark}
\begin{remark}\label{rem:isom permutations}
Since~$w$ acts on~$\lie g^{\oplus n}$ by permutation of factors and $(w\tensor w)(\mathbf r^{\sgna{\id,
\boldsymbol d})})=\mathbf r^{(\sgna{w,w(\boldsymbol d)})}$,
it follows that all r-matrices described in Theorem~\ref{thm:thm 1} are equivalent
to $\mathbf r(\sgna{\id},\boldsymbol d)$
for some~$\boldsymbol d\in \{1,-1\}^n$,
and the corresponding bialgebra 
structures on~$\lie g^{\oplus n}$ are isomorphic.
\end{remark}

\subsection{Diagonal embedding}
Given a Lie algebra~$\lie g$, the diagonal map~$\Delta^{(n)}:\lie g\to \lie g^{\oplus n}$,
$x\mapsto (x,\dots,x)$, $x\in \lie g$, is obviously a homomorphism of Lie algebras.
However, if~$\lie g$ is a Lie bialgebra with the cobracket~$\delta:\lie g\to 
\lie g\tensor\lie g$, then $\Delta^{(n)}$
is not a homomorphism of Lie bialgebras $(\lie g,\delta)\to 
(\lie g^{\oplus n},\delta_{\lie g^{\oplus n}})$.
\begin{remark}
The notation~$\Delta^{(n)}$ is justified by the fact that, on the level of universal enveloping algebras, the diagonal embedding $\lie g\hookrightarrow \lie g^{\oplus n}$
corresponds to the iterated comultiplication~$\Delta^{(n)}:U(\lie g)\to U(\lie g)^{\tensor n}$.
\end{remark}
\begin{theorem}\label{thm:diag embed bialg}
Let~$(\lie g,\delta)$ be a quasi-triangular Lie bialgebra with a family~$\mathbf r=\{r^{(c)}\}_{c\in C}$ of classical r-matrices corresponding to~$\delta$. 
Then for~$n\ge 2$ and $\mathbf c:I_n\to C$ transitive, $\Delta^{(n)}$ is a homomorphism of Lie bialgebras
$(\lie g,\delta)\to (\lie g^{\oplus n},\tilde\delta_{\mathbf c})$ where~$\tilde\delta_{\mathbf c}:=\widetilde{\delta_{\lie g^{\oplus n}}}_{j_{\mathbf c}(\mathbf r)}$
is the standard cobracket~$\delta_{\lie g^{\oplus n}}$ on~$\lie g^{\oplus n}$ twisted by~$j_{\mathbf c}(\mathbf r)$.
\end{theorem}
\begin{proof}
Let~$\boldsymbol{\alpha}\in C^n$ and let~$\mathbf r(\mathbf c,\boldsymbol\alpha)$
be as in Theorem~\ref{thm:main thm trans}. 
Let~$x\in \lie g$. Then $\Delta^{(n)}(x)=\sum_{1\le i\le n} x_i$ in~$U(\lie g)^{\tensor n}$
where~$x_i=1^{\tensor (i-1)}\tensor x\tensor 1^{\tensor (n-i)}$ and so
\begin{align*}
\tilde\delta_{\mathbf c}&(\Delta^{(n)}(x))=[\mathbf r(\mathbf c,\boldsymbol{\alpha}),\Delta_{U(\lie g)^{\tensor n}}(\Delta^{(n)}(x))]\\
&=
\sum_{k\in[n]} \Big(\sum_{i\in[n]} [r_{i,n+i}^{(\alpha_i)},x_k+x_{n+k}]
+\sum_{(i,j)\in I_n} [r_{j,n+i}^{(\mathbf c(i,j))},x_k+x_{n+k}]+
[-r_{n+j,i}^{(\mathbf c(i,j))},x_k+x_{n+k}]\Big)
\\
&=
\sum_{i\in[n]} [r_{i,n+i}^{(\alpha_i)},x_i+x_{n+i}]
+\sum_{(i,j)\in I_n} [r_{j,n+i}^{(\mathbf c(i,j))},x_j+x_{n+i}]+
[-r_{n+j,i}^{(\mathbf c(i,j))},x_i+x_{n+j}]\\
&=
\sum_{i\in[n]} [r^{(\alpha_i)},\Delta(x)]_{i,n+i}
+\sum_{(i,j)\in I_n} [r^{(\mathbf c(i,j))},\Delta(x)]_{j,n+i}+
[-\tau(r^{(\mathbf c(i,j))}),\Delta(x)]_{i,n+j}\\
&=\sum_{i,j\in[n]}\delta(x)_{i,j+n}
=(\Delta^{(n)}\tensor \Delta^{(n)})(\delta(x)),
\end{align*}
since~$[r^{(c)},\Delta(x)]=[-\tau(r^{(c)}),\Delta(x)]=\delta(x)$ for all~$c\in C$, 
$x\in\lie g$.
\end{proof}

We now establish the necessary and sufficient condition
on our family~$\{r^{(c)}\}_{c\in C}$
for~$\mathbf r^{(\mathbf c)}$ to be an r-matrix for~$\lie g^{\oplus n}$, assuming that it satisfies CYBE.
\begin{lemma}\label{lem:quasi invariant}
Let~$\mathbf r=\{ r^{(c)}\}_{c\in C}\subset \lie g\tensor \lie g$ be such that $r^{(c)}+r_{2,1}^{(c)}$, $c\in C$ is $\lie g$-invariant. Then for any~$\mathbf c:[n]\times[n]\to C$, $\mathbf r^{(\mathbf c)}+
\mathbf r^{(\mathbf c)}{}^{op}$ is $\lie g^{\oplus n}$-invariant if and only if 
$r^{(c_{i,k})}+\tau(r^{(c_{k,i})})\in \lie z(\lie g)\tensor \lie z(\lie g)$ for all~$(i,k)\in I_n$, where~$\lie z(\lie g)$
is the center of~$\lie g$.
\end{lemma}
\begin{proof}
Clearly, $\mathbf f\in \lie g^{\oplus n}\tensor \lie g^{\oplus n}$ is~$\lie g^{\oplus n}$-invariant
if and only if~$[\mathbf f,x_k+x_{k+n}]=0$ for all~$x\in\lie g$ and~$k\in[n]$, where, as before, $x_r=1^{\tensor(r-1)}\tensor x\tensor 1^{\tensor(n-r)}$, $r\in [n]$.
We have 
\begin{align*}
[\mathbf r^{(\mathbf c)}&+\mathbf r^{(\mathbf c)}{}^{op},x_k+x_{k+n}]=
\sum_{i,j\in[n]} [r_{i,j+n}^{(c_{j,i})}+
r_{i+n,j}^{(c_{j,i})},x_k+x_{k+n}]\\
&=[r_{k,k+n}^{(c_{k,k})}+r_{k+n,k}^{(c_{k,k})},x_k+x_{k+n}]
+\sum_{i\in[n]\setminus\{k\}}
([r_{k,i+n}^{(c_{i,k})}+r_{i+n,k}^{(c_{k,i})},x_k]
+[r_{i,k+n}^{(c_{k,i})}+r_{k+n,i}^{(c_{i,k})},x_{k+n}])\\
&=
\sum_{i\in[n]\setminus\{k\}}
[r_{k,i+n}^{(c_{i,k})}+r_{i+n,k}^{(c_{k,i})},x_k]
+\sum_{i\in[n]\setminus\{k\}}
[r_{k+n,i}^{(c_{i,k})}+r_{i,k+n}^{(c_{k,i})},x_{k+n}]
\end{align*}
since the first term is $[r^{(c_{k,k})}+\tau(r^{(c_{k,k})}),x\tensor 1+1\tensor x]_{k,k+n}=0$ as $r^{(c)}+\tau(r^{(c)})$ is $\lie g$-invariant for all~$c\in C$. Since the second sum
is obtained by applying ${}^{op}$ to the first,
it follows that
$[\mathbf r^{(\mathbf c)}+\mathbf r^{(\mathbf c)}{}^{op},x_k+x_{k+n}]=0$
if and only if
$$
\sum_{i\in[n]\setminus \{k\}}[r_{k,i+n}^{(c_{i,k})}+r_{i+n,k}^{(c_{k,i})},x_k]=0,
$$
which in turn is equivalent to
$$
\sum_{1\le i\le k-1} 
[(\tau(r^{(c_{i,k})})+r^{(c_{k,i})})_{i,k},x_k]
+\sum_{k+1\le i\le n}
[(r^{(c_{i,k})}+\tau(r^{(c_{k,i})}))_{k,i},x_k]=0
$$
as an element of~$U(\lie g)^{\tensor n}$. Since~$x$
is arbitrary, this is equivalent to
$$
\sum_{1\le i\le k-1} 
(\tau(r^{(c_{i,k})})+r^{(c_{k,i})})_{i,k}
+\sum_{k+1\le i\le n}
(r^{(c_{i,k})}+\tau(r^{(c_{k,i})}))_{k,i}
\in 
U_1(\lie g)^{\tensor(k-1)}\tensor\lie z(\lie g)\tensor 
U_1(\lie g)^{\tensor (n-k)},
$$
or, finally, to~$r^{(c_{i,k})}+\tau(r^{(c_{k,i})})
\in (\lie g\tensor \lie z(\lie g))
\cap (\lie z(\lie g)\tensor\lie g)=
\lie z(\lie g)\tensor\lie z(\lie g)$, $i\not=k\in[n]$.
\end{proof}

\begin{remark}
The condition that~$r^{(c)}+r_{2,1}^{(c)}$ is an invariant
for each~$c\in C$ is reminiscent of~\cite{FR}. However,
the authors do not consider CYBE involving more than one member of the family $\{r^{(c)}\}_{c\in C}$.
\end{remark}

\subsection{Poisson algebras and proof of Theorem~\ref{thm:Poisson mult}}\label{subs:Poisson}
Let~$\mathbf c:I_n\to C$ be transitive. By Proposition~\ref{prop:pois twist}, $\kk[G]^{\tensor n}$
acquires a Poisson algebra structure via
$$
\{ f,f'\}_{\mathbf c}:=\mu( \mathbf r(\mathbf c,\boldsymbol d)\bowtie (f\tensor f')),\qquad f,f'\in\kk[G]^{\tensor n}
$$
for any~$\boldsymbol d\in C^n$, where~$\mu:\kk[G]^{\tensor n}\tensor\kk[G]^{\tensor n}\to \kk[G]^{\tensor n}$ is the multiplication map. It should be noted that this bracket is independent of~$\boldsymbol d$, which explains the notation. More explicitly, $\kk[G]$ is 
generated, as an algebra, by $f^{(k)}:=1^{\tensor(k-1)}\tensor f\tensor 1^{\tensor (n-k)}$, $f\in\kk[G]$,
$k\in[n]$. Thus, the Poisson bracket~$\{\cdot,\cdot\}_{\mathbf c}$ is uniquely determined
by $\{f'{}^{(k')},f^{(k)}\}_{\mathbf c}$ where~$f,f'\in\kk[G]$ and~$k\le k'\in [n]$.
Furthermore, by the definition~\eqref{eq:r c d} of~$\mathbf r(\mathbf c,\boldsymbol d)$ and Proposition~\ref{prop:pois twist},
\begin{equation}\label{eq:c Poi bracket}
\{ f'{}^{(k')},f^{(k)}\}_{\mathbf c}=
\begin{cases}
(r^{(\mathbf c(k,k'))}\bowtie(f'\tensor f))_{k,k'},&k<k',\\
\{f',f\}^{(k)},&k=k',
\end{cases}
\end{equation}
for all~$f,f'\in\kk[G]$, $k\le k'\in [n]$, where~$\{\cdot,\cdot\}$ is the Poisson bracket on~$\kk[G]$
corresponding to~$r^{(c)}$ for {\em any} $c\in C$.
\begin{proof}[Proof of Theorem~\ref{thm:Poisson mult}]
It suffices to prove the assertion for
generators~$f^{(k)}$, $f\in\kk[G]$, $k\in[n]$. The iterated multiplication map
$\mu^{(n)}:\kk[G]^{\tensor n}\to \kk[G]$ is then given by 
$f^{(k)}\mapsto f$, $f\in \kk[G]$. Also, since the Poisson bracket is skew-symmetric, it suffices to prove that $\mu^{(n)}(\{f'{}^{(k')},f^{(k)}\}_{\mathbf c})=\{f',f\}$
for~$f,f'\in\kk[G]$ and~$k\le k'\in[n]$. This is obvious for~$k=k'$, while for~$k<k'$,
$\mu^{(n)}(\{f'{}^{(k')},f^{(k)}\})=\mu(r^{(\mathbf c(k,k'))}\bowtie (f'\tensor f))=\{f',f\}$ for all~$f,f'\in\kk[G]$.
\end{proof}

\section{Main results: the quantum case}\label{sect:main quantum}

\subsection{\texorpdfstring{$C$}{C}-quasi-triangular monoid}\label{subs:QTR(C)}
Let~$C$ be a set and~$n\in\ZZ_{>0}$.
The~$C$-{\em quasi-triangular monoid}~$\mathsf{QTr}^+_n(C)$ is generated by
the $\sfR_{i,j}^{(c)}$, $1\le i\not=j\le n$, $c\in C$ subjects to relations
\begin{alignat}{3}
&\sfR_{i,j}^{(c)}\sfR_{k,l}^{(c')}=\sfR_{k,l}^{(c')}\sfR_{i,j}^{(c)},&&
i\not=j\not=k\not=l\in[n], &\quad & c,c'\in C \label{eq:defn rel QTr comm}
\\
\label{eq:defn rel QTr}
&\sfR_{i,j}^{(c)} \sfR_{i,k}^{(c')} \sfR_{j,k}^{(c'')}=\sfR_{j,k}^{(c'')}\sfR_{i,k}^{(c')} \sfR_{i,j}^{(c)},&\quad&  i\not=j\not=k\in[n], && c'\in\{c,c''\}\subset C.
\end{alignat}
For~$|C|=1$, it 
has the same defining relations as the quasi-triangular group~$\mathsf{QTr}_n$
introduced in~\cite{BEER}.

Let~$N\ge n$, let $\mathbf i=(i_1,\dots,i_n)\in [N]^n$ with~$i_k\not=i_l$, $1\le k<l\le n$ and let~$C'\supset C$. Then the assignments $\sfR_{k,l}^{(c)}\mapsto 
\sfR_{i_k,i_l}^{(c)}$, $1\le k<l\le n$, $c\in C$
define a homomorphism of monoids~$\phi_{\mathbf i}:\mathsf{QTr}^+_n(C)\to \mathsf{QTr}^+_N(C')$.

Let~$H$ be a bialgebra and suppose that~$\boldsymbol R=\{R^{(c)}\}_{c\in C}$ is a family of R-matrices for~$H$ with respect to the same comultiplication. It follows from Proposition~\ref{prop:cQYBE} that
the assignments $\mathsf R_{i,j}^{(c)}\mapsto 
R_{i,j}^{(c)}$, $i\not=j\in[n]$, $c\in C$, define a homomorphism of monoids 
$\Phi^{(n)}_{\boldsymbol R}:\mathsf{QTr}^+_n(C)\to H^{\tensor n}$ for any~$n\ge 2$.

Mirroring~\eqref{eq:jc defn} we define, for any~$n\ge 2$ and~$\boldsymbol\gamma:I_n\to C$,
\begin{equation}\label{eq:Jc defn}
\mathsf J_{\boldsymbol\gamma}=\ascprod_{2\le i\le n}\,\dscprod_{1\le j\le i-1} \sfR_{i,n+j}^{(\boldsymbol\gamma(j,i))}\in\mathsf{QTr}^+_{2n}(C).
\end{equation}
\begin{lemma}\label{lem:Jc rec}
Let~$n\ge 3$. Then
for all~$\boldsymbol\gamma:I_n\to C$
\begin{align}\label{eq:rec Jc}
\mathsf J_{\boldsymbol\gamma}&=\phi_{[1,2n-1]\setminus\{n\}}(\mathsf J_{\boldsymbol\gamma^-})\dscprod_{1\le i\le n-1}
\sfR_{n,i+n}^{(\boldsymbol\gamma(i,n))}\\
&=\phi_{[2,2n]\setminus\{n+1\}}(\mathsf J_{\boldsymbol\gamma^+})\ascprod_{2\le i\le n} \sfR_{i,n+1}^{(\boldsymbol\gamma(1,i))}\label{eq:rec Jc-2}.
\end{align}
\end{lemma}
\begin{proof}
Abbreviate $\gamma_{ij}=\boldsymbol\gamma(i,j)$, $1\le i<j\le n$. 
We have 
$$
\mathsf J_{\boldsymbol\gamma}=\Big(\ascprod_{2\le i\le n-1}\, 
\dscprod_{1\le j\le i-1} \sfR^{(\gamma_{j,i})}_{i,n+j}\Big) \Big(\dscprod_{1\le j\le n-1} \sfR_{n,j+n}^{(\gamma_{j,n})}\Big).
$$
Since
$$
\mathsf J_{\boldsymbol\gamma^-}=\ascprod_{2\le i\le n-1}\,
\dscprod_{1\le j\le i-1} \sfR_{i,n-1+j}^{(\gamma_{ji})},
$$
the first identity follows. To prove the second, write
$$
\mathsf J_{\boldsymbol\gamma}=\ascprod_{2\le i\le n}\Big(
\Big( 
\dscprod_{2\le j\le i-1} \sfR^{(\gamma_{j,i})}_{i,n+j}\Big)\sfR_{i,n+1}^{(\gamma_{1,i})}\Big).
$$
By~\eqref{eq:defn rel QTr comm}, $\sfR_{i,n+1}^{(c)}$ commutes with the~$\sfR_{k,n+j}^{(c')}$, 
$i<k\le n-1$, $2\le j\le k-1$, $c,c'\in C$. Therefore, 
\begin{align*}
\mathsf J_{\boldsymbol\gamma}&=\Big(\ascprod_{3\le i\le n}
\,
\dscprod_{2\le j\le i-1} \sfR^{(\gamma_{j,i})}_{i,n+j}\Big)\Big(\ascprod_{2\le i\le n}\sfR_{i,n+1}^{(\gamma_{1,i})}\Big)=\Big(\ascprod_{2\le i\le n-1}\,
\dscprod_{1\le j\le i-1} \sfR^{(\gamma_{j+1,i+1})}_{i+1,n+j+1}\Big)\Big(\ascprod_{2\le i\le n}\sfR_{i,n+1}^{(\gamma_{1,i})}\Big)\\
&
=\Big(\ascprod_{2\le i\le n-1}
\,
\dscprod_{1\le j\le i-1} \sfR^{(\boldsymbol\gamma^+(i,j))}_{i+1,n+j+1}\Big)\Big(\ascprod_{2\le i\le n}\sfR_{i,n+1}^{(\gamma_{1,i})}\Big)=\phi_{[2,2n]\setminus\{n+1\}}(\mathsf J_{\boldsymbol\gamma^+})\Big(\ascprod_{2\le i\le n}\sfR_{i,n+1}^{(\gamma_{1,i})}\Big).\qedhere
\end{align*}
\end{proof}

The following can be viewed as a quantum analogue of Proposition~\ref{prop:identity jc} 
and plays a crucial role in our proof of Theorem~\ref{thm:main thm 2}.
\begin{proposition}\label{prop:key relation}
Let~$m\ge 2$, $\boldsymbol\gamma:I_m\to C$, $\boldsymbol\alpha=(\alpha_1,\dots,\alpha_m)\in C^m$
and suppose that~$\boldsymbol\gamma^{\boldsymbol\alpha}:I_{m+1}\to C$ is transitive. Then in~$\mathsf{QTr}^+_{2m-1}(C)$
\begin{equation}
\mathsf J_{\boldsymbol\gamma}\Big(\dscprod_{1\le j\le m-1}
\sfR_{1,m+j}^{(\alpha_j)}\Big)\Big(\dscprod_{2\le j\le m}
\sfR_{1,j}^{(\alpha_j)}\Big)
=\sfR_{1,m}^{(\alpha_m)}\Big(\dscprod_{2\le j\le m-1}
\sfR_{1,m+j}^{(\alpha_j)}\sfR_{1,j}^{(\alpha_j)}\Big)\sfR_{1,m+1}^{(\alpha_1)}\,
\mathsf J_{\boldsymbol\gamma}.\label{eq:compressed-monoid}
\end{equation}
\end{proposition}
\begin{proof}
We use induction on~$m$.
For~$m=2$, $\mathsf J_{\boldsymbol\gamma}=\sfR_{2,3}^{(\boldsymbol\gamma(1,2))}$ and so~\eqref{eq:compressed-monoid}
reads 
\begin{equation}\label{eq:interm III}
\sfR_{2,3}^{(\boldsymbol\gamma(1,2))}\sfR_{1,3}^{(\alpha_1)}\sfR_{1,2}^{(\alpha_2)}=
\sfR_{1,2}^{(\alpha_2)}\sfR_{1,3}^{(\alpha_1)}\sfR_{2,3}^{(\boldsymbol\gamma(1,2))}.
\end{equation}
Since~$\boldsymbol\gamma^{\boldsymbol{\alpha}}$ is transitive, $\alpha_1\in
\{\boldsymbol\gamma(1,2),\alpha_2\}$ by Lemma~\ref{lem:rest ext trans}\ref{lem:rest ext trans.a} and so~\eqref{eq:interm III} follows from~\eqref{eq:defn rel QTr}.

To prove the inductive step, we need the following
\begin{lemma}\label{lem:move}
For all~$1\le k\le m$,
\begin{align*}
\Big(\dscprod_{1\le i\le m-1}&
\sfR_{m,i+m}^{(\boldsymbol\gamma(i,m))}\Big)\Big(\dscprod_{1\le j\le m-1}
\sfR_{1,m+j}^{(\alpha_j)}\Big)\Big(\dscprod_{2\le j\le m}
\sfR_{1,j}^{(\alpha_j)}\Big)\\
&=\Big(\dscprod_{k\le i\le m-1}
\sfR_{m,i+m}^{(\boldsymbol\gamma(i,m))}\Big)
\Big(\dscprod_{k\le j\le m-1}
\sfR_{1,m+j}^{(\alpha_j)}\Big)\sfR_{1,m}^{(\alpha_m)}
\Big(\dscprod_{1\le j\le k-1}
\sfR_{1,m+j}^{(\alpha_j)}\Big)\times\\
&\qquad \Big(\dscprod_{2\le j\le m-1}
\sfR_{1,j}^{(\alpha_j)}\Big)\Big(\dscprod_{1\le i\le k-1}
\sfR_{m,i+m}^{(\boldsymbol\gamma(i,m))}\Big).
\end{align*}
\end{lemma}
\begin{proof}
The argument is by induction on~$k$, the case~$k=1$ being trivial.

For the inductive step, suppose that~$1\le k<m$.
Since by~\eqref{eq:defn rel QTr comm} $\sfR_{m,k+m}^{(\boldsymbol\gamma(k,m))}$
commutes with the $\sfR_{1,m+j}^{(\alpha_j)}$, $k+1\le j\le m-1$, we obtain by the induction hypothesis,
\begin{align*}
\dscprod_{1\le i\le m-1}&
\sfR_{m,i+m}^{(\boldsymbol\gamma(i,m))}\Big(\dscprod_{1\le j\le m-1}
\sfR_{1,m+j}^{(\alpha_j)}\Big)\Big(\dscprod_{2\le j\le m}
\sfR_{1,j}^{(\alpha_j)}\Big)\\
&=\Big(\dscprod_{k+1\le i\le m-1}
\sfR_{m,i+m}^{(\boldsymbol\gamma(i,m))}\Big)
\Big(\dscprod_{k+1\le j\le m-1}
\sfR_{1,m+j}^{(\alpha_j)}\Big)\sfR_{m,m+k}^{(\boldsymbol\gamma(k,m))} \sfR_{1,m+k}^{(\alpha_k)}\sfR_{1,m}^{(\alpha_m)}
\times\\
&\qquad \Big(\dscprod_{1\le j\le k-1}
\sfR_{1,m+j}^{(\alpha_j)}\Big)\Big(\dscprod_{2\le j\le m-1}
\sfR_{1,j}^{(\alpha_j)}\Big)\Big(\dscprod_{1\le i\le k-1}
\sfR_{m,i+m}^{(\boldsymbol\gamma(i,m))}\Big).
\end{align*}
Since~$\boldsymbol\gamma^{\boldsymbol\alpha}$ is transitive, $\alpha_k\in \{\boldsymbol\gamma(k,m),\alpha_m\}$ by Lemma~\ref{lem:rest ext trans}\ref{lem:rest ext trans.a}. Therefore, by~\eqref{eq:defn rel QTr}
\begin{align*}
\dscprod_{1\le i\le m-1}&
\sfR_{m,i+m}^{(\boldsymbol\gamma(i,m))}\Big(\dscprod_{1\le j\le m-1}
\sfR_{1,m+j}^{(\alpha_j)}\Big)\Big(\dscprod_{2\le j\le m}
\sfR_{1,j}^{(\alpha_j)}\Big)\\
&=\Big(\dscprod_{k+1\le i\le m-1}
\sfR_{m,i+m}^{(\boldsymbol\gamma(i,m))}\Big)
\Big(\dscprod_{k+1\le j\le m-1}
\sfR_{1,m+j}^{(\alpha_j)}\Big) \sfR_{1,m}^{(\alpha_m)}\sfR_{1,m+k}^{(\alpha_k)}\sfR_{m,m+k}^{(\boldsymbol\gamma(k,m))}
\times\\
&\qquad \Big(\dscprod_{1\le j\le k-1}
\sfR_{1,m+j}^{(\alpha_j)}\Big)\Big(\dscprod_{2\le j\le m-1}
\sfR_{1,j}^{(\alpha_j)}\Big)\Big(\dscprod_{1\le i\le k-1}
\sfR_{m,i+m}^{(\boldsymbol\gamma(i,m))}\Big)\\
&=\Big(\dscprod_{k+1\le i\le m-1}
\sfR_{m,i+m}^{(\boldsymbol\gamma(i,m))}\Big)
\Big(\dscprod_{k+1\le j\le m-1}
\sfR_{1,m+j}^{(\alpha_j)}\Big) \sfR_{1,m}^{(\alpha_m)}
\Big(\dscprod_{1\le j\le k}
\sfR_{1,m+j}^{(\alpha_j)}\Big)\times\\
&\qquad \Big(\dscprod_{2\le j\le m-1}
\sfR_{1,j}^{(\alpha_j)}\Big)\Big(\dscprod_{1\le i\le k}
\sfR_{m,i+m}^{(\boldsymbol\gamma(i,m))}\Big),
\end{align*}
which completes the proof of the inductive step.
\end{proof}
Applying the Lemma with~$k=m$, we obtain
\begin{align*}
    \Big(\dscprod_{1\le i\le m-1}&
\sfR_{m,i+m}^{(\boldsymbol\gamma(i,m))}\Big)\Big(\dscprod_{1\le j\le m-1}
\sfR_{1,m+j}^{(\alpha_j)}\Big)\Big(\dscprod_{2\le j\le m}
\sfR_{1,j}^{(\alpha_j)}\Big)\\
&=\sfR_{1,m}^{(\alpha_m)}
\Big(\dscprod_{1\le j\le m-1}
\sfR_{1,m+j}^{(\alpha_j)}\Big)
\Big(\dscprod_{2\le j\le m-1}
\sfR_{1,j}^{(\alpha_j)}\Big)\Big(\dscprod_{1\le i\le m-1}
\sfR_{m,i+m}^{(\boldsymbol\gamma(i,m))}\Big).
\end{align*}
Therefore, by Lemma~\ref{lem:Jc rec},
\begin{align*}
    \mathsf J&_{\boldsymbol\gamma}\Big(\dscprod_{1\le j\le m-1}
\sfR_{1,m+j}^{(\alpha_j)}\Big)\Big(\dscprod_{2\le j\le m}
\sfR_{1,j}^{(\alpha_j)}\Big)
\\
&=\phi_{[1,2m-1]\setminus\{m\}}(\mathsf J_{\boldsymbol\gamma^-})
\sfR_{1,m}^{(\alpha_m)}
\Big(\dscprod_{1\le j\le m-1}
\sfR_{1,m+j}^{(\alpha_j)}\Big)\Big(\dscprod_{2\le j\le m-1}
\sfR_{1,j}^{(\alpha_j)}\Big)\Big(\dscprod_{1\le i\le m-1}
\sfR_{m,i+m}^{(\boldsymbol\gamma(i,m))}\Big)\\
&=\sfR_{1,m}^{(\alpha_m)}\sfR_{1,2m-1}^{(\alpha_{m-1})}\phi_{[1,2m-1]\setminus\{m\}}(
\mathsf J_{\boldsymbol\gamma^-})
\Big(\dscprod_{1\le j\le m-2}
\sfR_{1,m+j}^{(\alpha_j)}\Big) 
\Big(\dscprod_{2\le j\le m-1}
\sfR_{1,j}^{(\alpha_j)}\Big)\Big(\dscprod_{1\le i\le m-1}
\sfR_{m,i+m}^{(\boldsymbol\gamma(i,m))}\Big)\\
&=\sfR_{1,m}^{(\alpha_m)}\sfR_{1,2m-1}^{(\alpha_{m-1})}\phi_{[1,2m-1]\setminus\{m\}}\Big(
\mathsf J_{\boldsymbol\gamma^-}
\Big(\dscprod_{1\le j\le m-2}\!\!
\sfR_{1,m-1+j}^{(\alpha_j)}\Big) 
\Big(\dscprod_{2\le j\le m-1}\!\!
\sfR_{1,j}^{(\alpha_j)}\Big)\Big)\Big(\dscprod_{1\le i\le m-1}\!\!
\sfR_{m,i+m}^{(\boldsymbol\gamma(i,m))}\Big).
\end{align*}
By Lemma~\ref{lem:rest ext trans}\ref{lem:rest ext trans.b},
$(\boldsymbol\gamma^-)^{\boldsymbol{\alpha}^-}:I_m\to C$ 
is transitive.
Thus, the induction hypothesis applies to~$\boldsymbol\gamma^-:I_{m-1}\to C$ and~$\boldsymbol\alpha^-:[m-1]\to C$, and  we conclude that 
\begin{align*}
\mathsf J_{\boldsymbol\gamma^-}\Big(\dscprod_{1\le j\le m-2}
\sfR_{1,m-1+j}^{(\alpha_j)}\Big) 
\Big(\dscprod_{2\le j\le m-1}
\sfR_{1,j}^{(\alpha_j)}\Big)=\sfR_{1,m-1}^{(\alpha_{m-1})}\Big(\dscprod_{2\le j\le m-2}
\sfR_{1,m-1+j}^{(\alpha_j)}\sfR_{1,j}^{(\alpha_j)}\Big)\sfR_{1,m}^{(\alpha_1)}
\mathsf J_{\boldsymbol\gamma^-}.
\end{align*}
Therefore,
\begin{align*}
    \mathsf J&_{\boldsymbol\gamma}\Big(\dscprod_{1\le j\le m-1}
\sfR_{1,m+j}^{(\alpha_j)}\Big)\Big(\dscprod_{2\le j\le m}
\sfR_{1,j}^{(\alpha_j)}\Big)\\
&=\sfR_{1,m}^{(\alpha_m)}\sfR_{1,2m-1}^{(\alpha_{m-1})}\phi_{[1,2m-1]\setminus\{m\}}\Big(\sfR_{1,m-1}^{(\alpha_{m-1})}\Big(\dscprod_{2\le j\le m-2}\!\!
\sfR_{1,m-1+j}^{(\alpha_j)}\sfR_{1,j}^{(\alpha_j)}\Big)\sfR_{1,m}^{(\alpha_1)}
\mathsf J_{\boldsymbol\gamma^-}\Big)\Big(\dscprod_{1\le i\le m-1}\!\!
\sfR_{m,i+m}^{(\boldsymbol\gamma(i,m))}\Big)\\
&=\sfR_{1,m}^{(\alpha_m)}\sfR_{1,2m-1}^{(\alpha_{m-1})}\sfR_{1,m-1}^{(\alpha_{m-1})}\Big(\dscprod_{2\le j\le m-2}\!\!
\sfR_{1,m+j}^{(\alpha_j)}\sfR_{1,j}^{(\alpha_j)}\Big)\sfR_{1,m+1}^{(\alpha_1)}\phi_{[1,2m-1]\setminus\{m\}}(\mathsf J_{\boldsymbol\gamma^-})\Big(\dscprod_{1\le i\le m-1}
\sfR_{m,i+m}^{(\boldsymbol\gamma(i,m))}\Big)\\
&=\sfR_{1,m}^{(\alpha_m)}\Big(\dscprod_{2\le j\le m-1}
\sfR_{1,m+j}^{(\alpha_j)}\sfR_{1,j}^{(\alpha_j)}\Big)\sfR_{1,m+1}^{(\alpha_1)}
\mathsf J_{\boldsymbol\gamma},
\end{align*}
which completes the proof of the inductive step.
\end{proof}

\subsection{A family of Drinfeld twists}\label{subs:fam Drinf twist}
We will now prove the main result of this section.
\begin{proof}[Proof of Theorem~\ref{thm:trans twist}]
Note that~$J_{\mathbf c}=J_{\mathbf c}(\mathbf R)$
defined in Theorem~\ref{thm:trans twist} equals
$\Phi^{(2n)}_{\mathbf  R}(\mathsf J_{\mathbf c})$.
Abbreviate $c_{j,i}=\mathbf c(j,i)$, $1\le j<i\le n$. We use the induction on~$n$, the case~$n=1$ being trivial. For $n=2$, we have~$J_{\mathbf c}=R_{23}^{(c_{12})}$, and the assertion follows from Proposition~\ref{prop:basic twist}.

Suppose that~$n\ge 3$ and let~$\mathbf c:I_n\to C$ be transitive. Then~$\mathbf c^-$
is transitive and so, 
by the induction hypothesis, $J_{\mathbf c^-}$ is a Drinfeld twist for~$H^{\tensor(n-1)}$ with its standard comultiplication.
The key ingredient in our proof is
\begin{proposition}\label{prop:key proposition}
Suppose that~$n\ge 3$. Let~$A=H^{\tensor (n-1)}$ with the comultiplication
$\Delta_A:A\to A\tensor A$ defined by~$\Delta_A(x)=J_{\mathbf c^-}{}^{-1}\Delta_{H^{\tensor(n-1)}}(x)J_{\mathbf c^-}$, $x\in H^{\tensor(n-1)}$, and let~$B=H$ with~$\Delta_B=\Delta$. Then
$F=\dscprod\limits_{1\le j\le n-1} R_{1,j+1}^{(c_{j,n})}\in B\tensor A$
is a relative Drinfeld twist for $(A\tensor B,\Delta_{A\tensor B})$.
\end{proposition}
\begin{proof}
Since~$(\varepsilon_B\tensor\id_A)(F)$ and~$(\id_B\tensor\varepsilon_A)(F)$ are manifestly central in respective algebras, 
by Proposition~\ref{prop:rel Drinfeld twist} we only need to prove that 
\begin{equation}\label{eq:rel twist to prove}
[(\Delta_B\tensor\id_A)(F)]_{(\mathbf 1,\mathbf 3,\mathbf 4)}(F\tensor 1_B\tensor 1_A)
=[(\id_B\tensor \Delta_A)(F)]_{(\mathbf 1,\mathbf 2,\mathbf 4)}(1_B\tensor 1_A\tensor F).
\end{equation}
We have
\begin{align}
[(\Delta_B\tensor\id_A)(F)]_{(\mathbf 1,\mathbf 3,\mathbf 4)}&=\dscprod_{2\le j\le n}
[(\Delta\tensor\id_{H^{\tensor (n-1)}})(R_{1,j}^{(c_{j-1,n})})]_{\{1\}\cup[n+1,2n]}\nonumber\\
&=\dscprod_{2\le j\le n}
(R_{1,n+j}^{(c_{j-1,n})}R_{n+1,n+j}^{(c_{j-1,n})})=\Big(\dscprod_{2\le j\le n}
R_{1,n+j}^{(c_{j-1,n})}\Big)\Big(\dscprod_{2\le j\le n}
R_{n+1,n+j}^{(c_{j-1,n})}\Big)
\nonumber\\
&=\Big(\dscprod_{2\le j\le n}
R_{1,n+j}^{(c_{j-1,n})}\Big)(1_B\tensor 1_A\tensor F).
\label{eq:F13,4}
\end{align}
Since~$F$ is invertible, it is therefore sufficient to prove that 
\begin{equation}\label{eq:F step I}
\Big(\dscprod_{2\le j\le n}
R_{1,n+j}^{(c_{j-1,n})}\Big)\Big(\dscprod_{2\le j\le n}
R_{1,j}^{(c_{j-1,n})}\Big)=[(\id_B\tensor \Delta_A)(F)]_{(\mathbf 1,\mathbf 2,\mathbf 4)}.
\end{equation}
Since both sides of this expression are contained in~$B\tensor A\tensor 1_B\tensor A$, \eqref{eq:F step I} is equivalent to
$$
\Big(\dscprod_{2\le j\le n}
R_{1,n+j-1}^{(c_{j-1,n})}\Big)\Big(\dscprod_{2\le j\le n}
R_{1,j}^{(c_{j-1,n})}\Big)=(\id_B\tensor \Delta_A)(F)
$$
in~$B\tensor A\tensor A$.
Furthermore, 
\begin{align*}
(\id_B&\tensor \Delta_A)(F)
=(1_B\tensor J_{\mathbf c^-}{}^{-1})\cdot (\id_B\tensor \Delta_{H^{\tensor(n-1)}})(F)\cdot(1_B\tensor J_{\mathbf c^-})\\
&=(1_B\tensor J_{\mathbf c^-}{}^{-1})\cdot (\Delta_{2,n+1}\cdots\Delta_{n,2n-1})\Big(\dscprod_{2\le j\le n}
R_{1,j}^{(c_{j-1,n})}
\Big)
\cdot(1_B\tensor J_{\mathbf c^-})\\
&=(1_B\tensor J_{\mathbf c^-}^{-1})\cdot \Big(\dscprod_{2\le j\le n}
R_{1,n+j-1}^{(c_{j-1,n})}R_{1,j}^{(c_{j-1,n})}
\Big)
\cdot(1_B\tensor J_{\mathbf c^-})
\end{align*}
in~$B\tensor A\tensor A$.
Thus, \eqref{eq:rel twist to prove} is equivalent to
$$
(1_H\tensor J_{\mathbf c^-})\Big(\dscprod_{2\le j\le n}
R_{1,n+j-1}^{(c_{j-1,n})}\Big)\Big(\dscprod_{2\le j\le n}
R_{1,j}^{(c_{j-1,n})}\Big)=\Big(\dscprod_{2\le j\le n}
R_{1,n+j-1}^{(c_{j-1,n})}R_{1,j}^{(c_{j-1,n})}
\Big)(1_H\tensor J_{\mathbf c^-})
$$
inside~$H^{\tensor n}\tensor H^{\tensor n}$. Note that,
since~$J_{\mathbf c^-}\in 1_H\tensor (H^{\tensor n-2})^{\tensor 2}\tensor 1_H\subset 
(H^{\tensor (n-1)})^{\tensor 2}$, it follows that $1_H\tensor J_{\mathbf c^-}$
commutes with $R_{1,2}^{(c_{1n})}$ and 
with~$R_{1,2n-1}^{(c_{n-1,n})}$. Therefore,
\eqref{eq:rel twist to prove} is equivalent to
\begin{multline*}
(1_H\tensor J_{\mathbf c^-})\Big(\dscprod_{2\le j\le n-1}
R_{1,n+j-1}^{(c_{j-1,n})}\Big)\Big(\dscprod_{3\le j\le n}
R_{1,j}^{(c_{j-1,n})}\Big)\\
=R_{1,n}^{(c_{n-1,n})}\Big(\dscprod_{3\le j\le n-1}
R_{1,n+j-1}^{(c_{j-1,n})}R_{1,j}^{(c_{j-1,n})}\Big)R_{1,n+1}^{(c_{1,n})}
(1_H\tensor J_{\mathbf c^-}).
\end{multline*}
Finally, observe that both sides of the above 
equation are contained in $H\tensor 1_H\tensor H^{\tensor 2n-3}\tensor 1_H$. Therefore, \eqref{eq:rel twist to prove} is equivalent to
\begin{align}
J_{\mathbf c^-}\Big(\dscprod_{1\le j\le n-2}\!\!
R_{1,n+j-1}^{(c_{j,n})}\Big)&\Big(\dscprod_{2\le j\le n-1}\!\!
R_{1,j}^{(c_{j,n})}\Big)=R_{1,n-1}^{(c_{n-1,n})}\Big(\dscprod_{2\le j\le n-2}\!\!
R_{1,n+j-1}^{(c_{j,n})}R_{1,j}^{(c_{j,n})}\Big)R_{1,n}^{(c_{1,n})}
J_{\mathbf c^-}\label{eq:compressed}
\end{align}
in~$(H^{\tensor (n-1)})^{\tensor 2}$.  By Lemma~\ref{lem:rest ext trans}\ref{lem:rest ext trans.c}, $(\mathbf c^-)^{\boldsymbol\alpha_{\mathbf c}}=\mathbf c$ is transitive. Then~\eqref{eq:compressed} follows 
from Proposition~\ref{prop:key relation} with $m=n-1$, $\boldsymbol{\gamma}=\mathbf c$ 
and~$\boldsymbol{\alpha}=\boldsymbol{\alpha}_{\mathbf c}$
by applying
the homomorphism $\Phi^{(2n)}_{\mathbf R}$.
\end{proof}
By~\eqref{eq:rec Jc}, $J_{\mathbf c}=(J_{\mathbf c^-})_{[1,2n-1]\setminus\{n\}}(1_A\tensor F\tensor 1_B)=(J_{\mathbf c^-})_{\mathbf 1,\mathbf 3}F_{\mathbf 2,\mathbf 3}$
where~$\mathbf 1=[n-1]$, $\mathbf 2=\{n\}$, $\mathbf 3=\{n+1,2n-1\}$. Then
by
Proposition~\ref{prop:key proposition} and Corollary~\ref{cor:twist from rel twist},
$J_{\mathbf c}$ 
is a Drinfeld twist for~$H^{\tensor n}$ with its standard comultiplication. This proves the first
assertion of Theorem~\ref{thm:trans twist}. Since
$\prod_{i\in[n]}R_{i,i+n}^{(d_i)}$ is an R-matrix
for~$\Delta_{H^{\tensor n}}$ for any $d_1,\dots,d_n\in C$,
the second 
assertion follows from the first and Proposition~\ref{prop:Drinfeld twist}\ref{prop:Drinfeld twist.b}.
\end{proof}
\begin{corollary}\label{cor:other recursion Jc}
For any transitive~$\mathbf c:I_n\to C$, $\ascprod\limits_{2\le i\le n} R_{i-1,n}^{(\mathbf c(1,i))}$ is a relative Drinfeld twist 
for~$A\tensor B$ where~$A=H$ and~$B=H^{\tensor (n-1)}$ with~$\Delta_B$ obtained by twisting~$\Delta_{H^{\tensor(n-1)}}$ by~$J_{\mathbf c^+}$.
\end{corollary}
\begin{proof}
This is an immediate consequence of Lemmata~\ref{lem:Jc rec} and~\ref{lem:rel twist from twist} and Theorem~\ref{thm:trans twist}.
\end{proof}
\begin{proof}[Proof of Theorem~\ref{thm:main thm 2}]
Given a permutation~$w\in S_n$ and an~R-matrix $R$ for~$H$, $J_w(R)$ defined in Theorem~\ref{thm:main thm 2}
coincides with~$J_{\sgna{w}}(\mathbf  R)$ where
$\sgna{w}:I_n\to\{1,-1\}$ is defined as in Lemma~\ref{lem:trans sign} and~$\mathbf R=\{R^{(1)},R^{(-1)}\}=
\{R,R_{2,1}^{-1}\}$. It remains to apply Theorem~\ref{thm:trans twist} with~$\mathbf c=\sgna{w}$.
\end{proof}
We now prove a special case of Conjecture~\ref{conj:QCYBE trans 1,-1}.
\begin{proposition}\label{prop:Spec trans conj}
Let~$H$ be a quasi-triangular bialgebra with an R-matrix~$R$,
and let $\mathbf R=\{R^{(1)},R^{(-1)}\}=\{R,R_{2,1}^{-1}\}$.
Suppose that~$\boldsymbol a=(a_{i,j})_{i,j\in[n]}$,
$a_{i,j}\in \{1,-1\}$, $i,j\in[n]$ is
transitive and satisfies $a_{j,i}=-a_{i,j}$, $(i,j)\in I_n$. Then~$\mathbf R^{(\boldsymbol a)}$ solves QYBE.
\end{proposition}
\begin{proof}
By Lemma~\ref{lem:In trans to trans}, 
there exists~$w\in S_n$ such that $\boldsymbol a=\boldsymbol\epsilon(w,\boldsymbol d)$ where
$\boldsymbol d=(a_{i,i})_{i\in [n]}$.
We claim that
$$
\mathbf R^{(\boldsymbol a)}=
\dscprod_{1\le j\le n} \ascprod_{1\le i\le n} 
R_{i,n+j}^{(a_{j,i})}=\mathbf R(\sgna{w},\boldsymbol d).
$$
In particular, $\mathbf R^{(\boldsymbol a)}$
is an R-matrix by Theorem~\ref{thm:main thm 2}\ref{thm:main thm 2.b} and
hence satisfies QYBE.

To prove the claim, we use induction on~$n$, the case~$n=1$ being trivial. For the inductive step, we have
\begin{align*}
\mathbf R^{(\boldsymbol{a})}&=
\Big(\dscprod_{2\le j\le n} R_{1,n+j}^{(a_{j,1})}\Big(\ascprod_{2\le i\le n} R_{i,n+j}^{(a_{j,i})}\Big)\Big)R_{1,n+1}^{(a_{1,1})}
\ascprod_{2\le i\le n} R_{i,n+1}^{(a_{1,i})}\\
&=\Big(\dscprod_{2\le i\le n} R_{1,n+i}^{(-a_{1,i})}
\Big) (\mathbf R^{(\boldsymbol{a}^+)})_{[2,2n]\setminus\{n+1\}}
R_{1,n+1}^{(a_{1,1})}
\ascprod_{2\le i\le n} R_{i,n+1}^{(a_{1,i})},
\end{align*}
where, as before, $\boldsymbol a^+:[n-1]\times[n-1]\to\{1,-1\}$ is defined by $\boldsymbol a^+(i,j)=\boldsymbol a(i+1,j+1)$. By the induction hypothesis,
$$\mathbf R^{(\boldsymbol a^+)}=\mathbf R(\sgna{w}^+,\boldsymbol d^+)=(J_{\sgna{w}^+}^{op})^{-1}
\prod_{i\in[n-1]} R_{i,n-1+i}^{(a_{i+1,i+1})} J_{\sgna{w}}.
$$
Since~$R^{(-\epsilon)}=R_{2,1}^{(\epsilon)}{}^{-1}$, it follows that
\begin{align*}
\mathbf R^{(\boldsymbol{a})}&=\Big(\Big((J_{\sgna{w}^+})_{[2,2n]\setminus\{n+1\}}\ascprod_{2\le i\le n} R_{i,n+1}^{(a_{1,i})}\Big)^{op}\Big)^{-1}
\prod_{i\in[n]} R_{i,i+n}^{(a_{i,i)}} \Big((J_{\sgna{w}^+})_{[2,2n]\setminus\{n+1\}}\ascprod_{2\le i\le n} R_{i,n+1}^{(a_{1,i})}\Big)\\
&=(J_{\sgna w}^{op})^{-1}\Big(\prod_{i\in[n]} R_{i,i+n}^{(a_{i,i)}}\Big) J_{\sgna w}=
\mathbf R(\sgna{w},\boldsymbol{d}),
\end{align*}
where we used~\eqref{eq:rec Jc-2}.
\end{proof}

\subsection{Diagonal homomorphism of bialgebras}
We now establish a quantum analogue of Theorem~\ref{thm:diag embed bialg}.
\begin{theorem}\label{thm:diag embed quantum}
Let~$H$ be a quasi-triangular bialgebra with a family of ~R-matrices
$\mathbf R=\{ R^{(c)}\}_{c\in C}$ and let~$\mathbf c:I_n\to C$, $n\in\ZZ_{\ge1}$ be transitive.
Then the iterated comultiplication $\Delta^{(n)}:=
\dscprod\limits_{0\le t\le n-2}(\Delta\tensor\id_H^{\tensor t})=
\dscprod\limits_{0\le t\le n-2}(\id_H^{\tensor t}\tensor\Delta)$ is
a homomorphism of bialgebras~$H\to H^{\tensor n}$ where
the standard comultiplication $\Delta_{H^{\tensor n}}$ is twisted by~$J_{\mathbf c}(\mathbf R)$.
\end{theorem}
\begin{proof}
The argument is by induction on~$n$, the case~$n=1$ being
trivial. For the inductive step, let 
$A=H$ and let $B=H^{\tensor (n-1)}$ with the 
comultiplication twisted by~$J_{\mathbf c^+}(\mathbf R)$.
By the induction hypothesis, $\Delta^{(n-1)}:H\to B$
is a homomorphism of bialgebras, and  
$F=\ascprod\limits_{2\le i\le n} R_{i-1,n}^{(\mathbf c(1,i))}$ is a relative Drinfeld twist for~$A\tensor B$
by Corollary~\ref{cor:other recursion Jc}.
Since~$\Delta^{(n)}=(\id_H\tensor\Delta^{(n-1)})\circ\Delta$, 
by Proposition~\ref{prop:rel twist comult prop}
with~$U=H$, $\psi_A=\id_H$ and~$\psi_B=\Delta^{(n-1)}$
it suffices to prove that  
$$
F\cdot\Delta^{(n)}(h)=(\tau_{\mathbf 1,\mathbf 2}\circ\Delta^{(n)})(h)\cdot F,\qquad h\in H
$$
where we abbreviate~$\mathbf 1=\{1\}$, $\mathbf 2=[2,n]$, or in Sweedler notation
\begin{equation}\label{eq:hom coalg id to prove}
\Big(\ascprod\limits_{1\le i\le n-1} R_{i,n}^{(\mathbf c(1,i+1))}\Big)
\Big(\ascprodtens_{1\le t\le n} h_{(t)}\Big)=
\Big(\Big(\ascprodtens_{2\le t\le n}h_{(t)}\Big)\tensor h_{(1)}\Big)\Big(
\ascprod\limits_{1\le i\le n-1} R_{i,n}^{(\mathbf c(1,i+1))}\Big).
\end{equation}
We need the following
\begin{lemma}\label{lem:hom bialg intern}
For all~$h\in H$, $c_1,\dots,c_{n-1}\in C$
and~$1\le k\le n$
\begin{align*}
\Big(\ascprod\limits_{1\le i\le n-1} R_{i,n}^{(c_i)}\Big)
\cdot\Big(\ascprodtens_{1\le t\le n} h_{(t)}\Big)
=\Big(\ascprod_{1\le i\le k-1} R_{i,n}^{(c_i)}
\Big)\cdot\Big(\Big(\ascprodtens_{t\in [n]\setminus\{k\}}\!\! h_{(t)}\Big)\tensor h_{(k)}\Big)
\cdot\Big(\ascprod_{k\le i\le n-1}\!\! R_{i,n}^{(c_i)}\Big).
\end{align*}
\end{lemma}
\begin{proof}
The argument is by descending induction on~$k$, the case~$k=n$ being trivial. For the inductive step,
it suffices to observe that, since~$R^{(c)}\cdot \Delta(h')=
\tau\circ \Delta(h')\cdot R^{(c)}$ for all~$h'\in H$, $c\in C$,
\begin{align*}
R_{k,n}^{(c_k)}\cdot\Big(\Big(&\ascprodtens_{t\in [n]\setminus\{k+1\}}\!\! h_{(t)}\Big)\tensor h_{(k+1)}\Big)\\
&=
\Big(\Big(\ascprodtens_{t\in[k-1]} h_{(t)}\Big)\tensor 1
\tensor \Big(\ascprodtens_{t\in[k+2,n]} h_{(t)}\Big)\tensor 1\Big)
\cdot (R^{(c_k)}\cdot (h_{(k)}\tensor h_{(k+1)}))_{k,n}\\
&=
\Big(\Big(\ascprodtens_{t\in[k-1]} h_{(t)}\Big)\tensor 1
\tensor \Big(\ascprodtens_{t\in[k+2,n]} h_{(t)}\Big)\tensor 1\Big)
\cdot ((h_{(k+1)}\tensor h_{(k)})\cdot R^{(c_k)})_{k,n}\\
&=\Big(\Big(\ascprodtens_{t\in[n]\setminus\{k\}}
h_{(t)}\Big)\tensor h_{(k)}\Big)\cdot R_{k,n}^{(c_k)}.
\qedhere
\end{align*}
\end{proof}
Applying the Lemma with~$k=1$ and~$c_i=\mathbf c(1,i+1)$, $i\in[n-1]$
yields~\eqref{eq:hom coalg id to prove} and completes the 
proof of the inductive step and hence of Theorem~\ref{thm:diag embed quantum}.
\end{proof}
\begin{remark}
It should be noted that, while Lemma~\ref{lem:hom bialg intern} does not require the transitivity of~$\mathbf c$, we need it to ensure that~$J_{\mathbf c}$ is a Drinfeld twist and,
therefore, $\Delta_{H^{\tensor n}}$ twisted by~$J_{\mathbf c}$ remains coassociative.
\end{remark}

\subsection{The dual picture}\label{subs:dual Jc}
Retain the notation from~\S\ref{subs:quant dual}.
Let~$H$ be a bialgebra and let~$\underline{\mathcal R}=\{\mathcal R^{(c)}\}\subset \Hom_\kk(H\tensor H,\kk)$ be a family of co-quasi-triangular structures on~$H$. As in~\S\ref{subs:QTR(C)}, by Proposition~\ref{prop:dual cQYBE}
the assignments~$\sfR_{i,j}^{(c)}\mapsto \mathcal R_{i,j}^{(c)}$, $i,j\in[n]$, $c\in C$, define a homomorphism of monoids
$\Phi^{(n)}_{\underline{\mathcal R}}:\mathsf{QTr}^+_n(C)\to \Hom_\kk(H^{\tensor n},\kk)$.
Let
\begin{equation}\label{eq:Jc dual product}
\mathcal J_{\mathbf c}=\Phi^{(2n)}_{\underline{\mathcal R}}(\mathsf J_{\mathbf c})=\ascprodst_{2\le i\le n}
\dscprodst_{1\le j\le i-1} \mathcal R_{i,j+n}^{(\mathbf c(j,i))}. 
\end{equation}
\begin{theorem}\label{thm:dual Jc twist}
If~$\mathbf c:I_n\to C$ is transitive then~${\mathcal J}_{\mathbf c}$ is a dual Drinfeld twist
for~$H^{\tensor n}\tensor H^{\tensor n}$.
\end{theorem}
\begin{proof}
The proof is similar to that of Theorem~\ref{thm:trans twist} and uses induction on~$n$, the induction base being Lemma~\ref{lem:basic rel dual}. To prove the 
inductive step, note that $\mathcal J_{\mathbf c^-}$
is a dual Drinfeld twist for~$H^{\tensor (n-1)}$
by the induction hypothesis.
We prove that  
$\mathcal F=\dscprodst \limits_{1\le i\le n-1}
\mathcal R_{1,i+1}^{(c_{i,n})}$
is a relative dual Drinfeld twist for~$A\tensor B$ where~$B=H$ and~$A=H^{\tensor(n-1)}$
with the multiplication~$\bullet_{\mathcal J_{\mathbf c^-}}$ and, as before, $c_{i,j}=\mathbf c(i,j)$, $(i,j)\in I_n$.
Indeed, note that $\mathcal F_{\mathbf{13},\mathbf 4}
=(\mathcal F\circ m_B\tensor\id_A)_{\mathbf 1,\mathbf 3,\mathbf 4}$ as an element of~$\Hom_\kk((B\tensor A)^{\tensor 2},\kk)$ which we identify with~$
\Hom_\kk(H^{\tensor 2n},\kk)$; we abbreviate~$\mathbf 1=\{1\}$,
$\mathbf 2=[2,n]$, $\mathbf 3=\{n\}$, $\mathbf 4=[n+1,2n]$. Then, similarly to~\eqref{eq:F13,4}
$$
\mathcal F_{\mathbf{13},\mathbf 4}=\Big(\dscprodst_{2\le j\le n}
\mathcal R_{1,n+j}^{(c_{j-1,n})}\Big)\ast \mathcal F_{\mathbf 3,\mathbf 4}.
$$
Since~$\mathcal F$ is invertible, it is therefore sufficient to prove that
$$
\Big(\dscprodst_{2\le j\le n}
\mathcal R_{1,n+j}^{(c_{j-1,n})}\Big)\ast 
\Big(\dscprodst_{2\le j\le n}
\mathcal R_{1,j}^{(c_{j-1,n})}\Big)
=\mathcal F_{\mathbf 1,\mathbf{24}}
$$
(compare with~\eqref{eq:F step I}) or, equivalently that
\begin{equation}\label{eq:dual F step I}
\Big(\dscprodst_{2\le j\le n}
\mathcal R_{1,n+j-1}^{(c_{j-1,n})}\Big)\ast 
\Big(\dscprodst_{2\le j\le n}
\mathcal R_{1,j}^{(c_{j-1,n})}\Big)
=\mathcal F\circ(\id_B\tensor m_A)
\end{equation}
in~$\Hom_\kk(B\tensor A\tensor A,\kk)$.
Now, since 
\begin{align*}\mathcal F\circ(\id_B\tensor m_A)(b\tensor a\tensor a')&=\mathcal F(b,a\bullet_{\mathcal J_{\mathbf c^-}} a')
=\mathcal J_{\mathbf c^-}{}^{-1}(a_{(1)},a'_{(1)})\mathcal F(b,a_{(2)}a'_{(2)})
\mathcal J_{c^-}(a_{(3)},a'_{(3)})\\
&=(\varepsilon_B\tensor \mathcal J_{\mathbf c^-}{}^{-1})\ast 
\mathcal F\circ(\id_B\tensor m_{H^{\tensor n}})\ast 
(\varepsilon_B\tensor \mathcal J_{\mathbf c^-})(b\tensor a\tensor a'),
\end{align*}
for all $a,a'\in A$, $b\in B$,
the identity~\eqref{eq:dual F step I} is equivalent to
$$
(\varepsilon_H\tensor \mathcal J_{\mathbf c^-})\ast 
\Big(\dscprodst_{2\le j\le n}
\mathcal R_{1,n+j-1}^{(c_{j-1,n})}\Big)\ast 
\Big(\dscprodst_{2\le j\le n}
\mathcal R_{1,j}^{(c_{j-1,n})}\Big)=
\Big(\dscprodst_{2\le j\le n}
\mathcal R_{1,n+j-1}^{(c_{j-1,n})}\ast \mathcal R_{1,j}^{(c_{j-1,n})}\Big)\ast (\varepsilon_H\tensor \mathcal J_{\mathbf c^-}),
$$
which reduces to the same identity~\eqref{eq:compressed}
with products replaced by convolution products and elements of~$H^{\wh\tensor k}$ replaced by their dual counterparts. Thus, \eqref{eq:dual F step I} follows from Proposition~\ref{prop:key relation}.
To complete the proof of the inductive step,
it remains to apply Lemma~\ref{lem:Jc rec}
together with Corollary~\ref{cor:dual twist from rel twist}.
\end{proof}
Similarly to Corollary~\ref{cor:other recursion Jc},
we have
\begin{corollary}\label{cor:other rec dual Jc}
For any transitive~$\mathbf c:I_n\to C$, $\ascprodst\limits_{2\le i\le n} \mathcal R_{i-1,n}^{(\mathbf c(1,i))}$ is a relative dual Drinfeld twist 
for~$A\tensor B$ where~$A=H$ and~$B=H^{\tensor (n-1)}$ with the multiplication twisted by~$\mathcal J_{\mathbf c^+}$.
\end{corollary}
Abbreviate~$\Psi^{(c)}:=\Psi_{\mathcal R^{(c)}}\in\End_{\kk}(H\tensor H)$
in the notation of Corollary~\ref{cor:dual twist from rel twist}, that is
\begin{equation}\label{eq:Psi via R}
\Psi^{(c)}(h\tensor h')=(\mathcal R^{(c)})^{\ast-1}(h_{(1)},h'_{(1)})\mathcal R^{(c)}(h_{(3)},h'_{(3)})
h'_{(2)}\tensor h_{(2)},\qquad h,h'\in H.
\end{equation}
\begin{proposition}\label{prop:Jc twisted mult}
Let~$n\in\ZZ_{>1}$, let~$\mathbf c:I_n\to C$ be
transitive and abbreviate~$\bullet_{\mathbf c}:=
\bullet_{\mathcal J_{\mathbf c}}$. Then
for all~$\mathbf h,\mathbf h'\in H^{\tensor n}$
\begin{equation}\label{eq:product via braid}
\mathbf h\bullet_{\mathbf c} \mathbf h'
=m_H{}^{\tensor n}\circ \Big(\dscprod\limits_{1\le i\le n-1}
\ascprod\limits_{i+1\le j\le n}
\id_H^{\tensor(i+j-2)}\tensor \Psi^{(\mathbf c(i,j))}\tensor\id_H^{\tensor (2n-i-j)}\Big)(\mathbf h\tensor \mathbf h').
\end{equation}
Explicitly, if~$\mathbf h=a^1\tensor\cdots\tensor a^n$,
$\mathbf h'=b^1\tensor\cdots \tensor b^n$, $a^i,b^i\in H$,
$i\in[n]$ then
\begin{align}
\mathbf h\bullet_{\mathbf c}
\mathbf h'
&=
\prod_{1\le j<i\le n} (\mathcal R^{(\mathbf c(j,i))})^{\ast-1}
(a^i_{(j)},b^j_{(n+1-i)})\mathcal R^{(\mathbf c(j,i))}(a^i_{(2i-j)},b^j_{(n+i+1-2j)})
\ascprodtens_{1\le i\le n} a^i_{(i)}b^i_{(n+1-i)}.
\label{eq:c product}
\end{align}
In particular, if
all the~$\mathcal R^{(c)}$, $c\in C$ are 
counital in the sense of~\eqref{eq:counital}
then
\begin{align*}
(1^{\tensor(l-1)}\tensor h
\tensor 1^{\tensor(n-l)})\bullet_{\mathbf c} 
(1^{\tensor(k-1)}\tensor h'\tensor 1^{\tensor(n-k)})=
\begin{cases}
(hh')_k,&k=l,\\
(h\tensor h')_{l,k},&k>l,\\
(\Psi^{(\mathbf c(k,l))}(h\tensor h'))_{k,l},
&k<l,
\end{cases}
\end{align*}
for all~$h,h'\in H$, $k,l\in[n]$.
\end{proposition}
\begin{proof}
We prove~\eqref{eq:product via braid} by induction on~$n$. The case~$n=2$ is
immediate from Lemma~\ref{lem:basic rel dual} and Corollary~\ref{cor:dual twist from rel twist} with~$A=B=H$, $\mathcal J_A=\mathcal J_B=\varepsilon\tensor\varepsilon$. 
For the inductive step, let
$a,a'\in A=H$, $\mathbf b,\mathbf b'\in B=H^{\tensor (n-1)}$, the multiplication in~$B$ being~$\bullet_{\mathbf c^+}$. 
Using Corollaries~\ref{cor:dual twist from rel twist} and~\ref{cor:other rec dual Jc} and the
induction hypothesis, we obtain
\begin{align*}
(a'&\tensor \mathbf b)\bullet_{\mathbf c}(a\tensor \mathbf  b')
=(m_H\tensor \bullet_{\mathbf c^+})(\id_A\tensor \Psi_{\mathcal F}
\tensor \id_B)( a'\tensor\mathbf b\tensor a\tensor\mathbf  b')\\
&=
m_H{}^{\tensor n}\circ
\Big(\dscprod\limits_{2\le i\le n-1}
\ascprod\limits_{i+1\le j\le n}
\id_H^{\tensor(i+j-2)}\tensor \Psi^{(\mathbf c(i,j))}\tensor\id_H^{\tensor (2n-i-j)}\Big)
(a'\tensor \Psi_{\mathcal F}(\mathbf b\tensor  a)
\tensor \mathbf b'),
\end{align*}
where~$\mathcal F=\ascprodst_{1\le j\le n-1} \mathcal R_{j,n}^{(\mathbf c(1,j+1))}$. Furthermore,
if~$\mathbf b=h^1\tensor\cdots \tensor h^{n-1}$,
$h^j\in H$, $j\in[n-1]$ then
\begin{align}
\Psi_{\mathcal F}(&\mathbf b\tensor a)
=\mathcal F^{\ast-1}(\mathbf b_{(1)}, a_{(1)})
\mathcal F(\mathbf b_{(3)},a_{(3)})
a_{(2)} \tensor \mathbf b_{(2)} \nonumber\\
&=\dscprodst_{1\le j\le n-1} (\mathcal R_{j,n}^{(\mathbf c(1,j+1))})^{\ast-1}(\mathbf b_{(1)},a_{(1)})\ascprodst_{1\le j\le n-1} \mathcal R_{j,n}^{(\mathbf c(1,j+1))}(\mathbf b_{(3)},a_{(3)}) a_{(2)}\tensor \mathbf b_{(2)}\nonumber\\
&=
\prod_{1\le j\le n-1} (\mathcal R^{(\mathbf c(1,j+1))})^{\ast-1}(h^{j}_{(1)},a_{(n-j)})
\mathcal R^{(\mathbf c(1,j+1))}(h^{j}_{(3)},a_{(n+j)}) a_{(n)}\tensor h^1_{(2)}
\tensor\cdots\tensor h^{n-1}_{(2)}\nonumber\\
&=\Big(\ascprod_{2\le j\le n} (\id_H^{\tensor (j-2)}
\tensor \Psi^{(\mathbf c(1,j))}\tensor \id_H^{\tensor (n-j)})\Big)(\mathbf b\tensor a).\label{eq:id Psi b a}
\end{align}
The inductive step and hence~\eqref{eq:product via braid} are now immediate. To prove the second assertion,
note that for~$\mathbf h=a^1\tensor\cdots\tensor a^n$,
$\mathbf h'=b^1\tensor\cdots\tensor b^n$, $a^i,b^i\in H$,
$i\in [n]$ it follows from~\eqref{eq:Jc dual product} that  
\begin{align}
&\mathcal J_{\mathbf c}(\mathbf h,
\mathbf h')
=\varepsilon(a^1)\varepsilon(b^{n})
\prod_{1\le j<i\le n}
\mathcal R^{(\mathbf c(j,i))}(a^{i}_{(i-j)},b^{j}_{(i-j)}),\nonumber\\
&\mathcal J_{\mathbf c}^{\ast-1}(\mathbf h,\mathbf h')=
\varepsilon(a^1)\varepsilon(b^{n})
\prod_{1\le j<i\le n} 
(\mathcal R^{(\mathbf c(j,i))})^{\ast-1}(a^i_{(j)},b^{j}_{(n+1-i)}),\label{eq:Jc on simple tens}
\end{align}
which immediately yields~\eqref{eq:c product}.
Finally, if all the~$\mathcal R^{(c)}$, $c\in C$
are counital then
for~$\mathbf h=1^{\tensor (l-1)}\tensor h\tensor 1^{\tensor (n-l)}$, $\mathbf h'=1^{\tensor (k-1)}
\tensor h'\tensor 1^{\tensor(n-k)}$, $h,h'\in H$, $k,l\in[n]$, \eqref{eq:Jc on simple tens} then yields
\begin{equation*}
\mathcal J_{\mathbf c}^{\ast\pm 1}(\mathbf h,\mathbf h')
=\begin{cases}
\varepsilon(h)\varepsilon(h'),& k\ge l,\\
(\mathcal R^{(\mathbf c(k,l))})^{\ast\pm 1}(h,h'),& k<l.
\end{cases}
\end{equation*}
The last assertion is now immediate.
\end{proof}

Thus, if all the~$\mathcal R^{(c)}$, $c\in C$
are counital, $(H^{\tensor n},\bullet_{\mathbf c})$
is generated by the $h^{(k)}$, $h\in H$, $k\in [n]$
subject to relations
\begin{align*}
&a^{(k)}\bullet_{\mathbf c}b^{(k)}=(ab)^{(k)},\qquad k\in[n],\\
&a^{(l)}\bullet_{\mathbf c}
b^{(k)}=(\mathcal R^{(\mathbf c(k,l))})^{\ast-1}(a_{(1)},b_{(1)})
\mathcal R^{(\mathbf c(k,l))}(a_{(3)},b_{(3)})
b_{(2)}{}^{(k)}\bullet_{\mathbf c}a_{(2)}{}^{(l)},\,\, k<l\in[n],\,a,b\in H.
\end{align*}
\begin{proposition}\label{prop:perm isom}
Suppose that all the~$\mathcal R^{(c)}$, $c\in C$ are counital.
Let~$\mathbf c,\mathbf c':I_n\to C$ be transitive 
and let~$\sigma\in S_n$. Then the assignments
$h^{(i)}\mapsto h^{(\sigma(i))}$, $i\in [n]$, $h\in H$
define 
an isomorphism of bialgebras $(H^{\tensor n},\bullet_{\mathbf c})\to (H^{\tensor n},\bullet_{\mathbf c'})$
if and only if, for all $h,h'\in H$, $k\not=l\in[n]$
$$
\mathcal Z_{k,l;\mathbf c,\mathbf c',\sigma}(h_{(1)}\tensor h'_{(1)})h_{(2)}\tensor h'_{(2)}=
\mathcal Z_{k,l;\mathbf c,\mathbf c',\sigma}(h_{(2)},h'_{(2)})
h_{(1)}\tensor h'_{(1)}
$$
where~$\mathcal Z_{k,l;\mathbf c,\mathbf c',\sigma}=
(\mathcal R^{(\mathbf c(k,l))})^{\Upsilon(l-k)}
\ast (\mathcal R^{(\mathbf c'(\sigma(k),\sigma(l)))})^{\ast-\Upsilon(\sigma(l)-\sigma(k))}\in\Hom_\kk(H\tensor H,\kk)$.
\end{proposition}
\begin{proof}
Let~$k\not=l\in[n]$, $h,h'\in H$ and let~$k'=\sigma(k)$, $l'=\sigma(l)$. 
We have 
$$
\sigma(h^{(l)}\bullet_{\mathbf c} h'{}^{(k)})
=\begin{cases}
(\mathcal R^{(\mathbf c(k,l))})^{\ast-1}(h_{(1)},h'_{(1)})
\mathcal R^{(\mathbf c(k,l))}(h_{(3)},h'_{(3)})(h_{(2)}\tensor h'_{(2)})_{l',k'},&k<l,\\
(h\tensor h')_{l',k'},&k>l.
\end{cases}
$$
On the other hand,
$$
h^{(l')}\bullet_{\mathbf c'} h'{}^{(k')}=
\begin{cases}
(\mathcal R^{(\mathbf c'(k',l'))})^{\ast-1}(h_{(1)},h'_{(1)})
\mathcal R^{(\mathbf c'(k',l'))}(h_{(3)},h'_{(3)})(h_{(2)}\tensor h'_{(2)})_{l',k'},&k'<l',\\
(h\tensor h')_{l',k'},&k'>l'.
\end{cases}
$$
The assertion follows by applying Lemma~\ref{lem:old mult from twisted}.
\end{proof}
In particular, for~$n=2$, given any co-quasi-triangular structure~$\mathcal R\in\Hom_\kk(H^{\tensor 2},\kk)$,
we obtain two possible multiplications
$\bullet_\pm:H^{\tensor 2}\tensor H^{\tensor 2}\to H^{\tensor 2}$ on~$H^{\tensor 2}$,
corresponding to~$\mathcal R^{(\pm 1)}$ 
where~$\mathcal R^{(1)}=\mathcal R$
and~$\mathcal R^{(-1)}=\mathcal R^{\ast-1}\circ\tau$
(cf. Lemma~\ref{lem:basic coquas family}).
Then $\tau:H\tensor H\to H\tensor H$
is an isomorphism of algebras~$(H^{\tensor 2},
\bullet_+)\to (H^{\tensor 2},\bullet_-)$ if and only if
\begin{equation}\label{eq:n=2 isom cond}
\mathcal R(h_{(1)},h'_{(1)})h_{(2)}\tensor h'_{(2)}=
\mathcal R(h_{(2)},h'_{(2)})h_{(1)}\tensor h'_{(1)},\qquad h,h'\in H,
\end{equation}
which holds automatically if~$H$ is cocommutative.

We complete this chapter with the counterpart of Theorem~\ref{thm:diag embed quantum}.
\begin{theorem}\label{thm:mult is homomorphism}
For any transitive~$\mathbf c:I_n\to C$, the iterated multiplication map~$m_H^{(n)}:
H^{\tensor n}\to H$
is a homomorphism of bialgebras~$(H^{\tensor n},
\bullet_{\mathbf c})\to H$.
\end{theorem}
\begin{proof}
Since~$m_H^{(n)}$ is a homomorphism of coalgebras,
we only need to prove that it is a homomorphism of algebras.
The argument by induction on~$n$, the case~$n=1$ being trivial. For the inductive step, let~$A=H$
and let~$B=(H^{\tensor(n-1)},\bullet_{\mathbf c^+})$.
Since
$m^{(n-1)}_H:(H^{\tensor(n-1)},\bullet_{\mathbf c^+})\to H$
is a homomorphism of algebras by the induction hypothesis
and~$m^{(n)}_H=m_H\circ(\id_H\tensor m_H^{(n-1)})$,
by Corollary~\ref{cor:other rec dual Jc} and Proposition~\ref{prop:twist hom alg} used with~$C=H$, $\varphi_A=\id_H$, $\varphi_B=m_H^{(n-1)}$
and~$\mathcal F$ as in the proof of Proposition~\ref{prop:Jc twisted mult},
to complete the inductive step it remains to show
that, for all~$a\in H$, $\mathbf b\in H^{\tensor(n-1)}$, 
\begin{equation}\label{eq:interm step hom}
\mathcal F(\mathbf b_{(2)},a_{(2)})
a_{(1)}m_H^{(n-1)}(\mathbf b_{(1)})=\mathcal F(\mathbf b_{(1)},a_{(1)})
m_H^{(n-1)}(\mathbf b_{(2)})a_{(2)}.
\end{equation}
Clearly, it suffices to prove~\eqref{eq:interm step hom}
for~$\mathbf b=h^1\tensor\cdots\tensor h^{n-1}$,
$h^j\in H$, $1\le j\le n-1$.
We need the following
\begin{lemma}
For all~$a,h^1,\dots,h^{n-1}\in H$, and~$1\le k\le n$
\begin{multline*}
\mathcal F(h^1_{(2)}\tensor\cdots\tensor h^{n-1}_{(2)},
a_{(2)}) a_{(1)}\Big(\ascprod_{1\le j\le n-1} h^j_{(1)}\Big)
\\=
\prod_{1\le j\le k-1}\!\!\!\!\mathcal R^{(\mathbf c(1,j+1))}
(h^j_{(1)},a_{(j)})\prod_{k\le j\le n-1}\!\!\!\!
\mathcal R^{(\mathbf c(1,j+1))}(h^j_{(2)},a_{(j+1)})
\Big(\ascprod_{1\le j\le k-1}\!\!\!\! h^j_{(2)}\Big)
a_{(k)}\Big(\ascprod_{k\le j\le n-1}\!\!\!\!h^j_{(1)}\Big)
\end{multline*}
\end{lemma}
\begin{proof}
The argument is by induction on~$k$, the case~$k=1$ being immediate from the definition of~$\mathcal F$. To prove
the inductive step, it suffices to observe that,
by~\eqref{eq:coquas triang 1}, 
\begin{align*}
&\prod_{1\le j\le k-1}\!\!\!\!\mathcal R^{(\mathbf c(1,j+1))}
(h^j_{(1)},a_{(j)})\prod_{k\le j\le n-1}\!\!\!\!
\mathcal R^{(\mathbf c(1,j+1))}(h^j_{(2)},a_{(j+1)})
\Big(\ascprod_{1\le j\le k-1}\!\!\!\! h^j_{(2)}\Big)
a_{(k)}\Big(\ascprod_{k\le j\le n-1}\!\!\!\!h^j_{(1)}\Big)
\\
&=\prod_{1\le j\le k-1}\!\!\!\!\mathcal R^{(\mathbf c(1,j+1))}
(h^j_{(1)},a_{(j)})\prod_{k+1\le j\le n-1}\!\!\!\!
\mathcal R^{(\mathbf c(1,j+1))}(h^j_{(2)},a_{(j+1)})
\Big(\ascprod_{1\le j\le k-1}\!\!\!\! h^j_{(2)}\Big)
\times\\
&\mskip200mu\mathcal R^{(\mathbf c(1,k+1))}(h^k_{(2)},a_{(k+1)})
a_{(k)}h^k_{(1)}\Big(\ascprod_{k+1\le j\le n-1}\!\!\!\!h^j_{(1)}\Big)\\
&=\prod_{1\le j\le k-1}\!\!\!\!\mathcal R^{(\mathbf c(1,j+1))}
(h^j_{(1)},a_{(j)})\prod_{k+1\le j\le n-1}\!\!\!\!
\mathcal R^{(\mathbf c(1,j+1))}(h^j_{(2)},a_{(j+1)})
\Big(\ascprod_{1\le j\le k-1}\!\!\!\! h^j_{(2)}\Big)
\times\\
&\mskip200mu\mathcal R^{(\mathbf c(1,k+1))}(h^k_{(1)},a_{(k)})
h^k_{(2)} a_{(k+1)}\Big(\ascprod_{k+1\le j\le n-1}\!\!\!\!h^j_{(1)}\Big)
\\
&=\prod_{1\le j\le k}\!\!\!\!\mathcal R^{(\mathbf c(1,j+1))}
(h^j_{(1)},a_{(j)})\prod_{k+1\le j\le n-1}\!\!\!\!
\mathcal R^{(\mathbf c(1,j+1))}(h^j_{(2)},a_{(j+1)})\times\\
&\mskip200mu
\Big(\ascprod_{1\le j\le k-1}\!\!\!\! h^j_{(2)}\Big)
a_{(k+1)}\Big(\ascprod_{k+1\le j\le n-1}\!\!\!\!h^j_{(1)}\Big).\qedhere
\end{align*}
\end{proof}
It remains to observe that, for~$k=n$, the 
right hand side of the identity in the Lemma is
precisely the right hand side of~\eqref{eq:interm step hom} with~$\mathbf b=h^1\tensor \cdots\tensor h^{n-1}$,
$h^i\in H$, $1\le i\le n-1$. 
\end{proof}
\begin{remark}
If all the~$\mathcal R^{(c)}$ are counital, the argument simplifies dramatically since in that case it suffices to prove the assertion for generators $h^{(k)}$, $h\in H$, $k\in[n]$ (cf.~\S\ref{subs:Poisson}).
\end{remark}

\section{Examples}

\subsection{Twisted tensor powers of quantum matrices}\label{subs:Poisson ex}
\label{subs:twisted tens quantum matrices}
We begin with the classical picture.
Retain the notation of~\S\ref{subs:Poisson}
and abbreviate~$\mathcal A_m=\kk[\Mat_m]$. Let~$E_{a,b}$, $a,b\in [m]$ be 
the standard basis of~$\Mat_m$ or of~$\lie g=\lie{gl}_m$
and define $x_{i,j}\in \lie g^*$, $i,j\in[m]$, by $x_{i,j}(E_{a,b})=
\delta_{i,a}\delta_{j,b}$, $a,b\in [m]$. Then~$\mathcal A_m$ identifies with~$\kk[x_{i,j}\,:\, i,j\in[m]]$
and is naturally a sub-bialgebra of~$U(\lie g)^\smallsquare\subset U(\lie g)^*$ in the notation 
of~\S\ref{subs:compl}.
The algebra~$\kk[GL_m]$ is the localization of~$\mathcal A_m$ by the determinant~$\det$. 
Then standard
right and left actions of~$\lie g$ on~$\mathcal A_m$
are given by, respectively, by
$$
x_{i,j}\triangleleft E_{a,b}=\delta_{a,i} x_{b,j},\quad E_{a,b}\triangleright x_{i,j}=\delta_{b,j}x_{i,a},\qquad a,b,i,j\in [m].
$$
Let
$$
r=\sum_{1\le i\le m} E_{i,i}\tensor E_{i,i}+2\sum_{1\le i<j\le m} E_{i,j}\tensor E_{j,i}
$$
be the standard r-matrix for~$\lie g$ (see e.g.~\cite{CPbook}*{\S2.2}). 
\begin{lemma}\label{lem:r eps acts}
For all $i,i',j,j'\in [m]$, $\epsilon\in\{1,-1\}$
\begin{align*}
r^{(\epsilon)}\bowtie x_{i,j}\tensor x_{i',j'}
=(\sign(i'-i)+\epsilon) x_{i',j}\tensor x_{i,j'}-
(\sign(j-j')+\epsilon)x_{i,j'}\tensor x_{i',j}.
\end{align*}
\end{lemma}
\begin{proof}
We have 
\begin{align*}
\sum_{a\in[m]} (E_{aa}\tensor E_{aa})&\bowtie (x_{i,j}\tensor x_{i',j'})=
\sum_{a\in [m]} (\delta_{a,i}\delta_{a,i'} x_{a,j}\tensor x_{a,j'}
-\delta_{a,j}\delta_{a,j'} x_{i,a}\tensor x_{i',a})\\
&=\delta_{i,i'} x_{i,j}\tensor x_{i,j'}-\delta_{j,j'}x_{i,j}\tensor x_{i',j}
=\delta_{i,i'} x_{i',j}\tensor x_{i,j'}-\delta_{j,j'}x_{i,j'}\tensor x_{i',j},
\end{align*}
while
\begin{align*}
\sum_{1\le a<b\le m} (E_{ab}\tensor E_{ba})\bowtie (x_{i,j}\tensor x_{i',j'})&=
\sum_{1\le a<b\le m} (\delta_{a,i}\delta_{b,i'} x_{b,j}\tensor x_{a,j'}
-\delta_{b,j}\delta_{a,j'} x_{i,a}\tensor x_{i',b})\\ 
&=\Upsilon(i'-i) x_{i',j}\tensor x_{i,j'}-\Upsilon(j-j') x_{i,j'}\tensor x_{i',j}.
\end{align*}
Therefore,
$$
r^{(\epsilon)}\bowtie (x_{i,j}\tensor x_{i',j'})=
\epsilon(\delta_{i,i'}+2\Upsilon(\epsilon(i'-i))x_{i',j}\tensor x_{i,j'}-
(\delta_{j,j'}+2\Upsilon(\epsilon(j-j')))x_{i,j'}\tensor x_{i',j}).
$$
By~\eqref{eq:Ups formula}, $\delta_{k,l}+2\Upsilon(\epsilon(l-k))=1+\sign(\epsilon(l-k))=1+\epsilon\sign(l-k)$
for all $k,l\in\ZZ$, $\epsilon\in\{1,-1\}$, and
the assertion follows.
\end{proof}

Using~\eqref{eq:c Poi bracket} and the Lemma, we immediately obtain the standard Poisson bracket 
on~$\mathcal A_m$
\begin{equation}
\{ x_{i,j},x_{i',j'}\}=(\sign(i'-i)+\sign(j'-j))x_{i,j'}x_{i',j}\label{eq:Am Pois bracket}
\end{equation}
and a Poisson bracket on~$\mathcal A_m^{\tensor n}$
\begin{align}
\{ x_{i,j}^{(k)},x_{i',j'}^{(k')}\}_{\sgna{\id}}
&=(\sign(i'-i)+\sign(k-k'))x_{i',j}^{(k)}x_{i,j'}^{(k')}\nonumber\\
&\phantom{=}+(\sign(j'-j)-\sign(k-k'))x_{i,j'}^{(k)}x_{i',j}^{(k')}\label{eq:Am,n Pois bracket}
\end{align}
for all~$i,i',j,j'\in[m]$, $k,k'\in[n]$.
Evidently, the assignments $x_{i,j}^{(k)}\mapsto 
x_{i,j}$, $i,j\in[m]$, $k\in[n]$ define a homomorphism
of Poisson algebras, as stipulated in Theorem~\ref{thm:Poisson mult}.
\begin{remark}\label{rem:Pois perms}
We can define Poisson brackets using the classical
Drinfeld twist corresponding to $\sgna{w}$ for any permutation~$w\in S_n$. However, the natural action of~$S_n$
on $\mathcal A_m^{\tensor n}$ by permutations of factors
yields an isomorphism between that Poisson algebra
and the one we just described, since the corresponding
Lie bialgebras are isomorphic (see Remark~\ref{rem:isom permutations}).
\end{remark}
\begin{remark}
It is easy to check, using~\eqref{eq:g act k[G]}, that
$E_{a,b}\triangleright\det=\det\triangleleft E_{a,b}=
\delta_{a,b}\det$, $a,b\in[m]$ whence $r^{(\epsilon)}\bowtie 
\det\tensor x_{i,j}=0$, $\epsilon\in\{1,-1\}$, $i,j\in[m]$ and so~$\det$ is Poisson-central
in~$\mathcal A_m$ and~$\det^{(k)}$, $k\in[n]$ are Poisson-central
in~$\mathcal A_{m}^{\tensor n}$. Thus, our Poisson bracket remains the same on~$\kk[G]^{\tensor n}$.
\end{remark}

Let~$q\in\kk^\times$ and assume that~$q$ is not a root of unity. The quantum analogue~$\mathcal A_{q,m}$ of~$\mathcal A_m$ is
generated by the $x_{i,j}$, $i,j\in [m]$ subject to relations
\begin{equation}\label{eq:q-matrices}
q^{\delta_{i,i'}}x_{i',j'}x_{i,j}-q^{\delta_{j,j'}}
x_{i,j}x_{i',j'}=(q-q^{-1})(\Upsilon(j-j')-
\Upsilon(i'-i))x_{i,j'}x_{i',j},
\end{equation}
for all $i,i',j,j'\in[m]$.
For instance, if~$m=2$ we obtain the familiar 
relations of the quantum coordinate ring
of $2\times 2$ matrices 
\begin{alignat*}{2}
&x_{1,1}x_{1,2}=q x_{1,2}x_{1,1},&\quad&
x_{1,1}x_{2,1}=q x_{2,1}x_{1,1},\\
&x_{1,2}x_{2,2}=q x_{2,2}x_{1,2},
&&x_{2,1}x_{2,2}=q x_{2,2}x_{2,1},\\
&x_{1,2}x_{2,1}=x_{2,1}x_{1,2},&&
x_{1,1}x_{2,2}-x_{2,2}x_{1,1}=(q-q^{-1})x_{1,2}x_{2,1}
\end{alignat*}
(see e.g.~\cite{Ge} or~\cite{CPbook}*{\S7.1} with~$q=e^{-h}$). 
The $q\to 1$ limit 
of this algebra is~$\mathcal A_m$
with the Poisson bracket given by~\eqref{eq:Am Pois bracket}. The coalgebra structure of~$\mathcal A_{q,m}$ is given by
$\Delta(x_{i,j})=\sum_{k\in[m]}x_{i,k}\tensor x_{k,j}$,
$i,j\in[m]$, while its counital co-quasi-triangular structure
$\mathcal R\in\Hom_\kk(\mathcal A_{q,m}\tensor\mathcal A_{q,m},\kk)$ is defined by
$$
\mathcal R(x_{k,l},x_{k',l'})=q^{\delta_{k,k'}}\delta_{k,l}\delta_{k',l'}
+(q-q^{-1})\Upsilon(k-k')\delta_{k,l'}\delta_{k',l},
\qquad k,k',l,l'\in[m].
$$

Let~$\mathcal R^{(1)}=\mathcal R$ and~$\mathcal R^{(-1)}=
\mathcal R^{\ast-1}\circ\tau$ (cf. Lemma~\ref{lem:basic coquas family}). In particular,
$(\mathcal R^{(\epsilon)})^{\ast-1}=\mathcal R^{(-\epsilon)}\circ\tau$, $\epsilon\in\{1,-1\}$.
It is easy to check 
that, for all $k,k',l,l'\in [m]$, $\epsilon\in\{1,-1\}$,
$$
\mathcal R^{(\epsilon)}(x_{k,l},x_{k',l'})=q_\epsilon^{\delta_{k,k'}}\delta_{k,l}\delta_{k',l'}+(q_\epsilon-q_\epsilon^{-1})\Upsilon(\epsilon(k-k'))\delta_{k,l'}\delta_{k',l},
$$
where~$q_\epsilon=q^{\epsilon}$. 
By~\eqref{eq:Psi via R}, for all
$i,i',j,j'\in[m]$, $\epsilon\in\{1,-1\}$,
\begin{align*}
\Psi^{(\epsilon)}(x_{i',j'}\tensor x_{i,j})&=
q_\epsilon^{\delta_{j,j'}-\delta_{i,i'}}x_{i,j}\tensor x_{i',j'}
-q_\epsilon^{\delta_{j,j'}}(q_\epsilon-q_\epsilon^{-1})\Upsilon(\epsilon(i'-i))x_{i',j}\tensor x_{i,j'}\\
&\phantom{=}+q_\epsilon^{-\delta_{i,i'}}(q_\epsilon-q_\epsilon^{-1})
\Upsilon(\epsilon(j-j'))x_{i,j'}\tensor x_{i',j}\\
&\phantom{=}
-(q-q^{-1})^2\Upsilon(\epsilon(i'-i))\Upsilon(\epsilon(j-j'))x_{i',j'}\tensor x_{i,j}.
\end{align*}
Given~$w\in S_n$, let~$\sgna{w}:I_n\to\{1,-1\}$ be as in Lemma~\ref{lem:trans sign} and
abbreviate~$\bullet_w=\bullet_{\sgna{w}}$.
Let, like in~\S\ref{subs:Poisson} and~\S\ref{subs:dual Jc}, $x^{(k)}=1^{\tensor (k-1)}\tensor x\tensor 1^{(n-k)}$, $x\in\mathcal A_{q,m}$, $k\in[n]$.
By Proposition~\ref{prop:Jc twisted mult},
$x_{i',j'}^{(k')}\bullet_w x_{i,j}^{(k)}=
x_{i',j'}^{(k')}\cdot x_{i,j}^{(k)}$ if~$k'\le k$, while 
for~$k'>k$
\begin{align*}
x_{i',j'}^{(k')}&\bullet_w x_{i,j}^{(k)}
=q_{k,k'}^{\delta_{j,j'}-\delta_{i,i'}}x_{i,j}^{(k)}\cdot x_{i',j'}^{(k')}
-q_{k,k'}^{\delta_{j,j'}}(q_{k,k'}-q_{k,k'}^{-1})\Upsilon(\sgna{w}_{k,k'}(i'-i))x_{i',j}^{(k)}\cdot x_{i,j'}^{(k')}\\
&+q_{k,k'}^{-\delta_{i,i'}}(q_{k,k'}-q_{k,k'}^{-1})
\Upsilon(\sgna{w}_{k,k'}(j-j'))x_{i,j'}^{(k)}\cdot x_{i',j}^{(k')}\\
&-(q-q^{-1})^2\Upsilon(\sgna{w}_{k,k'}(i'-i))
\Upsilon(\sgna{w}_{k,k'}(j-j'))x_{i',j'}^{(k)}\cdot x_{i,j}^{(k')},
\end{align*}
where~$\cdot$ stands for the usual multiplication in~$\mathcal A_{q,m}{}^{\tensor n}$ and~$q_{k,k'}=q_{\sgna{w}_{k,k'}}$, $(k,k')\in I_n$.
Thus,
$\mathcal A_{q,m}^{\tensor n,w}:=\mathcal A_{q,m}{}^{\tensor n}$ as a coalgebra, is generated, as an algebra, by the $x_{i,j}^{(k)}$, $i,j\in[m]$, $k\in[n]$,
such that for each~$k\in[n]$, the subalgebra
generated by the $x_{i,j}^{(k)}$, $i,j\in[m]$
is isomorphic to~$\mathcal A_{q,m}$ and
\begin{align*}
q_{k,k'}^{\delta_{i,i'}}x_{i',j'}^{(k')}&\bullet_w x_{i,j}^{(k)}
=q_{k,k'}^{\delta_{j,j'}}x_{i,j}^{(k)}\bullet_w x_{i',j'}^{(k')}
-q_{k,k'}^{\delta_{j,j'}}(q_{k,k'}-q_{k,k'}^{-1})\Upsilon(\sgna{w}_{k,k'}(i'-i))x_{i',j}^{(k)}\bullet_w x_{i,j'}^{(k')}
\\
&+(q_{k,k'}-q_{k,k'}^{-1})
\Upsilon(\sgna{w}_{k,k'}(j-j'))x_{i,j'}^{(k)}\bullet_w x_{i',j}^{(k')}\\
&
-(q-q^{-1})^2\Upsilon(\sgna{w}_{k,k'}(i'-i))\Upsilon(\sgna{w}_{k,k'}(j-j'))x_{i',j'}^{(k)}\bullet_w x_{i,j}^{(k')},\quad (k,k')\in I_n,
\end{align*}
for all $i,i',j,j'\in[m]$. Unlike in the classical case 
described~\S\ref{subs:Poisson ex}, it is quite an exercise to verify directly that the assignments $x_{i,j}^{(k)}\mapsto
x_{i,j}$, $i,j\in[m]$, $k\in[n]$ define the homomorphism 
of bialgebras~$\mathcal A_{q,m}^{\tensor n,w}\to 
\mathcal A_{q,m}$ provided by Theorem~\ref{thm:mult is homomorphism}. 
The Poisson bracket~\eqref{eq:Am,n Pois bracket} on~$\mathcal A_{m}^{\tensor n}$ 
is obtained
as a ``dequantization'' of~$\mathcal A_{q,m}^{\tensor n,\id}$. Note also, that 
since~$\mathcal A_{q,m}$ has a Poincar\'e-Birkhoff-Witt (PBW) basis, so does~$\mathcal A_{q,m}^{\tensor n,w}$
for any~$n\ge 2$ and~$w\in S_n$.

Needless to say, the $q=1$ limit of~$\mathcal A_{q,m}^{\tensor n,w}$ is also a Poisson algebra which, as discussed in Remark~\ref{rem:Pois perms}, is isomorphic to~$\mathcal A_{m}^{\tensor n}$ as a Poisson algebra via the natural action of~$w$ by permutation of factors. However, this no longer yields an isomorphism of algebras~$\mathcal A_{q,m}^{\tensor n,\id}\to 
\mathcal A_{q,m}^{\tensor n,w}$. Indeed, let~$n=2$. By~\eqref{eq:n=2 isom cond}, if
$\tau$ was an isomorphism of algebras 
$\mathcal A_{q,m}^{\tensor 2,\id}\to \mathcal A_{q,m}^{\tensor 2,(1,2)}$,
$x^{(k)}\mapsto x^{(3-k)}$, $k\in\{1,2\}$, $x\in \mathcal A_{q,m}$, then we would have, for all
$i,i',j,j'\in[m]$
$$
\sum_{s,s'\in[m]}\mathcal R(x_{i,s},x_{i',s'})x_{s,j}\tensor x_{s',j'}
=\sum_{s,s'\in[m]}\mathcal R(x_{s,j},x_{s',j'})x_{i,s}\tensor x_{i',s'},
$$
which is equivalent to
$$
(q^{\delta_{i,i'}}-q^{\delta_{j,j'}})x_{i,j}\tensor x_{i',j'}=(q-q^{-1})(\Upsilon(j'-j)x_{i,j'}\tensor x_{i',j}-\Upsilon(i-i')x_{i',j}\tensor x_{i,j'}).
$$
Let~$i>j\in[m]$ and let~$i'=j$, $j'=i$. Then
the above yields
$$
(q-q^{-1})(x_{i,i}\tensor x_{j,j}-x_{j,j}\tensor x_{i,i})=0,
$$
which is a contradiction for~$q\not=\pm1$.

\subsection{Families of classical r-matrices for Takiff
Lie algebras}\label{subs:Takiff}
Let~$\lie g$ be a Lie algebra and let~$f:\lie g\to V$ be a surjective
homomorphism of~$\lie g$-modules. Let~$\lie t=\lie t(\lie g,V)=V\rtimes \lie g$, which is equal to $V\oplus \lie g$ as a vector space with the Lie bracket
defined by $[(v,x),(v',x')]_{\lie t}=(xv'-x'v,[x,x']_{\lie g})$. The Lie algebra~$\lie t$ is a generalization 
of the Takiff Lie algebra, which corresponds to the case when~$f$ is an isomorphism and in turn is isomorphic to the truncated current algebra $\lie g\tensor \mathbb k[t]/(t^2)$.
It follows from the definition that~$\lie g$ (respectively, $V$) identifies with a subalgebra (respectively, an abelian ideal) of~$\lie t$.
\begin{lemma}
Let~$(\lie g,\delta)$ be a Lie bialgebra and let~$f:\lie g\to V$ be a surjective homomorphism of $\lie g$-modules.
Then~$\lie t=V\rtimes \lie g$ is a Lie bialgebra with
$\wh\delta:\lie t\to \lie t\tensor\lie t$ defined by
$\wh\delta(x)=(f\tensor\id_{\lie g}+\id_{\lie g}\tensor f)\delta(x)$, $\wh\delta(f(x))=(f\tensor f)\delta(x)$,
$x\in\lie g$.
\end{lemma}
\begin{proof}
Write~$\delta(z)=z_1\tensor z_2$, $z\in\lie g$ in Sweedler-like notation. We have, 
for all $x,y\in \lie g$
\begin{align*}
\wh\delta([x,y]_{\lie t})&=(f\tensor \id_{\lie g}+\id_{\lie g}\tensor f)(\delta([x,y]_{\lie g})=(f\tensor\id_{\lie g}+
\id_{\lie g}\tensor f)([\delta(x),\Delta(y)]-[\delta(y),\Delta(x)])\\
&
=(f\tensor\id_{\lie g}+
\id_{\lie g}\tensor f)([x_1,y]\tensor x_2+x_1\tensor [x_2,y]-[y_1,x]\tensor y_2-y_1\tensor [y_2,x])\\
&=f([x_1,y])\tensor x_2+f(x_1)\tensor [x_2,y]-
f([y_1,x])\tensor y_2-f(y_1)\tensor [y_2,x]
\\
&\phantom{-}+[x_1,y]\tensor f(x_2)+x_1\tensor f([x_2,y])-[y_1,x]\tensor f(y_2)-y_1\tensor f([y_2,x])\\
&=[f(x_1),y]\tensor x_2+f(x_1)\tensor [x_2,y]-
[f(y_1),x]\tensor y_2-f(y_1)\tensor [y_2,x]
\\
&\phantom{-}+[x_1,y]\tensor f(x_2)+x_1\tensor [f(x_2),y])-[y_1,x]\tensor f(y_2)-y_1\tensor [f(y_2),x])\\
&=[\wh\delta(x),\Delta(y)]-[\wh\delta(y),\Delta(x)],
\\
\intertext{while}
\wh\delta([&x,f(y)]_{\lie t})=\wh\delta(f([x,y]_{\lie g}))
=(f\tensor f)\delta([x,y]_{\lie g})=(f\tensor f)([\delta(x),\Delta(y)]-[\delta(y),\Delta(x)])\\
&=f([x_1,y])\tensor f(x_2)+f(x_1)\tensor f([x_2,y])-
f([y_1,x])\tensor f(y_2)-f(y_1)\tensor f([y_2,x]) \\
&=[x_1,f(y)]\tensor f(x_2)+f(x_1)\tensor [x_2,f(y)]
-[f(y_1),x]\tensor f(y_2)-f(y_1)\tensor [f(y_2),x]\\
&=[\wh\delta(x),\Delta(f(y))]-[f(x_1),f(y)]\tensor x_2-
x_1\tensor [f(x_2),f(y)]-[\wh\delta(f(y)),\Delta(x)]\\
&=[\wh\delta(x),\Delta(f(y))]-[\wh\delta(f(y)),\Delta(x)].
\end{align*}
Finally, $\widehat \delta([f(x),f(y)]_{\lie t})=0=[\widehat\delta(f(x)),\Delta(f(y))]-[\widehat\delta(f(y)),\Delta(f(x))]$ for all~$x,y\in\lie g$.
Thus, $\wh\delta$ satisfies 
\ref{cnd.3}. The condition~\ref{cnd.1} is evident. To verify~\ref{cnd.2}, note that for all $x\in\lie g$
\begin{align*}
(\wh\delta\tensor\id_{\lie t})\wh\delta(x)&=
(\wh\delta\tensor\id_{\lie t})(f(x_1)\tensor x_2+x_1\tensor f(x_2))\\
&=(f\tensor f\tensor\id_{\lie g}+f\tensor\id_{\lie g}\tensor f+\id_{\lie g}\tensor f\tensor f)(\delta(x_1)\tensor x_2)\\
&=(f\tensor f\tensor\id_{\lie g}+f\tensor\id_{\lie g}\tensor f+\id_{\lie g}\tensor f\tensor f)\circ(\delta\tensor\id_{\lie g})\circ \delta(x),
\\
\intertext{while}
(\wh\delta\tensor\id_{\lie t})\wh\delta(f(x))
&=(\wh\delta\tensor\id_{\lie t})(f(x_1)\tensor f(x_2))
=(f\tensor f\tensor f)\circ (\delta\tensor\id_{\lie g})\circ\delta(x).
\end{align*}
It is now immediate that~$\wh\delta$ 
satisfies~\ref{cnd.2}.
\end{proof}
Given $r\in \lie g\tensor \lie g$, denote 
$\wh r:=(f\tensor\id_{\lie g}+\id_{\lie g}\tensor f)(r)$
regarded as an element of~$\lie t\tensor\lie t$.
We will now construct an infinite family of non-equivalent r-matrices for~$\lie t$.
\begin{proposition}\label{prop:takiff}
Let~$\lie g$ be a quasi-triangular Lie bialgebra with an r-matrix $r$.
Then 
\begin{enumalph}
    \item\label{prop:takiff.a}
    $\lie t$ is a quasi-triangular Lie bialgebra with an r-matrix $\wh{r}$ and~$\wh{\delta_r}=\delta_{\wh r}$;
    \item\label{prop:takiff.b} if~$r+\tau(r)=0$ then for any~$\Omega\in (V\tensor V)^{\lie g}$,
    $\{\wh{r}+\lambda \Omega\}_{\lambda\in\kk}$ is 
    a family of classical r-matrices for the same Lie cobracket 
    $\delta=\delta_{\wh{r}}$ on~$\lie t$.
\end{enumalph}
\end{proposition}
\begin{proof}
We need the following
\begin{lemma}\label{lem:main cas Takiff}
If~$\Omega\in \lie g\tensor\lie g$ is $\lie g$-invariant
then~$\wh\Omega$
is $\lie t$-invariant. 
\end{lemma}
\begin{proof}
Since~$f\tensor\id_{\lie g}$ and~$\id_{\lie g}\tensor f$ are homomorphisms of~$\lie g$-modules
$\lie g\tensor\lie g\to V\tensor \lie g$ and~$\lie g\tensor \lie g\to \lie g\tensor V$, it follows that~$\wh{\Omega}$ is $\lie g$-invariant in~$\lie t\tensor\lie t$. Let~$v\in V$. Then $v=f(x)$ for some~$x\in\lie g$. Write $\Omega=\Omega_1\tensor \Omega_2$ in Sweedler-like notation. Then 
\begin{align*}
(\ad v&\tensor \id_{\lie g}+\id_{\lie g} \tensor \ad v)(\wh{\Omega})=(\ad f(x)\tensor \id_{\lie g}+\id_{\lie g}\tensor\ad f(x))(f(\Omega_1)\tensor \Omega_2+\Omega_1\tensor f(\Omega_2))\\
&=[f(x),f(\Omega_1)]\tensor \Omega_2+[f(x),\Omega_1]\tensor f(\Omega_2)
+f(\Omega_1)\tensor [f(x),\Omega_2]+\Omega_1\tensor [f(x),f(\Omega_2)]\\
&=(f\tensor f)([x,\Omega_1]\tensor \Omega_2+\Omega_1\tensor [x,\Omega_2])=(f\tensor f)(\ad x\tensor\id_{\lie g}+\id_{\lie g}\tensor\ad x)(\Omega)=0.\qedhere
\end{align*}
\end{proof}
Since~$r+\tau(r)$ is a $\lie g$-invariant in~$\lie g\tensor\lie g$
by Proposition~\ref{prop:compat cond}\ref{prop:compat cond.b}, we have $\wh{r}+\tau(\wh{r})=(f\tensor \id_{\lie g}+\id_{\lie g}\tensor f)(r+\tau(r))$ which is $\lie t$-invariant by Lemma~\ref{lem:main cas Takiff}. Note that 
$\wh{r}_{i,j}=(f_i+f_j)(r_{i,j})$, $1\le i<j\le 3$ where
$f_k:=\id_{\lie g}^{\tensor(k-1)}
\tensor f\tensor\id_{\lie g}^{\tensor(3-k)}$, $k\in\{1,2,3\}$.
Then
\begin{align*}
\cybe{\wh{r}}
&=[(f_1+f_2)(r_{1,2}),
(f_1+f_3)(r_{1,3})]+[(f_1+f_2)(r_{1,2}),
(f_2+f_3)(r_{2,3})]\\
&\phantom{=}+[(f_1+f_3)(r_{1,3}),
(f_2+f_3)(r_{2,3})]\\
&=[f_1(r_{1,2}),f_1(r_{1,3})]+
[f_2(r_{1,2}),f_2(r_{2,3})]+
[f_3(r_{1,3}),f_3(r_{2,3})]\\
&\phantom{=}
+[f_1(r_{1,2}),
f_3(r_{1,3})]+[f_2(r_{1,2}),
f_1(r_{1,3})]+[f_2(r_{1,2}),
f_3(r_{1,3})]\\
&\phantom{=}+[f_1(r_{1,2}),
f_2(r_{2,3})]+[f_1(r_{1,2}),
f_3(r_{2,3})]+[f_2(r_{1,2}),
f_3(r_{2,3})]\\
&\phantom{=}+[f_1(r_{1,3}),
f_2(r_{2,3})]+[f_1(r_{1,3}),
f_3(r_{2,3})]+[f_3(r_{1,3}),
f_2(r_{2,3})].
\end{align*}
Write $r=r_1\tensor r_2=r'_1\tensor r'_2$ in Sweedler-like notation. Then,
$$
[f_1(r_{1,2}),f_1(r_{1,3})]=[f(r_1),f(r'_1)]\tensor f_2\tensor f'_2=0,
$$
while 
\begin{align*}
[f_1(r_{1,2}),
f_3(r_{1,3})]=[f(r_1),r'_1]\tensor r_2\tensor f(r'_2)=
f([r_1,r'_1])\tensor r_2\tensor f(r'_2)=(f_1\circ f_3)[r_{1,2},r_{1,3}],
\end{align*}
since~$f$ is a homomorphism of~$\lie g$-modules. 
Using similar computations for other indices and noting that~$f_i\circ f_j=f_j\circ f_i$, $i\not=j$, we obtain
\begin{align*}
\cybe{\wh{r}}=(f_1\circ f_3+f_1\circ f_2+f_2\circ f_3)(\cybe{r})=0.
\end{align*}
Then part~\ref{prop:takiff.a} follows from Proposition~\ref{prop:compat cond}. 

To prove part~\ref{prop:takiff.b}, we need the following immediate
\begin{lemma}\label{lem:takiff inv}
If~$\Omega\in V\tensor V$ is $\lie g$-invariant then $\Omega$ is also $\lie t$-invariant. In particular, 
if $\Omega\in \lie g\tensor\lie g$ is $\lie g$-invariant then $(f\tensor f)(\Omega)$ is $\lie t$-invariant.
\end{lemma}
By Lemma~\ref{lem:takiff inv},
$\delta_{\wh{r}}=\delta_{\wh{r}+\lambda\Omega}$ for all~$\lambda\in\kk$. Thus, it remains to prove that~$\cybe{\wh{r}+\lambda\Omega}=0$. Since $\Omega\in (V\tensor V)^{\lie t}$ by Lemma~\ref{lem:takiff inv}, we have 
\begin{align*}
\cybe{\wh{r}+\lambda\Omega}&=
\cyb{\wh{r}}{\wh{r}}+\lambda\cyb{\wh{r}}{\Omega}
+\lambda\cyb{\Omega}{\wh{r}}+\lambda^2\cybe{\Omega}\\
&=\lambda[(\id_{\lie t}\tensor\Delta)(\wh{r}),\Omega_{1,3}+\Omega_{2,3}]
+\lambda[\Omega_{1,2}+\Omega_{1,3},(\Delta\tensor\id_{\lie t})(\wh{r})]\\
&=\lambda[\wh{r}_{1,2}+\wh{r}_{1,3},\Omega_{1,3}]
+\lambda[\Omega_{1,3},\wh{r}_{1,3}+\wh{r}_{2,3}]\\
&=\lambda[\wh{r}_{1,2}-\wh{r}_{2,3},\Omega_{1,3}]
=\lambda\tau_{2,3}([\wh{r}_{1,3}-\wh{r}_{3,2},\Omega_{1,2}]).
\end{align*}
Since~$\tau(r)=-r$, $\tau(\wh{r})=-\wh{r}$ and, therefore, 
\begin{equation*}
\cybe{\wh{r}+\lambda\Omega}=\lambda\tau_{2,3}([(\Delta\tensor\id_{\lie t})(\wh{r}),\Omega_{1,2}]=0.\qedhere
\end{equation*}
\end{proof}
Now, let~$\mathbf c:I_n\to \kk$ be transitive, let~$\boldsymbol{\alpha}:[n]\to \kk^n$, let~$\Omega\in (V\tensor V)^{\lie g}$ and
let~$r$
be a skew-symmetric classical r-matrix for~$\lie g$. By Theorem~\ref{thm:trans classical twists},
$$
\mathbf r(\mathbf c,\boldsymbol{\alpha})=
\sum_{i,j\in [n]} \wh{r}_{i,j+n}+
\sum_{i\in[n]}
\boldsymbol{\alpha}(i) \Omega_{i,i+n}
+\sum_{1\le i<j\le n} \mathbf c(i,j)(\Omega_{j,i+n}-\Omega_{j+n,i})
$$
is a classical r-matrix for~$\lie t^{\oplus n}$. In particular, if
$\tau(\Omega)=\Omega$ then
$$
\mathbf r(\mathbf c,\boldsymbol{\alpha})=\sum_{i,j\in[n]} (\wh{r}_{i,j+n}+\mathbf c(i,j)\Omega_{i,j+n}),
$$
where~$\mathbf c$ is extended to a map $[n]\times[n]\to\kk$ via $\mathbf c(i,i)=\boldsymbol{\alpha}(i)$, $i\in [n]$ and~$\mathbf c(j,i)=-\mathbf c(i,j)$, $(i,j)\in I_n$.

\subsection{Poisson structures for functions on Takiff groups}\label{subs:Poiss Takiff}
Let~$G$ be an algebraic group whose Lie algebra is~$\lie g$ and
consider $\kk[V\rtimes G]\cong \kk[V]\tensor \kk[G]$ 
as an algebra. Furthermore, $\kk[V]$ is generated by~$V^*$ as an algebra. Note that the group $G$ acts on~$\kk[V]$ on the right by algebra automorphisms
via
$$
(\xi\triangleleft g)(v)=\xi(gv),\qquad \xi\in V^*,\,g\in G,\,v\in V,
$$
and acts naturally on~$\kk[G]$ both on the left and on the right.

Fix a basis~$B$ of~$V$
and a dual basis~$\{\xi^b\}_{b\in B}$ of~$V^*$. Then~$V^*\tensor V$ is naturally a 
coalgebra via $\Delta(\xi\tensor v)=\xi\tensor \sum_{b\in B} b\tensor \xi^b\tensor v$, $v\in V$, $\xi\in V^*$, which, clearly,
is independent of the choice of~$B$.
\begin{lemma}\label{lem:hom-coalg}
The assignments $\xi\tensor v\mapsto \rho_{\xi,v}$ where~$\rho_{\xi,v}(g)=\xi(gv)$, $\xi\in V^*$, $v\in V$, $g\in G$, define a homomorphism
of coalgebras $\rho:V^*\tensor V\to \kk[G]$. In particular,
$\Delta(\rho_{\xi,v})=\sum_{b\in B} \rho_{\xi,b}\tensor \rho_{\xi^b,v}$, $\xi\in V^*$,
$v\in V$.
\end{lemma}
\begin{proof}
Since~$u=\sum_{b\in B}\xi^b(u)b$ for all~$u\in V$,
we have, for all $v\in V$, $\xi\in V^*$ and~$g,g'\in G$,
\begin{align*}
(((\rho\tensor\rho)\circ\Delta)(&\xi\tensor v))(g\tensor g')=
\Big(\sum_{b\in B} \rho_{\xi,b}\tensor \rho_{\xi^b,v}\Big)(g\tensor g')\\
&=\sum_{b\in B} \xi(g b)\xi^b(g'v)=\xi(gg'v)=\rho_{\xi,v}(gg')
=\Delta(\rho_{\xi,v})(g\tensor g').\qedhere
\end{align*}
\end{proof}
\begin{lemma}
The natural left and right actions of~$V\rtimes G$
on~$\kk[V]\tensor \kk[G]$ are defined by
\begin{alignat*}{2}
&v\triangleright \xi=\xi+\rho_{\xi,v},&\qquad &
\xi\triangleleft v=\xi(v)+\xi,\end{alignat*}
for all~$v\in V$, $\xi\in V^*$, together
with the trivial left $G$-action on~$\kk[V]$,
natural right $G$-action on~$\kk[V]$,
natural left and right actions of~$G$ on~$\kk[G]$
and trivial left and right actions of~$V$ on~$\kk[G]$.
\end{lemma}
\begin{proof}
Let~$\xi\in V^*$. We have, for all $v,v',v''\in V$
and~$g,g',g''\in G$
\begin{align*}
((v,g)&\lact ((v',g')\lact \xi))(v'',g'')=
(v,g)\lact (\xi+\rho_{\xi,v'})(v'',g'')\\&
=\xi(v'')+\rho_{\xi,v}(g'')+(g\lact \rho_{\xi,v'})(g'')
\\
&=\xi(v'')+\xi(g''v)+\rho_{\xi,v'}(g''g)=\xi(v'')+\xi(g''v)+\xi(g''gv')\\
&=\xi(v'')+\rho_{\xi,v}(g'')+\rho_{\xi,gv'}(g'')
=((v+gv',gg')\lact \xi)(v'',g'').
\end{align*}
Similarly,
\begin{align*}
((\xi\triangleleft (v,g))&\triangleleft (v',g'))(v'',g'')
=((\xi(v)+\xi\triangleleft g)\triangleleft (v',g'))(v'',g'')
\\
&=\xi(v)+(\xi\triangleleft g)(v')+((\xi\triangleleft g)\triangleleft g')(v'')=
\xi(v)+\xi(gv')+\xi(gg'v''),
\end{align*}
while~$
(\xi\triangleleft (v+gv',gg'))(v'',g'')
=\xi(v+gv')+\xi(gg'v'')$.
All remaining cases are obvious.
\end{proof}
\begin{corollary}
The natural right and left action of
$\lie t=V\rtimes \lie g$ on~$\kk[V]\tensor\kk[G]$ is determined by $\kk[G]$-linear derivations 
$\partial_v,{}_v\partial$ of~$\kk[V]\tensor\kk[G]$
defined by
$$
\partial_v(\xi)=\xi(v),\qquad {}_v\partial(\xi)=\rho_{\xi,v},\qquad 
v\in V,\, \xi\in V^*
$$
together with the zero left action of~$\lie g$ on~$\kk[V]$, the natural right action of~$\lie g$
on~$\kk[V]$, zero left and right actions of~$V$ on~$\kk[G]$ and natural left and right actions of~$\lie g$
on~$\kk[G]$ given by~\eqref{eq:g act k[G]}.
\end{corollary}

Now, suppose that~$\lie g$ is quasi-triangular with an r-matrix $r=r_1\tensor r_2$ in Sweedler-like notation and let~$f:\lie g\to V$ be a surjective homomorphism of~$\lie g$-modules. The Poisson 
bracket on~$\kk[V\rtimes G]$ induced by~$\wh r\in\lie t\tensor\lie t$ is  given by
\begin{align*}
&\{\varphi,\varphi'\}=0,\\
&\{\xi,\varphi\}=\xi(f(r_1))(\varphi\triangleleft r_2)-
\rho_{\xi,f(r_1)}\cdot(r_2\triangleright \varphi),\\
&\{\xi,\xi'\}=\xi(f(r_1))(\xi'\triangleleft r_2)+\xi'(f(r_1))(\xi\triangleleft r_2),
\end{align*}
for all~$\xi,\xi'\in V^*$, $\varphi,\varphi'\in\kk[G]$. Thus, $\kk[G]$
is a Poisson-commutative Poisson ideal in~$\kk[V\rtimes G]$ and~$\kk[V]$
is its Poisson subalgebra. Note that the Poisson
bracket~$\{\xi,\xi'\}$, $\xi,\xi'\in V^*$ is just the Lie bracket
on~$V^*$ induced by the Lie coalgebra structure
on~$\lie t$.

Let~$\Omega=\Omega_1\tensor\Omega_2\in V\tensor V$ in Sweedler-like notation be 
$\lie g$-invariant. 
Then $\Omega\bowtie (\varphi\tensor\varphi')=0=
\Omega\bowtie (\xi\tensor\varphi)$,
$\varphi,\varphi'\in\kk[G]$, $\xi\in V^*$ while 
for all $\xi,\xi'\in V^*$ 
$$
\Omega\bowtie(\xi\tensor \xi')=(\xi\tensor \xi')(\Omega)-\rho_{\xi,\Omega_1}\tensor \rho_{\xi',\Omega_2}.
$$
Assume for simplicity that~$\tau(\Omega)=\Omega$ and let~$\mathbf c:I_n\to \kk$ be transitive. Extend~$\mathbf c$
to a map~$[n]\times[n]\to \kk$ by $\mathbf c(j,i)=-\mathbf c(i,j)$, $i,j\in[n]$.
Using the notation from~\S\ref{subs:Poisson} and~\eqref{eq:c Poi bracket} we obtain
\begin{align*}
\{ \varphi^{(k)},\varphi'{}^{(l)}\}_{\mathbf c}&=0,\\
\{ \xi^{(k)},\varphi^{(l)}\}_{\mathbf c}&=\xi(f(r_1))(\varphi\triangleleft r_2)^{(l)}-
\rho_{\xi,f(r_1)}{}^{(k)}\cdot(r_2\triangleright \varphi)^{(l)},\\
\{\xi^{(k)},\xi'{}^{(l)}\}_{\mathbf c}&=\xi(f(r_1))(\xi'\triangleleft r_2)^{(l)}+\xi'(f(r_1))(\xi\triangleleft r_2)^{(k)}\\
&\qquad\qquad+\mathbf c(k,l)
(\xi(\Omega_1)
\xi'(\Omega_2)-\rho_{\xi,\Omega_1}{}^{(k)}\cdot \rho_{\xi',\Omega_2}{}^{(l)}),
\end{align*}
for all~$k,l\in[n]$, $\xi,\xi'\in V^*$ and~$\varphi,\varphi'\in \kk[G]$. Note that $\kk[G]^{\tensor n}$
is still a Poisson-commutative Poisson ideal of~$\kk[V\rtimes G]^{\tensor n}$
but~$\kk[V]^{\tensor n}$ is no longer a Poisson subalgebra.

We will now provide more explicit calculations for~$G=GL_m$, $\lie g=\lie{gl}_m$ and an isomorphism~$f:\lie g_{\ad}\to V$ of $\lie g$-modules. Denote
by $x_{i,j}\in\kk[G]$ the usual matrix element function (cf.~\S\ref{subs:Poisson ex}) and define
$\bar x_{i,j}\in\kk[G]$ by $\bar x_{i,j}(g)=x_{i,j}(g^{-1})$, $g\in G$. Then~$\kk[G]$ identifies with the localization of~$\kk[x_{i,j}\,:\, i,j\in[m]]$
by~$\det$.
It is easy to check that 
$\sum_{k\in[m]} x_{i,k}\bar x_{k,j}=\delta_{i,j}=\sum_{k\in[m]}\bar x_{i,k}x_{k,j}$ and~$\Delta(\bar x_{i,j})=
\sum_{k\in[m]} \bar x_{k,j}\tensor \bar x_{i,k}$, $i,j\in[m]$. 

Let~$\{y_{i,j}\}_{i,j\in[m]}$
be the basis of~$V^*$ dual to~$\{ f(E_{i,j})\}_{i,j\in [m]}$. 
We claim that
$$
\rho_{y_{i,j},f(E_{a,b})}=x_{i,a}\bar x_{b,j},\qquad i,j,a,b\in[m].
$$
Indeed, it is easy to check that the assignment $y_{i,j}\tensor f(E_{a,b})\mapsto 
x_{i,a}\bar x_{b,j}$, $i,j,a,b\in[m]$ define a homomorphism of coalgebras $V^*\tensor V\to \kk[G]$. 
Since~$G$ is generated by invertible diagonal matrices and by the $\exp(t E_{r,s})$, $t\in\kk$, $r\not=s$, it is enough to verify that~$y_{i,j}(f(g E_{a,b}g^{-1}))=
x_{i,a}(g)\bar x_{b,j}(g)=x_{i,a}(g)x_{b,j}(g^{-1})$ for any~$g\in G$ in one of these forms. If~$g=\sum_{r\in[m]}\lambda_r E_{r,r}$ with~$\lambda_r\in\kk^\times$, 
$r\in [m]$, we have
$$
y_{i,j}(f(g E_{a,b} g^{-1}))=\sum_{r,s}y_{i,j}( \lambda_r\lambda_s^{-1} E_{r,r}E_{a,b}E_{s,s})=\lambda_a\lambda_b^{-1}\delta_{i,a}\delta_{j,b}
=x_{i,a}(g)x_{b,j}(g^{-1}),
$$
while for~$g=\exp(t E_{r,s})=1+t E_{r,s}$, $r\not=s\in[m]$, $t\in\kk$,
\begin{align*}
y_{i,j}(f(gE_{a,b}g^{-1}))&=
y_{i,j}(f((1+t E_{r,s})E_{a,b}(1-t E_{r,s})))\\
&=\delta_{i,a}\delta_{j,b}+t\delta_{s,a}\delta_{i,r}\delta_{j,b}-
t\delta_{i,a}\delta_{j,s}\delta_{r,b}-t^2
\delta_{s,a}\delta_{r,b}\delta_{i,r}\delta_{j,s}\\
&=(\delta_{i,a}+t \delta_{i,r}\delta_{s,a})(\delta_{j,b}-t \delta_{j,s}\delta_{r,b})\\
&=x_{i,a}(1+t E_{r,s})x_{b,j}(1-t E_{r,s})=x_{i,a}(g) x_{b,j}(g^{-1}).
\end{align*}
Summarizing, we have for all~$i,j,a,b\in[m]$.
$$
f(E_{a,b})\triangleright y_{i,j}=x_{i,a}\bar x_{b,j},\quad 
y_{i,j}\triangleleft f(E_{a,b})=\delta_{i,a}\delta_{j,b},
\quad 
E_{a,b}\triangleright y_{i,j}=0,\,\, y_{i,j}\triangleleft E_{a,b}=\delta_{i,a}y_{b,j}-\delta_{b,j}y_{i,a}.
$$

Let~$\Omega=\sum_{a,b\in[m]} E_{a,b}\tensor E_{b,a}$ be the canonical $\lie g$-invariant in~$\lie g\tensor\lie g$. Then
\begin{align*}
(f\tensor f)(\Omega)&\bowtie(y_{i,j}\tensor y_{i',j'})=
\delta_{i,j'}\delta_{i',j}(1\tensor1)-
\sum_{a,b\in[m]} x_{i,a}\bar x_{b,j}\tensor x_{i',b}\bar x_{a,j'}.
\end{align*}
Note that applying the multiplication map to this tensor yields~$0$, as expected.

To write explicit Poisson brackets, we now need to choose a skew-symmetric r-matrix. For instance, by~\cite{GG}*{Example~5.2},
\begin{equation}\label{eq:skew symm r matr GLm}
r=\sum_{t\in[m]}d_t (E_{1,t}\tensor E_{t,m}-E_{t,m}\tensor E_{1,t}),
\end{equation}
where~$d_t=2-\delta_{t,1}-\delta_{t,m}$, $t\in [m]$, is a skew-symmetric
classical~r-matrix for~$\lie{gl}_n$; in particular, the well-known unique skew-symmetric
solution of~CYBE for~$\lie{gl}_2$ belongs to this family. 
We have, for all~$i,j,i',j'\in[m]$ 
\begin{align*}
\wh r\bowtie &(y_{i,j}\tensor x_{i',j'})= 
\sum_{t\in[m]} d_t \Big(y_{i,j}\triangleleft f(E_{1,t})\tensor x_{i',j'}\triangleleft 
E_{t,m}-y_{i,j}\triangleleft f(E_{t,m})\tensor x_{i',j'}\triangleleft 
E_{1,t}\\
&\phantom{(y_{i,j}\tensor x_{i',j'})=\sum_{t\in[m]}d_t} -f(E_{1,t})\triangleright y_{i,j}\tensor 
E_{t,m}\triangleright x_{i',j'}+f(E_{t,m})\triangleright y_{i,j}\tensor 
E_{1,t}\triangleright x_{i',j'}\Big)\\
&=d_j \delta_{1,i}
\delta_{i',j}\tensor x_{m,j'}-
d_i \delta_{m,j}\delta_{1,i'} \tensor x_{i,j'}+d_{j'} x_{i,j'}\bar x_{m,j}
\tensor 
x_{i',1}-\delta_{j',m}
x_{i,1}\sum_{t\in[m]}d_t \bar x_{t,j}
\tensor 
x_{i',t},\\
\wh r\bowtie &(y_{i,j}\tensor y_{i',j'})= 
\sum_{t\in[m]} d_t \Big(y_{i,j}\triangleleft f(E_{1,t})\tensor y_{i',j'}\triangleleft 
E_{t,m}-y_{i,j}\triangleleft f(E_{t,m})\tensor y_{i',j'}\triangleleft 
E_{1,t}\\
&\phantom{(y_{i,j}\tensor y_{i',j'})=\sum_{t\in[m]}d_t}+y_{i,j}\triangleleft E_{1,t}\tensor y_{i',j'}\triangleleft 
f(E_{t,m})-y_{i,j}\triangleleft E_{t,m}\tensor y_{i',j'}\triangleleft 
f(E_{1,t})\Big)\\
&= 
d_j \delta_{i,1}\tensor (\delta_{i',j} y_{m,j'}-\delta_{m,j'}y_{i',j})-d_i \delta_{j,m}\tensor (\delta_{1,i'}y_{i,j'}-\delta_{i,j'}y_{i',1})\\
&\phantom{=}+ d_{i'} (\delta_{1,i} y_{i',j}-\delta_{i',j}y_{i,1})\tensor \delta_{j',m}
-d_{j'} (\delta_{i,j'} y_{m,j}-\delta_{j,m}y_{i,j'})\tensor \delta_{i',1},
\end{align*}
while~$\wh r\bowtie (x_{i,j}\tensor x_{i',j'})=0$.
This yields the following Poisson bracket on~$\kk[V\rtimes G]$
\begin{align}
\{x_{i,j},x_{i',j'}\}&=0,\nonumber\\
\{y_{i,j},x_{i',j'}\}&=d_j \delta_{i',j}(\delta_{1,i}
x_{m,j'}-\delta_{j',m}
x_{i,1})+
d_i \delta_{m,j}\delta_{1,i'} x_{i,j'}+d_{j'} x_{i,j'}\bar x_{m,j}
x_{i',1},\nonumber\\
\{y_{i,j},y_{i',j'}\}&=
(\delta_{1,i}+\delta_{i,m}-\delta_{1,j'}-\delta_{j',m})\delta_{1,i'}\delta_{j,m} y_{i,j'}-
(\delta_{1,i'}+\delta_{i',m}-\delta_{1,j}-\delta_{j,m})\delta_{1,i}\delta_{j',m}y_{i',j}\nonumber\\
&\phantom{=}+\delta_{i,j'}d_i (\delta_{j,m} y_{i',1}-\delta_{i',1}y_{m,j})
-\delta_{i',j}d_{i'}(\delta_{j',m}y_{i,1}-\delta_{i,1} y_{m,j'})\label{eq:Lie algebra V*}
\end{align}
for~$i,j,i',j'\in[m]$. The last bracket, 
as already mentioned, is the Lie bracket on~$V^*$ induced by the Lie coalgebra structure on~$\lie t$.
Then for~$n\ge 2$ and
transitive $\mathbf c:I_n\to \kk$ extended 
to a skew-symmetric $\mathbf c:[n]\times[n]\to\kk$
we get the following Poisson bracket 
on~$\kk[V\rtimes G]^{\tensor n}$
\begin{align}
\{y_{i,j}^{(k)},y_{i',j'}^{(l)}\}_{\mathbf c}&=\delta_{1,i}\delta_{j',m}(d_{i'}y_{i',j}^{(k)}-d_j y^{(l)}_{i',j})
-\delta_{1,i'}\delta_{j,m}(d_i y_{i,j'}^{(l)}-d_{j'}y_{i,j'}^{(k)})\nonumber\\
&\phantom{=}+\delta_{i,j'}d_i(\delta_{j,m}y_{i',1}^{(l)}-
\delta_{1,i'}y_{m,j}^{(k)})
-\delta_{i',j}d_{i'}(\delta_{j',m} y^{(k)}_{i,1}-\delta_{1,i}y^{(l)}_{m,j'})\nonumber\\
&\phantom{=}+\mathbf c(k,l)\Big(\delta_{i,j'}\delta_{i',j}-\sum_{a,b\in[m]}
x_{i,a}^{(k)}\bar x_{b,j}^{(k)}x_{i',b}^{(l)}\bar x_{a,j'}^{(l)}\Big)
,\nonumber\\
\{y_{i,j}^{(k)},x_{i',j'}^{(l)}\}_{\mathbf c}&=
d_j \delta_{1,i}
\delta_{i',j}x^{(l)}_{m,j'}-
d_i \delta_{m,j}\delta_{1,i'}x^{(l)}_{i,j'}+d_{j'} x^{(k)}_{i,j'}\bar x_{m,j}^{(k)}
x^{(l)}_{i',1}-\delta_{j',m}
x_{i,1}^{(k)}\sum_{t\in[m]}d_t x^{(l)}_{i',t}\bar x^{(k)}_{t,j},\nonumber\\
\{x_{i,j}^{(k)},x_{i',j'}^{(l)}\}_{\mathbf c}&=0,\mskip200mu 
i,j,i',j'\in[m],\, k,l\in[n].\label{eq:Poiss Takiff tens}
\end{align}

\subsection{Small quantum algebra at roots of unity}
\label{subs:small quantum}
Following~\cite{LN}, we describe here an example of a family of
R-matrices for the small quantum group at roots of unity corresponding 
to~$\lie{sl}_2$. Let~$\kk=\mathbb C$ and let $q\in\kk^\times$ be a primitive root of unity of order~$2\ell>2$. 
The algebra~$H=u_{\lie q}(\lie{sl}_2)$
is generated by $E$, $F$ and~$L^{\pm 1}$  subject to relations
$$
[E,F]=\frac{L^2-L^{-2}}{q-q^{-1}},\quad 
LE L^{-1}=q E,\quad 
LF L^{-1}=q^{-1}F,\quad E^{\ell}=F^{\ell}=0,\quad 
L^{2\ell}=1.
$$
This is a Hopf algebra of dimension~$4\ell^3$ with the comultiplication
$$
\Delta(E)=E\tensor L^2+1\tensor E,\quad 
\Delta(F)=F\tensor 1+L^{-2}\tensor F,
$$
$L^{\pm1}$ being group-like. By~\cite{LN}, we obtain a 
family of R-matrices parametrized by~$\{1,-1\}^2$ for the same comultiplication, namely
\begin{align*}
R^{(\epsilon,1)}&=\frac1{4\ell}\sum_{\substack{0\le k\le \ell-1\\
0\le i,j\le 4\ell-1}}\, q^{\binom{k}2-\frac12 ij}\epsilon^{ij}\,
\frac{(q-q^{-1})^{k}}{[k]_q!}
 L^{i} E^k \tensor  L^{j}F^k,\\
R^{(\epsilon,-1)}&=\frac1{4\ell}\sum_{\substack{0\le k\le \ell-1\\
0\le i,j\le 4\ell-1}}\, q^{-\binom{k}2+\frac12 ij}\epsilon^{ij}\,
\frac{(q^{-1}-q)^{k}}{[k]_q!}
 F^k L^{i} \tensor  E^k L^{j}
\end{align*}
where~$\epsilon\in\{1,-1\}$ and~$[k]_q!=\prod_{1\le t\le k}(q^t-q^{-t})/(q-q^{-1})$; here~$R^{(\epsilon,-\epsilon')}=\tau(R^{(\epsilon,\epsilon')})^{-1}$, $(\epsilon,\epsilon')\in \{1,-1\}$.

We will now use this example to show that Drinfeld twists produced by our construction
are not, generally speaking, equivalent. Let~$R=R^{(1,1)}$ and consider two Drinfeld twists corresponding to the identity permutation and its counterpart corresponding to~$(1,2)$, that is
$J=R_{2,3}$, $J'=R_{3,2}^{-1}$.
We claim that~$J'$ cannot be obtained from~$J$ the same way the classical
twist corresponding to the transposition~$(1,2)$ is obtained from the classical twist corresponding 
to the identity permutation, that is, by conjugating with the permutation of factors (cf. Remark~\ref{rem:isom permutations}).
For, it is sufficient to show that
$$
(\tau_{1,2}\tensor \tau_{1,2})(J)J'^{-1}=
\tau_{1,2}\tau_{3,4}(R_{2,3})R_{3,2}
=R_{1,4}R_{3,2}
$$
does not commute with the action of~$H^{\tensor 2}$
given by the standard comultiplication
$\Delta_{H^{\tensor 2}}$
on the 4th tensor power of some $H$-module~$V$.  
Indeed, suppose that~$\ell>2$ and let~$V$ be the 3-dimensional $H$-module with the 
basis~$v_i$, $0\le i\le 2$ and the action given by 
$F v_i=v_{i+1}$, $i\in\{0,1\}$, $Fv_2=0$, $Ev_0=0$, $Ev_i=(q+q^{-1})v_{i-1}$,
$i\in\{1,2\}$ and~$Lv_i=q^{1-i}v_i$, $0\le i\le 2$. Note that 
$$
\frac1{4\ell}\sum_{0\le i,j\le 4\ell-1} q^{\pm\frac12 ij} L^{i}\tensor L^{j}(v_r\tensor v_s)=\Big(\frac1{4\ell}\sum_{0\le i\le 4\ell-1} q^{i(1-r)}
\sum_{0\le j\le 4\ell-1} q^{(1-s\pm\frac i2)j}\Big) v_r\tensor v_s.
$$
The inner sum is equal to~$0$ unless~$1-s\pm\frac12 i\equiv 0\pmod{2l}$ which
happens if and only if~$i=\mp2(1-s)$. Thus,
$$
\frac1{4\ell}\sum_{0\le i,j\le 4\ell-1} q^{\pm\frac12 ij} L^{i}\tensor L^{j}(v_r\tensor v_s)=q^{\mp 2(1-r)(1-s)}v_r\tensor v_s.
$$
Let~$u=v_2\tensor v_1\tensor v_2^{\tensor2}$. Since~$\Delta_{H^{\tensor 2}}(F)=F\tensor 1^{\tensor 3}+L^{-2}\tensor 1\tensor F\tensor 1$,
$\Delta_{H^{\tensor 2}}(F)(u)=0$.
On the other hand,
\begin{align*}
R_{1,4}R_{3,2}(u)&=R_{1,4}(u+(q^2-q^{-2}) v_2\tensor v_2\tensor v_1\tensor v_2)=q^2 u+(q^4-1) v_2\tensor v_2\tensor v_1\tensor v_2.
\end{align*}
Therefore, $\Delta_{H^{\tensor 2}}(F)
R_{1,4}R_{3,2}(u)=q^2(q^4-1)v_2^{\tensor 4}\not=0$
since we assumed that~$2\ell>4$.

More generally, one obtains families of~$R$ matrices in
extensions of the small quantum group at a root of unity by group like
elements. Such an extension depends on a pair of weight lattices~$\Lambda'\subset\Lambda$, and 
R-matrices are parametrized by certain pairings between subgroups of~$\Lambda/\Lambda'$
with values in~$\kk^\times$
(see~\cite{LN}*{Theorem~A}).

\BibSpec{book}{    +{}  {\PrintPrimary}                {transition}
    +{,} { \textit}                     {title}
    +{.} { }                            {part}
    +{:} { \textit}                     {subtitle}
    +{,} { \PrintEdition}               {edition}
    +{}  { \PrintEditorsB}              {editor}
    +{,} { \PrintTranslatorsC}          {translator}
    +{,} { \PrintContributions}         {contribution}
    +{,} { }                            {series}
    +{,} { \voltext}                    {volume}
    +{,} { }                            {publisher}
    +{,} { }                            {organization}
    +{,} { }                            {address}
    +{,} { \PrintDateB}                 {date}
    +{,} { }                            {status}
    +{}  { \parenthesize}               {language}
    +{}  { \PrintTranslation}           {translation}
    +{;} { \PrintReprint}               {reprint}
    +{.} { }                            {note}
    +{.} {}                             {transition}
    +{.} { \PrintDOI}                    {doi}
}

\renewcommand{\PrintDOI}[1]{DOI \href{http://dx.doi.org/#1}{#1}}
\renewcommand{\eprint}[1]{\href{http://arxiv.org/abs/#1}{arXiv:#1}}
\hbadness=99999


\begin{bibdiv}
\begin{biblist}

\bib{AM}{article}{
   author={Abedin, Raschid},
   author={Maximov, Stepan},
   title={Classification of classical twists of the standard Lie bi\-algebra
   structure on a loop algebra},
   journal={J. Geom. Phys.},
   volume={164},
   date={2021},
   pages={Paper No. 104149},
   issn={0393-0440},
   doi={10.1016/j.geomphys.2021.104149},
}

\bib{AN}{article}{
    author={Abedin, Raschid},
    author={Niu, Wenjun},
    title={Yangian for cotangent Lie algebras and spectral $R$-matrices},
    eprint={2405.19906}
}

\bib{ABGJ}{article}{
   author={Adin, Ron M.},
   author={Berenstein, Arkady},
   author={Greenstein, Jacob},
   author={Li, Jian-Rong},
   author={Marmor, Avichai},
   author={Roichman, Yuval},
   title={Transitive and Gallai colorings of the complete graph},
   journal={European J. Combin.},
   volume={130},
   date={2025},
   pages={Paper No. 104225},
   issn={0195-6698},
   doi={10.1016/j.ejc.2025.104225},
}

\bib{ABGJepr}{article}{
   author={Adin, Ron M.},
   author={Berenstein, Arkady},
   author={Greenstein, Jacob},
   author={Li, Jian-Rong},
   author={Marmor, Avichai},
   author={Roichman, Yuval},
   title={Transitive and Gallai colorings},
   eprint={2309.11203},
   date={2023},
}



\bib{BEER}{article}{
   author={Bartholdi, Laurent},
   author={Enriquez, Benjamin},
   author={Etingof, Pavel},
   author={Rains, Eric},
   title={Groups and Lie algebras corresponding to the Yang-Baxter
   equations},
   journal={J. Algebra},
   volume={305},
   date={2006},
   number={2},
   pages={742--764},
   issn={0021-8693},
   doi={10.1016/j.jalgebra.2005.12.006}
}

\bib{BB}{article}{
   author={Bazlov, Yuri},
   author={Berenstein, Arkady},
   title={Mystic reflection groups},
   journal={SIGMA Symmetry Integrability Geom. Methods Appl.},
   volume={10},
   date={2014},
   pages={Paper 040, 11},
   doi={10.3842/SIGMA.2014.040},
}

\bib{BoW}{article}{
   author={Borodin, Alexei},
   author={Wheeler, Michael},
   title={Colored stochastic vertex models and their spectral theory},
   journal={Ast\'erisque},
   number={437},
   date={2022},
       doi={10.24033/ast.1180},
}

\bib{Chaf}{article}{
   author={Chaffe, Matthew},
   title={Category $\mathcal{O}$ for Takiff Lie algebras},
   journal={Math. Z.},
   volume={304},
   date={2023},
   number={1},
   pages={Paper No. 14, 35},
   doi={10.1007/s00209-023-03262-1},
}

\bib{CG}{article}{
   author={Chari, Vyjayanthi},
   author={Greenstein, Jacob},
   title={A family of Koszul algebras arising from fi\-nite-di\-men\-sion\-al
   representations of simple Lie algebras},
   journal={Adv. Math.},
   volume={220},
   date={2009},
   number={4},
   pages={1193--1221},
   doi={10.1016/j.aim.2008.11.007},
}

\bib{ChKR}{article}{
   author={Chari, Vyjayanthi},
   author={Khare, Apoorva},
   author={Ridenour, Tim},
   title={Faces of polytopes and Koszul algebras},
   journal={J. Pure Appl. Algebra},
   volume={216},
   date={2012},
   number={7},
   pages={1611--1625},
   issn={0022-4049},
   doi={10.1016/j.jpaa.2011.10.014},
}

\bib{CPbook}{book}{
   author={Chari, Vyjayanthi},
   author={Pressley, Andrew},
   title={A guide to quantum groups},
   publisher={Cambridge University Press, Cambridge},
   date={1994},
}

\bib{Drinf-YBE}{article}{
   author={Drinfel\cprime d, V. G.},
   title={Constant quasiclassical solutions of the Yang-Baxter quantum
   equation},
   journal={Dokl. Akad. Nauk SSSR},
   volume={273},
   date={1983},
   number={3},
   pages={531--535},
   issn={0002-3264},
}

\bib{Drinf}{article}{
   author={Drinfel\cprime d, V. G.},
   title={Quasi-Hopf algebras},
   journal={Algebra i Analiz},
   volume={1},
   date={1989},
   number={6},
   pages={114--148},
   issn={0234-0852},
}

\bib{EH}{article}{
   author={Enriquez, Benjamin},
   author={Halbout, Gilles},
   title={Quantization of coboundary Lie bialgebras},
   journal={Ann. of Math. (2)},
   volume={171},
   date={2010},
   number={2},
   pages={1267--1345},
   issn={0003-486X},
   doi={10.4007/annals.2010.171.1267},
}

\bib{EK-I}{article}{
   author={Etingof, Pavel},
   author={Kazhdan, David},
   title={Quantization of Lie bialgebras. I},
   journal={Selecta Math. (N.S.)},
   volume={2},
   date={1996},
   number={1},
   pages={1--41},
   issn={1022-1824},
   doi={10.1007/BF01587938},
}

\bib{EK-II}{article}{
   author={Etingof, Pavel},
   author={Kazhdan, David},
   title={Quantization of Lie bialgebras. II},
   journal={Selecta Math. (N.S.)},
   volume={4},
   date={1998},
   number={2},
   pages={213--231},
   issn={1022-1824},
   doi={10.1007/s000290050030},
}

\bib{EK-III}{article}{
   author={Etingof, Pavel},
   author={Kazhdan, David},
   title={Quantization of Lie bialgebras. II, III},
   journal={Selecta Math. (N.S.)},
   volume={4},
   date={1998},
   number={2},
   pages={233--269},
   issn={1022-1824},
   doi={10.1007/s000290050030},
}


\bib{FZ}{article}{
   author={Foata, Dominique},
   author={Zeilberger, Doron},
   title={Graphical major indices},
   journal={J. Comput. Appl. Math.},
   volume={68},
   date={1996},
   number={1--2},
   pages={79--101},
   issn={0377-0427},
   doi={10.1016/0377-0427(95)00254-5},
}

\bib{FR}{article}{
   author={Fock, V. V.},
   author={Rosly, A. A.},
   title={Flat connections and polyubles},
   journal={Teoret. Mat. Fiz.},
   volume={95},
   date={1993},
   number={2},
   pages={228--238},
   issn={0564-6162},
   doi={10.1007/BF01017138},
}

\bib{GG}{article}{
   author={Gerstenhaber, M.},
   author={Giaquinto, A.},
   title={Boundary solutions of the classical Yang-Baxter equation},
   journal={Lett. Math. Phys.},
   volume={40},
   date={1997},
   number={4},
   pages={337--353},
   issn={0377-9017},
    doi={10.1023/A:1007363911649},
}

\bib{Ge}{article}{
   author={Goodearl, K. R.},
   title={Quantized coordinate rings and related Noetherian algebras},
   conference={
      title={Proceedings of the 35th Symposium on Ring Theory and
      Representation Theory},
      address={Okayama},
      date={2002},
   },
   book={
      publisher={Symp. Ring Theory Represent. Theory Organ. Comm., Okayama},
   },
   date={2003},
   pages={19--45},
 }

\bib{GM}{article}{
   author={Greenstein, Jacob},
   author={Mazorchuk, Volodymyr},
   title={Koszul duality for semidirect products and generalized Takiff
   algebras},
   journal={Algebr. Represent. Theory},
   volume={20},
   date={2017},
   number={3},
   pages={675--694},
   issn={1386-923X},
   doi={10.1007/s10468-016-9660-1},
}

\bib{Halb}{article}{
   author={Halbout, Gilles},
   title={Formality theorem for Lie bialgebras and quantization of twists
   and coboundary $r$-matrices},
   journal={Adv. Math.},
   volume={207},
   date={2006},
   number={2},
   pages={617--633},
   issn={0001-8708},
   doi={10.1016/j.aim.2005.12.006},
}

\bib{HK}{article}{
   author={Hetyei, G\'abor},
   author={Krattenthaler, Christian},
   title={The poset of bipartitions},
   journal={European J. Combin.},
   volume={32},
   date={2011},
   number={no.~8},
   pages={1253--1281},
   doi={10.1016/j.ejc.2011.03.019},
}

\bib{KPSST}{article}{
   author={Khoroshkin, S. M.},
   author={Pop, I. I.},
   author={Samsonov, M. E.},
   author={Stolin, A. A.},
   author={Tolstoy, V. N.},
   title={On some Lie bialgebra structures on polynomial algebras and their
   quantization},
   journal={Comm. Math. Phys.},
   volume={282},
   date={2008},
   number={3},
   pages={625--662},
   issn={0010-3616},
   doi={10.1007/s00220-008-0554-x},
}

\bib{LN}{article}{
   author={Lentner, Simon},
   author={Nett, Daniel},
   title={New R-matrices for small quantum groups},
   journal={Algebr. Represent. Theory},
   volume={18},
   date={2015},
   number={6},
   pages={1649--1673},
   issn={1386-923X},
   doi={10.1007/s10468-015-9555-6},
}

\bib{LuMou}{article}{
   author={Lu, Jiang-Hua},
   author={Mouquin, Victor},
   title={Mixed product Poisson structures associated to Poisson Lie groups
   and Lie bialgebras},
   journal={Int. Math. Res. Not. IMRN},
   date={2017},
   number={19},
   pages={5919--5976},
   issn={1073-7928},
      doi={10.1093/imrn/rnw189},
}

\bib{MS}{article}{
   author={Maillet, J. M.},
   author={Sanchez de Santos, J.},
   title={Drinfeld twists and algebraic Bethe ansatz},
   conference={
      title={L. D. Faddeev's Seminar on Mathematical Physics},
   },
   book={
      series={Amer. Math. Soc. Transl. Ser. 2},
      volume={201},
      publisher={Amer. Math. Soc., Providence, RI},
   },
   date={2000},
   pages={137--178},
   doi={10.1090/trans2/201/10},
}

\bib{Majid}{book}{
   author={Majid, Shahn},
   title={Foundations of quantum group theory},
   publisher={Cambridge University Press, Cambridge},
   date={1995},
   pages={x+607},
   isbn={0-521-46032-8},
   doi={10.1017/CBO9780511613104},
}

\bib{Mou}{article}{
   author={Mouquin, Victor},
   title={Quantization of a Poisson structure on products of principal
   affine spaces},
   journal={J. Noncommut. Geom.},
   volume={14},
   date={2020},
   number={3},
   pages={1049--1074},
   issn={1661-6952},
      doi={10.4171/jncg/386},
}


\bib{Ne}{article}{
   author={Negron, Cris},
   title={Small quantum groups associated to Belavin-Drinfeld triples},
   journal={Trans. Amer. Math. Soc.},
   volume={371},
   date={2019},
   number={8},
   pages={5401--5432},
   issn={0002-9947},
   doi={10.1090/tran/7438},
}

\bib{RSTS}{article}{
   author={Reshetikhin, N. Yu.},
   author={Semenov-Tian-Shansky, M. A.},
   title={Quantum R-matrices and factorization problems},
   journal={J. Geom. Phys.},
   volume={5},
   date={1988},
   number={4},
   pages={533--550 (1989)},
   issn={0393-0440},
   doi={10.1016/0393-0440(88)90018-6},
}

\bib{Pan}{article}{
   author={Panyushev, Dmitri I.},
   title={Semi-direct products of Lie algebras and their invariants},
   journal={Publ. Res. Inst. Math. Sci.},
   volume={43},
   date={2007},
   number={4},
   pages={1199--1257},
   issn={0034-5318},
}

\bib{OEIS}{book}{
editor={Sloane, N. J. A.},
title={The On-Line Encyclopedia of Integer Sequences},
note={Published electronically at \href{https://oeis.org}{https://oeis.org}}
}

\bib{Take}{article}{
   author={Takeuchi, Mitsuhiro},
   title={${\rm Ext}\sb{\rm ad}({\rm Sp}\,R,\mu \sp{A})\simeq \hat {\rm
   Br}(A/k)$},
   journal={J. Algebra},
   volume={67},
   date={1980},
   number={2},
   pages={436--475},
   issn={0021-8693},
   doi={10.1016/0021-8693(80)90170-2},
}



\bib{Tak}{article}{
   author={Takiff, S. J.},
   title={Rings of invariant polynomials for a class of Lie algebras},
   journal={Trans. Amer. Math. Soc.},
   volume={160},
   date={1971},
   pages={249--262},
   issn={0002-9947},
    doi={10.2307/1995803},
}



\bib{Ter}{article}{
   author={Terras, V.},
   title={Drinfel\cprime d{} twists and functional Bethe ansatz},
   journal={Lett. Math. Phys.},
   volume={48},
   date={1999},
   number={3},
   pages={263--276},
   issn={0377-9017},
   doi={10.1023/A:1007695001683},
}

\bib{Tor}{article}{
   author={Torrielli, Alessandro},
   title={On factorising twists in $AdS_3$ and $AdS_2$},
   journal={J. Geom. Phys.},
   volume={183},
   date={2023},
   pages={Paper No. 104690, 19},
   issn={0393-0440},
   doi={10.1016/j.geomphys.2022.104690},
}

\bib{Wag}{article}{
   author={Wagner, Carl G.},
   title={Enumeration of generalized weak orders},
   journal={Arch. Math. (Basel)},
   volume={39},
   date={1982},
   number={2},
   pages={147--152},
   issn={0003-889X},
   doi={10.1007/BF01899195},
}
	
\end{biblist}

\end{bibdiv}
\end{document}